\documentclass[reqno,12pt,a4paper]{amsart}
\usepackage{amscd,amsmath,amssymb,verbatim,color,mathtools,bbm}
\usepackage{nicefrac}
\usepackage{todonotes}
\usepackage{hyperref}
\usepackage{url}
\usepackage{enumerate,enumitem}
\usepackage{appendix}
\usepackage{epsfig,subfigure}
\usepackage{graphicx}

\theoremstyle{plain}

\newtheorem{theorem}{Theorem}[section]
\newtheorem{proposition}[theorem]{Proposition}
\newtheorem{lemma}[theorem]{Lemma}
\newtheorem{corollary}[theorem]{Corollary}

\theoremstyle{remark}
\newtheorem{remark}[theorem]{Remark}
\newtheorem{example}[theorem]{Example}

\theoremstyle{definition}
\newtheorem{definition}[theorem]{Definition}

\usepackage{geometry}
\usepackage[scrtime]{prelim2e}

\newcommand{\linspan}{\mathop{\rm span}\nolimits}
\newcommand{\Ran}{\mathop{\rm Ran}\nolimits}

\newcommand{\supp}{\mathop{\rm supp}\nolimits}

\newcommand{\rest}{\left.\kern-2\nulldelimiterspace\right|_}
\newcommand{\norm}[2]{\left|#1\right|_{#2}}

\newcommand{\ex}{\mathrm{e}}
\newcommand{\p}{\partial}

\newcommand{\e}{\varepsilon}
\newcommand{\ed}{\mathrm d}

\newcommand*{\Bigcdot}{\raisebox{-.25ex}{\scalebox{1.25}{$\cdot$}}}

\newcommand{\N}{{\mathbb N}}
\newcommand{\R}{{\mathbb R}}

\newcommand{\D}{{\mathrm D}}

\newcommand{\CCC}{{\mathbf C}}
\newcommand{\DDD}{{\mathbf D}}
\newcommand{\RRR}{{\mathbf R}}
\newcommand{\HHH}{{\mathbf H}}

\newcommand{\LLL}{{\mathbf L}}
\newcommand{\GGG}{{\mathbf G}}
\newcommand{\XXX}{{\mathbf X}}

\newcommand{\nnn}{\mathbf n}
\newcommand{\ttt}{\mathbf t}
\newcommand{\ppp}{\mathbf p}
\newcommand{\eee}{\mathbf e}

\newcommand{\BB}{{\mathcal B}}
\newcommand{\CC}{{\mathcal C}}
\newcommand{\DD}{{\mathcal D}}

\newcommand{\FF}{{\mathcal F}}
\newcommand{\GG}{{\mathcal G}}

\newcommand{\LL}{{\mathcal L}}
\newcommand{\MM}{{\mathcal M}}
\newcommand{\NN}{{\mathcal N}}
\newcommand{\OO}{{\mathcal O}}
\newcommand{\PP}{{\mathcal P}}
\newcommand{\RR}{{\mathcal R}}

\newcommand{\TT}{{\mathcal T}}

\newcommand{\WW}{{\mathcal W}}

\newcommand{\ZZ}{{\mathcal Z}}

\newcommand{\Ma}{{\mathbf M}}
\newcommand{\St}{{\mathbf S}}

\begin{document}
\title{Stabilization to trajectories for parabolic equations}
\author{Duy Phan}\author{S\'ergio S.~Rodrigues}
\address{Johann Radon Institute for Computational and Applied Mathematics,
\"OAW,\newline\indent Altenbergerstra{\normalfont\ss}e 69, A-4040 Linz.\newline \indent
e-mails: {\tt duy.phan-duc@oeaw.ac.at,sergio.rodrigues@oeaw.ac.at}}
\thanks{The authors acknowledge support from the Austrian Science Fund~(FWF): P~26034-N25.}

\begin{abstract}
The feedback exponential stabilization to trajectories for semilinear parabolic equations in a given bounded domain is addressed. The controls take values in a
finite-dimensional space and are supported in a small region. Both internal and boundary controls are considered.

\



\end{abstract}

\maketitle

{\tiny
\tableofcontents
}

\pagestyle{myheadings} \thispagestyle{plain} \markboth{\sc D.Phan and S. S.
Rodrigues}{\sc Stabilization to trajectories for parabolic equations}

\section{Introduction.}
We consider controlled parabolic equations, for time~$t\ge0$, in a smooth domain $\Omega \in \R^d$ located locally on one side of its boundary~$\Gamma=\p\Omega$, with~$d$ a positive integer, 
either of the form
\begin{align}
\p_t y - \nu \Delta y + f(y,\nabla y) +\sum_{i=1}^M u_i\Phi_i&= 0;\quad
 &&y\rest\Gamma = g;\label{sys-y-cont_i}\\
\intertext{or of the form }
\p_t y - \nu \Delta y + f(y,\nabla y)  &= 0;\quad
 &&y\rest\Gamma = g +\sum_{i=1}^M u_i\Psi_i.\label{sys-y-cont_b}
\end{align}

In the variables~$(t,x,\bar x)\in(0,+\infty)\times\Omega\times\Gamma$, the unknown in the equation is the function $y=y(t,x)\in\R$.
The diffusion coefficient~$\nu > 0$ is a positive constant; the functions $g=g(t,\bar x)\in\R$ and
$f\colon\R\times\R^d\to\R$ are fixed.

In system~\eqref{sys-y-cont_i} the functions $\Phi_i=\Phi_i(x)$ are given and will play the role of actuators, while
in system~\eqref{sys-y-cont_b} that role will be played by the given functions $\Psi_i=\Psi_i(\bar x)$.
Finally, $M$~is a positive integer and, in either system, $u=u(t)\in\R^M$ is a (control) vector function at our disposal.

The problem we address here is the {\em local} exponential stabilization to trajectories for systems~\eqref{sys-y-cont_i} and~\eqref{sys-y-cont_b}. That is,
given a positive constant $\lambda>0$ and a solution $\hat{y}(t)=\hat{y}(t,\Bigcdot)$ of the (uncontrolled) system with~$u = 0$, we want to find a control function~$u$
such that the solution $y(t)\coloneqq y(t,\Bigcdot)$ of the system, supplemented with the initial condition
\[
y(0)\coloneqq y(0,x) = y_0(x), 
\]
is defined on $\left[0, + \infty \right)$ and approaches $\hat{y}(t)$ exponentially with rate~$\frac{\lambda}{2}$, provided $y(0) - \hat{y} (0)$ is {\em small enough}.
In other words, for a suitable Banach space~$X$ and positive constants~$C$ and $\epsilon$, we want to have that
\begin{equation}\label{goal}
\left| y(t) - \hat{y} (t)  \right|^2_{X}
\le C \mathrm{e}^{-\lambda t}  \left| y(0) - \hat{y} (0)  \right|^2_{X},\quad\mbox{provided}\quad\norm{y(0) - \hat{y} (0)}{X}<\epsilon
\end{equation}
with~$\epsilon$ {\em small enough}.
Notice that, the constants $C$ and $\epsilon$ may depend on $\lambda$, but neither on $\hat{y} (0)$ nor on ${y}(0)$.

We are particularly interested in actuators which are supported in a small domain: either $\supp\Phi_i\subset \omega\subseteq\Omega$ or
$\supp\Psi_i\subset \OO\subseteq\Gamma$, where $\omega$ and $\OO$ are given open subsets of~$\Omega$ and~$\Gamma$, respectively.

By following the arguments in~\cite{BarRodShi11,KroRod15,KroRod-ecc15,BreKunRod-pp15} we shall conclude that the answer is affirmative, for the case of
system~\eqref{sys-y-cont_i} under suitable conditions on the family of internal actuators~$\CC_\omega=\{\Phi_i\mid i\in\{1,2,\dots,M\}\}$. Moreover the stabilizing
control can be taken in feedback form~$u(t)=K_t(y(t)-\hat y(t))$.

For the case of
system~\eqref{sys-y-cont_b} the answer is less straightforward from the available results in the Literature, however we shall show, by combining some of the arguments
in~\cite{Badra09-cocv,Rod14-na,Rod15-cocv},  that the answer is again affirmative provided suitable conditions are satisfied by
the family of boundary actuators~$\CC_{\Gamma_{\rm c}}=\{\Psi_i\mid i\in\{1,2,\dots,M\}\}$. In this case, the control can be taken in {\em integral} feedback form
$y\rest\Gamma(t)=\upsilon_0+\int_0^tK_\tau(y(\tau)-\hat y(\tau))\,\ed\tau$, where~$\upsilon_0$ may/must be taken in an appropriate space.

Considering finite-dimensional controls is important for applications because usually we have at our disposal only a finite number of actuators which we can {\em tune}.
Suppose we can choose those (either internal or boundary) actuators from a family
\[
\CC^{+\infty}=\{\Theta_i\mid i\in\N_0\}, 
\]
 then we may ask how many of these actuators we need to stabilize the system. That is,  
what is~$M$ such that $\CC^{M}=\{\Theta_i\mid i\in\{1,2,\dots,M\}\}$ allow us to stabilize the system. This question will be answered for suitable {\em complete} families.
In this way we arrive to an estimate on the number of actuators we need to stabilize the system.

We will start by considering the linearization of systems~\eqref{sys-y-cont_i} and~\eqref{sys-y-cont_b} around~$\hat y$ and construct a feedback rule
stabilizing {\em globally} the solution~$v$ of the linearized system to zero. That is,
\begin{equation}\label{goal-L}
\left| v(t)  (t)  \right|^2_{X}
\le C \mathrm{e}^{-\lambda t}  \left| v(0)\right|^2_{X}.
\end{equation}
Another question we will address is how the constant~$C=C(\lambda)$, in~\eqref{goal-L}, does depend on~$\lambda$.
Some estimates will be given and the results of some numerical
simulations will be presented for Riccati based feedback. This is motivated by a sufficient condition for exponential stabilization given in~\cite{BreKunRod-pp15} for
the FitzHugh--Nagumo and Rogers--McCulogh systems. Another motivation is that in general, the value~$\epsilon$ in~\eqref{goal}
will decrease as~$C$ in~\eqref{goal-L} increases,
that is the feedback control
will work, for the nonlinear system, in a bigger neighborhood for a smaller~$C$.

The rest of the paper is organized as follows. In Section~\ref{S:red-stabil} we reduce our problem to the stabilization to zero of the difference $z=y-\hat y$, and write the system
for~$z$ in an appropriate way.  In Section~\ref{S:lin_int} we deal with the internal stabilization to zero of the linearized system for the difference.
In Section~\ref{S:nonlinear_int} we deal with the internal stabilization to zero of the nonlinear (full) system for the difference, under some conditions on the
nonlinearity. In Section~\ref{S:exa_poly},
 as an example, we check the conditions in Section~\ref{S:nonlinear_int}
for some nonlinearities of polynomial type. In Section~\ref{S:bdry} we deal with boundary controls. Finally,
Sections~\ref{S:disc_lin} and~\ref{S:num_exa} are concerned with the discretization of our equations and the presentation of the results of some numerical simulations.

\medskip\noindent
{\em Notation.}
We write~$\mathbb R$ and~$\mathbb N$ for the sets of real numbers and nonnegative
integers, respectively, and we define $\mathbb R_a:= (a,+\infty)$ for all $a\in\mathbb R$, and $\mathbb
N_0:= \mathbb N\setminus\{0\}$. We denote by $\Omega\subset\mathbb R^n$, $n\in\mathbb N_0$, a
bounded domain with a smooth boundary $\Gamma=\partial\Omega$.
Given a function $v\colon (t,x_1,x_2,\dots,x_n)\mapsto v(t,x_1,x_2,\dots,x_n)\in\mathbb R$, defined in
an open subset of $\mathbb R\times\Omega$, its partial
time derivative~$\frac{\partial v}{\partial t}$ will be denoted by
$\partial_tv$.

We use the standard notation for Bochner spaces $L^p(\Omega,X)$ where~$\Omega\subseteq\mathbb R^n$, $n\in\N_0$, $p\in[1,+\infty]$, and~$X$ is a Banach space.
The spaces~$L^p(\Omega)^m=L^p(\Omega,\R^m)$ will be denoted
by simply~$L^p$ whenever there is no ambiguity neither concerning the domain~$\Omega$ nor the superscript~$m\in\N_0$.

Given an open interval $I\subseteq\mathbb R$, and Banach spaces~$X$ and ~$Y$, we write
$W(I,X,Y):= \{f\in L^2(I,X)\mid \partial_tf\in L^2(I,Y)\}$,
where the derivative $\partial_tf$ is taken in the sense of
distributions. This space is endowed with the natural norm
$|f|_{W(I,X,Y)}:= \bigl(|f|_{L^2(I,X)}^2+|\partial_tf|_{L^2(I,Y)}^2\bigr)^{1/2}$.
The space of continuous linear mappings from $X$ into $Y$ will be denoted by $\mathcal L(X, Y)$. In  case $X=Y$ we
write $\mathcal L(X):=\mathcal L(X, X)$ instead. If the inclusion
$X\subseteq Y$ is continuous, we write $X\xhookrightarrow{} Y$; we write
$X\xhookrightarrow{\rm d} Y$, respectively $X\xhookrightarrow{\rm c} Y$,
if the inclusion is also dense, respectively compact. The kernel and range of a linear mapping $A\colon Z\to W$, between vector spaces~$Z$ and~$W$, will be denoted
${\rm Ker}\,A:=\{x\in Z\mid Ax=0\}$ and~${\rm Ran}\,A:=\{Ax\mid x\in Z\}$, respectively.

$\overline C_{\left[a_1,\dots,a_k\right]}$ denotes a function of nonnegative variables~$a_j$ that
 increases in each of its arguments, and $C,C_i$, $i=1,2,\dots$, stand for
positive constants.

\section{Reduction to stabilization to zero}\label{S:red-stabil}
Let $\hat y(t)$ solves the uncontrolled system. We we want the solution $y(t)$ to go to the reference trajectory $\hat y(t)$ exponentially, thus
it is natural to consider the dynamics of the difference $y-\hat y$.
\subsection{The case of internal controls}
By direct computations, we find that~$z\coloneqq y-\hat y$ solves
\[
  \p_t z- \nu \Delta z + f(y,\nabla y)-f(\hat y,\nabla \hat y) +\sum_{i=1}^M u_i\Phi_i
  = 0,\qquad z\rest\Gamma = 0.
\]
Writing $(\xi^1,\xi^2)\in\R\times\R^d$ we denote $\p_1f\coloneqq \frac{\p f}{\p_{\xi^1}}$ and~$\p_2f\coloneqq \frac{\p f}{\p_{\xi^2}}$. Formally, we can write
\[
\begin{split}
   f(y,\nabla y)-f(\hat y,\nabla \hat y)&\eqqcolon\begin{bmatrix}\p_1f\rest{(\hat y,\nabla \hat y)} & \p_2f\rest{(\hat y,\nabla \hat y)}\end{bmatrix}
   \begin{bmatrix}z\\ \nabla z\end{bmatrix}+F_{\hat y}(z)\\
   &=\hat a z+  \nabla\cdot(\hat b z)
   -\widehat\NN(z).
\end{split}
 \]
with
\begin{equation}\label{abN_lineariz}
 \hat a\coloneqq\p_1f\rest{(\hat y,\nabla \hat y)}-\nabla\cdot\p_2f\rest{(\hat y,\nabla \hat y)},\quad \hat b\coloneqq\p_2f\rest{(\hat y,\nabla \hat y)},
 \quad\mbox{and}\quad \widehat\NN(z)=-F_{\hat y}(z),
\end{equation}
where $\widehat\NN(\cdot):\R\to\R$ is a nonlinear function if so is~$y\mapsto f(y,\nabla y)$.

\medskip\noindent
{\em Rescaling time.} By technical reasons, in order to use available results in the Literature (for the case~$\nu=1$), it is convenient to rewrite the system as
\begin{equation}\label{sys-z-int}
  \p_\tau \breve z- \Delta\breve  z + \textstyle\frac{1}{\nu} \breve{\hat a}\breve  z + \nabla \cdot \left( \textstyle\frac{1}{\nu} \breve{\hat b}\breve z \right)+\sum_{i=1}^M \textstyle\frac{1}{\nu} \breve u_i\Phi_i
  = \textstyle\frac{1}{\nu} \breve{\widehat  \NN}(\breve z);\quad
 \breve z\rest\Gamma = 0,
\end{equation}
which we can do by rescaling time $t=\frac{\tau}{\nu}$ and setting $\breve p(\tau)\coloneqq p(\frac{\tau}{\nu})$, for a function $p$ defined for~$t\ge0$.

\subsection{The case of boundary controls}
As in the internal case, by direct computations we find that~$\breve z(\tau)=(y - \hat y)(\frac{\tau}{\nu})$ solves
\begin{equation}\label{sys-z-bdry}
  \p_\tau \breve z- \Delta\breve  z + \textstyle\frac{1}{\nu} \breve{\hat a}\breve  z + \nabla \cdot \left( \textstyle\frac{1}{\nu} \breve{\hat b}\breve z \right)
   = \textstyle\frac{1}{\nu}\breve{\widehat \NN}(\breve z);\quad
 \breve z\rest\Gamma = \sum_{i=1}^M \breve u_i\Psi_i.
\end{equation}

\subsection{Stabilization to zero}
We can see that our goal~\eqref{goal} is to find the control~$u$, in either system~\eqref{sys-z-int} or system~\eqref{sys-z-bdry}, such that
\[\left| \breve z(\tau)\right|^2_{X}
\le C \mathrm{e}^{-\bar\lambda\tau}  \left| z(0)  \right|^2_{X},\quad\mbox{provided}\quad\norm{z(0)}{X}<\epsilon.
\]
with~$\bar\lambda=\frac{\lambda}{\nu}$, for suitable positive constants~$C=C_{\bar\lambda}$ and~$\epsilon=\epsilon_{\bar\lambda}$.

We will follow a standard procedure. We will start by proving the {\em global} stabilization result for the linearized system, that is, in the case~$\breve{\widehat\NN}=0$. Then,
the {\em local} stabilization result will follow by a fixed point argument.

\section{Internal stabilization of the linearized system}\label{S:lin_int}

We consider a system in the form~\eqref{sys-z-int}, without the nonlinearity.
In order to study such system we start by denoting the Hilbert space~$H\coloneqq L^2 (\Omega, \R)$ which we will consider as a pivot space, $H'=H$.
We also denote $V \coloneqq  H^1_0 (\Omega, \R)$ and $\mathrm{D}( \Delta) \coloneqq  V \cap H^2 (\Omega, \R)$, which are supposed to be endowed with the scalar products
\[
 (v,w)_V\coloneqq(\nabla v,\nabla w)_{L^2 (\Omega, \R^d)}\quad\mbox{and}\quad(v,w)_{\D(\Delta)}\coloneqq(\Delta v,\Delta w)_H,
\]
and corresponding norms~$\norm{v}{V}\coloneqq(v,v)_V^\frac{1}{2}$ and~$\norm{v}{\D(\Delta)}\coloneqq(v,v)_{\D(\Delta)}^\frac{1}{2}$.

Moreover we have the inclusions
\[
\D(\Delta)\xhookrightarrow{\rm d,c} V\xhookrightarrow{\rm d,c} H\xhookrightarrow{\rm d,c} V'\xhookrightarrow{\rm d,c} \D(\Delta)',
\]
the increasing sequence of repeated eigenvalues~$\alpha_i$, $i=1,2,\dots$, of~$-\Delta$ satisfy
\[
0< \alpha_1<\alpha_2\le\alpha_3\le\dots,\qquad \lim_{i\to
+\infty}\alpha_i= +\infty
\]
and we have
\[
\langle v,w\rangle_{V',V} =(v,w)_H,\quad\mbox{for all}\quad (v,w)\in H\times V.
\]

\smallskip\noindent
{\em Boundedness assumption.}
For~$m\in\mathbb N_0$, in order to simplify the writing we denote
\begin{equation}\label{Wwspace}
\begin{array}{rcl}
 \WW^{J}&\coloneqq& L^\infty_w(J,L^d(\Omega,\R)\times L^\infty(\Omega,\R^d))\\
 \WW&\coloneqq&L^\infty_w(\mathbb R_0,L^d(\Omega,\R)\times L^\infty(\Omega,\R^d))
 \end{array}
\end{equation}
where~$J\subseteq(0,+\infty)$ is an open interval.

We also fix~$a$ and~$b$, which may depend on time and space,  and a constant~$C_{\mathcal W}\ge 0$, satisfying
\begin{equation}\label{rhosigma}
\norm{(a,b)}{\mathcal W}\coloneqq\left(\norm{a}{L^\infty(\mathbb R_0,L^d)}^2+\norm{b}{L^\infty_w(\mathbb R_0,L^\infty)}^2\right)^\frac{1}{2}\le C_{\mathcal W}.
\end{equation}

\begin{remark}({\em A technical measurability detail\/}).
The space~$L^\infty_w(\mathbb R_0,L^\infty(\Omega,\,\R))$ is the Bochner-like notation for~$L^\infty(\mathbb R_0\times\Omega,\,\R)$, where the subscript~$w$ stands for
weak measurability.
Bochner spaces are usually defined for strongly measurable functions, see~\cite[Section~2]{Bochner33}. The Bochner space~$L^\infty((a,b),L^\infty(\Omega,\R))$
consisting of strongly measurable functions is strictly contained in
$L^\infty_w(\mathbb R_0,L^\infty(\Omega,\R))= (L^1((a,b),L^1(\Omega,\R)))'$, see~\cite[Example~5.0.10]{Fattorini99}.
This is due to the fact that~$L^\infty(\Omega,\R)$ is not separable.
Recall also~\cite[Theorem~1.1]{Pettis38} for a relation between strong and weak measurability. However, the norm in~$L^\infty_w((a,b),L^\infty(\Omega,\,\R))$
is essentially the usual norm of~$L^\infty((a,b),L^\infty(\Omega,\,\R))$, see~\cite[Lemma~9.1.2 (ii)]{Fattorini99}, \cite[Lemma~4.1.1]{Fattorini05}, and the miscellaneous
notes at the end of Section~4.1 in~\cite{Fattorini05}. Hereafter, we are going to use some arguments
from~\cite{Rod15-cocv,KroRod15,KroRod-ecc15}, whose results should be understood with the subscript~$w$ in~\cite[Equation~(2.2) and Remark~2.15]{Rod15-cocv},
 \cite[Equation~(2.1)]{KroRod15}, and~\cite[Equation~(5)]{KroRod-ecc15}.
\end{remark}

\subsection{Weak solutions}\label{sS:weak-sol}
Let us consider the interval~$I=(s_0,s_1)$ with $0\le s_0 < s_1$, whose length we denote by~$|I|\coloneqq s_1-s_0$.
Here we recall some regularity results for the weak solutions for systems as~\eqref{sys-z-int}. We start considering the more general system
\begin{subequations}\label{sys-z2}
\begin{align}
  &\partial_t z-\Delta z+ a z+\nabla\cdot(bz)+f=0,\label{sys-z2-eq}\\
  &z\rest\Gamma =0, \qquad z(0)=z_0.\label{sys-z2-bcic}
\end{align}
\end{subequations}
where the control is replaced by a general external force.

\begin{lemma}\label{L:estRC}
We have, for $z\in V$
\begin{align*}
 \langle az,z\rangle_{V',V}&\le C\norm{a}{L^d}\norm{z}{H}\norm{z}{V},
 \qquad \norm{az}{V'}\le C\norm{a}{L^d}\norm{z}{H}^\frac{1}{2}\norm{z}{V}^\frac{1}{2},&&\mbox{ for}\quad d\in\{1,2\}.\\
 \langle az,z\rangle_{V',V}&\le C\norm{a}{L^d}\norm{z}{H}\norm{z}{V},
 \qquad \norm{az}{V'}\le C\norm{a}{L^d}\norm{z}{H},&&\mbox{ for}\quad d\ge3.\\
 \langle \nabla\cdot(bz),z\rangle_{V',V}&\le C\norm{b}{L^\infty}\norm{z}{H}\norm{z}{V},
 \qquad \norm{\nabla\cdot(bz)}{V'}\le C\norm{b}{L^\infty}\norm{z}{H},&&\mbox{ for}\quad d\ge1.
 \end{align*}
for a suitable constant~$C\ge0$, depending only on~$(\Omega,d)$.
\end{lemma}
\begin{proof}
Concerning the reaction term, in the case~$d=1$, and $\Omega=(l,r)$ with $l<r$, from $\frac{\ed}{\ed x}\norm{z}{\R}^2=2z\frac{\ed}{\ed x}z$ we can see that for $z\in V$ and $s\in\Omega$, since~$z(l)=z(r)=0$,
\[
\norm{z(s)}{\R}^2=\norm{z(s)}{\R}^2-\norm{z(l)}{\R}^2=2\int_l^sz(\tau)\textstyle\frac{\ed}{\ed x}z(\tau)\,\ed\tau
\le 2\norm{z}{L^2(\Omega,\R)}\norm{\frac{\ed}{\ed x}z}{L^2(\Omega,\R)}.
\]
That is, we have the Agmon inequality
\begin{equation}\label{Agmon_d=1}
 \norm{z}{L^\infty}\le2^\frac{1}{2}\norm{z}{H}^\frac{1}{2}\norm{z}{V}^\frac{1}{2},\qquad \mbox{for}\quad d=1.
\end{equation}
Therefore, it follows that
\[
\langle az,w\rangle_{V',V}\le \norm{a}{L^1}\norm{z}{L^\infty}\norm{w}{L^\infty}
\le 2\norm{a}{L^1}\norm{z}{H}^\frac{1}{2}\norm{z}{V}^\frac{1}{2}\norm{w}{H}^\frac{1}{2}\norm{w}{V}^\frac{1}{2}.
\]

For the cases $d\ge2$, we will use suitable Sobolev embeddings (cf.\cite[Corollary~4.53]{DemengelDem12}).
 For $d=2$, from $H^\frac{1}{2}(\Omega,\R)\xhookrightarrow{}L^4(\Omega,\R)$ and an interpolation argument (cf.~\cite[chapter~1, section~9.1]{LioMag72-I}), we find
\[\langle az,w\rangle_{V',V}\le \norm{a}{L^2}\norm{z}{L^4}\norm{w}{L^4}
\le C_2\norm{a}{L^2}\norm{z}{H}^\frac{1}{2}\norm{z}{V}^\frac{1}{2}\norm{w}{H}^\frac{1}{2}\norm{w}{V}^\frac{1}{2}.\]
 For $d\ge3$, from $H^1(\Omega,\R)\xhookrightarrow{}L^\frac{2d}{d-2}(\Omega,\R)$, we find
\[\langle az,z\rangle_{V',V}\le \norm{a}{L^d}\norm{z}{L^\frac{2d}{d-2}}\norm{z}{L^2}\le C_3\norm{a}{L^d}\norm{z}{V}\norm{z}{H}.\]

Finally, writing
$\langle \nabla\cdot(bz),z\rangle_{V',V}=(bz,\nabla z)_{L^2 (\Omega, \R^d)},$ it will follow the estimates for the convection term.
\end{proof}

\begin{lemma}\label{L:weak-z}
 Given $f\in L^2(I,V')$ and $z_0\in H$, there is a weak solution
 $z\in W(I,V,V')$ for~\eqref{sys-z2}. Moreover~$z$ is unique and depends continuously on the data:
 \[
 \norm{z}{W(I,V,V')}^2\le \overline C_{\left[|I|,C_{\mathcal W}\right]}
\left(\norm{z(s_0)}{H}^2+\norm{f}{L^2(I,V')}^2\right).
 \]
\end{lemma}
The procedure is well known, yet we will recall some steps of the proof since some estimates from the proof will be used later on.

Weak solutions for system~\eqref{sys-z2} are understood in the variational
sense. We will restrict ourselves to the derivation of some a priori (like)
estimates. In fact those estimates
will also hold for
Galerkin approximations of the system, for example using a basis of
eigenfunctions of the Laplace operator~$\Delta$, thus the estimates can be used
to precisely
derive the existence of weak solutions. See~\cite[Chapter~1, Section~6]{Lions69}, \cite[Chapter~1, Section~3]{Temam95}, and~\cite[Chapter~3, Sections~1.3, 1.4, and~3.2]{Temam01}
for more details on the procedure.

We start with the following auxiliary result.
\begin{corollary}
 We have, for $z\in V$
 \begin{align*}
 2\langle az+\nabla\cdot(bz),z\rangle_{V',V}&\le D_{\rm rc,1}\norm{(a,b)}{\WW}^2 \norm{z}{H}^2 + \textstyle\frac{1}{2}\norm{z}{V}^2,\\
 2\langle az+\nabla\cdot(bz),z\rangle_{V',V}&\le D_{\rm rc} \norm{(a,b)}{\WW}^2\norm{z}{H}^2 + \textstyle\frac{3}{2}\norm{z}{V}^2.
  \end{align*}
  for suitable constants~$D_{\rm rc}\le D_{\rm rc,1}$ depending only on~$(\Omega,d)$.
\end{corollary}
 \begin{proof}
 From Lemma~\ref{L:estRC}, and for any $\alpha>0$, we have $2\langle az+\nabla\cdot(bz),z\rangle_{V',V}
 \le 2C(\norm{a}{L^d}+\norm{b}{L^\infty})\norm{z}{H}\norm{z}{V}
 \le \frac{1}{\alpha}C^2(\norm{a}{L^d}+\norm{b}{L^\infty})^2\norm{z}{H}^2 +\alpha\norm{z}{V}^2
 \le\frac{2}{\alpha}C^2(\norm{a}{L^d}^2+\norm{b}{L^\infty}^2)\norm{z}{H}^2 +2\alpha\norm{z}{V}^2.$
 \end{proof}

\begin{proof}[Proof of Lemma~\ref{L:weak-z}]
Multiplying~\eqref{sys-z2-eq} by~$2z$, formally we find
\[
\begin{split}
 \frac{\mathrm d}{\mathrm d t}\norm{z}{H}^2+2\norm{z}{V}^2=2\langle az+\nabla\cdot(bz),z\rangle_{V',V}+2\langle f,z\rangle_{V',V},
 \end{split}
\]
from which, by appropriate Young inequalities, we can obtain
\begin{subequations}
\begin{align}
 \frac{\mathrm d}{\mathrm d t}\norm{z}{H}^2+\norm{z}{V}^2&\le D_{\rm rc,1}\norm{(a,b)}{\WW}^2\norm{z}{H}^2
 +2\norm{f}{V'}^2;\label{zwYoung}\\
 \frac{\mathrm d}{\mathrm d t}\norm{z}{H}^2&\le D_{\rm rc}\norm{(a,b)}{\WW}^2\norm{z}{H}^2
 +2\norm{f}{V'}^2.\label{zwYoung-H}
\end{align}
\end{subequations}
By~\eqref{zwYoung-H} and the Gronwall inequality, it follows that for all $s\in I$
\begin{equation}\label{zwGron}
 \norm{z(s)}{H}^2\le \ex^{D_{\rm rc}\norm{(a,b)}{\WW}^2(s-s_0)}\left(\norm{z(s_0)}{H}^2+2\norm{f}{L^2(I,V')}^2\right),
\end{equation}
and integrating~\eqref{zwYoung},
\begin{equation}\label{zwYoungi}
 \norm{z(s)}{H}^2+\norm{z}{L^2((s_0,s),V)}^2\le \norm{z(s_0)}{H}^2
 +D_{\rm rc,1}\norm{(a,b)}{\WW}^2\norm{z}{L^2(I,H)}^2+2\norm{f}{L^2(I,V')}^2.
\end{equation}
From~\eqref{sys-z2-eq} and Lemma~\ref{L:estRC} 
we can also derive
\[
 \norm{\partial_t z}{L^2(I,V')}\le \norm{z}{L^2(I,V)}+ \overline C_{[C_\WW]}\norm{ z}{L^2(I,H)}^2+\norm{f}{L^2(I,V')},
\]
from which, using~\eqref{zwGron} and~\eqref{zwYoungi}  we can conclude that
\[
\norm{z}{W(I,V,V')}^2\le \overline C_{\left[s_1-s_0,C_\WW\right]}
\left(\norm{z(s_0)}{H}^2+\norm{f}{L^2(I,V')}^2\right).
\]
Finally the uniqueness of~$z$, follows from the fact that if~$\tilde{z}$ is another weak solution, then $ e=z-\tilde{z}$,
solves~\eqref{sys-z2} with $ e (s_0)=0$ and $f=0$. From~\eqref{zwGron} it will follow that
$\norm{ e (s)}{H}=0$
for all~$s\in I$.
\quad\end{proof}

\subsection{Null controllability}\label{sS:nullcont}
Here we recall the relation between null controllability of system~\eqref{sys-z2} and a suitable observability inequality for the adjoint
system.

Consider, in the bounded cylinder $I\times\Omega$, $I=(s_0,s_1)$, the controlled system
\begin{subequations}\label{sys-z2c}
\begin{align}
  &\partial_t z-\Delta z+ a z+\nabla\cdot (bz)+B\eta=0,\label{sys-y2c-eq}\\
  &z\rest \Gamma =0, \qquad z(s_0)=z_0,\label{sys-y2c-bcic}
\end{align}
\end{subequations}
where now our control is a function $\eta\in L^2(I,H)$ and $B\in\mathcal L(H)$ with adjoint denoted by $B^*$. Let us also consider in $I\times\Omega$ the adjoint system
\begin{subequations}\label{sys-q}
\begin{align}
  &-\partial_t q-\Delta q+ a q-b\cdot\nabla q=0,\label{sys-q-eq}\\
  &q\rest \Gamma =0, \qquad q(s_1)=q_1\in H,\label{sys-q-bcic}
\end{align}
\end{subequations}
and let $z(z_0,\eta)(t)\coloneqq z(t)$ and~$q(q_1)(t)\coloneqq q(t)$ denote the solutions of~\eqref{sys-z2c} and~\eqref{sys-q}, for given data~$(z_0,u)$
and~$q_1$, respectively.
Notice that, proceeding as in section~\ref{sS:weak-sol}, we can prove the existence of weak solutions~$q\in W(I,V,V')$
for system~\eqref{sys-q}.

\begin{definition}
(i) We say that~\eqref{sys-z2c} is null controllable in $I$ if there exists a
 family~$\{\eta(z_0)\mid z_0\in H\}\subset L^2(I,H)$ such that $z(z_0,\eta(z_0))(s_1)=0$, for
 $z_0\in H$.\\
 (ii) We say that~\eqref{sys-q} is $B^*$-observable in $I$ if there exists a constant $C_{\rm obs}>0$ such that
 for all $q_1\in H$ we have that the corresponding weak solution $q$ satisfies the inequality
 \begin{equation}\label{obs-in-B*}
 \norm{q(q_1)(s_0)}{H}\le C_{\rm obs}\norm{B^*q(q_1)}{L^2(I,H)}.
 \end{equation}
\end{definition}
The constant~$C_{\rm obs}$ in~\eqref{obs-in-B*} depends, in general, on~$\Omega$, $\omega$, $I$, $B$, and also on the coefficient functions~$a$ and~$b$.

The following lemma, can be proven by a standard procedure. For details, we refer to~\cite[Section~2]{Amm-KBenG-BT11}, \cite[Chapter~2]{Coron07}.

\begin{lemma}\label{L:ex0cont}
System~\eqref{sys-q} is $B^*$-observable in $I$ if, and only if,
system~\eqref{sys-z2c} is
null controllable in $I$ and the family of controls~$\{\eta(z_0)\mid z_0\in H\}$ can be chosen as a bounded linear function of~$z_0$:
\begin{align*}
\norm{\eta(z_0)}{L^2(I,H)}\le C_{\rm obs}\norm{z_0}{H}, \text{ where } C_{\rm obs} \text{ is as in }~\eqref{obs-in-B*}.
\end{align*}
\end{lemma}

\medskip\noindent
{\em Controls supported in a subset.}
Given an open subset $\omega\subseteq\Omega$, then from~\cite[Theorem~2.1]{DuyckZhangZua08} and~\cite[Theorem~2.3]{DouFCaraGBurgZuazua02}
(e.g., reversing time in system~\eqref{sys-q}) we have that there exists
a constant $ C_{\omega,\Omega}>0$, depending on~$\omega$ and~$\Omega$, such that
the weak solution~$q$ for~\eqref{sys-q} satisfies
\begin{equation*}
\norm{q(0)}{H}^2\le \ex^{C_{\omega,\Omega}\Theta\left(|I|,\norm{a}{L^\infty(I,L^d)}, \norm{b}{L^\infty_{w}(I,L^\infty)},d\right)}
\norm{q}{L^2(I,L^2(\omega,\R)}^2.
\end{equation*}
with
\begin{equation}\label{const_obs1}
\Theta(r,\theta_1,\theta_2,d)\coloneqq 1+\theta_1^2+d\theta_2^{2}
+\frac{1}{r}+r\left(\theta_1+d\theta_2^{2}\right),
\end{equation}
for~$r>0,$ $\theta_1\ge0\le\theta_2$, and~$d\in\N_0$.

Therefore, in the case we take $B=1_\omega\in\mathcal L(H)$ with
\[
1_\omega \eta(x)\coloneqq\left\{\begin{array}{ll} \eta(x),&\mbox{ if }x\in\omega\\
0,&\mbox{ if }x\in\Omega\setminus\overline{\omega}\end{array}\right.,
\]
then we  have $B^*=B$ and $\norm{q}{L^2(I,L^2(\omega,\R))}^2=\norm{B^*q}{L^2(I,H)}^2$, and we can conclude that~\eqref{obs-in-B*}
holds with~$
C_{\rm obs}=\ex^{C_{\omega,\Omega}\Theta\left(|I|,\norm{a}{L^\infty(I,L^d)},\norm{b}{L^\infty_w(I,L^\infty)},d\right)}$.
Therefore we have the following.
\begin{theorem}\label{T:exact.contB=1}
Let $B=1_\omega$ and let $I=(s_0, s_1)$ be arbitrary, then, there exists a
family $\{\eta(z_0)\mid z_0\in H\}\subseteq L^2(I,H)$ such that the solutions
$z(z_0,\eta(z_0))$ to \eqref{sys-z2c} satisfy
$z(z_0,\eta(z_0))(s_1)=0$ and, for a constant $\widehat C=C(\omega,\Omega)$, we have that
\begin{align*}
\norm{\eta(z_0)}{L^2(I,H)}\le\ex^{\widehat C\Theta}\norm{z_0}{H},
\end{align*}
with $\Theta=\Theta\left(|I|,\norm{a}{L^\infty(\R_0,L^d)},\norm{b}{L^\infty_w(\R_0,L^\infty)},d\right)$ given by~\eqref{const_obs1}.
\end{theorem}
Notice that since~\eqref{obs-in-B*} holds with $C_{\rm obs}=\ex^{\widehat C\Theta}$ and
$B=1_\omega$.
Proceeding as in~\cite[Section~A.2]{BarRodShi11} we can conclude that~\eqref{obs-in-B*} also holds with
$C_{\rm obs}=C_\chi\ex^{\widehat C_\chi \Theta}\le\ex^{\widehat D \Theta}$ and
$B^*q\coloneqq \chi 1_\omega q=1_\omega\chi 1_\omega q$, where $\widehat D=\log(C_\chi)+\widehat C_\chi$ and $\chi\in C^\infty(\overline\Omega)$ is any given smooth function
with~$\emptyset\ne \omega\cap\supp \chi$. Here $\widehat D=\widehat D(\chi,\omega,\Omega)>0$
depends only on~$(\chi,\omega,\Omega)$. Notice that~$\Theta(r,\theta_1,\theta_2,d)\ge1$.

\begin{corollary}\label{C:exact.contB=chi}
Theorem~\ref{T:exact.contB=1} holds in the more general case~$B=1_\omega\chi 1_\omega$, with~$\widehat D$
in the place of~$\widehat C$.
\end{corollary}

\subsection{Stabilization to zero by finite dimensional controls}

Here we analyze the case when stabilization of system~\eqref{sys-z2c} can be achieved by finite dimensional control action, of the form~$\sum\limits_{i=1}^Mu_i(t)\Phi_i(x)$.
Following the ideas in~\cite{KroRod15,KroRod-ecc15,BreKunRod-pp15}, we consider a family~$\widehat\CC_{\omega}=\{\widehat \Phi_i\in H\mid
i\in\{1,2,\dots,M\}\}\subset H$ and
denote by~$P_M$ the orthogonal projection
in~$H$ onto~$\mathcal S_{\widehat\CC_\omega}\coloneqq\mathrm{span}\,\widehat\CC_\omega$.

Let us also fix a positive constant~$\bar\lambda>0$ and consider, in~$\mathbb R_{s_0}\times\Omega$, the system:
\begin{subequations}\label{sys-z2cf}
\begin{align}
  &\partial_t z-\Delta z+ a z+\nabla\cdot (bz)+ 1_\omega\chi  P_M1_\omega \eta=0,\label{sys-z2cf-eq}\\
  &z\rest \Gamma =0, \qquad z(s_0)=z_0.\label{sys-z2cf-bcic}
\end{align}
\end{subequations}

\begin{definition}
We say that~\eqref{sys-z2cf} is exponentially stabilizable to zero, with rate $\frac{\bar\lambda}{2}$, if there
are a constant $C>0$ and a family $\{\eta=\eta(z_0)\mid z_0\in H\}\subseteq L^2(\mathbb R_{s_0},H)$
such that the corresponding global solution~$z(t)=z(z_0,\eta(z_0))(t)$ satisfies
\begin{equation*}
|z(t)|_H^2\leq C\,\ex^{-\bar\lambda (t-s_0)}|z_0|_H^2,\quad\mbox{for all $t\ge s_0$}.
\end{equation*}
\end{definition}

Notice that the stabilizing control in~\eqref{sys-z2cf-eq} takes its values in the finite dimensional space $ 1_\omega\chi\mathcal S_{\CC_\omega}=\mathrm{span}\,\{1_\omega\chi
\widehat \Phi_i\in H\mid i\in\{1,2,\dots,M\}\}$, for all $t\in\mathbb R_{s_0}$.
\[
1_\omega\chi  P_M1_\omega \eta= 1_\omega\chi\sum_{i=1}^M\eta_i\widehat\Phi_i=\sum_{i=1}^M u_i\Phi_i
\]
with~$u_i\coloneqq\eta_i$ and~$\Phi_i\coloneqq 1_\omega\chi\widehat\Phi_i$, $i\in\{1,2,\dots,M\}$.

\begin{remark}
Without loss of generality we can suppose that the family~$\widehat\CC_\omega$ is linearly independent and orthonormal.
In that case~$\eta_i(t)=\left(\eta(t,\Bigcdot),\widehat\Phi_i\right)_H$.
\end{remark}

In the following, the function~$\Theta$ and the constant~$\widehat D$ are as in
Theorem~\ref{T:exact.contB=1} and Corollary~\ref{C:exact.contB=chi}. 
Let $\chi\in C^\infty(\overline\Omega)$ satisfy $\emptyset\ne \omega\cap\supp \chi$, and consider the system
\begin{subequations}\label{sys-z2-1wM}
\begin{align}
  &\partial_t z-\Delta z+ (a-\textstyle\frac{\bar\lambda}{2}) z+\nabla\cdot (bz)+ 1_\omega\chi 1_\omega \eta=0,\label{sys-z2-1wM-eq}\\
  &z\rest \Gamma =0, \qquad z(s_0)=z_0.\label{sys-z2-1wM-bcic}
\end{align}
\end{subequations}

\begin{lemma}\label{L:defXi}
 Let $I=(s_0,s_1)$. The solution of system~\eqref{sys-z2-1wM}, in $I\times\Omega$, with the control~$\eta$
 given by Corollary~\ref{C:exact.contB=chi}, satisfies
 \begin{equation*}
  \begin{split}
 \norm{z(s_0+\tau)}{H}^2&\le \ex^{D_{\rm rc}\norm{(a-\frac{\bar\lambda}{2},b)}{\WW}^2 \tau}\norm{z(s_0)}{H}^2\\
 &\quad+2\norm{\iota}{\LL(H,V')}^2\ex^{D_{\rm rc}\norm{(a-\frac{\bar\lambda}{2},b)}{\WW}^2\tau+\widehat D\Theta\left(|I|,\norm{a}{L^\infty(\R_0,L^d)},
 \norm{b}{L^\infty_w(\R_0,L^\infty)},d\right)}\norm{z(s_0)}{H}^2,
 \end{split}
\end{equation*}
for all~$\tau\in I$, where~$\iota\colon H\to V'$ stands for the inclusion mapping $\iota z\coloneqq z$.
\end{lemma}
\begin{proof}
 Straightforward, from~\eqref{zwGron}. Notice that Corollary~\ref{C:exact.contB=chi} still holds true with~$a-\textstyle\frac{\bar\lambda}{2}$ in the role of~$a$.
\end{proof}

Inspired by Lemma~\ref{L:defXi}, and the procedure in~\cite{BreKunRod-pp15,KroRod-ecc15},
we consider the function~$\Xi\colon (0,+\infty)\to (0,+\infty)$ defined by
\begin{equation*}
\Xi(\tau)\coloneqq\mapsto 2\norm{\iota}{\LL(H,V')}^2\ex^{D_{\rm rc}\norm{(a-\frac{\bar\lambda}{2},b)}{\WW}^2\tau
+\widehat D\Theta\left(|I|,\norm{a-\frac{\bar\lambda}{2}}{L^\infty(\R_0,L^d)},\norm{b}{L^\infty_w(\R_0,L^\infty)},d\right)}
\end{equation*}
which we can extend to a function $\Xi^{\rm ex}\colon [0,+\infty]\to (0,+\infty]$, by setting
\[
 \Xi^{\rm ex}(\tau)\coloneqq
\left\{\begin{array}{ll}
  \Xi(\tau), &\mbox{if } \tau\in(0,+\infty),\\
 \lim\limits_{t\to\tau}\Xi(t), &\mbox{if } \tau\in\{0,+\infty\}.
 \end{array}\right.
\]
The minimum and minimizer of~$\Xi^{\rm ex}$ are denoted by~$\Upsilon$ and~$T_*$, respectively.
From
\[
\frac{\left.\frac{\rm d\Xi}{\rm d\tau}\right|_{\tau=t}}{\Xi(t)}=\textstyle\left(D_{\rm rc}\norm{(a-\frac{\bar\lambda}{2},b)}{\WW}^2
+\widehat D\left(-t^{-2}+\norm{ a-\frac{\bar\lambda}{2}}{L^\infty(\R_0,L^d)}
 +d\norm{ b}{L^\infty_w(\R_0,L^\infty)}^2\right)\right),
\]
we can conclude that $T_*>0$ can be defined by
\begin{align}\label{eq:Top}
 T_*^2=\frac{\widehat D}{D_{\rm rc}\norm{(a-\frac{\bar\lambda}{2},b)}{\WW}^2
 +\widehat D\left(\norm{ a-\frac{\bar\lambda}{2}}{L^\infty(\R_0,L^d)}
 +d\norm{ b}{L^\infty_w(\R_0,L^\infty)}^2\right)}.
\end{align}
Further~$T_*=+\infty$ if, and only if, both~$a-\frac{\bar\lambda}{2}$
and~$b$ vanish.

Thus, if $T_*\in\R_0$, we have that the minimum~$\Upsilon=\Xi^{\rm ex}(T_*)$ is given by
\begin{subequations}\label{Upsilon}
\begin{align} 
\Upsilon&=\textstyle2\norm{\iota}{\LL(H,V')}^2\ex^{\overline\Theta\left(\norm{a-\frac{\bar\lambda}{2}}{L^\infty(\R_0,L^d)}
,\norm{b}{L^\infty_w(\R_0,L^\infty)},\norm{ (a-\frac{\bar\lambda}{2},b)}{\WW},d\right)}, \label{Upsilon1}
\intertext{and, if $T_*=+\infty$ by}
\Upsilon&=\textstyle2\norm{\iota}{\LL(H,V')}^2\ex^{\widehat D}.\label{Upsilon2}
\end{align} 
\end{subequations}
where
\begin{equation}\label{barTheta}
\overline\Theta(\xi_1,\xi_2,\xi_3,d,\nu)\coloneqq\widehat D\left(1+\xi_1^2+d\xi_2^2\right)
 +2(\widehat D)^\frac{1}{2}\left(D_{\rm rc} \xi_3^2
 +\widehat D\left(\xi_1
  +d\xi_2^2\right)\right)^\frac{1}{2}.
 \end{equation}

 The following result gives us a sufficient condition on the family~$\widehat\CC_\omega$ for the existence of a stabilizing control. The proof can be done following the arguments
in~\cite{KroRod-ecc15,BreKunRod-pp15}.
\begin{theorem}\label{T:stab-v-gen}
Let us be given $\chi\in C^\infty(\overline\Omega)$ satisfying $\emptyset\ne\omega\cap\supp\chi$.  If
\begin{equation}\label{M_bound-gen}
T_*\in\mathbb R_0\quad\mbox{and}\quad\norm{1_\omega\chi(1- P_M)1_\omega}{\mathcal L(H, V')}^2
\le \Upsilon^{-1},
\end{equation}
with $T_*$ as in~\eqref{eq:Top} and $\Upsilon$ as in~\eqref{Upsilon1},
then system~\eqref{sys-z2cf} is stabilizable to zero with rate~$\frac{\bar\lambda}{2}$.
Moreover, we can set the stabilizing control function $\eta=\eta(z_0)$ such that
\begin{align*}
 |z(t)|_{H}^2
&\le \left(\Upsilon_0+\Upsilon\norm{B_M }{\mathcal L(H, V')}^2\right)\ex^{-\bar\lambda (t-s_0)}\norm{z_0}{H}^2, \mbox{ for }t\ge s_0,\\
\bigl|\ex^{\frac{\widehat\lambda}{2}\Bigcdot} \eta(z_0)\bigr|_{L^2(\mathbb R_{s_0},H)}^2
&\le \textstyle\frac{1}{1-\ex^{(\widehat\lambda-\bar\lambda)T_*}}\ex^{2\widehat D\Theta_*}\norm{z_0}{H}^2, \mbox{ for } \widehat\lambda<\bar\lambda,
\end{align*}
with $\Upsilon_0\coloneqq\ex^{D_{\rm rc}\norm{ (a-\frac{\bar\lambda}{2},b)}{\WW}^2 T_*}$ and
$\Theta_*\coloneqq \Theta(T_*,\norm{ a-\frac{\bar\lambda}{2}}{L^\infty(\R_0,L^d)},\norm{b}{L^\infty_w(\R_0,L^\infty)},d)$.
 
If $T_*=+\infty$, then setting $\eta=\eta(z_0)=0$ the solution~$z$ of system~\eqref{sys-z2cf} satisfies $|z(t)|_{H}^2\le \ex^{-\bar\lambda (t-s_0)}\norm{z_0}{H}^2,$ for $t\ge s_0$.
\end{theorem}

\subsection{The dimension of the control}\label{sS:examples_dimM}
Families~$\widehat\CC_\omega$ which
satisfy~\eqref{M_bound-gen} do exist, for example in the case the control domain is an open
nonempty rectangle $\omega=\omega_R\coloneqq\prod_{j=1}^{n}(l_{j,1},l_{j,2})\subset\Omega$. We give two examples

\begin{enumerate}[noitemsep,topsep=5pt,parsep=5pt,partopsep=0pt,leftmargin=0em]%
\renewcommand{\theenumi}{{\bf Ex\arabic{enumi}}} 
 \renewcommand{\labelenumi}{}

\item\theenumi\label{ex.1}. {\em $0\ne\chi\in C^\infty(\overline\Omega)$ such that
$\supp\chi\subseteq\overline{\omega_R}$
and $\widehat\CC_\omega\coloneqq\{\widehat\Psi_{R,i}\mid i\in\{1,2,\dots,M\}\}$ for big enough~$M$, where $\{\widehat\Psi_{R,i}\mid i\in\mathbb N_0\}$
is a complete system of eigenfunctions of the negative Laplacian~$-\Delta$
in~$\omega_R$ with homogeneous Dirichlet boundary conditions, which are
ordered according to the
increasing sequence of the (repeated) eigenvalues:
$0<\bar\lambda_i\le\bar\lambda_{i+1}$, $\lim_{i\to \infty}\bar\lambda_i=\infty$.
}
\item\theenumi\label{ex.2}. {\em $\chi=\mathbbm1_\Omega$ and we consider a family~$\widehat\CC_\omega=\left\{\Phi_i=\mathbbm{1}_{R_i}\mid
i\in\{1,2,\dots,M\}\right\}\in H$ where the $R_i$s are the sub-rectangles in a uniform partition of~$\omega_R$, and
\[
\mathbbm{1}_{\OO}(x)\coloneqq\begin{cases}1,&\mbox{if }x\in \OO,\\0,&\mbox{if }x\in \Omega\setminus\overline{\OO},
\end{cases}\qquad\mbox{for an open subset}\quad\OO\subseteq\Omega.
\]
That is, each
interval~$(l_{j,1},l_{j,2})$ is divided into~$p_j$ intervals:
$I_{j,k}=(l_{j,1}+k_j\frac{\overline
l_j}{p_j},l_{j,1}+(k_j+1)\frac{\overline l_j}{p_j})$, with
$k_j\in\{0,1,\dots,p_j-1\}$ and $\overline l_j\coloneqq l_{j,2} - l_{j,1}$.
In this way, our rectangle is divided into $M=\prod_{j=1}^d p_j$
sub-rectangles $\{R_i\mid i\in\{1,2,\dots,M\}\}=\{\prod_{j=1}^nI_{j,k_j}\mid
k_j\in\{0,1,\dots,p_j-1\}\}.$
}
\end{enumerate}

\medskip
Proceeding as in~\cite[Example~2.14]{BreKunRod-pp15}, we can see that in the case of example~\ref{ex.1} above, condition~\eqref{M_bound-gen}
is satisfied provided that~$4\norm{\chi}{C^1(\overline\Omega)}^2(\bar\lambda_{M}+1)^{-1}\le\Upsilon^{-1}$, that is, provided
$\bar\lambda_{M}+1\ge 4\norm{\chi}{C^1(\overline\Omega)}^2\Upsilon$. Furthermore,
from~\cite[Corollary~1]{LiYau83} we have the asymptotic behaviour $\bar\lambda_M\ge D_dM^\frac{2}{d}$
with
\begin{equation}\label{Dd-asymp}
D_d=\textstyle\frac{4d\pi^2}{(d+2)|w_R|^\frac{2}{d} |\BB_d|^\frac{2}{d}},
\end{equation}
where $|w_R|=\prod_{j=1}^d\bar l_j$ is the volume of~$w_R$ and~$|\BB_d|$
is the volume of the unit ball~$\BB_d\coloneqq\{x\in\R^d\mid\norm{x}{\R^d}\le1\}$, we can also
arrive at the sufficient condition
\begin{equation}\label{estM_eig}
M\ge
D_d^{-\frac{d}{2}}(4\norm{\chi}{C^1(\overline\Omega)}^2\Upsilon)^{\frac{d}{2}}, 
\end{equation}
which gives
us an upper bound on the number~$M$ of actuators which are needed to
stabilize the system.

We recall that $|\BB_1|=2$, $|\BB_2|=\pi$, and $|\BB_3|=\frac{4}{3}\pi$, and in general
$|\BB_{2k}|=\frac{\pi^k}{\varGamma(k+1)}=\frac{\pi^k}{k!}$ and also
$|\BB_{2k+1}|=\frac{\pi^{k+\frac{1}{2}}}{\varGamma(k+\frac{3}{2})}=\frac{2(k!)(4\pi)^k}{(2k+1)!}$, $k\in\N$, where~$\varGamma$ stands for the Euler gamma function.
See~\cite[Corollary]{Folland01} and~\cite[Eq.~(10)]{Birkhoff13}.  Notice that we can prove the last identity by induction
using~$\varGamma(k+\frac{3}{2})=\varGamma(k+\frac{1}{2}+1)=(k+\frac{1}{2})\varGamma((k-1)+\frac{3}{2})$ and~$\varGamma(\frac{1}{2})=\pi^\frac{1}{2}$.

Proceeding as in~\cite[Example~2.15]{BreKunRod-pp15}, we can see that in the case of example~\ref{ex.2}, condition~\eqref{M_bound-gen}
is satisfied provided that the partition is fine enough, so that $(\mu_M\pi^2)^{-1}\le\Upsilon^{-1}$
where~$\mu_M=\min\left\{\frac{p_j^2}{\bar l_j^2}\mid j\in\{1,2,\dots,M\}\right\}$. Notice that $\mu_M\pi^2$ is the smallest nonzero eigenvalue of the (negative)
Neumann Laplacian
in each rectangle~$R_j$. Thus \eqref{M_bound-gen} holds
provided $\frac{\,\underline p^2\,}{\,\overline l^2\,}\ge\frac{\Upsilon}{\pi^2}$, where $\underline p\coloneqq\min\{p_j\mid j\in\{1,2,\dots,M\}\}$ and
$\overline l\coloneqq\max\{\overline l_j\mid j\in\{1,2,\dots,M\}\}$. Then we conclude that
$\frac{M^2}{\bar l^{2d}}\ge\frac{\Upsilon^d}{\pi^{2d}} 
$
is sufficient for~\eqref{M_bound-gen}, and thus so is
\begin{equation}\label{estM_pc}
M\ge\textstyle\left(\frac{\overline l^2}{\pi^2}\Upsilon\right)^{\frac{d}{2}}.
\end{equation}

Notice that, from~\eqref{Upsilon} we can see that the above estimates for~$M$ depend exponentially on~$\norm{ a-\frac{\bar\lambda}{2}}{\mathcal W}$
and~$\norm{ b}{\mathcal W^{d}}$. In the one dimensional case, in~\cite{KroRod15} it is conjectured that a better estimate might be possible
(depending polynomially in~$\norm{ a-\frac{\bar\lambda}{2}}{\mathcal W}$ and~$\norm{ b}{\mathcal W^{d}}$, like as in~\eqref{Msimplecase} below), and this conjecture is supported
from the results of some numerical simulations.
Later on, we will present some further simulations, for the two dimensional case, that again support this conjecture.

\medskip\noindent
{\em The particular case $\omega=\Omega$.}  Proceeding as in~\cite[section~3.1]{KroRod15} 
we can see that in case there is no constraint
in the support of the control, and we take $\chi = \mathbbm{1}_{\Omega}$, then we can find a better estimate.

\begin{theorem}\label{T:part}
Let $(a,b)\in \WW$ and $\bar\lambda > 0$. Let us also take $\chi=\mathbbm{1}_{\Omega}$ and a family $\widehat\CC_\omega\coloneqq\{\widehat\Psi_{i}\mid i\in\{1,2,\dots,M\}\}$,
where $\{\widehat\Psi_{i}\mid i\in\mathbb N_0\}$
is a complete system of eigenfunctions of the negative Laplacian
in~$\Omega$ ordered as in example~\ref{ex.1} above.
If we take
\begin{align}
\label{Msimplecase}
M\ge (\textstyle\frac{D_{\rm rc}\ex^{1}(d+2)}{d\pi^2})^\frac{d}{2}|\Omega||\BB_d|\norm{(a-\textstyle\frac{\bar\lambda}{2},b)}{\WW}^d,
\end{align}
then for any
given $z_0 \in H$, there is a control $\eta=\eta^{\bar\lambda,a,b} (s_0,z_0) \in L^2(\R_{s_0}, H)$ such that the corresponding
solution $z$ of system~\eqref{sys-z2cf}, satisfies the inequality
\begin{align}\label{est_part_z}
|z(t)|^2_H &\le 2\left( \mathrm{e}^{\frac{1}{2}} + 1  \right) \mathrm{e}^{-\bar\lambda (t-s_0)} |z_0|^2_H, \qquad t \ge 0. 
\end{align}
Furthermore,
\begin{align}\label{est_part_eta}
 \norm{\ex^{\frac{\widehat\lambda}{2}(\Bigcdot-s_0)}\eta^{\bar\lambda,a,b} (s_0,z_0)}{L^2(\R_{s_0}, H)}^2
 &\le\textstyle\frac{\ex^\frac{1}{2}}{T_*(\bar\lambda-\widehat\lambda)}\norm{z_0}{H}^2,\quad\mbox{for all}\quad \widehat\lambda<\bar\lambda.
\end{align}
\end{theorem}

\begin{proof}
We may proceed as in~\cite[Section~3.1]{KroRod15}. We just recall the main steps.
The main idea is to find a suitable time $T_*>0$, such that for all~$\overline s_0\ge0$ we can find a control defined in~$(\overline s_0,\overline s_0+T_*)$ such
that at time~$\overline s_0+T_*>0$ we
have $\norm{\ex^{\frac{\bar\lambda}{2}T_*}z(\overline s_0+T_*)}{H}\le\norm{z(\overline s_0)}{H}$.

For a given forcing~$f$, let us consider the system
\begin{align}
\label{SystemwithRate}
\partial_t y -  \Delta y + (a-\textstyle\frac{\bar\lambda}{2})y + \nabla \cdot \left( b y  \right) + f=0,  \quad y|_{\Gamma} = 0, \quad y(\overline s_0) = y_0.
\end{align}
Let~$w$ solve~\eqref{SystemwithRate} in the time interval~$(\overline s_0,\overline s_0+T)$ with~$f=0$. Then, $\delta(t)\coloneqq\frac{\overline s_0+T-t}{T}w(t)$,
$t\in[\overline s_0,\overline s_0+T]$, solves~\eqref{SystemwithRate} with~$f=\frac{1}{T}w$. Let
$\delta_M$ solve~\eqref{SystemwithRate} with~$f=\frac{1}{T}P_M w$ where~$P_M$ stands for the orthogonal projection in~$H$ onto the space~$\linspan\{\widehat\Psi_i\mid
i\in\{1,2,\dots,M\}\}$ spanned by the first eigenfunctions of the Dirichlet Laplacian in~$\Omega$.

It follows that~$d=\delta-\delta_M$ solves~\eqref{SystemwithRate} with~$f=\frac{1}{T}(1-P_M)w$ and by~\eqref{zwGron} it follows that for all~$s\in[\overline s_0,\overline s_0+T]$
\begin{align*}
\norm{d(s)}{H}^2&\le \textstyle\frac{2}{T^2}\ex^{D_{\rm rc}\norm{(a-\frac{\bar\lambda}{2},b)}{\WW}^2T}\norm{(1-P_M)}{\LL(H,V')}^2\norm{w}{L^2((\overline s_0,\overline s_0+T),H)}^2 \\ 
&\le \textstyle\frac{2}{T}\ex^{2D_{\rm rc}\norm{(a-\frac{\bar\lambda}{2},b)}{\WW}^2T}\norm{(1-P_M)}{\LL(H,V')}^2\norm{y_0}{H}^2.
\end{align*}
Therefore, setting~$T=T_*$ minimizing the right hand side, that is, setting
\[
 T_*=(2D_{\rm rc}\norm{(a-\textstyle\frac{\bar\lambda}{2},b)}{\WW}^2)^{-1},
\]
we find
\[
\norm{d(s)}{H}^2\le 4D_{\rm rc}\ex^{1}\norm{(a-\textstyle\frac{\bar\lambda}{2},b)}{\WW}^2\norm{(1-P_M)}{\LL(H,V')}^2\norm{y_0}{H}^2,
\quad s\in[\overline s_0,\overline s_0+T_*].
\]
Hence, we will have
\[
\norm{\delta_M(\overline s_0+T_*)}{H}^2\le \norm{y_0}{H}^2, 
\]
provided
\begin{equation}\label{cond.part1}
\norm{(1-P_M)}{\LL(H,V')}^{-2}\ge 4D_{\rm rc}\ex^{1}\norm{(a-\textstyle\frac{\bar\lambda}{2},b)}{\WW}^2.
\end{equation}
Thus, from~$\norm{(1-P_M)}{\LL(H,V')}^2=\alpha_{M+1}^{-1}$ and~\cite[Corollary~1]{LiYau83}, see~\eqref{Dd-asymp}, we have that~\eqref{cond.part1} follows from
\[
\textstyle\frac{4d\pi^2}{(d+2)|\Omega|^\frac{2}{d} |\BB_d|^\frac{2}{d}}M^\frac{2}{d}\ge 4D_{\rm rc}\ex^{1}\norm{(a-\textstyle\frac{\bar\lambda}{2},b)}{\WW}^2, 
\]
that is, \eqref{cond.part1} follows from~\eqref{Msimplecase}.

By~\eqref{zwGron} and~\eqref{cond.part1}, we can also obtain
\begin{align*}
\norm{\delta_M(s)}{H}^2&\le 2\norm{\delta(s)}{H}^2+2\norm{d(s)}{H}^2\le 2\norm{w(s)}{H}^2+2\norm{d(s)}{H}^2\\
&\le 2\ex^{D_{\rm rc}\norm{(a-\frac{\bar\lambda}{2},b)}{\WW}^2T_*}\norm{z_0}{H}^2+ 2\norm{z_0}{H}^2=2\left(\ex^{\frac{1}{2}}+1\right)\norm{y_0}{H}^2,
\quad s\in[\overline s_0,\overline s_0+T_*].
\end{align*} 

We observe that the control~$\eta_M\coloneqq\frac{1}{T} w$ associated with~$\delta_M$ satisfies
\[
\norm{\eta_M(\overline s_0,y_0)}{L^2((s_0,s_0+T_*), H)}^2\le\textstyle\frac{1}{T_*}\ex^\frac{1}{2}\norm{y_0}{H}^2. 
\]

We also see that~$T_*$ does not depend on~$\overline s_0$. Then, for any given~$z_0\in H$, we can define the concatenated
control~$\widetilde\eta=\widetilde\eta^{\bar\lambda,a,b}(s_0,z_0)$ in the half line~$\R_{s_0}=(s_0,+\infty)$ time interval as follows
\[
\left\{\begin{array}{l}
\widetilde\eta\rest{(s_0,s_0+T_*)}= \eta_M(s_0,z_0),\\
\widetilde\eta\rest{(s_0+iT_*,s_0+(i+1)T_*)}= \eta_M(s_0+iT_*,y(z_0,\widetilde\eta\rest{(s_0,s_0+iT_*)})(s_0+iT_*)),\quad i\in\N_0,
\end{array}\right.
\]
where~$y(z_0,\widetilde\eta\rest{(s_0,s_0+iT_*)})(t)$ stands for the solution of~\eqref{SystemwithRate},
with the setting~$(\overline s_0,y_0,f)=(s_0,z_0,\widetilde\eta\rest{(s_0,s_0+iT_*)})$.

In this way we have that
\begin{align*}
\norm{y(s_0+iT_*)}{H}^2&\le\norm{y(s_0)}{H}^2,&\quad&\mbox{for all}\quad i\in\N, \\
\norm{y(s_0+iT_*+\tau)}{H}^2&\le2\left(\ex^{\frac{1}{2}}+1\right)\norm{y(s_0)}{H}^2,&\quad&\mbox{for all}\quad  \tau\in[0,T_*],\\
\norm{\widetilde\eta}{L^2((s_0+iT_*,s_0+(i+1)T_*), H)}^2&\le\textstyle\frac{1}{T_*}\ex^\frac{1}{2}\norm{y(s_0)}{H}^2,&\quad&\mbox{for all}\quad i\in\N.
\end{align*} 

Finally, we can conclude~\eqref{est_part_z} and~\eqref{est_part_eta} from the fact that $z=\ex^{-\frac{\bar\lambda}{2}(\Bigcdot-s_0)}y$ solves~\eqref{sys-z2cf} in~$\R_{s_0}$
with~$\chi=1$ and~$\eta^{\bar\lambda,a,b}=\ex^{-\frac{\bar\lambda}{2}(\Bigcdot-s_0)}\widetilde\eta$.
\end{proof}

Also, concerning the example~\ref{ex.2} above, again in the case~$\omega=\Omega$ and~$\chi=1$, an estimate similar to~\eqref{Msimplecase} will follow,
for example, in the case $\Omega$ is convex and we take piecewise constant controls related to a
partition of~$\Omega=\bigcup_{i=1}^M\Omega_j$ into suitable small convex sub-domains~$\Omega_i$.
Now, if $D$ is the diameter of~$\Omega$ we can
cover~$\Omega$ with~$n^d$ copies (obtained by suitable translations) of the rectangle~$\prod_{i=1}^d[0,l]$, with~$l=\frac{D}{n}$, and whose diameter is~$d^\frac{1}{2}l$.
Taking the intersection of these
rectangle copies with~$\Omega$ we still have that
we can cover~$\Omega$ with~$M\le n^d$ convex domains~$\{\Omega_i\mid i\in\{1,\,2,\,\dots,\,M\}\}$,
each with diameter not bigger than~$d^\frac{1}{2}l$. Proceeding as in the proof of Theorem~\ref{T:part}
we arrive to the sufficient condition~\eqref{cond.part1} for stabilization
\[
\norm{(1-P_M)}{\LL(H,V')}^{-2}\ge 4D_{\rm rc}\ex^{1}\norm{(a-\textstyle\frac{\bar\lambda}{2},b)}{\WW}^2.
\]

Now we follow~\cite[Section~IV.A]{KroRod-ecc15} and~\cite[Section~2.3, Example~2.15]{BreKunRod-pp15} to compute~$\norm{(1-P_M)}{\LL(H,V')}^{-2}$.
Let us set $\mathcal C
=\left\{\Psi_i=\frac{1}{\norm{1_{\Omega_i}}{H}}1_{\Omega_i}\mid
i\in\{1,\,2,\,\dots,\,M\right\}\in H$, $\omega=\Omega$, and $\chi=1$.
For given $v\in V$ and $z\in H$ we find that
\begin{align*}
&\quad\left(1_{\omega}(1-P_M)1_{\omega}z,\,v\right)_{V',\,V}=\left((1-P_M)z,\,v\right)_{V',\,V}
=\left(z,\,(1-P_M)v\right)_{H}\\
&=\Big(z,\,v-\textstyle\sum\limits_{i=1}^M(v,\,\Psi_i)_{H}\Psi_i\Big)_{H}
=\sum\limits_{i=1}^M\left(z\rest{\Omega_i},\varphi_i\right)_{L^2(\Omega_i)},
\end{align*}
where $\varphi_i\coloneqq
v\rest{\Omega_i}-(v\rest{\Omega_i},\,\Psi_i\rest{\Omega_i})_{L^2(\Omega_i)}\Psi_i\rest{\Omega_i}
=v\rest{\Omega_i}-\frac{1}{\norm{1}{L^2(\Omega_i)}^2}(v\rest{\Omega_i},\,1)_{L^2(\Omega_i)}$
has zero
average in~$\Omega_i$. This allows to conclude
that~$\norm{\nabla\varphi_i}{L^2(\Omega_i)^n}^2\ge
\mu_{1,i}\norm{\varphi_i}{L^2(\Omega_i)}^2$, where~$\mu_{1,i}$ is the first nonzero eigenvalue
$\mu_{1,i}$ of the Neumann Laplacian in~$\Omega_j$. Recall that in this case,
from~\cite[Equation~(1.9)]{PayneWein60}, we know that~$\mu_{1,i}\ge\frac{\pi^2}{l^2d}$.
Thus, we find for $z\in H$ and $v\in V$ with
$\norm{z}{H}=1$, $\norm{v}{V} =1$ the estimates
\begin{equation*}
\left(1_{\omega}(1-P_M)1_{\omega}z,\,v\right)_{V',\,V}
\le \textstyle\sum\limits_{i=1}^M
\mu_{1,i}^{-\frac{1}{2}}\norm{z\rest{\Omega_i}}{L^2(\Omega_i)}\norm{\nabla
v\rest{\Omega_i}}{L^2(\Omega_i)}
\le d^\frac{1}{2}l\pi^{-1}.
\end{equation*}
Therefore $\norm{(1-P_M)}{\LL(H,V')}^{2}\le l^2\pi^{-2}d$, and we can conclude that
\[
l^{-2}\ge 4\pi^{-2}dD_{\rm rc}\ex^{1}\norm{(a-\textstyle\frac{\bar\lambda}{2},b)}{\WW}^2
\]
is a sufficient condition for stabilization. That is, the~$M\le n^d=D^d l^{-d}$ piecewise constant actuators in~$\mathcal C$ above are able to stabilize the system.
That is, it is enough to take
\[
 M\ge (\textstyle\frac{4D_{\rm rc}\ex^{1}}{\pi^2})^\frac{d}{2}D^d d^\frac{d}{2}\norm{(a-\textstyle\frac{\bar\lambda}{2},b)}{\WW}^d
\]
piecewise constant actuators as above to stabilize the system.

\subsection{Feedback stabilizing rule and Riccati equation}
From Theorem~\ref{T:stab-v-gen} we know  that system~\eqref{sys-z2cf} is stabilizable with rate~$\frac{\bar\lambda}{2}$, provided~\eqref{M_bound-gen} holds true.
Here we recall that in that case the control can be taken in feedback form, i.e.
\begin{align*}
\eta=\FF_{\bar\lambda}(t)z=B_M^*\Pi_{\bar\lambda}(t)z,
\end{align*}
with
\begin{equation}\label{eq:kk2}
B_M\coloneqq1_\omega\chi  P_M1_\omega.
\end{equation}

Actually, the procedure is standard. We have just to follow the arguments in~\cite{BarRodShi11,BreKunRod-pp15}, by considering suitable minimization problems to conclude the following results.

\begin{lemma}\label{L:wot}
 Let $(\mathcal M,\mathcal R) =(1,1)$ or $(\mathcal M,\mathcal R) =
((-\Delta)^\frac{1}{2},1)$, where~$1$ here stands for identity operator. Then there exists a function $\Pi\colon s\mapsto \Pi(s)$, $s\ge0$, which belongs to
 \[
 \mathcal P\coloneqq\left\{P\in L^{\infty}(\mathbb R_0,\mathcal L(H))
 \left|\begin{array}{l}\hspace*{-0.4em}
                     P(t)\mbox{ is self-adjoint positive definite for all }t\ge0,\\
                     \hspace*{-0.4em}\mbox{the family } \{P(t)\mid t\ge0\}\mbox{ is continuous in the}\\
                     \hspace*{-0.4em}\mbox{weak operator topology}
                    \end{array}\right.\hspace*{-0.6em}\right\}\hspace*{-0.2em}
\]
 and satisfies the differential Riccati equation
 \begin{equation}\label{eq:riccati}
 \dot \Pi+\Pi\mathbb A^{a,b}+{\mathbb A^{a,b}}^*\Pi -\Pi  B_M  \mathcal R ^{-1}
B_M^* \Pi +\bar\lambda\Pi+\mathcal M^*\mathcal M=0,
\end{equation}
with $\mathbb A^{a,b}z\coloneqq \Delta z-
az- \nabla\cdot(bz)$.
Moreover, $\Pi=\Pi_{\bar\lambda}$ is the unique solution of~\eqref{eq:riccati} in the class~$\mathcal P$. Furthermore there is a
constant~$\overline C_{\left[C_{\mathcal W},\bar\lambda,\frac{1}{\bar\lambda}\right]}$ such that
\begin{equation}\label{costPis}
\norm{\Pi_{\bar\lambda}(s)}{\mathcal L(H)}\le
 \overline C_{\left[C_{\mathcal W},\bar\lambda,\frac{1}{\bar\lambda}\right]},\quad\mbox{for all}\quad s\ge0.
\end{equation}
\end{lemma}

We recall that the obtained feedback control is the one that minimizes the cost function
\begin{equation}\label{J_cost}
 J(\eta)=J(z(\eta),\eta)\coloneqq\int_{s_0}^{+\infty}\norm{\MM z(\tau)}{H}^2 +(\RR\eta(\tau),\eta(\tau))_H\ed\tau.
\end{equation}

Let us now consider the closed-loop system
\begin{subequations}\label{sys-tilzfeed}
\begin{align}
  &\partial_t \tilde z-\Delta \tilde z+ ( a-\textstyle\frac{\bar\lambda}{2}) \tilde z+ \nabla\cdot (b\tilde z)+B_M\RR^{-1}B_M^*\Pi_{\bar\lambda} \tilde z=0,\\
  &\tilde z\rest \Gamma =0, \qquad \tilde z(s_0)=z_0.
\end{align}
\end{subequations}

\begin{theorem}\label{T:feed}
 Let  $\chi$ and $P_M$ satisfy the conditions  in Theorem~\ref{T:stab-v-gen}, let $(\mathcal M,\mathcal R) =(1,1)$ or $(\mathcal M,\mathcal R) =
((-\Delta)^\frac{1}{2},1)$,
 and let~$z_0\in H$.
 Then the solution~$\tilde z$ for~\eqref{sys-tilzfeed}
 is defined for all $t\ge s_0$, and it satisfies
\begin{equation}\label{est-yfeedH}
  \norm{\tilde z}{W(\mathbb R_{s_0},V,V')}^2\le\overline C_{\left[C_{\mathcal W},\bar\lambda,\frac{1}{\bar\lambda}\right]}\norm{z_0}{H}^2.
\end{equation}
\end{theorem}

Notice that that~$\tilde z$ solves~\eqref{sys-tilzfeed} if, and only if, $z=\ex^{-(\Bigcdot-s_0)\frac{\bar\lambda}{2}}\tilde z(t)$ solves
\begin{subequations}\label{sys-zfeed}
\begin{align}
  &\partial_t z-\Delta z+  a z+ \nabla\cdot (b z)+ B_M  \mathcal
R ^{-1}B_M^*\Pi_{\bar\lambda} z=0,\\
  &z\rest \Gamma =0, \qquad z(s_0)=z_0.
\end{align}
\end{subequations}
Therefore we can conclude the next result.
\begin{corollary}\label{C:feedz} Under the assumptions of  Theorem~\ref{T:feed}, let $\Pi_{\bar\lambda} \in {\mathcal {P}}$ be the unique solution
of~\eqref{eq:riccati}. Then for any ~$z_0\in H$, the solution~$z$ of~\eqref{sys-zfeed}
 is defined globally and satisfies, for all $t\ge s_0$,
\begin{equation*}
 \ex^{(t-s_0)\bar\lambda}\norm{z(t)}{H}^2+\int_{s_0}^t \ex^{(\tau-s_0)\bar\lambda}(\norm{z(\tau)}{V}^2
  +\norm{\partial_tz(\tau)}{V'}^2)\,{\rm d}\tau\le\overline C_{\left[C_{\WW,\rm st},\bar\lambda,\frac{1}{\bar\lambda}\right]}\norm{z_0}{H}^2.
\end{equation*}
\end{corollary}

\subsection{Dependence of the transient bound on the exponential rate}\label{sS:transient}
From Theorem~\ref{T:stab-v-gen} we have that, for a suitable open-loop control~$\eta$, the solution of system~\eqref{sys-z2cf} satisfies
\begin{subequations}
\begin{equation}\label{trans_bd1_z}
 |z(t)|_{H}^2
\le C_{\bar\lambda}\ex^{-\bar\lambda (t-s_0)}\norm{z_0}{H}^2, \mbox{ for }t\ge s_0,
\end{equation}
where
\begin{equation}\label{trans_bd1}
 C_{\bar\lambda}\le \ex^{\TT_{0,2}+\TT_{1,2}{\bar\lambda}^2},
\end{equation}
for suitable positive
constants $\TT_{0,2}$ and~$\TT_{1,2}$, which can be taken independent of~$\bar\lambda$. Taking, for example, $\widehat\lambda=\frac{\bar\lambda}{2}$, again by Theorem~\ref{T:stab-v-gen}, we have that
$J(z(\eta),\eta)\le \widetilde C_{\bar\lambda}\ex^{-\bar\lambda (t-s_0)}\norm{z_0}{H}^2$, with $\widetilde C_{\bar\lambda}\le\ex^{\widetilde \TT_{0,2}+\widetilde \TT_{1,2}\bar\lambda^2}$
and~$J$ as in~\eqref{J_cost}.
In particular, the solution of the closed-loop system~\eqref{sys-tilzfeed} also satisfies
\begin{equation}\label{trans_bd1-f}
 |\tilde z(t)|_{H}^2\le\widehat C_{\bar\lambda}\norm{z_0}{H}^2,\qquad \widehat C_{\bar\lambda}\le \ex^{\widehat \TT_{0,2}+\widehat \TT_{1,2}\bar\lambda^2},
\end{equation}
\end{subequations}

\begin{remark}
In the case $a\in L^\infty(I,L^\infty)$ we have a slightly different observability inequality (cf.~\cite{DuyckZhangZua08}), and we can also prove that we can take 
$\Upsilon_0=\ex^{(A_1\bar\lambda+B_2)T_*}$ for suitable constants $A_1$ and $A_2$, in Theorem~\ref{T:stab-v-gen}. In this case we would arrive to an estimate 
in the form~$C_{\bar\lambda}\le \ex^{\TT_{0,1}+\TT_{1,1}\bar\lambda}$. Notice that in either case we have a
term~${\bar\lambda}^e$ with $e\ge1$. 
\end{remark}

Of course, the constants~$C_{\bar\lambda}$, $\widetilde C_{\bar\lambda}$, and $\widehat C_{\bar\lambda}$ above are related to each other.

We will prove that we can have~\eqref{trans_bd1} with~$C_{\bar\lambda}\le \ex^{\TT_{0,\frac{1}{2}}+\TT_{1,\frac{1}{2}}\bar\lambda^\frac{1}{2}}$.
The motivations of this study are the following:%
\begin{itemize}
\item a sufficient condition, on the smallness of~$\frac{\widehat  C_{\bar\lambda}}{\bar\lambda}$, 
given in~\cite{BreKunRod-pp15} for the stabilizability of the
FitzHugh--Nagumo system.
\item we need to know the constant~$\widehat  C_{\bar\lambda}$ to guarantee that at time~$t=\tau>s_0$
we will be closer to zero than we are at time~$t=s_0$, because from $\norm{\tilde z(\tau)}{H}^2\le C_{\bar\lambda}\ex^{-\bar\lambda(\tau-s_0)}\norm{\tilde z(s_0)}{H}^2$ it follows that
\[
\norm{\tilde z(\tau)}{H}^2<\norm{\tilde z(s_0)}{H}^2\quad\mbox{if}\quad \tau>s_0+\textstyle\frac{\log(\widehat  C_{\bar\lambda})}{\bar\lambda}. 
\]
\item as we will see, the value~$\epsilon$ in~\eqref{goal}
will decrease as~$\widehat  C_{\bar\lambda}$ increases,
that is the linearization based feedback control
will work, for the nonlinear system, in a bigger neighborhood for a smaller~$\widehat  C_{\bar\lambda}$ (cf.~Remark~\ref{R:trans_local} below). 
\end{itemize}
\begin{remark} Notice also that necessarily~$\widehat  C_{\bar\lambda}\ge1$. Then~$\frac{\widehat  C_{\bar\lambda}}{\bar\lambda}\to+\infty$ as~$\bar\lambda\to0$. 
It follows that, if~$\widehat  C_{\bar\lambda}\approx\ex^{\widehat  \TT_{0,\frac{1}{2}}+\widehat  \TT_{1,\frac{1}{2}}\bar\lambda^\frac{1}{2}}$ for big~$\bar\lambda$,
then the function~$\frac{\widehat  C_{\bar\lambda}}{\bar\lambda}$ has a strictly positive minimum. That is,
$\frac{\widehat  C_{\bar\lambda}}{\bar\lambda}$ cannot be made arbitrarily small by the choice of~$\bar\lambda$.
\end{remark}

Following~\cite{HinrichsenPlisPrit01}, since associated to a closed-loop system we will call
the constant $\widehat C_{\bar\lambda}$ the {transient bound} associated to that system. The following Theorem is, however, concerned with open-loop controls and with no restriction on
the dimension of the control.
\begin{theorem}
There exist $T_{**}>0$ and nonnegative constants~$\TT_{0,\frac{1}{2}}$, and $\TT_{1,\frac{1}{2}}$, such that for a suitable control function~$\eta$,
the solution of system~\eqref{sys-z2c}, with $B=1_\omega$, satisfies 
for all~$z_0\in H$
\begin{equation}\label{trans_bd3}
 |z(t)|_{H}^2
\le \ex^{\TT_{0,\frac{1}{2}}+\TT_{1,\frac{1}{2}}\bar\lambda^\frac{1}{2}}\ex^{-\bar\lambda (t-s_0)}\norm{z_0}{H}^2, \mbox{ for }t\ge s_0,
\end{equation}
and~$z(t)=0$ for all~$t\ge s_0+T_{**}$.
\end{theorem}

\begin{proof}
 Given $T>0$ and taking $s_1=s_0+T$, by Lemma~\ref{L:defXi} (extending the control by zero for $t>s_1$) we have
 \[
  \begin{split}
 \norm{z(t)}{H}^2&\le \ex^{D_{\rm rc}\norm{ (a,b)}{\WW^{I}}^2 T}\norm{z(s_0)}{H}^2\\
 &\quad+2\norm{\iota}{\LL(H,V')}^2\ex^{D_{\rm rc}\norm{ (a,b)}{\WW^{I}}^2
 T+\widehat D\Theta\left(T,\norm{a}{L^\infty(\R_0,L^d)},\norm{b}{L^\infty_w(\R_0,L^\infty)},d\right)}\norm{z(s_0)}{H}^2
 \end{split},
 \]
 for all $t> s_0$, and $z(t)=0$ for all~$t\ge s_0+T$. Thus we can write
\[
 \norm{z(t)}{H}^2\le\ex^{\rho_0(1+\frac{1}{T})+\rho_1(1+\bar\lambda)T}\norm{z(s_0)}{H}^2
\]
for suitable nonnegative constants~$\rho_0$ and $\rho_1$ (depending on $\norm{(a,b)}{\WW}$  and~$d$, but neither on~$T$ nor on~$\bar\lambda$).

Taking the minimizer~$T_{**}>0$ of~$\ex^{\rho_0(1+\frac{1}{\Bigcdot})+\rho_1(1+\bar\lambda)\Bigcdot}$ is given by~$\frac{\rho_0}{T_{**}^2}=\rho_1(1+\bar\lambda)$, we find
 \[
  |z(t)|_{H}^2
\le \ex^{\rho_0+2(\rho_0\rho_1)^\frac{1}{2}(1+\bar\lambda)^\frac{1}{2}}\ex^{-\bar\lambda (t-s_0)}\norm{z_0}{H}^2
\le\ex^{\rho_0+2(\rho_0\rho_1)^\frac{1}{2} +2(\rho_0\rho_1)^\frac{1}{2}\bar\lambda^\frac{1}{2}}\ex^{-\bar\lambda (t-s_0)}\norm{z_0}{H}^2,
 \]
for all $t\ge s_0$, and $z(t)=0$ for all~$t\ge s_0+T_{**}$.
\end{proof}

\subsection{Remark on the viscosity coefficient}
In section~\ref{S:red-stabil} we have rescaled time to normalize the viscosity parameter. We recall that the dimension of the stabilizing control may depend on the viscosity
parameter. If we are given the system
\[
\p_t z- \nu\Delta  z +  a z + \nabla \cdot \left( bz \right)+\sum_{i=1}^M u_i\Phi_i
  = 0;\quad
  z\rest\Gamma = 0.
\]
after rescaling time~$t=\frac{\tau}{\nu}$ we arrive to system
\[
\p_\tau \breve z- \Delta\breve  z + \frac{\breve  a}{\nu}\breve  z + \nabla \cdot \left( \frac{\breve b}{\nu}\breve z \right)+\sum_{i=1}^M \breve u_i\Phi_i
  = 0;\quad
 \breve z\rest\Gamma = 0.
\]
with~$\breve f(\tau)\coloneqq f(\frac{\tau}{\nu})$. Then if we want to have
\[
\left| z(t)\right|^2_{H}
\le C \mathrm{e}^{-\bar\lambda t}  \left| z(0)  \right|^2_{H},
\]
then we must have
\[
\left| \breve z(\tau)\right|^2_{H}
\le C \mathrm{e}^{-\frac{\bar\lambda}{\nu}\tau}  \left| z(0)  \right|^2_{H}.
\]
From Theorem~\ref{T:stab-v-gen}, condition~\eqref{M_bound-gen} is sufficient for stabilization. That condition reads now
\[
\begin{split}
\norm{1_\omega\chi(1- P_M)1_\omega}{\mathcal L(H, V')}
&\le 2\norm{\iota}{\LL(H,V')}^{-2}\ex^{-\overline\Theta\left(\norm{ \frac{\breve a}{\nu}-\frac{\bar\lambda}{2\nu}}{L^\infty(I,L^d)},
\norm{\frac{\breve b}{\nu}}{L^\infty_w(I,L^\infty)},\norm{ (\frac{\breve a}{\nu}-\frac{\bar\lambda}{2\nu},\frac{\breve b}{\nu})}{\WW^{I}},d,\nu\right)}\\
&=2\norm{\iota}{\LL(H,V')}^{-2}\ex^{-\overline\Theta\left(\norm{ \frac{a}{\nu}-\frac{\bar\lambda}{2\nu}}{L^\infty(I,L^d)},
\norm{\frac{b}{\nu}}{L^\infty_w(I,L^\infty)},\norm{ (\frac{a}{\nu}-\frac{\bar\lambda}{2\nu},\frac{b}{\nu})}{\WW^{I}},d,\nu\right)},
\end{split}
\]
with~$\overline\Theta$ as in~\eqref{barTheta}. Therefore, $P_M$ must be closer to the identity operator for smaller~$\nu$. This is natural, because for smaller~$\nu$
the system is ``less stable''. In other words, in order to stabilize the system, we expect to need more controls as~$\nu$ decreases. Recall the
estimates~\eqref{estM_eig} and~\eqref{estM_pc} in the examples in section~\ref{sS:examples_dimM}.

\section{Local internal stabilization of the nonlinear system}\label{S:nonlinear_int}
Here we show that the feedback rule ~$-{B_M}{\mathcal R}^{-1}B_M^*\Pi_{\bar\lambda}(t)$
constructed to stabilize exponentially
the linear system~\eqref{sys-zfeed} to zero, with rate~$\frac{\bar\lambda}{2}$, also stabilizes the
nonlinear system%
\begin{subequations}\label{sys-zfeed-nonl}
\begin{align}
  &\partial_t z-\Delta z+  a z+ \nabla\cdot (b z)+ B_M  \mathcal
R ^{-1}B_M^*\Pi_{\bar\lambda} z=\NN(z),\\
  &z\rest \Gamma =0, \qquad z(s_0)=z_0.
\end{align}
\end{subequations}
to zero, with the same rate, provided~${z}_0$ is small enough (cf.~system\eqref{sys-z-int}), and provided the nonlinearity~$\NN$ satisfies suitable boundedness properties.

We follow a standard procedure by using a suitable fixed point argument. To deal with the nonlinear system we will need to ask
more regularity for the reference trajectory, or more precisely to the function~$b$ we obtain
in the convection part of the linearized system~\eqref{sys-z2-eq}, after linearization around that reference trajectory.

This additional regularity will allow us to construct appropriate Banach spaces where we can construct an appropriate contraction mapping, whose fixed point
is the solution of the nonlinear system~\eqref{sys-zfeed-nonl}.

\smallskip\noindent
{\em Implicit boundedness assumption on the the reference trajectory.}
Consider the spaces
\begin{subequations}\label{Wwspace-st}
\begin{align}\label{Wwspace-st1-3}
&\begin{array}{rcl}
 \WW^{J}_{\rm st}&\coloneqq& \left\{b\in \WW^{J}\mid \nabla\cdot b\in L^\infty(J,L^2(\Omega,\R))\right\},
\\
 \WW_{\rm st}&\coloneqq&\left\{b\in \WW\mid \nabla\cdot b\in L^\infty(\R_0,L^2(\Omega,\R))\right\},
 \end{array}\qquad d\in\{1,2,3\},
\\
\label{Wwspace-st4-}
&\begin{array}{rcl}
 \WW^{J}_{\rm st}&\coloneqq& \left\{b\in \WW^{J}\mid \nabla\cdot b\in L^\infty_w(J,L^\infty(\Omega,\R))\right\},
\\
 \WW_{\rm st}&\coloneqq&\left\{b\in \WW\mid \nabla\cdot b\in L^\infty_w(\R_0,L^\infty(\Omega,\R))\right\},
 \end{array}\qquad d\ge4,
 \end{align}
\end{subequations}
 (cf.~\eqref{Wwspace}) which are supposed to be endowed, respectively, with the norms
\[
\begin{array}{rcl}
 \norm{(a,b)}{\WW^{J}_{\rm st}}&\coloneqq& \left((\norm{(a,b)}{\WW^{J}}^2+\norm{\nabla\cdot b}{L^\infty(J,L^r)}^2\right)^\frac{1}{2};
\\
\norm{b}{\WW_{\rm st}}&\coloneqq& \left((\norm{(a,b)}{\WW}^2+\norm{\nabla\cdot b}{L^\infty_w(\R_0,L^r)}^2\right)^\frac{1}{2}.
  \end{array}
\]
with $r=2$ for $d\in\{1,2,3\}$, and~$r=\infty$ for $d\ge4$.
The spaces~\eqref{Wwspace-st} are defined essentially for us to be able to derive the existence of strong solutions.
Notice that for $1< r<+\infty$ we have $L^\infty_w(\R_0,L^r)=L^\infty(\R_0,L^r)$, because~$L^r$ and~$L^{r_*}$, with $\frac{1}{r}+\frac{1}{r_*}=1$, are both separable.
See~\cite[Lemma~9.1.2]{Fattorini99} and~\cite[Section~4.1]{Fattorini05}.

Now, we have the following estimates for the convection term
\[
\norm{\nabla\cdot(bz)}{H}=\norm{(\nabla\cdot b)z + b\cdot\nabla z}{H}\le\norm{(\nabla\cdot b)}{L^2}\norm{z}{L^\infty} + \norm{b}{L^\infty}\norm{z}{V} 
\]
and, by using the Agmon inequality (cf.~\eqref{Agmon_d=1} and~\cite[chapter~II, section~1.4]{Temam97}),
 \begin{align*}
 \norm{\nabla\cdot(bz)}{H}&\le C\norm{(\nabla\cdot b)}{L^2}\norm{z}{H}^\frac{1}{2}\norm{z}{V}^\frac{1}{2}+\norm{b}{L^\infty}\norm{z}{V},&&\mbox{for}\quad d=1.\\
 \norm{\nabla\cdot(bz)}{H}&\le C\norm{(\nabla\cdot b)}{L^2}\norm{z}{H}^\frac{1}{2}\norm{z}{\D(\Delta)}^\frac{1}{2}+\norm{b}{L^\infty}\norm{z}{V},&&\mbox{for}\quad d=2.\\
 \norm{\nabla\cdot(bz)}{H}&\le C\norm{(\nabla\cdot b)}{L^2}\norm{z}{V}^\frac{1}{2}\norm{z}{\D(\Delta)}^\frac{1}{2}+\norm{b}{L^\infty}\norm{z}{V},&&\mbox{for}\quad d=3.
 \end{align*}

 For $d\ge4$, the Agmon inequality does not allow us to bound the $L^\infty$-norm by the~$\D(\Delta)$-norm. This is the reason we (need to) take different spaces in~\eqref{Wwspace-st4-}.
Notice that 
\begin{align*}
   \norm{\nabla\cdot(bz)}{H}&  \le C\norm{\nabla\cdot b}{L^\infty}\norm{z}{H}+\norm{b}{L^\infty}\norm{z}{V},&&\mbox{ for all}\quad d\ge1.
 \end{align*}
where~$C>0$ is a positive constant.

\begin{remark}
Considering~\eqref{Wwspace-st} we will be able to treat a wide class of nonlinearities.
However, it may happen that we do not need to deal with strong solutions (see for example the case of 1D Burgers
equation considered in~\cite{KroRod15}),  or it may happen that that we do not need to ask extra regularity for $(a,b)$ to have strong solutions
(see for example the case the convection term takes the form~$b\cdot\nabla z$ as in~\cite{BreKunRod-pp15}, suitable to deal with Neumann boundary conditions).
That is, depending on the particular system we are dealing with, we may consider different (less regular) spaces~\eqref{Wwspace-st}.
\end{remark}

\smallskip\noindent
{\em Boundedness assumption on the nonlinearity.}
We suppose the nonlinearity $\NN(z)$ in system~\eqref{sys-zfeed-nonl} satisfies for a suitable constant~$\widehat C\ge0$,
and all~$(z,\tilde z)\in\D(\Delta)\times\D(\Delta)$, the estimates
\begin{align}
|\NN(z)-\NN(\tilde z)|_{H}^2&\le \widehat C
|z-\tilde z|_V^{2}\left(1+\norm{z}{V}^{\e_1}
+\norm{\tilde z}{V}^{\e_2}\right)\left(\norm{z}{\D(\Delta)}^{2}+\norm{\tilde z}{\D(\Delta)}^{2}\right)\notag\\
&\quad+\widehat C|z-\tilde z|_{{\mathrm D}(A)}^2\left(|{z}|_{V}^{\e_3}
+|{\tilde z}|_{V}^{\e_4}\right),\label{exp-non2}
\end{align}
with $\{\e_1,\e_2\}\in[0,+\infty)$, and $\{\e_3,\e_4\}\in[2,+\infty)$.
Notice that~\eqref{exp-non2} implies that
\begin{align}
&\quad\left(\NN(z)-\NN(\tilde z),z-\tilde z\right)_H\notag\\
&\le 2\widehat C\left(|{z}|_{V}^{\e_3}
+|{\tilde z}|_{V}^{\e_4}+\left(|z|_V^{2}+|\tilde z|_V^{2}\right)\left(1+\norm{z}{V}^{\e_1}
+\norm{\tilde z}{V}^{\e_2}\right)\right)\left(\norm{z}{\D(\Delta)}^{2}+\norm{\tilde z}{\D(\Delta)}^{2}\right)\norm{z-\tilde z}{H}^2\notag\\
&\le \widehat C_1\left(1+\norm{z}{V}^{\e_5}
+\norm{\tilde z}{V}^{\e_6}\right)\left(\norm{z}{\D(\Delta)}^{2}+\norm{\tilde z}{\D(\Delta)}^{2}\right)\norm{z-\tilde z}{H}^2,
\label{exp-non3}
\end{align}
for suitable $\{\e_5,\e_6\}\in[2,+\infty)$.

\subsection{Strong solutions for the linearized systems}
We fix~$a$ and~$b$, which may depend on time and space, and a constant~$C_{\WW,\rm st}\ge 0$, satisfying
\begin{equation}\label{rhosigma-st}
\norm{a}{\WW}+\norm{b}{\WW^d_{\rm st}}\le C_{\WW,\rm st}.
\end{equation}

We present now some results, whose proofs can be done by following the arguments in~\cite[section~2]{BreKunRod-pp15}.

By multiplying~\eqref{sys-z2-eq} by~$-2\Delta z$, instead of by~$2z$ as in the proof of Lemma~\ref{L:weak-z}, we can also obtain the following.
\begin{lemma}\label{L:strong-z}
 Given $f\in L^2(I,H)$ and $z_0\in V$, there exists a strong
 solution
 $z\in W(I,\mathrm D(A),H)$ for system~\eqref{sys-z2}, which depends continuously on the data:
 \[
 \norm{z}{W(I,\mathrm D(A),H)}^2\le \overline C_{\left[|I|,C_{\mathcal W,\rm st}\right]}
\left(\norm{z(s_0)}{V}^2+\norm{f}{L^2(I,H)}^2\right).
 \]
\end{lemma}

The next lemma shows a certain smoothing property of system~\eqref{sys-z2}.
\begin{lemma}\label{L:smooth}
 Given $f\in L^2(I,H)$ and $z_0\in H$, let~$z$ be the weak solution
 for system~\eqref{sys-z2}. Then
 $y(t)\coloneqq (t-s_0)z(t)$ is in~$W(I,\mathrm D(A),H)$ and satisfies the estimates
 \begin{align*}
 &\quad\,\norm{y}{W(I,\mathrm D(A),H)}^2\le
 \overline C_{\left[|I|,C_{\mathcal W,\rm st}\right]}
 \left((s_1-s_0)^2\norm{f}{L^2(I,H)}^2+\norm{z}{L^2(I,H)}^2\right)\\
 &\le \overline C_{\left[|I|,C_{\mathcal W,\rm st}\right]}
\left(\norm{z(s_0)}{H}^2+\norm{f}{L^2(I,H)}^2\right).
 \end{align*}
\end{lemma}

\begin{definition}\label{D:globalw}
 For $f\in L^2_{\rm loc}(\mathbb R_{s_0},H)$ and $y_0\in H$
 the function $z$ defined in $\mathbb R_{s_0}\times\Omega$ by the property that~$z\rest{(s_0,\tau)}$
 coincides with the weak solution
 of~\eqref{sys-z2} in $(s_0,\tau)$, for all $ \tau>s_0$ is well defined. It is called the global weak solution
 of~\eqref{sys-z2} in $\mathbb R_{s_0}\times\Omega$.
\end{definition}

We have the following property for the solutions
of~\eqref{sys-z2} on the infinite time interval~$\mathbb R_{s_0}=(s_0,+\infty)$, $s_0\ge0$.

\begin{lemma}\label{L:global.bdd}
 For $f\in L^2(\mathbb R_{s_0},V')$ and $z_{0}\in H$,
 let~$z$ be the global weak solution
 of~\eqref{sys-z2} in $\mathbb R_{s_0}$, with $z(s_0)=z_{0}$. If $z\in L^2(\mathbb R_{s_0},H)$, then $z\in W(\mathbb R_{s_0},V,V')$, and
 \begin{equation*}
 \norm{z}{W(\mathbb R_{s_0},V,V')}\le
 \overline C_{\left[C_{\mathcal W,\rm st}\right]}
 \left(\norm{z(s_0)}{H}^2+\norm{f}{L^2(\mathbb R_{s_0},V')}^2+\norm{z}{L^2(\mathbb R_{s_0},H)}^2\right).
 \end{equation*}
\end{lemma}

\medskip\noindent
{\em Strong solutions for the closed-loop linear system.}
The above Lemmas allow us to conclude the next result (cf.~\cite[Corollary~2.19]{BreKunRod-pp15}).
\begin{corollary}\label{C:feedz-st}
Under the assumptions of  Theorem~\ref{T:feed}, let $\Pi \in {\mathcal {P}}$ be the unique solution of \eqref{eq:riccati},
and let~$(a,b)$ satisfy~\eqref{rhosigma-st}. Then for any ~$z_0\in V$, the solution~$z$ of~\eqref{sys-zfeed}
 is defined globally and satisfies, for all $t\ge s_0$,
\begin{align*}
\norm{z(t)}{V}^2+\norm{z}{L^2((t,t+1),\mathrm D(A))}^2
  \le \overline C_{\left[C_{\WW,\rm st},\bar\lambda,\frac{1}{\bar\lambda}\right]}\ex^{-(t-s_0)\bar\lambda}\norm{z_0}{V}^2.
\end{align*}
\end{corollary}

\subsection{Fixed point argument}\label{sS:fixedpoint-int}
We start with the following result which follows from Corollary~\ref{C:feedz-st} and Lemma~\ref{L:strong-z}.
\begin{corollary}\label{C:stab_cond_VH}
 Under the assumptions of Corollary~\ref{C:feedz-st}, with $z_0\in V$, the solution~$z$ of~\eqref{sys-zfeed}
 is defined globally and satisfies,
\[
 \;\sup\limits_{t\ge s_0}
\norm{{\ex}^{\bar\lambda(\Bigcdot-s_0)}{z}(\Bigcdot)}{W((t,t+1),{\mathrm D}(A),H)}^2
\le \overline C_{\left[C_{\WW,\rm st},\bar\lambda,\frac{1}{\bar\lambda}\right]}\norm{z_0}{V}^2.
\]
\end{corollary}

Inspired from Corollary~\ref{C:stab_cond_VH}, taking $s_0=0$, we define the Banach space
\[
\mathcal{Z}^{\bar\lambda}  \coloneqq \left\{{z}\in L^2_{\rm
loc}\left(\mathbb{R}_0,H\right)
\,\Bigl|\,
\;|{z}|_{\mathcal{Z}^{\bar\lambda}  }<\infty
\right\}
\]
endowed with the norm
$|{z}|_{\mathcal{Z}^{\bar\lambda}  }\coloneqq \sup_{r\geq 0}\norm{{\ex}^{\bar\lambda\Bigcdot}{z}}{W((r,r+1),{\mathrm D}(A),H)}$.
We also set
\[\mathcal{Z}^{\bar\lambda}_{\rm loc}
 \coloneqq
\left\{{z}\in L^2_{\rm loc}\left(\mathbb{R}_0,H)\right)
\,\Bigl|\,
\;\norm{{\ex}^{\bar\lambda\Bigcdot}{z}}{W((r,r+1),{\mathrm D}(A),H)}<\infty,\;\mbox{for all }r\ge0
\right\}.
\]

For a given constant~$\varrho>0$ and
${z}_0\in V$ we define the subset
\[
\mathcal{Z}^{\bar\lambda}  _\varrho\coloneqq \left\{ z\in\mathcal{Z}^{\bar\lambda}  \mid\,
| z|_{\mathcal{Z}^{\bar\lambda}  }^2\leq \varrho| z_0|_{V}^2\right\},
\]
and the mapping
$\Psi\colon \mathcal{Z}^{\bar\lambda}_\varrho\to\mathcal{Z}^{\bar\lambda}_{\rm loc}$,
${\bar z}\mapsto z$, taking a given vector~${\bar z}$ to the solution~$ z$ of
\begin{subequations}\label{sys-z=Bz}
\begin{align}
  &\partial_t z-\nu\Delta z+  a z+ \nabla\cdot (b z)+ B_M  \mathcal
R ^{-1}B_M^*\Pi_{\bar\lambda} z=\NN (\bar z),\\
  &z\rest \Gamma =0, \qquad z(s_0)=z_0.
\end{align}
\end{subequations}

\begin{lemma}\label{L:contract}
Under the hypotheses of Corollary~\ref{C:stab_cond_VH}, there exists $\varrho>0$
such that the following property holds: for any $\gamma\in(0,1)$ one can find a
constant $\epsilon=\epsilon_\gamma>0$ such that, for any $ z_0$ satisfying
and~$| z_0|_{V}\le\epsilon$, the mapping~$\Psi$
takes the set~$\mathcal{Z}^{\bar\lambda} _\varrho$ into itself and satisfies the inequality
\begin{equation} \label{contraction}
|\Psi({\bar z}_1)-\Psi({\bar z}_2)|_{\mathcal{Z}^{\bar\lambda}  }\leq\gamma|{\bar z}_1-{\bar z}_2|_{\mathcal{Z}^{\bar\lambda}  }\quad
\mbox{for all}\quad {\bar z}_1,{\bar z}_2\in\mathcal{Z}^{\bar\lambda}  _\varrho.
\end{equation}
\end{lemma}

\begin{proof} We sketch the proof into~\ref{stcontr} main steps:%
%
\begin{enumerate}[noitemsep,topsep=5pt,parsep=5pt,partopsep=0pt,leftmargin=0em]%
\renewcommand{\theenumi}{{\sf\arabic{enumi}}} 
 \renewcommand{\labelenumi}{}

\item \textcircled{\bf s}~Step~\theenumi:\label{stprel} {\em a preliminary estimate.} Consider the system
\begin{subequations}\label{sys-z=f}
\begin{align}
  &\partial_t z-\nu\Delta z+  a z+ \nabla\cdot (b z)+ B_M  \mathcal
R ^{-1}B_M^*\Pi_{\bar\lambda} z=f,\\
  &z\rest \Gamma =0, \qquad z(0)=z_0.
\end{align}
\end{subequations}
where $f\in L_{\rm loc}^2(\mathbb{R}_0,H)$.
If $ z$ is the solution of system~\eqref{sys-z=f} with $f=0$,
by Corollary~\ref{C:stab_cond_VH},
\begin{equation}\label{estz0}
\sup_{r\geq 0}|{\ex}^{\bar\lambda \Bigcdot } z(\Bigcdot)|_{W((r,r+1),{\mathrm D}(A), H)}^2
\le \overline C_{\left[C_{\WW,\rm st},\bar\lambda,\frac{1}{\bar\lambda}\right]}
| z_0|_{V}^2.
\end{equation}

Proceeding as in~\cite{BarRodShi11,KroRod15,BreKunRod-pp15,Rod-pp15}, it follows that for nonzero $f$ we also have 
\begin{align}
\sup_{r\geq 0}|{\ex}^{\bar\lambda \Bigcdot } z(\Bigcdot)|_{W((r,r+1),{\mathrm D}(A), H)}^2\label{estz-f}
\le \overline C_{\left[C_{\WW,\rm st},\bar\lambda,\frac{1}{\bar\lambda}\right]}
\left(| z_0|_{V}^2
+\sup_{\begin{subarray}{l}k\in\mathbb{N}\end{subarray}}\int_k^{k+1}{\ex}^{{4\bar\lambda}  s}
|f(s)|_{H}^2\,{\mathrm d} s\right).
\end{align}

\item \textcircled{\bf s}~Step~\theenumi:\label{stinto} {\em $\Psi$ maps $\mathcal{Z}^{\bar\lambda}_\varrho$ into itself,
if $| z_0|_{V}$ is small.} We will replace $f$ by $\NN({\bar z})$ in~\eqref{estz-f}.
From~\eqref{exp-non2}, with~$(z,\tilde z)=(\bar z,0)$, we find that
\begin{align*}
&\quad\,\sup_{\begin{subarray}{l}k\in\mathbb{N}\end{subarray}}\int_k^{k+1}{\ex}^{{4\bar\lambda}  s}
|\NN(\bar z)(s)|_{H}^2\,{\mathrm d} s\\&\le
\sup_{\begin{subarray}{l}k\in\mathbb{N}\end{subarray}}\sup_{\begin{subarray}{l}s\in[k,k+1]\end{subarray}}
\widehat C \left(|{\ex}^{{\bar\lambda}  s}{\bar z}(s)|_{V}^{2}+|{\ex}^{{\bar\lambda} s}{\bar z}(s)|_{V}^{2+\e_1}+|{\ex}^{{\bar\lambda} s}{\bar z}(s)|_{V}^{\e_3}\right)
\int_k^{k+1}\hspace*{-.3em}
|{\ex}^{{\bar\lambda}  s}{\bar z}(s)|_{{\mathrm D}(A)}^2\,{\mathrm d} s\\
&\le C_1 \left(|{\bar z}|_{\mathcal{Z}^{\bar\lambda}}^{4}+|{\bar z}|_{\mathcal{Z}^{\bar\lambda}}^{4+\e_1}
+|{\bar z}|_{\mathcal{Z}^{\bar\lambda}}^{\e_3+2}\right),
\end{align*}
because $W((k,k+1),\D(\Delta),H)\xhookrightarrow{} C([k,k+1],V)$ uniformly with respect to $k\in\N$.
Thus,
inequality~\eqref{estz-f} with $f=\NN(\bar z)$ and $\bar z\in\mathcal{Z}^{\bar\lambda}_\varrho$ gives us
\begin{align*} 
&\quad\,|\Psi({\bar z})|_{\mathcal{Z}^{\bar\lambda}  }^2
\le \overline C_{\left[C_{\WW,\rm st},\bar\lambda,\frac{1}{\bar\lambda}\right]}
\left(| z_0|_{V}^2
+C_1 \left(|{\bar z}|_{\mathcal{Z}^{\bar\lambda}}^{4}+|{\bar z}|_{\mathcal{Z}^{\bar\lambda}}^{4+\e_1}
+|{\bar z}|_{\mathcal{Z}^{\bar\lambda}}^{\e_3+2}\right)\right)\notag\\
&\le C_2
\left(1+\varrho^2| z_0|_{V}^{2} +\varrho^\frac{4+\e_1}{2}| z_0|_{V}^{2+\e_1}
+\varrho^\frac{\e_3+2}{2}| z_0|_{V}^{\e_3}\right)| z_0|_{V}^2
\end{align*}
and if we set~$\varrho=4 C_2$ and
$
\epsilon<\min\left\{\varrho^{-1},\varrho^{-\textstyle\frac{4+\e_1}{2(2+\e_1)}}
,\varrho^{-\textstyle\frac{\e_3+2}{2\e_3}}\right\},
$
then we obtain
\begin{equation}\label{Psiz2}
|\Psi({\bar z})|_{\mathcal{Z}^{\bar\lambda}  }^2\le C_2
\left(1+\varrho^2\epsilon^{2} +\varrho^\frac{4+\e_1}{2}\epsilon^{2+\e_1}
+\varrho^\frac{\e_3+2}{2}\epsilon^{\e_3}\right)\le 4 C_2| z_0|_{V}^2=\varrho| z_0|_{V}^2 
\end{equation}
if $| z_0|_{V}\le\epsilon$,
which means that~$\Psi({\bar z})\in\mathcal{Z}^{\bar\lambda}_\varrho$.

%
\item \textcircled{\bf s}~Step~\theenumi:\label{stcontr} {\em $\Psi$ is a contraction, if $| z_0|_{V}$ is smaller.}
It remains to prove~\eqref{contraction}.
Let us take two functions ${\bar z}_1, {\bar z}_2\in \mathcal{Z}^{\bar\lambda}  _\varrho$ and let~$\Psi({\bar z}_1)$ and~$\Psi({\bar z}_2)$
be the corresponding solutions for~\eqref{sys-z=Bz}. Set $e={\bar z}_1-{\bar z}_2$ and
$d^\Psi=\Psi({\bar z}_1)-\Psi({\bar z}_2)$. Then~$d^\Psi$ solves~\eqref{sys-z=f} with $d^\Psi(0)=0$ and
$f=\NN({\bar z}_1)-\NN({\bar z}_2)$.
Therefore, by inequality~\eqref{estz-f}, we have
\begin{align*} 
&\quad\,|\Psi({\bar z}_1)-\Psi({\bar z}_2)|_{\mathcal{Z}^{\bar\lambda}  }^2\le
\overline C_{\left[C_{\WW,\rm st},\bar\lambda,\frac{1}{\bar\lambda}\right]}
\sup_{t\ge0}\int_t^{t+1}{\ex}^{{4\bar\lambda}  s}|\NN({\bar z}_1)(s)-\NN({\bar z}_2)(s)|_{H}^2\,{\ed} s,
\end{align*}
and from
\begin{align*}
&\quad\,{\ex}^{{4\bar\lambda}  s}|\NN({\bar z}_1)(s)-\NN({\bar z}_2)(s)|_{H}^2\\
&\le|\ex^{\bar\lambda s}e(s)|_V^{2}\left(1+\norm{\ex^{\bar\lambda s}\bar z_1(s)}{V}^{\e_1}
+\norm{\ex^{\bar\lambda s}\bar z_2(s)}{V}^{\e_2}\right)\left(\norm{\ex^{\bar\lambda s}\bar z_1(s)}{\D(\Delta)}^{2}+\norm{\ex^{\bar\lambda s}\bar z_2(s)}{\D(\Delta)}^{2}\right)\\
&\quad+|\ex^{\bar\lambda s}e(s)|_{{\mathrm D}(A)}^2\left(|{\ex^{\bar\lambda s}\bar z_1(s)}|_{V}^{\e_3}
+|\ex^{\bar\lambda s}{\bar z_2(s)}|_{V}^{\e_4}\right),
\end{align*}
it follows that
\begin{align*}
|\Psi({\bar z}_1)-\Psi({\bar z}_2)|_{\mathcal{Z}^{\bar\lambda}  }^2
\le  C_{3}
|e|_{\mathcal{Z}^{\bar\lambda}}^2\left(1+|{\bar z}_1|_{\mathcal{Z}^{\bar\lambda}}^{\e_1}
+|{\bar z}_2|_{\mathcal{Z}^{\bar\lambda}}^{\e_2}\right)\left( |{\bar z}_1|_{\mathcal{Z}^{\bar\lambda}}^2
+|{\bar z}_2|_{\mathcal{Z}^{\bar\lambda}}^2+|{\bar z}_1|_{\mathcal{Z}^{\bar\lambda}}^{\e_3}
+|{\bar z}_2|_{\mathcal{Z}^{\bar\lambda}}^{\e_4}\right),
\end{align*}
and since ${\bar z}_1$ and ${\bar z}_2$ are both in~$\mathcal{Z}^{\bar\lambda}_\varrho$, we arrive to
\begin{equation*}
|d^\Psi|_{\mathcal{Z}^{\bar\lambda}  }^2
\le  C_{3}
|e|_{\mathcal{Z}^{\bar\lambda}}^2\left(1+\varrho^\frac{\e_1}{2}|v_0|_V^{\e_1}
+\varrho^\frac{\e_2}{2}|v_0|_V^{\e_2}\right)\left( 2\varrho|v_0|_V^2
+\varrho^\frac{\e_3}{2}|v_0|_V^{\e_3}
+\varrho^\frac{\e_4}{2}|v_0|_V^{\e_4}\right),
\end{equation*}
Choosing $\epsilon>0$, smaller than the one in Step~\ref{stinto}, such that
\[
\epsilon<\min\left\{\varrho^{-1},\varrho^{-\textstyle\frac{4+\e_1}{2(2+\e_1)}}
,\varrho^{-\textstyle\frac{\e_3+2}{2\e_3}},\varrho^{-\frac{1}{2}},(\textstyle\frac{\gamma^2}{18C_3})^\frac{1}{2}\varrho^{-\frac{1}{2}}
,(\textstyle\frac{\gamma^2}{9C_3})^\frac{1}{\e_3}\varrho^{-\frac{1}{2}},(\textstyle\frac{\gamma^2}{9C_3})^\frac{1}{\e_4}\varrho^{-\frac{1}{2}}\right\},
\]
then we have that $\Psi$ maps $\mathcal{Z}^{\bar\lambda}_\varrho$ into itself and
\begin{align*}
|d^\Psi|_{\mathcal{Z}^{\bar\lambda}  }^2
\le  C_{3}
|e|_{\mathcal{Z}^{\bar\lambda}}^2\left(1+\varrho^\frac{\e_1}{2}\epsilon^{\e_1}
+\varrho^\frac{\e_2}{2}\epsilon^{\e_2}\right)\left( 2\varrho\epsilon^2
+\varrho^\frac{\e_3}{2}\epsilon^{\e_3}
+\varrho^\frac{\e_4}{2}\epsilon^{\e_4}\right)< C_3|e|_{\mathcal{Z}^{\bar\lambda}}^2 3\textstyle\frac{3\gamma^2}{9C_3},
\end{align*}
provided {$| z_0|_{V}^2\le\epsilon.$} That is~\eqref{contraction} holds true: $|\Psi({\bar z}_1)-\Psi({\bar z}_2)|_{\mathcal{Z}^{\bar\lambda}  }^2
<\gamma^2|{\bar z}_1-{\bar z}_2|_{\mathcal{Z}^{\bar\lambda}}^2$.
\end{enumerate}
The proof of Lemma~\ref{L:contract} is complete.
\quad\end{proof}

The following result says that the feedback control locally stabilizes exponentially the nonlinear system~\eqref{sys-zfeed-nonl},
is locally exponentially stable with rate~$\frac{\bar\lambda}{2}$.

\begin{theorem}\label{T:ex.st-ct}
Under the hypotheses of Corollary~\ref{C:stab_cond_VH}, there is~$\epsilon>0$ with the following property: if $|{z}_0|_{V}\le\epsilon$, then
there exists a solution for the system~\eqref{sys-zfeed-nonl}, in $\mathbb R_0\times\Omega$,
which belongs to~{$L^2_{\rm loc}(\mathbb{R}_0,{\mathrm D}(A))\cap C([0,+\infty),V)$}, is unique, and satisfies
\begin{equation} \label{expH1.main}
|{z}(t)|_{V} \le C\ex^{-\bar\lambda(t-s_0)}|{z}_0|_{V},\quad\mbox{for all }t\ge s_0,
\end{equation}
for a suitable constant~$C$ independent of~$(\epsilon,{z}_0)$.
\end{theorem}

\begin{proof}
From Lemma~\ref{L:contract} and the contraction
mapping principle it follows that if $ z_0\in V$ is sufficiently small,
$| z_0|_{V}<\epsilon$, then there exists a unique fixed point $ z=\Psi({\bar z})={\bar z}\in
\mathcal{Z}^\epsilon_\varrho$ for~$\Psi$. It follows from the
definitions of~$\Psi$ and~$Z^\epsilon_\varrho$ that~$z$
solves the system~\eqref{sys-z=Bz}, with ${\bar z}= z$.
We can conclude that~$ z$ solves~\eqref{sys-zfeed-nonl}. 

Further, inequality~\eqref{expH1.main} can be concluded from~\eqref{Psiz2}.

Finally, it remains to prove the uniqueness of the solution for~\eqref{sys-zfeed-nonl} in the space
$Z\coloneqq L^2_{\rm loc}(\mathbb{R}_0,{\mathrm D}(A)\times H)\cap C([0,+\infty),V\times H)\supset \mathcal Z^\epsilon_\varrho$.
Let~$ z_1$ and~$ z_2$ be two solutions, in~$Z$, for~\eqref{sys-zfeed-nonl}.
It turns out that~$e=z_1- z_2$ solves~\eqref{sys-z=f} with $f=\NN({z_1})-\NN({z_2})$,
Using~\eqref{costPis} and~\eqref{exp-non3}, we can obtain
\begin{align*}
\langle{B_M}{\mathcal R}^{-1}B_M^*\Pi_{\bar\lambda} e,e\rangle_{V',V}
&\le\overline C_{\left[C_{\WW,\rm st},\bar\lambda,\frac{1}{\bar\lambda}\right]}\norm{e_v}{H}^2,\\
\langle\mathcal{F}({z_1})-\mathcal{F}({z_2}),e\rangle_{V',V}
&\le \widehat C_1\left(1+\norm{z_1}{V}^{\e_5}
+\norm{z_2}{V}^{\e_6}\right)\left(\norm{z_1}{\D(\Delta)}^{2}+\norm{z_2}{\D(\Delta)}^{2}\right)\norm{e}{H}^2,
\end{align*}
from which we can derive, proceeding as in the proof of Lemma~\eqref{L:weak-z}, that
\begin{equation*}
 \frac{\rm d}{\rm d t}\norm{e}{H}^2\le \overline C_{\left[C_{\WW,\rm st},\bar\lambda,\frac{1}{\bar\lambda},\frac{1}{\nu}\right]}
 \left(1+\norm{z_1}{V}^{\e_5}
+\norm{z_2}{V}^{\e_6}\right)\left(\norm{z_1}{\D(\Delta)}^{2}+\norm{z_2}{\D(\Delta)}^{2}\right)\norm{e}{H}^2.
\end{equation*}
Notice that the function
\[
s\mapsto \GG(s)\coloneqq \overline C_{\left[C_{\WW,\rm st},\bar\lambda,\frac{1}{\bar\lambda},\frac{1}{\nu}\right]}
 \left(1+\norm{z_1(s)}{V}^{\e_5}
+\norm{z_2(s)}{V}^{\e_6}\right)\left(\norm{z_1(s)}{\D(\Delta)}^{2}+\norm{z_2(s)}{\D(\Delta)}^{2}\right)
\]
is locally integrable, which allow us to write
\begin{equation*}
 \norm{e(t)}{H}^2\le \ex^{\int_0^t\GG(s)\,\ed s}\norm{e(0)}{H}^2=0,\quad\mbox{for all}\quad t\ge0.
\end{equation*}
That is, the uniqueness holds true: $z_1-z_2=e=0$.
\end{proof}
\begin{remark}\label{R:trans_local}
 Notice that from the proof of Lemma~\ref{L:contract} we see that~$\varrho$ increases and $\epsilon$ decreases as the transient bound~$\overline C_{C_{\WW,\rm st},\bar\lambda,\frac{1}{\bar\lambda}}$
 in~\eqref{estz0} increases.
\end{remark}

\section{Local internal stabilization to trajectories}\label{S:stab_traj}
As a straightforward corollary to Theorem~\ref{T:ex.st-ct}  it follows the
stabilization to trajectories for system~\eqref{sys-y-cont_i}.

Let us be given a solution~$\hat y$ for the uncontrolled system~\eqref{sys-y-cont_i} (with~$u=0$) with~$\hat y_0\coloneqq\hat y(0)\in H$ and such that the vector functions
in~\eqref{abN_lineariz} are such that~$\widehat\NN$ satisfies~\eqref{exp-non2} and \eqref{exp-non3}, and~$(\hat a,\hat b)$ satisfies~\eqref{rhosigma-st}, for suitable nonnegative
constants
$\widehat C$ and $C_{\WW,\rm st}$.
Notice that, recalling the notation in section~\ref{S:red-stabil}, in this case~$\NN\coloneqq\frac{1}{\nu}\breve{\widehat\NN}$ also
satisfies~\eqref{exp-non2} and \eqref{exp-non3}, and~$(a,b)=(\frac{1}{\nu}\breve{\hat a},\frac{1}{\nu}\breve{\hat b})$ also satisfies~\eqref{rhosigma-st},

We consider system~\eqref{sys-y-cont_i}
\begin{equation}\label{sys-y-cont_i-feed}
\begin{split}
 \p_t y - \nu \Delta y + f(y,\nabla y) +B_M {\mathcal R}^{-1}B_M^*\widehat\Pi_{\lambda} (y-\hat y)&= 0;\\
 y\rest\Gamma = g;\qquad
 y(0)&=y_0.
 \end{split}
\end{equation}
with $\widehat\Pi_{\lambda}$  solving 

 \begin{equation}\label{eq:riccati_nu}
 \dot{\widehat\Pi}_{\lambda}+\widehat\Pi_{\lambda}\mathbb A^{\hat a,\hat b}_\nu+{\mathbb A^{\hat a,\hat b}_\nu}^*\widehat\Pi_{\lambda}
 -\widehat\Pi_{\lambda}  B_M  \mathcal R ^{-1}B_M^* \widehat\Pi_{\lambda} +{\lambda}\widehat\Pi_{\lambda}+\mathcal M^*\mathcal M=0,
\end{equation}
with $\mathbb A^{\hat a,\hat b}_\nu z\coloneqq \nu\Delta z-
\hat a z- \nabla\cdot(\hat bz)$.

Observe that $\breve z(\tau)\coloneqq (y-\hat y)(\frac{\tau}{\nu})$ solves
\[
\begin{split}
 \p_\tau \breve z - \Delta \breve z + a\breve z+ \nabla\cdot(b\breve z) +\textstyle\frac{1}{\nu}B_M {\RR}^{-1}B_M^*\breve{\widehat\Pi}_{\lambda}
 \breve z&= \NN(\breve z);\\
 \breve z\rest\Gamma = 0;\qquad
 \breve z(0)&=y_0-\hat y_0,
 \end{split} 
\]
and $\breve{\widehat\Pi}_{\lambda}(\tau)=\widehat\Pi_{\lambda}(\frac{\tau}{\nu})$ solves~\eqref{eq:riccati} with a different pair~$(\RR,\MM)$:
\[
\begin{split}
\textstyle\frac{\ed}{\ed\tau}{\breve{\widehat\Pi}}_{\lambda}+\breve{\widehat\Pi}_{\lambda}\textstyle\mathbb A^{a,b}
+\textstyle{\mathbb A^{a,b}}^*\breve{\widehat\Pi}_{\lambda}
 -\breve{\widehat\Pi}_{\lambda}  B_M  (\nu\mathcal R)^{-1}B_M^* \breve{\widehat\Pi}_{\lambda}
 +\textstyle\frac{\lambda}{\nu}\breve{\widehat\Pi}_{\lambda}+((\textstyle\frac{1}{\nu})^\frac{1}{2}\mathcal M)^*((\textstyle\frac{1}{\nu})^\frac{1}{2}\mathcal M)=0.
\end{split}  
\]
Therefore, from Theorem~\ref{T:ex.st-ct}, we have that
\[
 |\breve z(\tau)|_{V}^2 \le C\mathrm{e}^{-\textstyle\frac{\lambda}{\nu}\tau}
|\breve z(0)|_{V}^2,\quad\mbox{for all }\tau\ge0,
\]
provided $|\breve z(0)|_{V}$ is small enough. This implies that
\[
 |y(t)-\hat y(t)|_{V}^2 \le C\mathrm{e}^{-\lambda t}
|y_0-\hat y_0|_{V}^2,\quad\mbox{for all }t\ge0,
\]
and we can conclude that the following theorem holds true.

\begin{theorem}\label{T:st-traj}
Under the hypotheses of Theorem~\ref{T:ex.st-ct}, there exists~$\epsilon>0$ with the following properties.
If~$y_0\in H$ is such that
\[
y_0-\hat y_0\in V\quad\mbox{and}\quad\norm{y_0-\hat y_0}{V}<\epsilon,
\]
then the solution~$y$ of the system~\eqref{sys-y-cont_i-feed}
goes exponentially to~$\hat y$ with rate~$\frac{\lambda}{2}$, that is,
\[
 |y(t)-\hat y(t)|_{V}^2 \le C\mathrm{e}^{-\lambda t}
|y_0-\hat y_0|_{V}^2,\quad\mbox{for all }t\ge0,
\]
for a suitable constant~$C$ independent of~$(\epsilon,y_0-\hat y_0)$. Furthermore, the solution~$y$ is, and is unique, in the affine space
$\hat y+L^2_{\rm loc}(\mathbb{R}_0,{\mathrm D}(A))\cap C([0,+\infty),V)$.
\end{theorem}

Notice that the feedback control stabilizes the linearized system to zero globally, that is, we have the following theorem.
\begin{theorem}\label{T:st-traj-lin}
Under the hypotheses of Theorem~\ref{T:ex.st-ct},
given~{$z_0\in V$}, the solution of
\begin{equation}\label{sys-int-lin-RT}
\begin{split}
 \p_t z - \nu\Delta z + \hat a z+ \nabla\cdot(\hat b z) +B_M {\RR}^{-1}B_M^*{\widehat\Pi}_{\lambda}
 z&= 0;\\
 z\rest\Gamma = 0;\qquad
  z(0)&=z_0,
 \end{split} 
\end{equation}
satisfies
\[
|z(t)|_{V}^2 \le C\mathrm{e}^{-\lambda t}
|z_0|_{V}^2,\quad\mbox{for all }t\ge0,
\]
for a suitable constant~$C$ independent of~$z_0$. Furthermore, the solution~$z$ is, and is unique, in the space
$L^2_{\rm loc}(\mathbb{R}_0,{\mathrm D}(A))\cap C([0,+\infty),V)$.
\end{theorem}

\section{Example. Polynomial nonlinearities}\label{S:exa_poly}
Many systems modelling real evolutions involve polynomial nonlinearities, for example the Fisher-like {equations~\cite{Fisher37,VolpertPetr09}} modelling population dynamics
and the Burgers-like
equations~\cite{KroRod15,BuchotRaymondTiago15} modelling fluid (e.g., traffic) flow. Here we check the previous assumptions for the case the
function~$f(y,\nabla y)$ takes the form~$f_{\rm r}(y)+f_{\rm c}(y)\cdot\nabla y$ where~$f_{\rm r}$ and~$f_{\rm c}=[{f_{\rm c}}_1\;\;{f_{\rm c}}_2\;\;\dots\;\;{f_{\rm c}}_d]^\top$, are polynomials:
\[
f_{\rm r}(y)=\sum_{j=0}^{\bar p} \bar r_jy^j\quad\mbox{and}\quad {f_{\rm c}}_k(y)=\sum_{j=0}^{p_k} r_{k,j}y^j,
\]
with~$\bar r_j$ and $r_{k,j}$ real numbers, and $p_k\in\N$ for $k\in\{1,\dots,d\}$.

For illustration, we consider here the case $d=3$. The following estimates will be also valid for $d\in\{1,2\}$, though in those cases better estimates may hold true.
On the other hand some of the following arguments will not work in dimension~$d\ge4$, in that case some changes are needed.

It is enough to analyze the case of monomials, with degree greater than or equal to~$2$:
\begin{align*}
\qquad f(y)&=y^{\bar n}\mbox{ with }\bar n\ge 2
\intertext{and}
\qquad f(y)&=y^{n}\textstyle\p_{x_{\bar k}}y=\frac{1}{n+1}\textstyle\p_{x_{\bar k}}y^{n+1},\;\; \mbox{ with } n\ge1\;\mbox{for some } \bar k\in\{1,\dots,d\}.
\end{align*}
In this case, recalling the notation in Section~\ref{S:red-stabil}, for a given trajectory~$\hat y$, we obtain
respectively either
\begin{align*}
\hat a&=\bar n\hat y^{\bar n-1}\quad\mbox{and}\quad\hat b=0,
\intertext{or}
 \hat a&=0\quad\mbox{and}\quad \hat b=[\hat b_1\;\;\hat b_2\;\;\dots\;\;\hat b_d]^\top,\quad\mbox{with}\quad\hat b_k=\left\{\begin{array}{ll}
                                                                                                                        0&\mbox{if }k\ne\bar k,\\
                                                                                                                        \hat y^{n}&\mbox{if }k=\bar k.
                                                                                                                       \end{array}\right.
\end{align*}

Observe that in the case of a reaction nonlinearity $f(y)=y^{\bar n}$, condition~\eqref{rhosigma} is satisfied provided $\hat y\in L^\infty(\R_0,L^{3(\bar n-1)})$.
In the case of a convection nonlinearity $f(y)=\textstyle\p_{x_{\bar k}}y^{n+1}$, conditions~\eqref{rhosigma} and~\eqref{rhosigma-st} are satisfied
provided $\hat y\in L^\infty_w(\R_0,L^\infty)\cap L^\infty(\R_0,W^{1,3}(\Omega,\R))$.

\medskip\noindent
{\em The nonlinearity.}
Concerning~\eqref{exp-non2}, we need to be more careful, and  need a bit more of work. Again we consider only monomials. We also consider the
case~$\hat y\in L^\infty_w(\R_0,L^\infty)$.
\begin{example}
 In the case~$\NN(z)=\hat y^m z^2$, $m\in\N$, \eqref{exp-non2} holds true. We may write
 \[
\norm{\NN(z)-\NN(\tilde z)}{H}^2=\norm{\hat y^m(z-\tilde z)(z+\tilde z)}{H}^2
\le\norm{z-\tilde z}{H}^2\norm{\hat y^m}{L^\infty}^2\norm{z+\tilde z}{L^\infty}^2,
\]
and
\[
\norm{\NN(z)-\NN(\tilde z)}{H}^2
\le C\norm{\hat y^m}{L^\infty}^2\norm{z-\tilde z}{V}^2\left(\norm{z}{\D(\Delta)}^2+\norm{\tilde z}{\D(\Delta)}^2\right),
\]
which shows that~\eqref{exp-non2} holds true.
\end{example}

\begin{example}
 In the case~$\NN(z)=\hat y^mz^n$, $m\in\N$ and $n=\{3,4,5\}$, \eqref{exp-non2} holds true. 
 We may write, for suitable nonzero constants~$r_j$,
 \[
\NN(z)-\NN(\tilde z)=\hat y^m(z-\tilde z)\textstyle\sum\limits_{j=0}^{n-1} r_jz^j\tilde z^{n-1-j}
\]
where in the sum we have monomials of degree~$n-1$. For example for~$z^1\tilde z^{n-2}$, by standard (yet appropriate) Young, H\"older, Sobolev, and Agmon inequalities, we may write
\[
\begin{split}
\norm{(z-\tilde z)z^1\tilde z^{n-2}}{H}^2&=\norm{(z-\tilde z)^2z^2\tilde z^{2n-4}}{L^1}
\le\norm{z-\tilde z}{L^\infty}^2\norm{z}{L^\infty}\norm{\tilde z}{L^\infty}\norm{z\tilde z^{2n-5}}{L^1}\\
&\le C_1\norm{z-\tilde z}{V}\norm{z-\tilde z}{\D(\Delta)}\norm{z}{V}^\frac{1}{2}\norm{z}{\D(\Delta)}^\frac{1}{2}
\norm{\tilde z}{V}^\frac{1}{2}\norm{\tilde z}{\D(\Delta)}^\frac{1}{2}\norm{z}{L^6}\norm{\tilde z}{L^\frac{6(2n-5)}{5}}^{2n-5},
\end{split}
\]
and, since $H^1(\Omega,\R)\xhookrightarrow{}L^6(\Omega,\R)\xhookrightarrow{}L^\frac{6(2n-5)}{5}(\Omega,\R)$,
\[
\begin{split}
\norm{(z-\tilde z)z^1\tilde z^{n-2}}{H}^2
&\le \textstyle\frac{C_1}{2}\norm{z-\tilde z}{V}^2\norm{z}{\D(\Delta)}\norm{\tilde z}{\D(\Delta)}
+C_2\norm{z-\tilde z}{\D(\Delta)}^2\norm{z}{V}^3
\norm{\tilde z}{V}^{1+2(2n-5)}\\
&\le \textstyle\frac{C_1}{4}\norm{z-\tilde z}{V}^2\left(\norm{z}{\D(\Delta)}^2+\norm{\tilde z}{\D(\Delta)}^2\right)
+C_3\norm{z-\tilde z}{\D(\Delta)}^2\left(\norm{z}{V}^{4n-6}+\norm{\tilde z}{V}^{4n-6}\right).
\end{split}
\]
For the other monomials we can obtain analogous estimates, which give us
\[
\begin{split}
|\NN(z)-\NN(\tilde z)|_{H}^2
&\le C_4\norm{\hat y^m}{L^\infty}^2\norm{z-\tilde z}{V}^2\left(\norm{z}{\D(\Delta)}^2+\norm{\tilde z}{\D(\Delta)}^2\right)\\
&\quad+C_4\norm{\hat y^m}{L^\infty}^2\norm{z-\tilde z}{\D(\Delta)}^2\left(\norm{z}{V}^{4n-6}+\norm{\tilde z}{V}^{4n-6}\right),
\end{split}
\]
which shows that~\eqref{exp-non2} holds true.
\end{example}
\begin{example}
 In the case~$\NN(z)=\hat y^mz^6$, we were not able to derive~\eqref{exp-non2}. Proceeding as above, for suitable nonzero constants~$r_j$,
 \[
\NN(z)-\NN(\tilde z)=\hat y^m(z-\tilde z)\textstyle\sum\limits_{j=0}^{4} r_jz^j\tilde z^{5-j},
\]
where in the sum we have now monomials of degree~$5$. If for example for~$z^1\tilde z^4$,  we proceed as above and write
\[
\begin{split}
\norm{(z-\tilde z)z^1\tilde z^4}{H}^2&=\norm{(z-\tilde z)^2z^2\tilde z^8}{L^1}
\le\norm{z-\tilde z}{L^\infty}^2\norm{z}{L^\infty}\norm{\tilde z}{L^\infty}\norm{z\tilde z^7}{L^1}\\
&\le C_1\norm{z-\tilde z}{V}\norm{z-\tilde z}{\D(\Delta)}\norm{z}{V}^\frac{1}{2}\norm{z}{\D(\Delta)}^\frac{1}{2}
\norm{\tilde z}{V}^\frac{1}{2}\norm{\tilde z}{\D(\Delta)}^\frac{1}{2}\norm{z\tilde z^7}{L^1},
\end{split}
\]
we cannot bound the term~$\norm{z\tilde z^7}{L^1}$ by the $V$-norms of~$z$ and $\tilde z$ (for~$d=3$). Trying to use again the~$\D(\Delta)$-norms, we were not able to arrive
to~\eqref{exp-non2} (the~$\D(\Delta)$-norms will appear with a power strictly greater than~$2$).
\end{example}
\begin{example}
 In the case~$\NN(z)=\nabla\cdot(g(\hat y)z^n)$, where~$n\in\{2,3\}$ and~$g\colon \R\to\R^3$ is a smooth function, estimate~\eqref{exp-non2} holds true provided
 $g(\hat y)\in\WW_{\rm st}$. We consider only the case $n=3$. We write, for suitable nonzero constants~$r_j$,
 \[
\NN(z)-\NN(\tilde z)=\nabla\cdot\Bigl(g(\hat y)(z-\tilde z)\textstyle\sum\limits_{j=0}^{2} r_jz^j\tilde z^{2-j}\Bigr)
\]
where in the sum we have monomials of degree~$2$. For example for~$z\tilde z$ we find
\[
\begin{split}
&\quad\norm{\nabla\cdot\left( g(\hat y)(z-\tilde z)z\tilde z \right) }{H}^2\\
&\le\norm{(\nabla\cdot g(\hat y))^2(z-\tilde z)^2z^2\tilde z^{2}}{L^1}
+\norm{g(\hat y)}{L^\infty}^2\norm{(\nabla((z-\tilde z)z\tilde z))^2}{L^1}\\
&\le\norm{(\nabla\cdot g(\hat y))}{L^3}^2\norm{z-\tilde z}{L^6}^2\norm{z^2\tilde z^{2}}{L^\infty}\\
&\quad+\norm{g(\hat y)}{L^\infty}^2\Bigl(\norm{(\nabla(z-\tilde z))^2}{L^1}\norm{z^2\tilde z^{2}}{L^\infty}
+\norm{z-\tilde z}{L^\infty}^2\norm{(\nabla(z\tilde z))^2}{L^1}\Bigr)\\
&\le C\norm{z-\tilde z}{V}^2\norm{z}{L^\infty}^2\norm{\tilde z}{L^\infty}^{2}
+C\norm{z-\tilde z}{V}\norm{z-\tilde z}{\D(\Delta)}\left(\norm{z}{V}^2\norm{\tilde z}{L^\infty}^{2}
+\norm{z}{L^\infty}^2\norm{\tilde z}{V}^{2}\right)\\
&\le C_1\norm{z-\tilde z}{V}^2\left(\norm{z}{V}^2+\norm{z}{V}^{2}\right)\left(\norm{z}{\D(\Delta)}^2+\norm{\tilde z}{\D(\Delta)}^2\right)\\
&\quad+C_1\norm{z-\tilde z}{V}\norm{z-\tilde z}{\D(\Delta)}\left(\norm{z}{V}^2\norm{\tilde z}{V}\norm{\tilde z}{\D(\Delta)}
+\norm{z}{V}\norm{z}{\D(\Delta)}\norm{\tilde z}{V}^2\right)\\
&\le C_2\norm{z-\tilde z}{V}^2\left(\norm{z}{V}^2+\norm{z}{V}^{2}+1\right)\left(\norm{z}{\D(\Delta)}^2+\norm{\tilde z}{\D(\Delta)}^2\right)
+C_2\norm{z-\tilde z}{\D(\Delta)}^2\left(\norm{z}{V}^6+\norm{\tilde z}{V}^6\right).
\end{split}
\]
We can obtain analogous estimates for the other monomials, and conclude that
\[
\begin{split}
\norm{\NN(z)-\NN(\tilde z)}{H}^2
&\le C_3\norm{z-\tilde z}{V}^2\left(\norm{z}{V}^2+\norm{z}{V}^{2}+1\right)\left(\norm{z}{\D(\Delta)}^2+\norm{\tilde z}{\D(\Delta)}^2\right)\\
&\quad+C_3\norm{z-\tilde z}{\D(\Delta)}^2\left(\norm{z}{V}^6+\norm{\tilde z}{V}^6\right).
\end{split}
\]
which shows that~\eqref{exp-non2} holds true.
\end{example}

\section{Boundary stabilization}\label{S:bdry}
We start by considering a linear system in the form~\eqref{sys-z-bdry}, without the nonlinearity, which we rewrite in the more general form
\begin{equation}\label{sys-z_linbdry}
  \begin{split}
  \p_t z- \Delta z + az + \nabla \cdot (bz)
  &= 0;\quad
 z\rest\Gamma = B_\Gamma\zeta;\\
 z(0)&=z_0.
  \end{split}
\end{equation}
where now our control is a function~$\zeta\in \ZZ$, where~$\ZZ$ is an Hilbert space, and $B_\Gamma\in\LL(\ZZ,G^1_{\rm loc}(\R_0,\Gamma))$ with 
\[
\begin{split}
G^1_{\rm loc}(\R_0,\Gamma)&\coloneqq\bigcap_{T>0}G^1((0,T),\Gamma),\\
G^1((s_0,s_1),\Gamma)&\coloneqq W((s_0,s_1),H^1(\Omega,\R),V')\rest\Gamma\coloneqq\{v\rest\Gamma\mid v\in W((s_0,s_1),H^1(\Omega,\R),V')\},
\end{split}
\]
for all $s_1>s_0\ge0$.

\subsection{Weak solutions}\label{sS:weak-sol_bdry}
Let us consider the more general system
\begin{subequations}\label{sys-z2_bdry}
\begin{align}
  &\partial_t z-\Delta z+ a z+\nabla\cdot(bz)+g=0,\label{sys-z2b-eq}\\
  &z\rest\Gamma =\gamma, \qquad z(0)=z_0.\label{sys-z2b-bcic}
\end{align}
\end{subequations}
with an external body forcing~$g$ and where the control in~\eqref{sys-z_linbdry} is replaced by a general external boundary forcing~$\gamma$.

Notice that for~$I=(s_0,s_1)$ and for given $v\in W(I,H^1(\Omega,\R),V')$, we have~$v\rest\Gamma=0$ if, and only if, $v\in W(I,V,V')$.
Therefore since $v\in W(I,V,V')$ is a closed subspace of $W(I,H^1(\Omega,\R),V')$ we have that the function
\[
\begin{split}
 E_1\colon G^1(I,\Gamma)&\to W(I,H^1(\Omega,\R),V');\\
 \gamma&\mapsto E_1\gamma\in W(I,V,V')^\perp,\quad\mbox{with}\quad E_1(\gamma)\rest\Gamma=\gamma
\end{split}
 \]
is well defined. Notice also that in this case we can see that the trace space $G^1(I,\Gamma)$ is an Hilbert space, with the scalar product
\[
(\gamma,\xi)_{G^1(I,\Gamma)}\coloneqq(E_1(\gamma),E_1(\xi))_{W(I,H^1(\Omega,\R),V')}. 
\]
The corresponding induced norm corresponds to the trace norm (or range norm, cf.~\cite[Lemma~3.1]{Rod14-na})
\[
 \norm{\gamma}{G^1(I,\Gamma)}\coloneqq\inf_{\gamma=v\rest\Gamma}\norm{v}{W(I),H^1(\Omega,\R),V')}.
\]

\begin{definition}\label{D:weaksolbdry}
We say that $z$ is a weak solution for system~\eqref{sys-z2_bdry} if $y=z-E_1\gamma$ is a weak solution for the system~\eqref{sys-z2}
 with $f=g+\p_tE_1\gamma-\Delta E_1\gamma+ a E_1\gamma+\nabla\cdot(bE_1\gamma)$, and $y(0)=z_0-E_1\gamma(0)$.
\end{definition}

\begin{lemma}\label{L:weakbdry-z}
 Given~$(a,b)\in\WW$ satisfying~\eqref{rhosigma}, $g\in L^2(I,V')$, $\gamma\in G^1(I,\Gamma)$ and $z_0\in H$, there exists a weak solution
 $z\in W(I,H^1(\Omega,\R),V')$ for \eqref{sys-z2}. Moreover~$z$ is unique and depends continuously on the data:
 \[
 \norm{z}{W(I,H^1(\Omega,\R),V')}^2\le \overline C_{\left[|I|,C_{\WW}\right]}
\left(\norm{z(s_0)}{H}^2+\norm{g}{L^2(I,V')}^2+\norm{\gamma}{G^1(I,\Gamma)}^2\right).
 \]
\end{lemma}
\begin{proof}
 The proof is straightforward from Definition~\ref{D:weaksolbdry} and Lemma~\ref{L:weak-z} (cf.~\cite[Theorem~3.2]{Rod14-na}).
\end{proof}

\subsection{Strong solutions}\label{sS:strong-sol_bdry}
In order to define strong solutions, we {introduce}
\[
\begin{split}
G^2_{\rm loc}(\R_0,\Gamma)&\coloneqq\bigcap_{T>0}G^1((0,T),\Gamma),\\
G^2((s_0,s_1),\Gamma)&\coloneqq W((s_0,s_1),H^2(\Omega,\R),H)\rest\Gamma,
\end{split}
\]
for all $s_1>s_0\ge0$, and consider the extension
\[
\begin{split}
 E_2\colon G^2((s_0,s_1),\Gamma)&\to W((s_0,s_1),H^2(\Omega,\R),H);\\
 \gamma&\mapsto E_2\gamma\in W((s_0,s_1),\D(\Delta),H)^\perp,\quad\mbox{with}\quad E_2(\gamma)\rest\Gamma=\gamma.
\end{split}
 \]
The trace space $G^2((s_0,s_1),\Gamma)$ is endowed with the scalar product
\[
(\gamma,\xi)_{G^2((s_0,s_1),\Gamma)}\coloneqq(E_2(\gamma),E_2(\xi))_{W((s_0,s_1),H^2(\Omega,\R),H)} 
\]
and induced norm 
\[
 \norm{\gamma}{G^2((s_0,s_1),\Gamma)}\coloneqq\inf_{\gamma=v\rest\Gamma}\norm{v}{W((s_0,s_1),H^2(\Omega,\R),H)}.
\]

\begin{definition}\label{D:strongsolbdry}
 We say that $z$ is a strong solution for system~\eqref{sys-z2_bdry} if $y=z-E_2\gamma$ is a strong solution for the system~\eqref{sys-z2}
 with $f=g+\p_tE_2\gamma-\Delta E_2\gamma+ a E_2\gamma+\nabla\cdot(bE_2\gamma)$, and $y(0)=z_0-E_2\gamma(0)\in V$.
\end{definition}

\begin{lemma}\label{L:strongbdry-z}
 Given~$(a,b)\in\WW_{\rm st}$ satisfying~\eqref{rhosigma-st}, $g\in L^2(I,H)$, $\gamma\in G^2(I,\Gamma)$ and $z_0\in H^1(\Omega,\R)$, with $z_0-E_2\gamma(0)\in V$, then there exists a strong solution
 $z\in W(I,H^2(\Omega,\R),H)$ for~\eqref{sys-z2_bdry}. Moreover~$z$ is unique and depends continuously on the data
 \[
 \norm{z}{W(I,H^2(\Omega,\R),H)}^2\le \overline C_{\left[|I|,C_{\WW,\rm st}\right]}
\left(\norm{z(s_0)}{H}^2+\norm{g}{L^2(I,H)}^2+\norm{\gamma}{G^2(I,\Gamma)}^2\right).
 \]
\end{lemma}
\begin{proof}
 The proof is straightforward from Definition~\ref{D:strongsolbdry} and Lemma~\ref{L:strong-z}.
\end{proof}

Also in the case of nonhomogeneous boundary conditions, we have the following smoothing property.
\begin{lemma}\label{L:smooth-prop}
Let us be given $(a,b)\in\WW_{\rm st}$ satisfying~\eqref{rhosigma-st},
$z_0\in H$, $g\in L^2(I,H)$, and $\gamma\in G^2(I,\Gamma))$;
then for the weak solution $z$ of system~\eqref{sys-z2_bdry},
we have that
$(\Bigcdot-s_0)z\in W(I,H^2(\Omega,\R)
,H)$, and
\begin{align*}
(|(\Bigcdot-s_0)z|_{W(I,H^2(\Omega,\R),H))}^2
&\le\overline C_{\left[C_{\WW,\rm st}\right]}
\left(|z_0|_{H}^2+\norm{g}{L^2(I,H)}^2 +|\gamma|_{G^2(I,\Gamma)}^2\right).
\end{align*}
\end{lemma}

\begin{proof}
Since $z$ solves~\eqref{sys-z2_bdry}, it turns out that also $w=(\Bigcdot-s_0)z$ does, with different data:
\begin{align*}
  &\partial_t w-\Delta w+ a w+\nabla\cdot(bw)+(\Bigcdot-s_0)g-z=0,\\
  &w\rest\Gamma =(\Bigcdot-s_0)\gamma, \qquad w(0)=0.
\end{align*}
Then, from Lemma~\ref{L:strongbdry-z}, we can derive that
\[
 \norm{w}{W(I,H^2(\Omega,\R),H)}^2\le \overline C_{\left[|I|,C_{\WW,\rm st}\right]}
\left(\norm{(\Bigcdot-s_0)g}{L^2(I,H)}^2+\norm{z}{L^2(I,H)}^2+\norm{(\Bigcdot-s_0)\gamma}{G^2(I,\Gamma)}^2\right).
 \]
Thus, the result follows from $|z|_{L^2(I,H)}^2
\le|z|_{W(I,H^1(\Omega,\R),H^{-1}(\Omega,\R))}^2$ and from Lemma~\ref{L:weak-z}.
\end{proof}

\begin{lemma}\label{L:L2H1b-H1}
Let $(a,b)\in\WW_{\rm st}$ satisfy~\eqref{rhosigma-st}. Let $z$ solve system~\eqref{sys-z2_bdry}, with $g\in L^2(\R_0,H)$ and $\gamma\in G^2(\R_0,\Gamma)$. If
$z\in L^2(\R_0,H)$, then 
$z\in W(\R_0,H^1(\Omega,\R),H^{-1}(\Omega,\R))$, with
\[
|z|_{W(\R_0,H^1(\Omega,\R),H^{-1}(\Omega,\R))}^2
\le
\overline C_{\left[C_{\WW,\rm st}\right]}\left(|z_0|_{H}^2
+|z|_{L^2(\R_0,H)}^2+\norm{g}{L^2(\R_0,H)}^2
+|\gamma|_{G^2(\R_0,\Gamma)}^2\right).
\]
\end{lemma}

\begin{proof}
From Theorem~\ref{L:weak-z}  we find
\begin{align}
|z|_{W((0,1),H^1(\Omega,\R),H^{-1}(\Omega,\R))}^2
&\le\overline C_{\left[C_{\WW,\rm st}\right]}\left(|z_0|_{H}^2+\norm{g}{L^2((0,1),V')}^2
+|\gamma|_{G^1((0,1),\Gamma)}^2\right)\label{ev.W1}
\end{align}
and, from Lemma~\ref{L:smooth-prop} we have that for all $t\ge1$
\[
|z(t)|_{H^1(\Omega,\R)}^2\le
\overline C_{\left[C_{\WW,\rm st}\right]}\left(|v(t-1)|_{H}^2+\norm{g}{L^2((t-1,t),H)}^2
+|\gamma|_{G^2((t-1,t),\Gamma)}^2\right),
\]
which allow us to obtain
\begin{align*}
&\quad|z|_{L^2(\R_1,H^1(\Omega,\R))}^2
=\sum_{n=1}^{+\infty} |z|_{L^2((n,n+1),H^1(\Omega,\R))}^2\\
&\le\overline C_{\left[C_{\WW,\rm st}\right]}\sum_{n=1}^{+\infty}
\int_{n}^{n+1} |v(t-1)|_{H}^2+\norm{g}{L^2((t-1,t),H)}^2
+|\gamma|_{G^2((t-1,t),\Gamma)}^2\,\ed t\\
&\le\overline C_{\left[C_{\WW,\rm st}\right]}\left(|z|_{L^2(\R_0,H)}^2
+\sum_{n=1}^{+\infty}\int_{n}^{n+1}\norm{g}{L^2((t-1,t),H)}^2+|E_2\gamma|_{W^1((t-1,t),H^2(\Omega,\R),H)}^2\,\ed t\right)\\
&\le\overline C_{\left[C_{\WW,\rm st}\right]}\left(|z|_{L^2(\R_0,H)}^2
+\sum_{n=1}^{+\infty}\norm{g}{L^2((n-1,n+1),H)}^2+|E_2\gamma|_{W^1((n-1,n+1),H^2(\Omega,\R),H)}^2\right).
\end{align*}
Hence, it follows
\begin{align*}
|z|_{L^2(\R_1,H^1(\Omega,\R))}^2
&\le\overline C_{\left[C_{\WW,\rm st}\right]}\left(|z|_{L^2(\R_0,H)}^2+2\norm{g}{L^2(\R_0,H)}^2
+2|E_2\gamma|_{W^1(\R_0,H^2(\Omega,\R),H)}^2\right) \\
&\le2\overline C_{\left[C_{\WW,\rm st}\right]}\left(|z|_{L^2(\R_0,H)}^2+\norm{g}{L^2(\R_0,H)}^2
+|\gamma|_{G^2(\R_0,\Gamma)}^2\right),
\end{align*}
which, together with~\eqref{ev.W1} gives us
\[
|z|_{L^2(\R_0,H^1(\Omega,\R))}^2
\le\overline C_{\left[C_{\WW,\rm st}\right]}\left(|z_0|_H^2+|z|_{L^2(\R_0,H)}^2+\norm{g}{L^2(\R_0,H)}^2
+|\gamma|_{G^2(\R_0,\Gamma)}^2\right). 
\]
Since $z$ solves system~\eqref{sys-z2_bdry}, we can obtain that
\[
|\p_tz|_{L^2(\R_0,H^{-1}(\Omega,\R))}^2
\le \overline C_{\left[C_{\WW,\rm st}\right]}
|z|_{L^2(\R_0,H^1(\Omega,\R))}^2 ,
\]
which finishes the proof.
\end{proof}

\subsection{Null controllability}\label{sS:nullcont_bdry}
Consider, in the bounded cylinder $I\times\Omega$, $I=(s_0,s_1)$, the controlled system~\eqref{sys-z_linbdry}.
and also the adjoint system~\eqref{sys-q}.

Let~$z(\Bigcdot)=z(z_0,\zeta)(\Bigcdot)$ and~$q(\Bigcdot)=q(q_1)(\Bigcdot)$ solve~\eqref{sys-z_linbdry} and~\eqref{sys-q}, respectively. Thus, integrating
$\frac{\ed}{\ed t}(z,q)$, we find following the arguments in~\cite[section~4]{Rod14-na} and in~\cite[section~3.1]{Rod15-cocv} that
\[
\begin{split}
 (z(s_1),q_1)_H-(z_0,q(s_0))_H&=\langle \nnn\cdot\nabla q,B_\Gamma\zeta\rangle_{G^1((s_0,s_1),\Gamma)',G^1((s_0,s_1),\Gamma)}\\
 &=\langle B_\Gamma^*\circ(\nnn\cdot\nabla) q,\zeta\rangle_{\ZZ',\ZZ}
 \end{split}
\]
where~$B_\Gamma^*\in\LL(G^1((s_0,s_1),\Gamma)',\ZZ')$ is the adjoint of $B_\Gamma$, and the symbol $\circ$ stands for the composition of two operators.

Also in the boundary case we have the following lemma (cf. Lemma~\ref{L:ex0cont}).

\begin{lemma}\label{L:ex0cont_bdry}
System~\eqref{sys-q} is $B_\Gamma^*\circ(\nnn\cdot\nabla)$ observable in $I$ with
\begin{equation}\label{obs-in-BGamma*}
 \norm{q(q_1)(s_0)}{H}\le C_{\rm obs}\norm{B_\Gamma^*\circ(\nnn\cdot\nabla) q(q_1)}{\ZZ'}
 \end{equation}
if, and only if,
system~\eqref{sys-z_linbdry} is
null controllable in $I$ and the family of controls~$\{u(z_0)\mid z_0\in H\}$ is a bounded linear function of~$z_0$:
\begin{align*}
\norm{\zeta(z_0)}{\ZZ}\le C_{\rm obs}\norm{z_0}{H}, \text{ where } C_{\rm obs} \text{ is as in }~\eqref{obs-in-BGamma*}.
\end{align*}
\end{lemma}

\medskip\noindent
{\em Controls supported in a subset.}
Given an open subset $\Gamma_{\rm c}\subseteq\Gamma$, we define the spaces
\[
\begin{split}
G^1_{\rm c}((s_0,s_1),\Gamma)&\coloneqq\{\gamma\in G^1((s_0,s_1),\Gamma)\mid \gamma\rest{\Gamma\setminus\overline{\Gamma_{\rm c}}}=0\},\\
G^2_{\rm c}((s_0,s_1),\Gamma)&\coloneqq\{\gamma\in G^2((s_0,s_1),\Gamma)\mid \gamma\rest{\Gamma\setminus\overline{\Gamma_{\rm c}}}=0\}. 
\end{split}
\]
From the results in the section~\ref{sS:nullcont} and following the arguments
in~\cite[section~4]{Rod14-na}
we have that we can construct open subsets~$\tilde\omega$ with $\Omega\cap\tilde\omega=\emptyset$ and $\Gamma\cap\p\tilde\omega=\overline{\Gamma_{\rm c}}$ leading to
the existence of a constant $ C_{\tilde\omega,\Omega}>0$, depending on~$\tilde\omega$ and~$\Omega$, such that
the weak solution~$\tilde q$ for~\eqref{sys-q} in~$I\times\widetilde\Omega$ with $\widetilde\Omega=\Omega\cup\omega\cup\Gamma_{\rm c})$, that is, $\tilde q$ solving
\begin{align*}
  &\partial_t \tilde q-\Delta \tilde q+ \tilde a \tilde q-\tilde b\cdot\nabla\tilde q=0,\\
  &\tilde q\rest\Gamma =0, \qquad \tilde q(s_1)=\tilde q_1,
\end{align*}
satisfies
\[
\norm{\tilde q(0)}{L^2(\widetilde\Omega,\R)}^2\le  \ex^{C_{\tilde\omega,\Omega}\Theta\left(|I|,\norm{a}{L^\infty(I,L^d)},\norm{b}{L^\infty_w(I,L^\infty)},d\right)}
\norm{\tilde q}{L^2(I,L^2(\tilde\omega,\R)}^2.
\]
with~$\Theta$ as in~\eqref{const_obs1}. Here the functions~$\tilde a$ and $\tilde b$ are extensions of~$a$ and~$b$ by zero outside~$\Omega$.

By Lemmas~\ref{L:ex0cont_bdry} and~\ref{L:weak-z} we can find a control $\tilde\eta\in L^2(I,L^2(\widetilde\Omega,\R))$ such that the corresponding solution of the system
\begin{align*}
  &\partial_t \tilde z-\Delta \tilde z+ \tilde a \tilde z+\nabla\cdot(\tilde b\tilde z)+1_{\tilde\omega}\tilde\eta=0,\\
  &\tilde z\rest{\widetilde\Gamma} =0, \qquad \tilde z(s_0)=\tilde z_0,
\end{align*}
where $\tilde z_0$ is the extension of~$z_0$ by zero outside~$\Omega$, satisfies
\[
 \norm{\tilde z}{W(I,H^1_0(\widetilde\Omega,\R),H^{-1}(\widetilde\Omega,\R))}^2\le \overline C_{\left[|I|,C_{\mathcal W}\right]}
\left(1+\ex^{C_{\tilde\omega,\Omega}\Theta}\right)\norm{\tilde z(s_0)}{L^2(\widetilde\Omega,\R)}^2,
 \]
Therefore $z\coloneqq\tilde z\rest\Omega$ solves~\eqref{sys-z2_bdry} with~$\gamma=\tilde z\rest\Gamma\in G^1_{\rm c}((s_0,s_1),\Gamma)$ satisfying
\[
 \norm{\gamma}{G^1_{\rm c}((s_0,s_1),\Gamma)}^2\le C\overline C_{\left[|I|,C_{\mathcal W}\right]}
\left(1+\ex^{C_{\tilde\omega,\Omega}\Theta}\right)\norm{\tilde z(s_0)}{L^2(\widetilde\Omega,\R)}^2\le \overline C_{\left[|I|,C_{\mathcal W}\right]}
\ex^{C_{\tilde\omega,\Omega}\Theta}\norm{z(s_0)}{H}^2.
 \]

Since the choice of such subset~$\tilde\omega$ is at our disposal, and using Lemmas~\ref{L:ex0cont} and~\ref{L:ex0cont_bdry}
we can conclude that there exists a constant $ C_{\Gamma_c,\Omega}>0$ depending on~$\Gamma_c$ and~$\Omega$,
such that
\begin{subequations}\label{obs-ineq-b23}
\begin{equation}
\norm{q(0)}{H}^2\le  \overline C_{\left[|I|,C_{\mathcal W}\right]}\ex^{C_{\Gamma_c,\Omega}\Theta\left(|I|,\norm{a}{L^\infty(I,L^d)},\norm{b}{L^\infty_w(I,L^\infty)},d\right)}
\norm{1_{\Gamma_{\rm c}}(\nnn\cdot\nabla q(q_1))}{G^1_{\rm c}((s_0,s_1),\Gamma)'}^2,
\end{equation}
if~$(a,b)\in\WW$. Further, if~$(a,b)\in\WW_{\rm st}$ (cf.~\cite[Remark~3.3]{Rod15-cocv}) we can derive that
\begin{equation}
\norm{q(0)}{H}^2\le  \overline C_{\left[|I|,C_{\WW,\rm st}\right]}\ex^{C_{\Gamma_c,\Omega}\Theta\left(|I|,\norm{a}{L^\infty(I,L^d)},\norm{b}{L^\infty_w(I,L^\infty)},d\right)}
\norm{1_{\Gamma_{\rm c}}(\nnn\cdot\nabla q(q_1))}{G^2_{\rm c}((s_0,s_1),\Gamma)'}^2.
\end{equation}
\end{subequations}
Thus, in the case we take~$B_\Gamma=\iota_{\Gamma_{\rm c}}\in\LL(G^1_{\rm c}((s_0,s_1),\Gamma),G^1((s_0,s_1),\Gamma))$, $\gamma\mapsto\gamma$, as the inclusion operator,
then we have $B_\Gamma^*=1_{\Gamma_{\rm c}}$ and we can conclude that~\eqref{obs-in-BGamma*}
holds with~$C_{\rm obs}=C\ex^{C_{\Gamma_c,\Omega}\Theta\left(|I|,\norm{a}{L^\infty(I,L^d)},\norm{b}{L^\infty_w(I,L^\infty)},d\right)}$.
Therefore we have the following.
\begin{theorem}\label{T:exact.contBGamma=1}
Let $(a,b)\in\WW$ and $B_\Gamma=\iota_{\Gamma_{\rm c}}$. Then, there exists a
family $\{\bar\zeta(z_0)\mid z_0\in H\}\subseteq G^1_{\rm c}((s_0,s_1),\Gamma)$ such that the solutions
$z(z_0,\bar\zeta(z_0))$ to \eqref{sys-z_linbdry} satisfy
$z(z_0,\bar\zeta(z_0))(s_1)=0$ and, for a constant $\widehat C=C(\Gamma_c,\Omega)=C_{[C_\WW]}$, we have that
\begin{align*}
\norm{\bar\zeta(z_0)}{G^1_{\rm c}((s_0,s_1),\Gamma)}\le\overline C_{\left[|I|,C_{\mathcal W}\right]}\ex^{\widehat C\Theta}\norm{z_0}{H},
\end{align*}
with $\Theta=\Theta\left(|I|,\norm{a}{L^\infty(\R_0,L^d)},\norm{b}{L^\infty_w(\R_0,L^\infty)},d\right)$ given by~\eqref{const_obs1}.

Furthermore, if $(a,b)\in\WW_{\rm st}$ then there exists a
family $\{\bar\zeta(z_0)\mid z_0\in H\}\subseteq G^2_{\rm c}((s_0,s_1),\Gamma)$ such that, still $z(z_0,\bar\zeta(z_0))(s_1)=0$,
and we have
\begin{align*}
\norm{\bar\zeta(z_0)}{G^2_{\rm c}((s_0,s_1),\Gamma)}\le\overline C_{\left[|I|,C_{\WW,\rm st}\right]}\ex^{\widehat C\Theta}\norm{z_0}{H}.
\end{align*}
\end{theorem}

Given a nonzero smooth function~$\chi_\Gamma\colon\Gamma\to\R$ with $\supp \chi_\Gamma\subseteq\overline{\Gamma_{\rm c}}$, we also have the following (cf.~\cite[section~3.3]{Rod15-cocv}).
\begin{corollary}\label{C:exact.contBGamma=chi}
Theorem~\ref{T:exact.contBGamma=1} holds in the more general case
\[
B_\Gamma=1_{\Gamma_{\rm c}}\chi_\Gamma 1_{\Gamma_{\rm c}}\in
\LL(G^1_{\rm c}((s_0,s_1),\Gamma),G^1((s_0,s_1),\Gamma)) ,
\]
with~$\widehat D=\widehat D(\chi_\Gamma,\Gamma_{\rm c},\Omega)>0$
in the place of~$\widehat C$.
\end{corollary}

\subsection{Stabilization to zero by finite dimensional controls}

Now we consider the case of finite dimensional controls, of the form~$\sum\limits_{i=1}^Mu_i(t)\Phi_i(x)$.
Let us consider a family~$\widehat\CC_{\Gamma}=\{\widehat \Psi_i\in H^\frac{3}{2}(\Gamma,\R)\mid i\in\{1,2,\dots,M\}\}$
satisfying~$1_{\Gamma_{\rm c}}\chi_\Gamma\widehat\CC_{\Gamma}\subset H^\frac{3}{2}(\Gamma,\R)$, and
denote by~$P_M$ the orthogonal projection
in~$L^2(\Gamma,\R)$ onto~$\mathcal S_{\widehat\CC_\Gamma}\coloneqq\mathrm{span}\,\widehat\CC_\Gamma$.

Let us also fix a positive constant~$\bar\lambda>0$ and consider, in~$\mathbb R_{s_0}\times\Omega$, the system:
\begin{subequations}\label{sys-z2cf-bdry}
\begin{align}
  &\partial_t z_{\bar\lambda}-\Delta z_{\bar\lambda}+ (a-\textstyle\frac{\bar\lambda}{2}) z_{\bar\lambda}+\nabla\cdot (bz_{\bar\lambda})=0,\label{sys-z2cf-bdry-eq}\\
  &z_{\bar\lambda}\rest \Gamma = 1_{\Gamma_{\rm c}}\chi_\Gamma  P_M1_{\Gamma_{\rm c}} \zeta, \qquad z_{\bar\lambda}(s_0)=z_0.\label{sys-z2cf-bdry-bcic}
\end{align}
\end{subequations}

\begin{definition}
We say that~\eqref{sys-z2cf-bdry} is exponentially stabilizable to zero, with rate $0$, if there
are constants $C_1>0$ and $C_2>0$, and a bounded family $\{\zeta=\zeta(z_0)\mid z_0\in H\}\subseteq G^1_{\rm c}(\R_0,\Gamma)$,
\[
\norm{\zeta(z_0)}{G^1_{\rm c}(\R_0,\Gamma)}^2\le C_1\norm{z_0}{H}^2
\]
such that the corresponding global solution~$z_{\bar\lambda}(t)=z_{\bar\lambda}(z_0,\zeta(z_0))(t)$ satisfies
\begin{equation*}
|z_{\bar\lambda}(t)|_H^2\leq C_2\,|z_0|_H^2,\quad\mbox{for all $t\ge s_0$}.
\end{equation*}
\end{definition}

Notice that we may write
\[
 1_{\Gamma_{\rm c}}\chi_\Gamma  P_M1_{\Gamma_{\rm c}} \zeta=1_{\Gamma_{\rm c}}\chi_\Gamma  \sum_{i=1}^M\zeta_i\widehat\Psi_i=\sum_{i=1}^M u_i \Psi_i
\]
with~$u_i\coloneqq\zeta_i$ and~$\Psi_i\coloneqq 1_\Gamma\chi_\Gamma  \widehat\Psi_i$, $i\in\{1,2,\dots,M\}$.

Henceforth we use the control operator
\begin{equation}\label{eq:kk2-bdry}
B_M^\Gamma\coloneqq 1_{\Gamma_{\rm c}}\chi_\Gamma  P_M1_{\Gamma_{\rm c}}.
\end{equation}
Further~$\Theta$ and~$\widehat D$ are as in
Theorem~\ref{T:exact.contBGamma=1} and Corollary~\ref{C:exact.contBGamma=chi}.

Let $z_{\bar\lambda}$ solve~\eqref{sys-z_linbdry} with~$B_\Gamma=1_{\Gamma_{\rm c}}\chi_\Gamma 1_{\Gamma_{\rm c}}$ and $(a-\frac{\bar\lambda}{2})$ in the place of $a$,
and with the corresponding control $\zeta=\bar\zeta(z_0)\in
G^2_{\rm c}((s_0,s_1),\Gamma)$ given by Corollary~\ref{C:exact.contBGamma=chi}, and let~$z_M$ solve~\eqref{sys-z2cf-bdry} also with $\zeta=\bar\zeta(z_0)$.
Then,
$d=z-z_M$ solves
\begin{align*}
  &\partial_t d-\Delta d+ (a-\textstyle\frac{\bar\lambda}{2}) d+\nabla\cdot (bd)=0,\\
  &d\rest \Gamma = 1_{\Gamma_{\rm c}} \chi_\Gamma(1-P_M)1_{\Gamma_{\rm c}} \zeta(z_0), \qquad d(s_0)=0.
\end{align*}
If $(a,b)\in\WW_{\rm st}$ satisfies~\eqref{rhosigma-st}, from Lemma~\ref{L:weakbdry-z} and Corollary~\ref{C:exact.contBGamma=chi} it follows
\[
 \norm{d}{W(I,H^1(\Omega,\R),V')}^2\le \Xi_\Gamma(|I|)
\norm{1_{\Gamma_{\rm c}}\chi_\Gamma(1- P_M)1_{\Gamma_{\rm c}}}{\LL(G^2_{\rm c}((s_0,s_1),\Gamma)),G^1_{\rm c}((s_0,s_1),\Gamma)}^2\norm{z_0}{H}^2
 \]
with
\[
\Xi_\Gamma(\tau)\coloneqq\overline C_{\left[|I|,C_{\WW,\rm st}\right]}\ex^{\widehat D\Theta\left(\tau,\norm{a-\textstyle\frac{\bar\lambda}{2}}{L^\infty(\R_0,L^d)},
\norm{b}{L^\infty_w(\R_0,L^\infty)},d\right)},\quad \tau>0.
\]

We can see that when~$(a-\textstyle\frac{\bar\lambda}{2},b)=(0,0)$ then~\eqref{sys-z2cf-bdry} is exponentially stabilizable to zero, with rate $0$, just by setting~$\zeta(z_0)=0$
for all~$z_0\in H$. Therefore from now we consider the case $(a-\textstyle\frac{\bar\lambda}{2},b)\ne(0,0)$
where we can see that it holds
\[
\lim\limits_{\tau\to+\infty}\Xi_\Gamma(\tau)=+\infty\quad\mbox{and}\quad\lim\limits_{\tau\to0}\Xi_\Gamma(\tau)=+\infty.
\]
Hence we can set~$T_*>0$ such that $\Xi_\Gamma(T_*)=\min\limits_{\tau>0}\Xi_\Gamma(\tau)\eqqcolon\Upsilon_\Gamma$.

This allow us to derive the following result on a sufficient condition on the family~$\widehat\CC_\Gamma$ for the existence of a stabilizing control.
The proof can be done following the arguments used in the internal controls case as in
in~\cite{KroRod-ecc15,BreKunRod-pp15}.
\begin{theorem}\label{T:stab-v-gen-bdry}
Let us be given a nonzero $\chi_\Gamma\in C^\infty(\overline\Omega)$ satisfying $\supp\chi_\Gamma\subseteq\overline{\Gamma_{\rm c}}$.  If
\begin{equation}\label{M_bound-gen-bdry}
\norm{1_{\Gamma_{\rm c}}\chi_\Gamma(1- P_M)1_{\Gamma_{\rm c}}}{\LL(G^2_{\rm c}((s_0,s_1),\Gamma)),G^1_{\rm c}((s_0,s_1),\Gamma)}^{2}
\le \Upsilon_\Gamma^{-1},
\end{equation}
then system~\eqref{sys-z2cf-bdry} is stabilizable to zero with rate~$0$.
\end{theorem}

\subsection{Feedback stabilizing rule and Riccati equation}
Here we follow the ideas from~\cite{Badra09-cocv,Rod-pp15}, by considering a suitable extended system. This is done in order to be able to use the Dynamical Programming Principle.

Here we suppose that~\eqref{M_bound-gen-bdry} holds true, so that from Theorem~\ref{T:stab-v-gen-bdry} we know that there
exists~$\zeta\in G^2_{\rm c}((s_0,s_1),\Gamma)$ stabilizing system~\eqref{sys-z2cf-bdry} to zero with rate~$0$. 
Writing
\[
 1_{\Gamma_{\rm c}}\chi_\Gamma  P_M1_{\Gamma_{\rm c}} \zeta=\sum_{i=1}^M u_i 1_{\Gamma_{\rm c}}\chi_\Gamma\widehat\Psi_i, 
\]
we know that $E_2 1_{\Gamma_{\rm c}}\chi_\Gamma  P_M1_{\Gamma_{\rm c}} \zeta\in W^1((s_0,s_1),H^2(\Omega,\R),H)$, but
we do not have necessarily that $u_i\in H^1((s_0,s_1),\R)$, for each $i\in\{1,2,\dots,M\}$. Though, we can say that $u_i\in H^\frac{3}{4}((s_0,s_1),\R)$,
see~\cite[chapter~4, section~2.2]{LioMag72-II}. 

Since we would like to follow the procedure in~\cite{Rod-pp15} we need~$u_i\in H^1((s_0,s_1),\R)$.
This regularity in time could be guaranteed (for a suitable class of actuators) by following the arguments based
on suitable truncated observability inequalities, in both space and time variable, from~\cite[section~4.3]{Rod15-cocv},
see also~\cite{Shirikyan-arx11}. 
Skiping those technical details, we would arrive to the following result for a suitable projection~$P^t_{\tilde M}$ defined in~$L^2((s_0,s_1),\R)$, with range contained
in~$H^1((s_0,s_1),\R)$.
\begin{proposition}\label{P:stab-v-gen-bdry}
Let us be given $\chi_\Gamma$ as in Theorem~\ref{T:stab-v-gen-bdry}.  If
\begin{equation}\label{cond.spacetime}
\norm{1_{\Gamma_{\rm c}}\chi_\Gamma(1- P^t_{\tilde M}P_M)1_{\Gamma_{\rm c}}}{\LL(G^2_{\rm c}((s_0,s_1),\Gamma)),G^1_{\rm c}((s_0,s_1),\Gamma)}^{2}
\le \Upsilon_\Gamma^{-1},
\end{equation}
then system~\eqref{sys-z2cf-bdry}, with $P^t_{\tilde M}P_M$ in the place of~$P_M$, is stabilizable to zero with rate~$0$.
\end{proposition}

\begin{remark}
A characterization of~$G^2_{\rm c}((s_0,s_1),\Gamma)$  in terms of fractional Sobolev spaces can be found in~\cite[chapter~4, section~2.2]{LioMag72-II}. However,
we do not know a similar characterization of~$G^1_{\rm c}((s_0,s_1),\Gamma)$. Thus, we cannot follow word by word the procedure in~\cite[section~4.3]{Rod15-cocv}, but
we can use the main idea.
Starting by writing $1- P^t_{\tilde M}P_M=1- P_M
+(1- P^t_{\tilde M})P_M$, we find (using the identity~$P_M=P_MP_M$)
\[
\begin{split}
&\quad\norm{1_{\Gamma_{\rm c}}\chi_\Gamma(1- P^t_{\tilde M}P_M)1_{\Gamma_{\rm c}}}{\LL(G^2_{\rm c}((s_0,s_1),\Gamma)),G^1_{\rm c}((s_0,s_1),\Gamma)}\\
&\le\norm{1_{\Gamma_{\rm c}}\chi(1- P_M)1_{\Gamma_{\rm c}}}{\LL(G^2_{\rm c}((s_0,s_1),\Gamma)),G^1_{\rm c}((s_0,s_1),\Gamma)}\\
&\qquad+\norm{1_{\Gamma_{\rm c}}\chi(1- P^t_{\tilde M})P_M}{\LL(G^2_{\rm c}((s_0,s_1),\Gamma)),G^{\frac{3}{2}}_{\rm c}((s_0,s_1),\Gamma)}
\norm{P_M 1_{\Gamma_{\rm c}}}{\LL(G^{\frac{3}{2}}_{\rm c}((s_0,s_1),\Gamma)),G^1_{\rm c}((s_0,s_1),\Gamma)}\\
&\le\norm{1_{\Gamma_{\rm c}}\chi(1- P_M)1_{\Gamma_{\rm c}}}{\LL(G^2_{\rm c}((s_0,s_1),\Gamma)),G^1_{\rm c}((s_0,s_1),\Gamma)}\\
&\qquad+C_{\chi_\Gamma,M}\norm{(1- P^t_{\tilde M})}{\LL(H^\frac{3}{4}((s_0,s_1),\R)),H^{\frac{2}{3}}((s_0,s_1),\R))}
\norm{P_M 1_{\Gamma_{\rm c}}}{\LL(G^{\frac{3}{2}}_{\rm c}((s_0,s_1),\Gamma)),G^1_{\rm c}((s_0,s_1),\Gamma)},
\end{split}
\]
where $G^{\frac{3}{2}}_{\rm c}((s_0,s_1),\Gamma)\coloneqq [G^{2}_{\rm c}((s_0,s_1),\Gamma),G^{1}_{\rm c}((s_0,s_1),\Gamma)]_\frac{1}{2}$
is an interpolation space in the sense of~\cite{LioMag72-I}.
If $P_M$ satisfies~\eqref{M_bound-gen-bdry} with, say, $\frac{\Upsilon_\Gamma^{-1}}{2}$ in the place of $\Upsilon_\Gamma^{-1}$, then we can
set~$P^t_{\tilde M}$ (i.e., set $\tilde M$ big enough) such that~\eqref{cond.spacetime} holds true. Finally, notice that
with $X=W^1((s_0,s_1),H^2(\Omega,\R),H)$ and~$Y=W^1((s_0,s_1),H^1(\Omega,\R),H^{-1}(\Omega,\R))$ we have
(using some results from~\cite[chapter~4, section~2.2]{LioMag72-II})
\[
\begin{split}
G^{\frac{3}{2}}_{\rm c}((s_0,s_1),\Gamma)&=[X,Y]_\frac{1}{2}\rest\Gamma,\\
G^{2}_{\rm c}((s_0,s_1),\Gamma)&\xhookrightarrow{\rm c} G^{\frac{3}{2}}_{\rm c}((s_0,s_1),\Gamma)\xhookrightarrow{\rm c} G^{1}_{\rm c}((s_0,s_1),\Gamma),\\
W^1((s_0,s_1),H^\frac{3}{2}(\Omega,\R),H)\rest\Gamma&=L^2((s_0,s_1),H^1(\Gamma,\R))\cap H^\frac{2}{3}((s_0,s_1),L^2(\Gamma,\R)),\\
W^1((s_0,s_1),H^\frac{3}{2}(\Omega,\R),H)\rest\Gamma&=[X,W^1((s_0,s_1),H^1(\Omega,\R),H)]_\frac{1}{2}\rest\Gamma\\
&\xhookrightarrow{}G^{\frac{3}{2}}_{\rm c}((s_0,s_1),\Gamma),\\
H^\frac{3}{4}((s_0,s_1),\R)&\xhookrightarrow{\rm c} H^\frac{2}{3}((s_0,s_1),\R).
\end{split}
\]
\end{remark}

\subsubsection{The auxiliary extended system}
Once we have that $u_i\in H^1((s_0,s_1),\R)$, for each $i\in\{1,2,\dots,M\}$, then we can rewrite~\eqref{sys-z2cf-bdry}
in the variables $(y_{\bar\lambda},\kappa_{\bar\lambda})=(z_{\bar\lambda}-B_\Psi\kappa_{\bar\lambda},\kappa_{\bar\lambda})$, as the extended system
\begin{align}
  &\partial_t y_{\bar\lambda}-\Delta y_{\bar\lambda}+ (a-\textstyle\frac{\bar\lambda}{2}) y_{\bar\lambda}+\nabla\cdot (by_{\bar\lambda}) -\Delta B_\Psi\kappa_{\bar\lambda}
  + a B_\Psi\kappa_{\bar\lambda}+\nabla\cdot (b B_\Psi\kappa_{\bar\lambda})+{\bar\varsigma} B_\Psi\kappa_{\bar\lambda}= B_\Psi\varkappa_{\bar\lambda},\notag\\
  &\p_t \kappa_{\bar\lambda}-\textstyle\frac{\bar\lambda}{2}\kappa_{\bar\lambda}+{\bar\varsigma}\kappa_{\bar\lambda}=\varkappa_{\bar\lambda},\label{sys-z2cf-ext}\\
  &y_{\bar\lambda}(s_0)=y_0=z_0- B_\Psi\kappa_0, \qquad \kappa_{\bar\lambda}(s_0)= \kappa_0,\qquad y_{\bar\lambda}\rest \Gamma = 0,\notag
\end{align}
where $B_\Psi\colon\R^M\to H^2(\Omega,\R)$ is given by
\begin{equation}\label{BPsi-ext}
B_\Psi\kappa\coloneqq\sum\limits_{i=1}^M\kappa_i\widetilde\Psi_i,\quad\mbox{with}\quad
\widetilde\Psi_i\in H^2(\Omega,\R),\quad\widetilde\Psi_i\rest\Gamma=\Psi_i\coloneqq 1_{\Gamma_{\rm c}}\chi_\Gamma\widehat\Psi_i,
\end{equation}
where the extensions~$\widetilde\Psi_i$ of the~$\Psi_i$s are fixed. Recall that
by assumption $\Psi_i\in H^\frac{3}{2}(\Gamma,\R)=H^2(\Omega,\R)\rest\Gamma$ for all $i\in\{1,2,\dots,M\}$.
Furthermore, ${\bar\varsigma}\in\R$ is a parameter at our disposal, notice that it does not appear in~\eqref{sys-z2cf-bdry}.
 
From Proposition~\ref{P:stab-v-gen-bdry} we can conclude that there exists $\kappa=u\in H^1((s_0,s_1),\R^M)$ such that system~\eqref{sys-z2cf-ext}
is stabilizable to zero with rate~$0$, provided~\eqref{cond.spacetime} holds true.

Hereafter we will use a particular extension, namely, 
the given actuators $\Psi_i$, defined on the boundary~$\Gamma$, are extended to~$\widetilde\Psi_i$, defined in~$\Omega$, by solving the elliptic system
\begin{equation}\label{part-extPsi}
 -\Delta\widetilde\Psi+{\bar\varsigma}\widetilde\Psi_i=0,\qquad \widetilde\Psi_i\rest\Gamma=\Psi_i.
\end{equation}
The following discretization is valid only for this particular extension, which as we see leads to suitable simplifications on the corresponding Riccati equations, namely the
2nd order space derivatives term~$-\Delta B_{\widetilde\Psi}\kappa$ disappears in~\eqref{sys-z2cf-ext},
because $(A^{a,b} +{\bar\varsigma})B_{\widetilde\Psi}$ reduces to~$aB_{\widetilde\Psi}+\nabla\cdot(b B_{\widetilde\Psi})$. To simplify the exposition, we rewrite system~\{\eqref{sys-z2cf-ext},\eqref{part-extPsi}\}, for $(y,\kappa)\in H\times\R^M$, as
\begin{align*}
  \partial_t \begin{bmatrix}y_{\bar\lambda}\\\kappa_{\bar\lambda}\end{bmatrix}
               +\begin{bmatrix}-\Delta +K^{a,b,\bar\lambda}&K^{a,b,0}\\0&{\bar\varsigma}_{\bar\lambda}\end{bmatrix}
               \begin{bmatrix}y_{\bar\lambda}\\\kappa_{\bar\lambda}\end{bmatrix}
               - \begin{bmatrix}B_\Psi\\1\end{bmatrix}\varkappa_{\bar\lambda}=0,\qquad
  \begin{bmatrix}y_{\bar\lambda}(s_0)\\\kappa_{\bar\lambda}(s_0)\end{bmatrix}=\begin{bmatrix}y_0\\\kappa_0\end{bmatrix}, 
\end{align*}
where ${\bar\varsigma}_{r}$ and $K^{a,b,r}\in\LL(V+B_\Psi\R^M,V')$, for $r\ge0$,  are given by
\begin{equation}\label{sigKlambda}
 {\bar\varsigma}_{r}\coloneqq{\bar\varsigma}-\textstyle\frac{r}{2},\qquad K^{a,b,r}w\coloneqq a w+\nabla\cdot (bw)-\textstyle\frac{r}{2}.
\end{equation}

As in the internal case we can prove that the control can be taken in feedback form,
\begin{align*}
\varkappa_{\bar\lambda}={\FF_{\bar\lambda}}(t)(y_{\bar\lambda},\kappa_{\bar\lambda})=-{B_M^{\Gamma}}^*\Pi^\Gamma_{\bar\lambda}(t)(y_{\bar\lambda},\kappa_{\bar\lambda}),
\end{align*}
with $B_M^\Gamma=\begin{bmatrix}B_\Psi\\1\end{bmatrix}$, with adjoint~${B_M^\Gamma}^*=\begin{bmatrix}B_\Psi^*&1\end{bmatrix}$. 
Furthermore the operator~$\Pi^\Gamma_{\bar\lambda}$ can be chosen to satisfy a differential Riccati equation
\begin{equation}\label{eq:riccati-bdry}
\frac{\mathrm d}{\mathrm dt} \Pi^\Gamma_{\bar\lambda}+\Pi^\Gamma_{\bar\lambda}\mathbb A^{a,b}_{\bar\lambda,{\bar\varsigma}}
 +{\mathbb A^{a,b}_{\bar\lambda,{\bar\varsigma}}}^*\Pi^\Gamma_{\bar\lambda} -\Pi^\Gamma_{\bar\lambda}  B_M^\Gamma  \mathcal R ^{-1}
{B_M^\Gamma}^* \Pi^\Gamma_{\bar\lambda} +\mathcal M^*\mathcal M=0,
\end{equation}
with
\[\mathbb A^{a,b}_{\bar\lambda,{\bar\varsigma}}
=-\begin{bmatrix}-\Delta +K^{a,b,\bar\lambda}&K^{a,b,0}\\0&{\bar\varsigma}_{\bar\lambda}\end{bmatrix},
\]
and for suitable~$\RR\in\LL(\R^M)$ and~$\MM\in
\LL(V\times\R^M\to V'\times\R^M)$. In this case, the obtained feedback control is the one that minimizes
\[
 \int_{s_0}^{+\infty}\norm{\MM\begin{bmatrix}y_{\bar\lambda}(\tau)\\\kappa_{\bar\lambda}(\tau)\end{bmatrix}}{H\times\R^M}^2 +\varkappa_{\bar\lambda}(\tau)^\top \RR\varkappa_{\bar\lambda}(\tau)\ed\tau.
\]
For example we can take~$\RR=1$ and~$\MM=\begin{bmatrix}(-\Delta)^\frac{1}{2}&0\\0&1\end{bmatrix}$ as in~\cite{BarRodShi11,KroRod15}.
From Lemma~\ref{L:L2H1b-H1} we can also take $\MM=\begin{bmatrix}1&0\\0&1\end{bmatrix}$ instead (cf.~\cite{Rod-pp15,BreKunRod-pp15}).

It follows (cf. Theorem~\ref{T:feed} and Corollary~\ref{C:feedz}) that the solution $(y_{\bar\lambda},\kappa_{\bar\lambda})$ of the system
\begin{align}\label{sys-ext-Feed}
  \partial_t \begin{bmatrix} y_{\bar\lambda}\\ \kappa_{\bar\lambda}\end{bmatrix}
                { -\mathbb A^{a,b}_{\bar\lambda,{\bar\varsigma}}\begin{bmatrix}y_{\bar\lambda}\\ \kappa_{\bar\lambda}\end{bmatrix} }
                   + \begin{bmatrix}B_\Psi\\1\end{bmatrix}\begin{bmatrix}B_\Psi^*&1\end{bmatrix}\Pi^\Gamma_{\bar\lambda}\begin{bmatrix} y_{\bar\lambda}\\ \kappa_{\bar\lambda}\end{bmatrix}=0,\quad
  \begin{bmatrix} y_{\bar\lambda}(s_0)\\ \kappa_{\bar\lambda}(s_0)\end{bmatrix}=\begin{bmatrix}y_0\\\kappa_0\end{bmatrix}, 
\end{align}
remains bounded:
$
 \norm{( y_{\bar\lambda}, \kappa_{\bar\lambda})}{H\times\R^M}^2\le\overline C_{\left[C_{\mathcal W},\bar\lambda,\frac{1}{\bar\lambda}\right]}\norm{(y_0,\kappa_0}{H\times\R^M}^2. 
$
Therefore, the solution of
 \begin{align}\label{sys-ext-Feed-orig}
  \partial_t \begin{bmatrix}y\\\kappa\end{bmatrix}
              + {\mathbb A^{a,b}_{0,{\bar\varsigma}}}\begin{bmatrix} y\\ \kappa\end{bmatrix}
               + \begin{bmatrix}B_\Psi\\1\end{bmatrix}\begin{bmatrix}B_\Psi^*&1\end{bmatrix}\Pi^\Gamma_{\bar\lambda}\begin{bmatrix}y\\\kappa\end{bmatrix}=0,\quad
  \begin{bmatrix}y(s_0)\\\kappa(s_0)\end{bmatrix}=\ex^{-\frac{\bar\lambda}{2}s_0}\begin{bmatrix}y_0\\\kappa_0\end{bmatrix}, 
\end{align}
 goes exponentially to~$0$ with rate~$\frac{\bar\lambda}{2}$. Notice that $(y_{\bar\lambda},\kappa_{\bar\lambda})$ solves~\eqref{sys-ext-Feed} if, and only if,
 $(y,\kappa)=\ex^{-\frac{\bar\lambda}{2}t}(y_{\bar\lambda} , \kappa_{\bar\lambda})$ solves~\eqref{sys-ext-Feed-orig}.

\subsubsection{From the auxiliary to the original system}
Once we have the feedback rule
\[
\Pi^\Gamma_{\bar\lambda}\eqqcolon\begin{bmatrix}\Pi^{\Gamma,\bar\lambda}_{1,1}&\Pi^{\Gamma,\bar\lambda}_{1,2}\\{\Pi^{\Gamma,\bar\lambda}_{1,2}}^*&\Pi^{\Gamma,\bar\lambda}_{2,2}\end{bmatrix}\in\LL(H\times\R^M)
\]
we observe that
\[
 \varkappa=-\begin{bmatrix}B_\Psi^* &1\end{bmatrix}\Pi^\Gamma_{\bar\lambda}\begin{bmatrix}y\\\kappa\end{bmatrix}=-(B_\Psi^*\Pi^{\Gamma,\bar\lambda}_{1,1}+{\Pi^{\Gamma,\bar\lambda}_{1,2}}^*)y
 - (B_\Psi^*\Pi^{\Gamma,\bar\lambda}_{1,2}+\Pi^{\Gamma,\bar\lambda}_{2,2})\kappa,
\]
and that $z_{\bar\lambda}=y_{\bar\lambda}+B_\Psi\kappa_{\bar\lambda}$ solves for time~$t\ge s_0$
 \begin{subequations}\label{sys-Feed-orig1}
 \begin{align}
  &\partial_t z_{\bar\lambda}-\Delta z_{\bar\lambda}+ (a-\textstyle\frac{\bar\lambda}{2})z_{\bar\lambda}+\nabla\cdot (bz_{\bar\lambda})=0,
\quad z_{\bar\lambda}(s_0)={y_{0}+B_\Psi\kappa_{0}},\label{sys-Feed-orig1-z}\\
  &z_{\bar\lambda}\rest\Gamma=B_\Psi^\Gamma\left(\ex^{-{\bar\varsigma} t}{\kappa_{0}}-\int_{s_0}^{\Bigcdot}\ex^{-{\bar\varsigma} (\Bigcdot-s)}
  \begin{bmatrix}B_\Psi^* &1\end{bmatrix}\Pi^\Gamma_{\bar\lambda}(s)\begin{bmatrix}y_{\bar\lambda}(s)\\\kappa_{\bar\lambda}(s)\end{bmatrix}\,\ed s\right),
  \end{align}
and remains bounded. Here $B_\Psi^\Gamma\colon\R^M\to\linspan 1_{\Gamma_{\rm c}}\chi_\Gamma\widehat\CC_\Gamma$ stands for the mapping
\begin{equation*}
 B_\Psi^\Gamma u \mapsto\sum_{i=1}^M u_i 1_{\Gamma_{\rm c}}\chi_\Gamma\widehat\Psi_i.
\end{equation*}

Notice that, without loss of generality, we can suppose that the functions $\Psi_i$ are linearly independent.
In this case, we have that $V\cap B_\Psi\R^M=\emptyset$. Thus, since for a.e. $t\ge s_0$ we have $z(t)\in V+B_\Psi\R^M$, we can write
$z(t)=y^z(t)+ B_\Psi\kappa^z(t)=0$ in a unique way, with~$(y^z(t),\kappa^z(t))\in V\times\R^M$. That is, $y_{\bar\lambda}(s)$ and $\kappa_{\bar\lambda}(s)$ appearing in~\eqref{sys-Feed-orig1}
can be constructed from~$z_{\bar\lambda}(s)$.

A procedure to construct the mapping $z\mapsto(y^z,\kappa^z)$ is the following:
we orthonormalize (e.g., by applying Gram--Schmidt procedure)
the family $\widetilde\Psi_i$, say in the $H$-scalar product. In this way we arrive to the orthonormal family
$\breve\CC=\{\breve\Psi_i\mid i\in\{1,2,\dots,M\}\}$,
Now we can write $v=w+\sum_{i=1}^M\xi_i\breve\Psi_i$, with $\xi_i\coloneqq (v,\breve\Psi_i)_H$. Finally, we construct~$\sigma$ by a matrix of change of coordinates
$\sum_{i=1}^M\sigma_i\widetilde\Psi_i\coloneqq\sum_{i=1}^M\xi_i\breve\Psi_i$, and still have $v=w+\sum_{i=1}^M\sigma_i\widetilde\Psi_i$.
That is, denoting $y^z=w$ and~$\kappa^z=\sigma$, the mapping $z(t)\mapsto(y^{z(t)},\kappa^{z(t)})$ is well defined and we can
rewrite the integral feedback rule in~\eqref{sys-Feed-orig1} as
\begin{equation}\label{sys-Feed-orig1-int}
z\rest\Gamma=B_\Psi^\Gamma\left(\ex^{-{\bar\varsigma} t}{\kappa_0}-\int_{s_0}^{\Bigcdot}\ex^{-{\bar\varsigma} (\Bigcdot-s)}
\begin{bmatrix}B_\Psi^* &1\end{bmatrix}\Pi^\Gamma_{\bar\lambda}(s)\begin{bmatrix}y^{z(s)}\\\kappa^{z(s)}\end{bmatrix}\,\ed s\right),
\end{equation}
\end{subequations}
which underlines the (integral) feedback nature of the control, also as a boundary control.

\subsection{The nonlinear systems}
As we have done for the case of internal controls, under suitable conditions on the nonlinear function~$\NN$ we can derive also that the local stabilization
result holds for nonlinear system in the form~\eqref{sys-Feed-orig1} with~$\NN$ as a perturbation
\begin{subequations}\label{sys-dynFeed-orig1}
\begin{align}
  \partial_t z_{\bar\lambda}+ A^{a,b}z_{\bar\lambda} -\textstyle\frac{\bar\lambda}{2}z_{\bar\lambda}&=\NN(z_{\bar\lambda}),
  \quad z_{\bar\lambda}\rest\Gamma=B_\Psi^\Gamma\kappa_{\bar\lambda},&
 z_{\bar\lambda}(s_0)&=y_{\bar\lambda}(s_0)+B_\Psi\kappa_{\bar\lambda}(s_0),\\
 \p_t \kappa_{\bar\lambda}-\textstyle\frac{\bar\lambda}{2}\kappa_{\bar\lambda}+{\bar\varsigma}\kappa_{\bar\lambda}&
 =-\begin{bmatrix}B_\Psi^* &1\end{bmatrix}\Pi^\Gamma_{\bar\lambda}\begin{bmatrix}y_{\bar\lambda}-B_\Psi^\Gamma\kappa_{\bar\lambda}\\ \kappa\end{bmatrix},
 & \kappa_{\bar\lambda}(s_0)&=\kappa_{0},\label{sys-dynFeed-orig1-k}
  \end{align} 
\end{subequations}

The procedure is analogous to the one in section~\ref{S:nonlinear_int} (cf.~\cite{Rod-pp15}), so we will not repeat the details here.
However, it is important to recall that since we need (in general) strong solutions to deal with the nonlinearity, it is important that the initial condition satisfies
the compatibility condition $z(s_0)=z_0\in V+B_\Psi\R^M$ (cf.Lemma~\ref{L:strongbdry-z}). This means that
when dealing with stabilization to trajectories (cf.~section~\ref{S:stab_traj}, for the internal controls case),
we need the compatibility condition
\begin{align}
 (y(s_0)-\hat y(s_0))\rest\Gamma\in B_\Psi^\Gamma\R^M=\linspan\{1_{\Gamma_{\rm c}}\chi_\Gamma\widehat\Psi_i\}.
 \label{CompatibilityCond}
\end{align}
At this point we also recall that the solutions given in Lemmas~\ref{L:weakbdry-z} and~\ref{L:strongbdry-z} do not depend on the extension~$E_1$ and~$E_2$, respectively.
See~\cite[Remark~3.2]{Rod14-na}.

Similar estimates on the transient bound, and on dimension of the control like~\eqref{estM_eig} and~\eqref{estM_pc}, depending exponentially on the data,
can also be derived.

It turns out that, in both internal and boundary cases, numerical simulations do suggest that a better estimate on the number of needed actuators might exist,
like the one in~\eqref{Msimplecase} depending polynomially on the data. So, we would like to finish this section with two very particular questions:
\begin{itemize}
 \item Can we find a particular case, in the boundary setting,
leading to an estimate like~\eqref{Msimplecase}? 
\item How does the observability constant~$C_{\rm obs}$ in~\eqref{obs-in-BGamma*} ``looks like'' when the control acts on all the boundary?
Can we get, in this case, a constant which is ``better'' than those in~\eqref{obs-ineq-b23}?
\end{itemize}%

\begin{remark}
Though the closed-loop systems~\{\eqref{sys-Feed-orig1-z},\,\eqref{sys-Feed-orig1-int}\} and~\eqref{sys-dynFeed-orig1} are formally equivalent,
the latter may have some advantages for numerical simulations, because the
dynamical equation in~\eqref{sys-dynFeed-orig1-k} is relatively easier to discretize than the integral equation in~\eqref{sys-Feed-orig1-int}.
Hereafter, we will consider only the discretization of~\eqref{sys-dynFeed-orig1}.
\end{remark}

\subsection{Back to original time. Stabilization to trajectories}
Let us be give a solution~$\hat y$ for the uncontrolled system~\eqref{sys-y-cont_b} (with~$u=0$) with~$\hat y_0\coloneqq\hat y(0)\in H$ and such that the vector functions
in~\eqref{abN_lineariz} are such that~$\widehat\NN$ satisfies~\eqref{exp-non2} and \eqref{exp-non3}, and~$(\hat a,\hat b)$ satisfies~\eqref{rhosigma-st}, for suitable nonnegative
constants
$\widehat C$ and $C_{\WW,\rm st}$.
Notice that, recalling the notation in section~\ref{S:red-stabil}, in this case~$\NN\coloneqq\frac{1}{\nu}\breve{\widehat\NN}$ also
satisfies~\eqref{exp-non2} and \eqref{exp-non3}, and~$(a,b)=(\frac{1}{\nu}\breve{\hat a},\frac{1}{\nu}\breve{\hat b})$ also satisfies~\eqref{rhosigma-st},

We consider system~\eqref{sys-y-cont_b} with a dynamical feedback as follows
\begin{subequations}\label{sys-y-cont_b-feed}
\begin{align}
 \p_t y - \nu \Delta y + f(y,\nabla y) = 0&,\qquad
 &y\rest\Gamma &= g-B_\Psi^\Gamma\kappa,\\
 \p_t \kappa+\varsigma\kappa
 =-\begin{bmatrix}B_\Psi^* &1\end{bmatrix}\widehat\Pi^\Gamma_{\lambda}\begin{bmatrix}y-\hat y-B_\Psi^\Gamma\kappa\\ \kappa\end{bmatrix}&,\qquad &y(0)&=y_0,
 \end{align}
 \end{subequations}
with $\widehat\Pi_{\lambda}$  solving

\begin{equation}\label{eq:riccati-bdry-traj}
 \textstyle\frac{\ed}{\ed t} {\widehat\Pi}^\Gamma_{\lambda}+{\widehat\Pi}^\Gamma_{\lambda}\mathbb A^{\hat a,\hat b,\nu}_{\lambda,\varsigma}
 +{\mathbb A^{\hat a,\hat b,\nu}_{\lambda,\varsigma}}^*{\widehat\Pi}^\Gamma_{\lambda}
 -{\widehat\Pi}^\Gamma_{\lambda}  B_M^\Gamma  \mathcal R ^{-1}
{B_M^\Gamma}^* {\widehat\Pi}^\Gamma_{\lambda} +\mathcal M^*\mathcal M=0,
\end{equation} 
with~$\RR$ and~$\MM$ as in~\eqref{eq:riccati-bdry} and
with
\begin{equation}\label{A-Ric-orig}
\mathbb A^{\hat a,\hat b,\nu}_{\lambda,\varsigma}
=-\begin{bmatrix}-\nu \Delta +K^{\hat a,\hat b,\lambda}&K^{\hat a,\hat b,0}\\0&{\varsigma}_{\lambda}\end{bmatrix},
\end{equation}
and $K^{\hat a,\hat b,r}$ and ${\varsigma}_{\lambda}$ defined as in~\eqref{sigKlambda}.

Observe that $(\breve z,\breve \kappa)(\tau)\coloneqq (y-\hat y,\kappa)(\frac{\tau}{\nu})$ solves, with
\begin{align*}
 \p_\tau \breve z - \Delta \breve z +  a\breve z+ \nabla\cdot(b\breve z) &= \NN(\breve z);&\quad
 \breve z\rest\Gamma &= -B_\Psi^\Gamma\breve\kappa,\\
 \p_t \breve \kappa+\textstyle\frac{\varsigma}{\nu}\breve \kappa&
 =-\begin{bmatrix}B_\Psi^* &1\end{bmatrix}\breve{\widehat\Pi}{}^\Gamma_{\lambda}\begin{bmatrix}\breve z-B_\Psi^\Gamma\breve \kappa\\ \breve \kappa\end{bmatrix},
 &\quad \breve z(0)&=y_0-\hat y_0,
 \end{align*}
and ${\breve{\widehat\Pi}}{}_{\lambda}^\Gamma(\tau)=\widehat\Pi_{\lambda}^\Gamma(\frac{\tau}{\nu})$ solves~\eqref{eq:riccati-bdry} with a different pair~$(\RR,\MM)$.
Indeed, from~\eqref{eq:riccati-bdry-traj}, we obtain that for $(\bar\lambda,\bar\varsigma)=(\frac{\lambda}{\nu},\frac{\varsigma}{\nu}),$
{
\[
\begin{split}
\textstyle\frac{\ed}{\ed\tau}{\breve{\widehat\Pi}}{}^\Gamma_{\lambda}+\breve{\widehat\Pi}_{\lambda}\textstyle\mathbb A^{a,b,1}_{\bar\lambda,\bar\varsigma}
+\textstyle{\mathbb A^{a,b,1}_{\bar\lambda,\bar\varsigma}}^*\breve{\widehat\Pi}_{\lambda}
 -\breve{\widehat\Pi}_{\lambda}  B_M  (\nu\mathcal R)^{-1}B_M^* \breve{\widehat\Pi}_{\lambda}
 +((\textstyle\frac{1}{\nu})^\frac{1}{2}\mathcal M)^*((\textstyle\frac{1}{\nu})^\frac{1}{2}\mathcal M)=0.
\end{split}  
\]
}%
Therefore, we can conclude that
\[
 |\breve z(\tau)|_{V}^2 \le {C\mathrm{e}^{-\bar\lambda\tau}}
|\breve z(0)|_{V}^2,\quad\mbox{for all }\tau\ge0,
\]
provided $|\breve z(0)|_{V}$ is small enough. This implies that
\[
 |y(t)-\hat y(t)|_{V}^2 \le C\mathrm{e}^{-\lambda t}
|y_0-\hat y_0|_{V}^2,\quad\mbox{for all }t\ge0,
\]
and we can conclude that the following theorem holds true. Recall the operators~\eqref{eq:kk2-bdry} and~\eqref{BPsi-ext}.
\begin{theorem}\label{T:st-traj-bdry}
Under the hypotheses of Proposition~\ref{P:stab-v-gen-bdry}, with
\[ \Ran P_M=\linspan\left\{\widehat\Psi_i\mid i\in\{1,2,\dots,M\}\right\}\subset H^\frac{3}{2}(\Gamma,\R),\]
there exists~$\epsilon>0$ with the following properties:
if~$y_0\in H$ is such that
\[
y_0-\hat y_0\in V+B_\Psi\R^M\quad\mbox{and}\quad\norm{y_0-\hat y_0}{V}<\epsilon,
\]
then the solution~$y$ of the system~\eqref{sys-y-cont_b-feed}
goes exponentially to~$\hat y$ with rate~$\frac{\lambda}{2}$, that is,
\[
 |y(t)-\hat y(t)|_{V}^2 \le C\mathrm{e}^{-\lambda t}
|y_0-\hat y_0|^2_{V},\quad\mbox{for all }t\ge0,
\]
for a suitable constant~$C$ independent of~$(\epsilon,y_0-\hat y_0)$. Furthermore, the solution~$y$ is, and is unique, in the affine space
$\hat y+L^2_{\rm loc}(\mathbb{R}_0,{\mathrm D}(A)+B_\Psi\R^M)\cap C([0,+\infty),V+B_\Psi\R^M)$.
\end{theorem}

Notice that the feedback control stabilizes the linearized system to zero globally, that is, we have the following theorem.
\begin{theorem}\label{T:st-traj-bdry-lin}
Under the hypotheses of Proposition~\ref{P:stab-v-gen-bdry}, with
\[\Ran P_M=\linspan\left\{\widehat\Psi_i\mid i\in\{1,2,\dots,M\}\right\}\subset H^\frac{3}{2}(\Gamma,\R),\]
given~$(z_0,\kappa_0)\in H\times\R^M$, the solution of
\begin{subequations}\label{sys-int-bdry-lin-RT}
\begin{align}
 \p_t z - \nu \Delta z +  a z+ \nabla\cdot(b z) = 0&,&\qquad
 z\rest\Gamma &= -B_\Psi^\Gamma\kappa,\\
 \p_t \kappa+\varsigma\kappa
 =-\begin{bmatrix}B_\Psi^* &1\end{bmatrix}\widehat\Pi^\Gamma_{\lambda}\begin{bmatrix}z-B_\Psi^\Gamma\kappa\\ \kappa\end{bmatrix}&,
 &\qquad \begin{bmatrix}z(0)\\ \kappa(0)\end{bmatrix}&=\begin{bmatrix}z_0\\ \kappa_0\end{bmatrix}.
 \end{align}
\end{subequations}
satisfies
\[
|z(t)|_{H}^2 \le C\mathrm{e}^{-\lambda t}
|z_0|_{H}^2,\quad\mbox{for all }t\ge0,
\]
for a suitable constant~$C$ independent of~$z_0$. Furthermore, the solution~$z$ is, and is unique, in the space
$L^2_{\rm loc}(\mathbb{R}_0,{\mathrm D}(A)+B_\Psi\R^M)\cap C([0,+\infty),V+B_\Psi\R^M)$.
\end{theorem}

\section{Discretization of the linear systems}\label{S:disc_lin}
We explain how we discretize a linear parabolic equation with nonhomogeneous Dirichlet boundary conditions. We will focus our simulations on the 2D case,
considering our domain to be the unit ball $\Omega=\mathbb D\coloneqq\{(x_1,x_2)\in\R^2\mid x_1^2+x_2^2<1\}$.

Here, we focus on the approximation of the linearized closed-loop systems~\eqref{sys-int-lin-RT} and~\eqref{sys-int-bdry-lin-RT} 
perturbed with the reaction term~$\frac{\lambda}{2}z$:
\begin{equation}\label{PSInternal}
\begin{aligned}
\partial_t  z-\nu\Delta z+ (\hat a-\textstyle\frac{\lambda}{2})  z+ \nabla\cdot (\hat b z)+{\FF_{\lambda}^{\rm in}} z&=0,\\
   z(0)&=z_0, &
 \end{aligned}\qquad
 \begin{aligned}
z \rest \Gamma &=0,\\&
 \end{aligned}
 \end{equation} 
 and
  \begin{subequations}\label{PSBoundary}
\begin{align}
\partial_t z-\nu\Delta z+ (\hat a-\textstyle\frac{\lambda}{2}) z+\nabla\cdot (\hat bz)&=0,&z\rest\Gamma&=B_\Psi^\Gamma\kappa,\label{PSBoundary-z}\\
\partial_t\kappa+(\varsigma-\textstyle\frac{\lambda}{2})\kappa+{\FF_{\lambda}^{\rm bo}}(z-B_\Psi\kappa,\kappa)&=0,&
(z(0),\kappa(0))&=(z_0,\kappa_0),\label{PSBoundary-k}
 \end{align}
 \end{subequations} 
with ${\FF_{\lambda}^{\rm in}}z\coloneqq B_M\RR^{-1}B_M^*\widehat\Pi_{\lambda} z$ and
${\FF_{\lambda}^{\rm bo}}(z-B_\Psi\kappa,\kappa)\coloneqq\begin{bmatrix}B_\Psi^* &1\end{bmatrix}\widehat\Pi^\Gamma_{\lambda}\begin{bmatrix}z-B_\Psi\kappa\\\kappa\end{bmatrix}$.
 
\begin{remark}
 Notice that we want to observe that the solutions of~\eqref{sys-int-lin-RT},
 respectively~\eqref{sys-int-bdry-lin-RT} go to zero exponentially with rate~$\frac{\lambda}{2}$. Thus we want
 the solutions of~\eqref{PSInternal} and~\eqref{PSBoundary} to remain bounded. Indeed~$z$ solves~\eqref{sys-int-lin-RT},
 respectively~$(z,\kappa)$ solves~\eqref{sys-int-bdry-lin-RT}, if and only if~$\ex^{\frac{\lambda}{2}\Bigcdot}z$ solves~\eqref{PSInternal},
 respectively~$\ex^{\frac{\lambda}{2}\Bigcdot}(z,\kappa)$ solves~\eqref{PSBoundary}.
\end{remark}

\subsection{Discretization in space}
The simulations are done in MATLAB. We approximate~$\Omega$ by a polygonal domain~${\Omega}_D$ and consider a partition of~${\Omega}_D$
into nonoverlapping triangles. For this we use the function~{\tt initmesh} from MATLAB. This function gives us a mesh triple~$(\mathbf p,\mathbf e,\mathbf t)$ where
\begin{itemize}
\item the point matrix $\mathbf p$ contains (information about) all the vertices of all triangles of the partition.
\item the edge matrix $\mathbf e$ contains (information about) all the {\em boundary} segments of~${\Omega}_D$.
\item the triangle matrix $\mathbf t$ contains (information about) the triangles of the partition.
\end{itemize}

Recall that given a triangle $\ttt_k$ with vertices~$(\ppp_{\ttt_k(1)},\ppp_{\ttt_k(2)},\ppp_{\ttt_k(3)})$, then any point~$x\in\ttt_k$ (inside or on the boundary of the triangle~$\ttt_k$)
can be written uniquely as a convex combination: there are nonnegative~{$x^{\ttt_k(1)}$, $x^{\ttt_k(2)}$, and~$x^{\ttt_k(3)}$} such that
\begin{align*}
1&=x^{\ttt_k(1)}+x^{\ttt_k(2)}+x^{\ttt_k(3)},\\
 x&=x^{\ttt_k(1)}\ppp_{\ttt_k(1)}+x^{\ttt_k(2)}\ppp_{\ttt_k(2)}+x^{\ttt_k(3)}\ppp_{\ttt_k(3)}.
\end{align*}
Any continuous function~$z\in C(\overline\Omega,\R)$ can, and will, be approximated by a sum
\[
 z(x)\approx\tilde z(x)\coloneqq\sum_{i=1}^{s_{\mathbf p}}z(\mathbf p_i)\widehat\phi_i(x),\quad\mbox{for all}\quad x\in\overline{\Omega_D},
\]
where $s_{\mathbf p}$ is the total number of points in the mesh; $\{\mathbf p_i\mid i=1,2,\dots,s_{\mathbf p}\}$ is the set of all mesh points and the $\widehat\phi_i$s are the classical piecewise linear hat functions defined as
\[
\widehat\phi_i(x)\coloneqq\begin{cases}
1,&\mbox{if }x=\mathbf p_i;\\
0,&\mbox{if }x=\mathbf p_j\mbox{ with }j\ne i;\\
\sum_{l=1}^3x^{\ttt_k(l)}\widehat\phi_i(\ppp_{\ttt_k(l)}),&\mbox{if }x\in\ttt_k.
\end{cases} 
\]
Notice that the support of $\phi_i$ consists of the triangles with the common vertex $\ppp_i$. We will denote the finite-dimensional space
\[
 V_D\coloneqq\linspan\{\widehat\phi_i\mid i\in\{1,2,\dots,s_{\ppp}\}\}.
\]
In other words the function~$z$ can be approximated by the {\em evaluation} vector
\[
\overline{z} \coloneqq\begin{bmatrix} z(\ppp_1) & z(\ppp_2) & \dots & z(\ppp_{s_\ppp}) \end{bmatrix}^{\top} \in \MM_{s_\ppp \times 1},
\]
where $A^{\top}$ stands for the transpose matrix of $A$.

The weak {\em discretization matrix} $\mathbf{L}_D$ of a given operator $L \in \LL (H^1 \rightarrow V')$, is defined so that (for smooth functions)
\begin{align*}
\overline{v}^{\top} \mathbf{L}_D \overline{u} \coloneqq (L\tilde u,\tilde v)_H=(\tilde v,L\tilde u)_H,
\end{align*}
that is, the entry in the $i$th row and $j$th column of~$\mathbf{L}_D$ is $(\mathbf{L}_D)_{(i,j)} = (\widehat\phi_i, L \widehat\phi_j)_{H}$.

We recall the so called {\em mass}~$\Ma$ and {\em stiffness}~$\St$ matrices are defined as
\[
 \Ma_{(i,j)}=(\widehat\phi_i, \widehat\phi_j)_H\quad\mbox{and}\quad \St_{(i,j)}\coloneqq (\nabla\widehat\phi_i, \nabla\widehat\phi_j)_H.
\]
Notice that $\Ma=(Id)_D$ where $Id\colon v\mapsto v$ is the identity/inclusion operator, and that $\St$ is related to the Laplacian:
for smooth $(u,v)\in H^1(\Omega)\times H^1_0(\Omega)$ we have $\langle -\Delta u, v\rangle_{V',V}=(\nabla u, \nabla v)_H\approx (\nabla \tilde u, \nabla \tilde v)_H
=\overline v^\top\St \overline u$.

We refer to~\cite[Section~1.3 to~1.5]{Chen05} for further details.

\subsubsection{Discretization of the heat equation} We are ready to semi-discretize the system
\begin{subequations}\label{heat-sys}
\begin{align}
\partial_t z  - \nu \Delta z + f &= 0, \label{heat-sys-dyn}\\
z\rest{\Gamma} = g, \qquad z(0)&=z_0,\label{heat-sys-data}
\end{align}
\end{subequations}
provided the functions~$f$ and~$g$ are known (and continuous in space variable). 

Inspired from Lemma~\ref{L:weakbdry-z} we look
for a solution in~$W(I,H^1,V')$, thus we look at~\eqref{heat-sys-dyn} as an identity in~$L^2(I,V')$. Therefore, it is enough to test this equation with 
elements $w$ in the dual~$L^2(I,V)$.
Replacing $z$ by $\tilde z$ and testing with a 
function~$\tilde w\in C(I,V_D)$ with~$\tilde w\rest\Gamma=0$, we find that~$\tilde z\in L^2(I,V)$ and
$(\Delta\tilde z,\tilde w)=-\overline w^\top\St\overline z$, and we arrive to
\begin{align*}
\overline w^\top\partial_t \Ma \overline z  + \nu \overline w^\top\St\overline z + \overline w^\top\Ma \overline f
= 0, \qquad \overline z\rest\Gamma = \overline g.  
\end{align*}

\subsubsection*{Reordering the mesh points} The vector $\mathbf p$ contains both interior and boundary points of~$\Omega$ (or~$\Omega_D$). We define the {\em permutation matrix}
$\mathbf{P^{ib}}$ such that the boundary points will appear at the end of the vector (and the relative order of interior (resp. boundary) points is unchanged).
We can find which indices correspond to the boundary points from the information we have in the edge vector~$\eee$.
In the new coordinates $\overline z_{*}=\mathbf{P^{ib}}\overline z$ we can write $\overline z_*\eqqcolon\begin{bmatrix} \overline z_{*\rm i}\\ \overline z_{*\rm b}\end{bmatrix}$ where
the vector~$\overline z_{*\rm i}$ correspond to the values at the interior points and $\overline z_{*\rm b}$ to the values at the boundary points. Thus we arrive at
\begin{align*}
\overline w_*^\top\partial_t\Ma_* \overline z_*  + \nu w_*^\top\St_*\overline z_*^\top + \overline w^\top_*\Ma_* \overline f_*
= 0, \qquad \overline z_{*\rm b}=\overline g;  
\end{align*}
with $\Ma_*=\mathbf{P^{ib}} \Ma \mathbf{P^{ib\top}}$ and $\St_*=\mathbf{P^{ib}}\St\mathbf{P^{ib\top}}$. Notice that the inverse and transpose of a permutation matrix do coincide.
Writing $\Ma_*$ and $\St_*$ in blocks notation
\begin{equation*}
\Ma_*\eqqcolon\begin{bmatrix}\Ma_{*\rm ii} &\Ma_{*\rm ib}\\ \Ma_{*\rm bi} &\Ma_{*\rm bb}\end{bmatrix}\quad\mbox{and}\quad
\St_*\eqqcolon\begin{bmatrix}\St_{*\rm ii} &\St_{*\rm ib}\\ \St_{*\rm bi} &\St_{*\rm bb}\end{bmatrix},
\end{equation*}
and recalling that $\overline w_{*\rm b}^\top=0$ we obtain
\begin{align*}
\overline w_{*\rm i}^\top 
\partial_t\begin{bmatrix}\Ma_{*\rm ii} &\Ma_{*\rm ib}\end{bmatrix}
\begin{bmatrix} \overline z_{*\rm i}\\ \overline z_{*\rm b}\end{bmatrix}
+ \nu \overline w_{*\rm i}^\top\begin{bmatrix}\St_{*\rm ii} &\St_{*\rm ib}\end{bmatrix}\begin{bmatrix} \overline z_{*\rm i}\\ \overline z_{*\rm b}\end{bmatrix}
+ \overline w_{*\rm i}^\top 
\begin{bmatrix}\Ma_{*\rm ii} &\Ma_{*\rm ib}\end{bmatrix}
\begin{bmatrix} \overline f_{*\rm i}\\ \overline f_{*\rm b}\end{bmatrix}
= 0.
\end{align*}

\begin{remark}\label{R:reorder}
In what follows we will work in the new coordinates and for simplicity we will skip the subscript~$*$. 
\end{remark}

Therefore, taking into account Remark~\ref{R:reorder}, we arrive to the semi-discrete system
\begin{equation}\label{heat-sys-sD}
\partial_t\Ma_{\rm ii} \overline z_{\rm i}
= -\nu \St_{\rm ii} \overline z_{\rm i} 
-\nu\St_{\rm ib}\overline z_{\rm b}
-\partial_t\Ma_{\rm ib} \overline z_{\rm b}
-\begin{bmatrix}\Ma_{\rm ii} &\Ma_{\rm ib}\end{bmatrix}\overline f,
\end{equation}
which underlines that when $\bar g$ (i.e., $\overline z_{\rm b}$) is known then the number of unknowns (i.e., the entries of~$\overline z_{\rm i}$)
is reduced to the number of interior points~$s_\ppp-s_\eee$,
where~$s_\eee$ is the total number of points in the boundary.

\begin{remark}
In the elliptic case, $\partial_t \Ma \overline z=0$, we see that~\eqref{heat-sys-sD} reduces to $\nu \St_{\rm ii} \overline z_{\rm i} =  
-\nu\St_{\rm ib}\overline z_{\rm b}
-\begin{bmatrix}\Ma_{\rm ii} &\Ma_{\rm ib}\end{bmatrix}\overline f$ which are the system we find in~\cite[section~2]{FixGunzburgerPeterson83}
when~$\overline z_{\rm b}$ is given. See also~\cite[section~1]{FixGunzburgerPeterson83} for references to other methods to deal with nonhomogeneous boundary conditions.
\end{remark}

\subsubsection{Discretization of a composition of linear operators}
In order to discretize systems~\eqref{PSInternal} and~\eqref{PSBoundary} we follow a simple
idea. See, for example,~\cite[Section 5.1]{KroRod15}.

Given two operators $L_1\in\LL(H^1,Z)$ and~$L_2\in\LL(Z,V')$, where~$Z\subset H^1$ is an Hilbert space
then the composition~$L_2\circ L_1$ is in $\LL(H^1,V')$.
Suppose we know the discretization matrix~$(L_2)_D$ and also an
{\em evaluation matrix}~$\overline{L_1}$ which approximates $L_1$, that is, $\overline{L_1}\overline v\approx\overline{L_1 v}$. 
Then we may write
\begin{align*}
 \overline w^\top(L_2\circ L_1)_D\overline z&\approx(L_2\circ L_1 z,w)_H\approx\overline w^\top(L_2)_D\overline{L_1}\overline z.
\end{align*}
Analogously if we know
the discretization matrix~$(L_1)_D$ and an
evaluation matrix~$\overline{L_2^*}$ of the adjoint~$L_2^*$, we may write
\begin{align*}
 \overline w^\top(L_2\circ L_1)_D\overline z&\approx(L_2\circ L_1 z,w)_H\approx\left(\overline{(L_2)^*}\,\overline w\right)^\top(L_1)_D\overline z.
\end{align*}
Therefore we have two candidates to approximate~$(L_2\circ L_1)_D$:
\begin{align*}
 (L_2)_D\overline{L_1}\approx(L_2\circ L_1)_D\approx\overline{(L_2)^*}^\top(L_1)_D.
\end{align*}
Further notice that $(L^*)_D=L_D^\top$ and that
\[
\overline{L^*}^\top\Ma\approx L_D\approx\Ma\overline{L},\qquad \overline{L}\approx\Ma^{-1}L_D,\qquad \overline{L^*}^\top\approx L_D\Ma^{-1}.
\]
\subsubsection{Discretization of the linear reaction and linear convection operators}
We can construct the discretization matrices
$\GGG_{x_k}\coloneqq(\p_{x_k})_D$, $(\GGG_{x_k})_{ij}\coloneqq( \phi_i, \partial_{x_k} \phi_j)_H$, of the directional derivatives operators, $k\in\{1,2\}$,
in the same way as we construct the mass and stiffness matrices.
Then for a function~$w$,  we set $\overline{w\Bigcdot}
\coloneqq\DD_{\overline {w}}$,
where~$\DD_{v}$ stands for the diagonal matrix whose diagonal is $v$. Observe that
$2{w}\Bigcdot=\iota_{V}\circ ({w}\Bigcdot)+({w}\Bigcdot)\circ \iota_{V}$. Finally for the linear reaction and convection operators, we take
\[
((\hat a-\textstyle\frac{\lambda}{2})\Bigcdot)_D\approx\frac{1}{2}\left(\Ma\DD_{\overline {\hat a-\frac{\lambda}{2}}}\,+\,\DD_{\overline {\hat a-\frac{\lambda}{2}}}\Ma\right)
\quad\mbox{and}\quad (\nabla\cdot (\hat b\Bigcdot))_D\approx\GGG_{x_1}\DD_{\overline{\hat b}_1}+\GGG_{x_2}\DD_{\overline{\hat b}_2}.
\]
Notice that ${w}\Bigcdot=({w}\Bigcdot)^*$ and by taking
the semi-sum above as an approximation for $({w}\Bigcdot)_D$ we preserve the symmetry. We must say, however that in our simulations we did not observe much difference
when we have simply taken~$({w}\Bigcdot)_D\approx\Ma\DD_{\overline {w}}$.

Notice also that the above discretization idea would lead us to take
$(\hat b\cdot\nabla \Bigcdot)\approx\DD_{\overline{\hat b}_1}\GGG_{x_1}+\DD_{\overline{\hat b}_2}\GGG_{x_2}$, and that
looking at the operators as~$\LL(H^1,\,V')$, we need to test the operators
with functions/vectors vanishing at the boundary. Now, 
for either $v\rest\Gamma=0$ or $w\rest\Gamma=0$, the definition of $\GGG_{x_k}$, gives us
$(\tilde v, \partial_{x_k} \tilde w)_H=-(\partial_{x_k} \tilde v,\tilde w)_H$, that is, $\overline v^\top\GGG_{x_k}\overline w=-\overline w^\top\GGG_{x_k}\overline v$.
Now our approximations give
$\overline v^\top(\GGG_{x_1}\DD_{\overline{\hat b}_1}+\GGG_{x_2}\DD_{\overline{\hat b}_2})\overline w
=\overline w^\top(\DD_{\overline{\hat b}_1}\GGG_{x_1}^\top+\DD_{\overline{\hat b}_2}\GGG_{x_2}^\top)\overline v
=\overline w^\top(-\DD_{\overline{\hat b}_1}\GGG_{x_1}-\DD_{\overline{\hat b}_2}\GGG_{x_2})\overline v$. Therefore,
$(\GGG_{x_1}\DD_{\overline{\hat b}_1}+\GGG_{x_2}\DD_{\overline{\hat b}_2})^\top=-\DD_{\overline{\hat b}_1}\GGG_{x_1}-\DD_{\overline{\hat b}_2}\GGG_{x_2}$,
which agrees with the fact that
we have $(w,\,\nabla\cdot ({\hat b}v))_H=(-{\hat b}\cdot\nabla w ,\,v)_H$.

\begin{remark}
In some situations, for example when the viscosity coefficient is small and the exact solution is known to have some
big magnitude directional space derivatives, it may be the case that the numerical solution present some oscillations.
For solving such problems with a reasonable number of mesh points, avoiding or minimizing the oscillations, we may need more sophisticated numerical discretizations.
We refer to~\cite{JohnSchmeyer08} and to~\cite{JohnKnobloch07,JohnKnobloch08} for such discretizations. 
\end{remark}
\begin{remark}
Below, we will consider the stabilization to a trajectory~$\hat y$ (given a priori) whose first order derivatives are not too big,
which is also motivated from the discussion in Section~\ref{S:stab_traj}.
Therefore, remaining close to $\hat y$ the solutions will not present big first order
derivatives and we may expect that no big oscillations will be observed on the numerical solution.
From the control point of view, if some small oscillations do appear, then we could thing of them as small perturbations of the solution and, since the control will be given in
feedback form, we may expect the control to be able to respond to such perturbations. Roughly speaking, we could see those small perturbations as a
further test for the robustness
of the control.
\end{remark}

\subsubsection{Discretization of the feedback control}\label{ssS:discRiccati}
We start by recalling that the feedback depends on the solution~$\Pi$ of the corresponding Riccati equation and is self-adjoint, that is,
we know that~$\Pi_D$ is symmetric. We look for an approximation of~$\Pi_D$ by solving a suitable matrix Riccati equation 
\begin{equation}\label{eq:riccati_nu-sD}
\p_t \Pi_D + \Pi_D \XXX + \XXX^{\top} \Pi_D - \Pi_D \RRR \RRR^{\top} \Pi_D +\CCC^\top\CCC = 0, \qquad t>0,
\end{equation}
For~$r\ge0$, let us introduce the matrices
\[
\mathbf{K_r}(t) \coloneqq   \mathbf{M}^{-1} \left(  \GGG_{x_1} \DD_{\overline{\hat b_1 (t)}} + \GGG_{x_2}  \DD_{\overline{\hat b_2 (t)}}  \right) 
+  \DD_{\overline{\hat a-\frac{r}{2} (t)}}\eqqcolon \begin{bmatrix}\mathbf{K}_{r,{\rm ii}}(t) & \mathbf{K}_{r,{\rm ib}}(t)
\\ \mathbf{K}_{r,{\rm bi}}(t) & \mathbf{K}_{r,{\rm bb}}(t)\end{bmatrix}.
\]

\subsubsection*{Internal feedback}
For system~\eqref{PSInternal} we need to solve~\eqref{eq:riccati_nu}. Since~$z\rest\Gamma=0$, we look for matrices in~$\MM_{(s_\ppp-s_\eee)\times(s_\ppp-s_\eee)}$.

We choose the matrices $\CCC^{\rm in}=\nu^\frac{1}{2}\St_{\rm ii,c}$ and
$\XXX^{\rm in}_\lambda = \XXX^{\rm in}_\lambda(t) =   -\nu \Ma_{\rm ii}^{-1} \St_{\rm ii} - \mathbf{K}_{\lambda,{\rm ii}}(t)$, where~$\St_{\rm ii,c}$
is the Cholesky factor
of~$\St_{\rm ii}$, which gives us $\left(\CCC^{\rm in}\right)^\top \left(\CCC^{\rm in}\right)=\nu\St_{\rm ii}$. To construct the matrix $\RRR^{\rm in}$ we set $\RR=1$
in~\eqref{eq:riccati_nu} and observe that we must have
\[
 \Pi_D \RRR^{\rm in} \RRR^{\rm in\top} \Pi_D\approx(\Pi 1_\omega \chi P_M1_\omega\RR 1_\omega P_M \chi 1_\omega\Pi)_D.
\]

We will also take actuators~$\{\overline{\Phi_i}\mid i\in{1,2,\dots,M}\}$ which are supported in~$\overline\omega$, thus
since~$P_M$ stands for the orthogonal projection onto~$\linspan\{\overline{\Phi_i}\mid i\in{1,2,\dots,M}\}$, we may write
\begin{align*}
(\Pi 1_\omega \chi P_M1_\omega\RR 1_\omega P_M\chi1_\omega\Pi)_D\approx
\Pi_D \overline{1_\omega\chi P_M\chi 1_\omega\Pi}
\end{align*}
and observe that the natural evaluation matrix of the multiplication~$\overline{1_\omega\chi \Bigcdot}=\overline{\chi 1_\omega\Bigcdot}$ is just the diagonal matrix
$\DD_{\overline{\chi 1_\omega}}$ when we identify $\chi 1_\omega$ with the function that takes the value~$1$ if $x\in\omega\cap\supp\chi$ and the
value~$0$ otherwise. To construct an evaluation matrix of the $H$-orthogonal projection~$P_M$, we start by
orthonormalizing, in the $\Ma_{\rm ii}$-scalar product, the family of actuators~$\{\overline{\Phi_i}_{\rm i}\mid i\in{1,2,\dots,M}\}$.
Then we denote the orthonormal family by~$\{\overline{\Phi_i}^o\mid i\in{1,2,\dots,M}\}$ and set 
\begin{align*}
\overline{P_M}=\mathbf{P}_M \coloneqq S_MS_M^\top\Ma_{\rm ii},\quad\mbox{with}\quad S_M\coloneqq\begin{bmatrix} \overline{\Phi_1}^o & \overline{\Phi_2}^o & \dots & 
\overline{\Phi_M}^o\end{bmatrix},
\end{align*}
which leads us to
\begin{align*}
(\Pi 1_\omega \chi P_M1_\omega\RR 1_\omega P_Ṃ\chi1_\omega\Pi)_D\approx
\Pi_D \DD_{\overline{\chi 1_\omega}} S_MS_M^\top\Ma_{\rm ii}\DD_{\overline{1_\omega\chi}}\overline{\Pi}\approx
\Pi_D \DD_{\overline{\chi 1_\omega}} S_MS_M^\top\DD_{\overline{1_\omega\chi}} \Pi_D
\end{align*}
and this is the reason we propose to take the matrix 
\begin{equation}
\label{RInternal}
\RRR^{\rm in} \coloneqq \DD_{\overline{\chi 1_\omega}} S_M.
\end{equation}
Therefore, as the approximation evaluation matrix of the feedback control rule  we set
\begin{equation*}
{\overline{\FF_{\lambda}^{\rm in}}} \coloneqq \left( \RRR^{\rm in} \right) \left( \RRR^{\rm in} \right)^\top \Pi_D.
\end{equation*}

We refer to~\cite[section~5.3.3]{KroRod15} for a different choice of~$\RRR^{\rm in}$, namely the less simple
expression~$\left( \mathbf{M}_{\rm ii,c} \mathbf{P}_M \DD_{\overline{\chi 1_\omega}} \mathbf{M}_{\rm ii}^{-1} \right)^{\top}$. 

\subsubsection*{Boundary feedback}
For system~\eqref{PSBoundary} we need to solve~\eqref{eq:riccati-bdry-traj}.
Notice that~\eqref{eq:riccati-bdry} is related with the extended system~\eqref{sys-z2cf-ext} whose state belongs to $V\times\R^M$, thus we look for matrices
in~$\MM_{(s_\ppp-s_\eee+M)\times(s_\ppp-s_\eee+M)}$.

{
First of all we given the boundary actuators~$\Phi$ we construct the extensions~$\widetilde\Phi$, defined in~$\Omega$, by solving (numerically) the elliptic system
\begin{equation*}
 -\nu\Delta\widetilde\Psi+\varsigma\widetilde\Psi_i=0,\qquad \widetilde\Psi_i\rest\Gamma=\Psi_i.
\end{equation*}
(which is equivalent to~\eqref{part-extPsi}, with~$\bar\varsigma=\frac{\varsigma}{\nu}$). }

We choose the matrices~$\CCC^{\rm bo}$ and~$\XXX^{\rm bo}_\lambda = \XXX^{\rm bo}_\lambda(t)$ as
\[
\CCC^{\rm bo}=\begin{bmatrix}
\nu^\frac{1}{2} \St_{\rm ii,c} & 0 \\ 0 & \mathbf{I}_M
\end{bmatrix}\quad\mbox{and}\quad
\XXX^{\rm bo}  = \begin{bmatrix}
 \XXX^{\rm in}_\lambda & -\mathbf{K}_{0,{\rm ii}}[B_{\overline{\widetilde{\Psi}}}]_{\rm i} \\
 0 & -\varsigma\mathbf{I}_M \end{bmatrix},
\]
where~$\mathbf I_M\in\MM_{M\times M}$ is the identity matrix and~$B_{\overline{\widetilde{\Psi}}}\coloneqq\begin{bmatrix}\overline{\widetilde{\Psi}_1}
 & \overline{\widetilde{\Psi}_2} & \dots &  \overline{\widetilde{\Psi}_M} \end{bmatrix}\eqqcolon\begin{bmatrix}[B_{\overline{\widetilde{\Psi}}}]_{\rm i}\\
 [B_{\overline{\widetilde{\Psi}}}]_{\rm b}\end{bmatrix}$.

 To choose an appropriate operator~$\RRR^{\rm bo}$ we start again by setting $\RR=1$ and, looking at the nonlinear term of the Riccati equation
we can formally write
\begin{align*}
\left(\Pi \begin{bmatrix}B_\Psi\\\mathbf{I}_M\end{bmatrix}\begin{bmatrix}B_\Psi^*&\mathbf{I}_M\end{bmatrix}\Pi\right)_D\approx
\Pi_D \begin{bmatrix}\overline{B_\Psi}\\\mathbf{I}_M\end{bmatrix}\begin{bmatrix}\overline{B_\Psi^*}&\mathbf{I}_M\end{bmatrix}\overline{\Pi}.
\end{align*}
From $(B_\Psi\kappa,\,y)_H=\sum_{j=1}^M\kappa_j(\widetilde\Psi_j,\,y)_H$ we may also set
$\overline{B_\Psi^*}=\overline{B_\Psi}^\top\Ma=B_{\overline{\widetilde{\Psi}}}^\top\Ma$. For given~$y\in V$ and $j\in\{1,2,\dots,M\}$ we obtain
\[
\begin{split}
(\widetilde\Psi_j,\,y)_H&\approx\overline{\widetilde\Psi_j}^\top\Ma\overline y
=\overline{\widetilde\Psi_j}^\top\begin{bmatrix}\Ma_{\rm ii}&\Ma_{\rm bi}\end{bmatrix}\overline y_{\rm i}
=\left(\overline{\widetilde\Psi_j}_{\rm i}^\top +\overline{\widetilde\Psi_j}_{\rm b}^\top\Ma_{\rm bi}\Ma_{\rm ii}^{-1}\right)\Ma_{\rm ii}\overline y_{\rm i}\\
&=\left(\overline{\widetilde\Psi_j}_{\rm i} +\Ma_{\rm ii}^{-1}\Ma_{\rm ib}\overline{\widetilde\Psi_j}_{\rm b}\right)^\top\Ma_{\rm ii}\overline y_{\rm i}.
\end{split}
\]
and this is why, looking at~$B_\Psi$ as an operator in $\LL(\R^M,\,V')$,
we take
\[
\overline{B_\Psi}=\begin{bmatrix}\overline{B_\Psi}_{(1,1)} &
                   \overline{B_\Psi}_{(1,2)} & \dots & \overline{B_\Psi}_{(1,M)}
                  \end{bmatrix}\qquad\mbox{and}\qquad
                  \overline{B_\Psi^*}=\overline{B_\Psi}^\top\Ma_{\rm ii},
\]
with $\overline{B_\Psi}_{(1,j)}\coloneqq\overline{\widetilde\Psi_j}_{\rm i} +\Ma_{\rm ii}^{-1}\Ma_{\rm ib}\overline{\widetilde\Psi_j}_{\rm b}$. 

Therefore, we arrive to
\[
 \left(\Pi \begin{bmatrix}B_\Psi\\\mathbf{I}_M\end{bmatrix}\begin{bmatrix}B_\Psi^\top&\mathbf{I}_M\end{bmatrix}\Pi\right)_D\approx
\Pi_D \begin{bmatrix}\overline{B_{\Psi}}\\\mathbf{I}_M\end{bmatrix}\begin{bmatrix}\overline{B_{\Psi}}^\top&\mathbf{I}_M\end{bmatrix}
\begin{bmatrix}\Ma_{\rm ii}&0\\0&\mathbf{I}_M\end{bmatrix}\overline{\Pi}.
\]
Since $\begin{bmatrix}\Ma_{\rm ii}&0\\0&\mathbf{I}_M\end{bmatrix}\overline{\Pi}\approx\Pi_D$ we
propose to take the matrix
\begin{align*}
\RRR^{\rm bo} \coloneqq \begin{bmatrix}\overline{B_{\Psi}}\\\mathbf{I}_M\end{bmatrix}
\end{align*}
and, as the approximation evaluation matrix of the feedback control rule,  we set
\begin{align*}
\overline{\FF_{\lambda}^{\rm bo}} \coloneqq \left(\RRR^{\rm bo}\right) \left( \RRR^{\rm bo} \right)^\top \Pi_D.
\end{align*}

\subsubsection{The semidicrete systems} We are now ready to present the semidiscrete versions of the linear closed-loop systems.
Introducing the matrices
\[
\begin{split}
 \LLL_\lambda&=\LLL_{(\hat a-\frac{\lambda}{2},\hat b)}\coloneqq
 \textstyle\frac{1}{2}\left(\Ma\DD_{\overline{\hat a-\frac{\lambda}{2}}}+\DD_{\overline{\hat a-\frac{\lambda}{2}}}\Ma\right)
 +\GGG_{x_1}\DD_{\overline{\hat b}_1}+\GGG_{x_2}\DD_{\overline{\hat b}_2},
 \quad\LLL_\lambda\eqqcolon\begin{bmatrix}\LLL_{\lambda,{\rm ii}}&\LLL_{\lambda,{\rm ib}}\\ \LLL_{\lambda,{\rm bi}}&\LLL_{\lambda,{\rm bb}}\end{bmatrix},
 \end{split}
\]
system~\eqref{PSInternal} is approximated by
\begin{equation}\label{SInternal-sD}
\partial_t\Ma_{\rm ii} \overline z_{\rm i}
= -\nu \St_{\rm ii} \overline z_{\rm i} 
-\LLL_{\lambda,{\rm ii}}\overline z_{\rm i}-\Ma_{\rm ii}\overline{\FF_{\lambda}^{\rm in}}\overline z_{\rm i},
\end{equation}
and system~\eqref{PSBoundary}, setting $\varsigma_\lambda=\varsigma-\frac{\lambda}{2}$, is approximated by
\begin{subequations}\label{SBoundary-sD}
\begin{align} 
\partial_t\kappa
&=-\varsigma_\lambda\kappa-\FF_{\lambda,{\rm b}}^{\rm bo}
\begin{bmatrix}\overline z_{\rm i}\\ \kappa\end{bmatrix},\label{SBoundary-sDk}\\
\partial_t\Ma_{\rm ii} \overline z_{\rm i}
&= -\nu \St_{\rm ii} \overline z_{\rm i} -\begin{bmatrix}\LLL_{\lambda,{\rm ii}} & \LLL_{\lambda,{\rm ib}}\end{bmatrix}\overline z
-(\nu\St_{\rm ib}+\partial_t\Ma_{\rm ib})B_{\overline\Psi}^\Gamma\kappa,\label{SBoundary-sDy}
\end{align}
\end{subequations}
with $\overline{\FF_{\lambda,{\rm b}}^{\rm bo}}\coloneqq\begin{bmatrix}0_{M\times(s_\ppp-s_\eee)}& \mathbf I_M\end{bmatrix}
\overline{\FF_{\lambda}^{\rm bo}}\begin{bmatrix}\mathbf I_{s_\ppp-s_\eee}&-\overline{B_\Psi}\\0&\mathbf I_M\end{bmatrix}$.

\subsection{Discretization in time}
 Now to be able to solve numerically~\eqref{SInternal-sD} and~\eqref{SBoundary-sD} in a given time interval $[0,T]$, with $T>0$, we need to discretize $[0,T]$.
 We introduce a uniform mesh~$[0,T]_D$ consisting of $N_t+1\ge3$ points
\begin{align}
\label{TimeMesh}
[0,T]_D \coloneqq \left( 0, kT,2 kT,\cdots,(N_t -1) kT,T \right),
\end{align}
where~$k\coloneqq \textstyle\frac{T}{N_t}$ is the time step.
Any function $z \in C([0,T]\times\Omega)$ will be approximated by the values $z_i^j\coloneqq z(jk,\,\ppp_i)$ taken in $[0,T]_D \times \Omega_D$, that is, we essentially
approximate $z = z(t, x)$ by a matrix  $[z]=[z_{(i,j)}]\in \MM_{ N_p \times (N_t+1)}$, $z_{(i,j)}\coloneqq z_i^j$. Notice that in the~$j$th column,
denoted~$\overline z^j$, we have an approximation of~$z$ at time~$t=jk$, $\overline z^j=\overline{z(jk,\Bigcdot)}\approx z(jk,\Bigcdot)$.

We will use the Crank-Nicolson scheme taking, for $t_f>t_i$, $\p_tz\rest{\frac{t_i+t_f}{2}}\approx\frac{z(t_f)-z(t_i)}{t_f-t_i}$ and
$z(\frac{t_i+t_f}{2})\approx\frac{z(t_f)+z(t_i)}{2}$. For system~\eqref{SInternal-sD}, we obtain
\begin{equation*}
\textstyle\frac{\Ma_{\rm ii}}{k} (\overline z_{\rm i}^{j+1}-\overline z_{\rm i}^{j})
= -\nu \frac{\St_{\rm ii}}{2} (\overline z_{\rm i}^{j+1}+\overline z_{\rm i}^{j})
-\frac{\LLL_{\lambda,{\rm ii}}^j\overline z_{\rm i}^{j}+\LLL_{\lambda,{\rm ii}}^{j+1}\overline z_{\rm i}^{j+1}}{2}
-\frac{\Ma_{\rm ii}\overline{\FF_{\lambda}^{j,\rm in}}\overline z_{\rm i}^{j}+\Ma_{\rm ii}\overline{\FF_{\lambda}^{j+1,\rm in}}\overline z_{\rm i}^{j+1}}{2},
\end{equation*}
where, recalling that~$\LLL_{\lambda,{\rm ii}}$ and~$\FF_{\lambda}^{\rm in}$ (may) do depend on time,
we denote~$\LLL_{\lambda,{\rm ii}}^j\coloneqq \LLL_{\lambda,{\rm ii}}(jk)$ and~$\overline{\FF_{\lambda}^{j,\rm in}}\coloneqq\overline{\FF_{\lambda}^{\rm in}(jk)}$. We wish to find~$z_{\rm i}^{j+1}$ once~$z_{\rm i}^{j}$ is known,
Since~$z_{\rm i}^{j+1}$ is unknown we make the following linear extrapolations
\begin{equation}\label{extrap}
\LLL_{\lambda,{\rm ii}}^{j+1}\overline z_{\rm i}^{j+1}\approx 2\LLL_{\lambda,{\rm ii}}^{j}\overline z_{\rm i}^{j}-\LLL_{\lambda,{\rm ii}}^{j-1}\overline z_{\rm i}^{j-1}
\quad\mbox{and}\quad \overline{\FF_{\lambda}^{j+1,\rm in}}\overline z_{\rm i}^{j+1}\approx
2\overline{\FF_{\lambda}^{j,\rm in}}\overline z_{\rm i}^{j}-\overline{\FF_{\lambda}^{j-1,\rm in}}\overline z_{\rm i}^{j-1}.
\end{equation}
Defining $\mathbf A_{\rm ii}^\oplus\coloneqq(2\Ma_{\rm ii}+k\nu\St_{\rm ii})$ and~$\mathbf A_{\rm ii}^\ominus\coloneqq(2\Ma_{\rm ii}-k\nu\St_{\rm ii})$, we obtain the system
\begin{equation}\label{SInternal-D}
\begin{split}
\mathbf A_{\rm ii}^\oplus \overline z_{\rm i}^{j+1}
= \mathbf A_{\rm ii}^\ominus\overline z_{\rm i}^{j}
-k(3\LLL_{\lambda,{\rm ii}}^j\overline z_{\rm i}^{j}-\LLL_{\lambda,{\rm ii}}^{j-1}\overline z_{\rm i}^{j-1})
-k\Ma_{\rm ii}(3\overline{\FF_{\lambda}^{j,\rm in}}\overline z_{\rm i}^{j}-\overline{\FF_{\lambda}^{j-1,\rm in}}\overline z_{\rm i}^{j-1}),
\end{split}
\end{equation}
which we invert to obtain~$\overline z_{\rm i}^{j+1}$.
To start the loop we define, at the ``ghost'' time instant $t=-kT$, the
terms~$\LLL_{\lambda,{\rm ii}}^{-1}\overline z_{\rm i}^{-1}\coloneqq\LLL_{\lambda,{\rm ii}}^{0}\overline z_{\rm i}^{0}$
and~$\overline{\FF_{\lambda}^{-1,\rm in}}\overline z_{\rm i}^{-1}\coloneqq\overline{\FF_{\lambda}^{0,\rm in}}\overline z_{\rm i}^{0}$.

\begin{remark}
Notice that since~$\Ma_{\rm ii}$ and~$\St_{\rm ii}$ are symmetric and positive definite, then $\mathbf A_{\rm ii}^\oplus$ is symmetric and positive definite (at least for
small enough~$k$). Notice also that the idea in~\eqref{extrap} is proposed in~\cite[section~5.4]{KroRod15}, but with an extrapolation
as~$L^{j+1}\approx (1+k)L^{j}-kL^{j-1}$, which approaches a zero order extrapolation as $k$ decreases to~$0$. 
With the linear extrapolation, we observed a better accuracy/convergence performance in some tests. Finally, the
feedback part of the control in~\cite[section~5.4]{KroRod15} is treated in a different way using a preliminary ``uncontrolled guess'' $z_G^{j+1}$ for~$\overline z^{j+1}_{\rm i}$
by solving the system with no control, the idea is to use the fact
that we know~$\FF^{j+1,\rm in}$. We do not use this fact with the linear extrapolation above. However, in the {2D} boundary case, the extrapolation seems to work better in our
simulations, this is why we take the extrapolation in both cases. Furthermore, the extrapolation approach is cheaper because we
need to solve the system only once (at each time step).
\end{remark}

Analogously, system~\eqref{SBoundary-sDk} is approximated by 
\begin{subequations}\label{SBoundary-D}
 \begin{equation}\label{SBoundary-Dk}
(2+k\varsigma_\lambda)\kappa^{j+1}
=(2-k\varsigma_\lambda)\kappa^{j}-k\left(3\overline{\FF_{\lambda,{\rm b}}^{j,\rm bo}}
\begin{bmatrix}\overline z_{\rm i}^{j}\\ \kappa^{j}\end{bmatrix}
-\overline{\FF_{\lambda,{\rm b}}^{j-1,\rm bo}}
\begin{bmatrix}\overline z_{\rm i}^{j-1}\\ \kappa^{j-1}\end{bmatrix}\right),
\end{equation}
with~$\overline{\FF_{\lambda,{\rm b}}^{-1,\rm bo}}\begin{bmatrix}\overline z_{\rm i}^{-1}\\ \kappa^{-1}\end{bmatrix}
\coloneqq \overline{\FF_{\lambda,{\rm b}}^{0,\rm bo}}\begin{bmatrix}\overline z_{\rm i}^{0}\\ \kappa^{0}\end{bmatrix}$, from which we obtain~$\kappa^{j+1}$.
Then, system~\eqref{SBoundary-sDy} is approximated by
\begin{equation*}
\begin{split}
\quad\mathbf A_{\rm ii}^\oplus\overline z_{\rm i}^{j+1}
&= \mathbf A_{\rm ii}^\ominus\overline z_{\rm i}^j 
-k\nu\St_{\rm ib}B_{\overline\Psi}^\Gamma(\kappa^{j+1}+\kappa^{j})
-2\Ma_{\rm ib}B_{\overline\Psi}^\Gamma(\kappa^{j+1}-\kappa^j)\\
&\quad-k\LLL_{\lambda,{\rm ii}}^{j+1}\overline z_{\rm i}^{j+1} -k\LLL_{\lambda,{\rm ii}}^{j}\overline z_{\rm i}^{j}
-k\LLL_{\lambda,{\rm ib}}^{j+1}B_{\overline\Psi}^\Gamma\kappa^{j+1}- k\LLL_{\lambda,{\rm ib}}^{j}B_{\overline\Psi}^\Gamma\kappa^{j}.
\end{split}
\end{equation*}
Extrapolating again for the unknown~$\LLL_{\lambda,{\rm ii}}^{j+1}\overline z_{\rm i}^{j+1}$,
with~$\LLL_{\lambda,{\rm ii}}^{-1}\overline z_{\rm i}^{-1}
\coloneqq\LLL_{\lambda,{\rm ii}}^{0}\overline z_{\rm i}^{0}$, we arrive to
\begin{equation}\label{SBoundary-Dy}
\begin{split}
\quad\mathbf A_{\rm ii}^\oplus\overline z_{\rm i}^{j+1}
&= \mathbf A_{\rm ii}^\ominus\overline z_{\rm i}^j -(\mathbf A_{\rm ib}^\oplus +k\LLL_{\lambda,{\rm ib}}^{j+1})B_{\overline\Psi}^\Gamma\kappa^{j+1}
+(\mathbf A_{\rm ib}^\ominus-k\LLL_{\lambda,{\rm ib}}^{j}) B_{\overline\Psi}^\Gamma \kappa^{j}\\
&\quad-k3\LLL_{\lambda,{\rm ii}}^{j}\overline z_{\rm i}^{j}+k\LLL_{\lambda,{\rm ii}}^{j-1}\overline z_{\rm i}^{j-1},
\end{split}
\end{equation}
which we invert to obtain the interior component~$\overline z_{\rm i}^{j+1}$ of~$\overline z^j$.
Notice that the boundary component is given by~$\overline z_{\rm b}^{j+1}=B_{\overline\Psi}^\Gamma \kappa^{j+1}$.
\end{subequations}

\subsubsection{Solving the Riccati systems}
It remains to explain how we compute the the feedbacks $\overline{\FF_{\lambda}^{\rm in}}$ and $\overline{\FF_{\lambda}^{\rm bo}}$ we need in~\eqref{SInternal-D} and~\eqref{SBoundary-Dk}.
That is, to explain how we solve the Ricatti system~\eqref{eq:riccati_nu-sD}, backwards in time and in an bounded interval of time~$[0,T]$.
We follow the procedure in~\cite[sections~5.3.2 and~5.3.3]{KroRod15}, with some changes.
\subsubsection*{Internal feedback}
Firstly, we look for a solution~$\Pi_D^T$ of the agebraic Riccati equation
\begin{equation*}
\Pi_D^T \XXX_{\lambda}^{\rm in}(T) + \XXX_{\lambda}^{\rm in}(T)^{\top} \Pi_D^T - \Pi_D^T \RRR \RRR^{\top} \Pi_D^T +\left(\CCC^{\rm in}\right)^\top \left(\CCC^{\rm in} \right)= 0,
\end{equation*}
with~$\RRR=\HHH_0=\Ma_{\rm ii,c}^{-1}$ as the Cholesky factor of~$\Ma_{\rm ii}^{-1}$. The main idea is to
have~$\HHH_0\HHH_0^\top=\Ma_{\rm ii}$ which corresponds to taking the identity as control operator (see~\cite[5.3.3]{KroRod15} where
it is chosen~$\HHH_0=(\Ma_{\rm ii,c}\Ma_{\rm ii}^{-1})^\top$).

Secondly, we connect~$\HHH_0$ to~$\RRR^{\rm in}$ by an homotopy. In the case where we have finite-dimensional controls~$\RRR^{\rm in}$ is a rectangular matrix, so we add the enough
zero columns to obtain a square matrix, and we consider
\begin{equation}\label{homotopy}
\HHH_\tau=(1-\tau)^2\HHH_0 +\tau^2 \begin{bmatrix}\RRR^{\rm in} & 0\end{bmatrix},\qquad\tau\in[0,\,1].
\end{equation}
Discretizing the homotopy interval~$[0,\,1]_D=[0,l,2l,\dots,(N_h-1)l,1]$, $N_h\in\N_0$ and $l=\frac{1}{N_h}$, we compute the solution corresponding to~$\RRR=\HHH_{ml}$
from those corresponding to~$\RRR=\HHH_{(m-1)l}$, $m\in\{1,2,\dots,N_h\}$, following~\cite{KroRod15}.
Let $\Pi_D^T$ be the solution corresponding to~$\HHH_{1}$, then since~$\HHH_{1}\HHH_{1}^\top=\RRR^{\rm in} \RRR^{\rm in\top}$, $\Pi_D^T$ is the solution
corresponding to~$\RRR=\RRR^{\rm in}$. 

Finally, we solve the equation~\eqref{eq:riccati_nu-sD} as in~\cite{KroRod15}, with the final condition $\Pi_D(T)=\Pi_D^T$ and
using Crank-Nicolson discretization in time variable,
by transforming the equation into an algebraic Riccati equation at each time step.

To solve the algebraic Riccati equations we use the software in~\cite{BennerRicSolver} (see also~\cite{Benner06}).

\begin{remark}
Notice that in~\cite{KroRod15} the solution~$\Pi_D^T$ is found in three steps and~$\RRR$ is always set to be a square matrix, this is because we need to connect
two matrices by an homotopy. Here we find~$\Pi_D^T$ in two steps and
we implicitly connect the square matrix~$\HHH_0$ to the rectangular one~$\RRR^{\rm in}$ by
connecting~$\HHH_0$ to~$\begin{bmatrix}\RRR^{\rm in}&0\end{bmatrix}$. We have used the homotopy~\eqref{homotopy} because in some situations we observed
that we need less homotopy steps than with the convex combination~$(1-\tau)\HHH_0 +\tau \begin{bmatrix}\RRR^{\rm in} & 0\end{bmatrix}$ like as in~\cite{KroRod15}, but we cannot
say a priori which homotopy is better.
\end{remark}

\subsubsection*{Boundary feedback}
We proceed as in the internal case to find a final condition for the differential Riccati equation. At final time~$t=T$ we start by solving the system 
\begin{equation*}
\Pi_D^T \XXX_{\lambda}^{\rm bo}(T) + \XXX_{\lambda}^{\rm bo}(T)^{\top} \Pi_D^T - \Pi_D^T \RRR \RRR^{\top} \Pi_D^T +\left(\CCC^{\rm bo}\right)^\top\left(\CCC^{\rm bo} \right)= 0
\end{equation*}
and then set~$\HHH_0=\begin{bmatrix}\Ma_{\rm ii,c}^{-1}&0\\0&\mathbf I_M\end{bmatrix}$ and
$\HHH_\tau=(1-\tau)^2\HHH_0 +\tau^2 \begin{bmatrix}\RRR^{\rm bo} & 0\end{bmatrix}$, $\tau\in[0,\,1]$.

\section{Numerical examples}\label{S:num_exa}
We present some results of numerical simulations which we have performed concerning the stabilization of
systems \eqref{PSInternal} with internal feedback control or \eqref{PSBoundary} with boundary feedback control to zero.
Below, $z_{\mathrm{u},0}$ stands for the solution of the uncontrolled discretized
systems (i.e., without the feedback term  and $\lambda = 0$), and  $z_{\lambda}$ stands for the solution of the discretized systems
under the action of a discretized feedback control.
We focus on the 2D case and our domain is the unit ball. In the internal feedback control case, we define a rectangle
subdomain $\omega \coloneqq \left(0, \frac{1}{2} \right) \times \left(0, \frac{1}{3} \right)$. Then, we take a regular partition of into~$M=mn$ subrectangles
\[
\omega_{l_1,l_2} \coloneqq \left(\textstyle\frac{l_1-1}{2m},\frac{l_1}{2m} \right) \times \left(\textstyle\frac{l_2-1}{3n},\frac{l_2}{3n} \right),
\qquad (l_1,l_2)\in\{1,2,\dots,m\}\times\{1,2,\dots,n\}.
\]
We take the~$M$ actuators~$1_{\omega_{l_1,l_2}}$, thus in each subrectangle~$\omega_{l_1,l_2}$ the control is constant.
As an illustration, we plot a linear combination of 4 piecewise-constant actuators in Figure \ref{Controllers}(a), corresponding to the arrangement~$(m,n)=(2,2)$.
For the boundary control case,
our boundary, once parametrized by arc length, is $\Gamma = [0, 2 \pi)$. We use boundary actuators whose form is 
\begin{equation}\label{sin_contbdry}
\Psi_i (\theta) = 1_{\left(\theta_0, \theta_1 \right)} \sin \left( \frac{i(\theta - \theta_0)}{\theta_1 - \theta_0} \right).
\end{equation}
with $\theta_0 = \pi$ and $\theta_1 = \frac{5\pi}{4}$.
As an illustration, the boundary actuator $\Psi_2$ is plotted in Figure \ref{Controllers}(b). 

\begin{figure}[ht]  
  \centering
\subfigure[A linear combination off 4 piecewise-constant internal actuators.]
{\includegraphics[width=.495\linewidth]{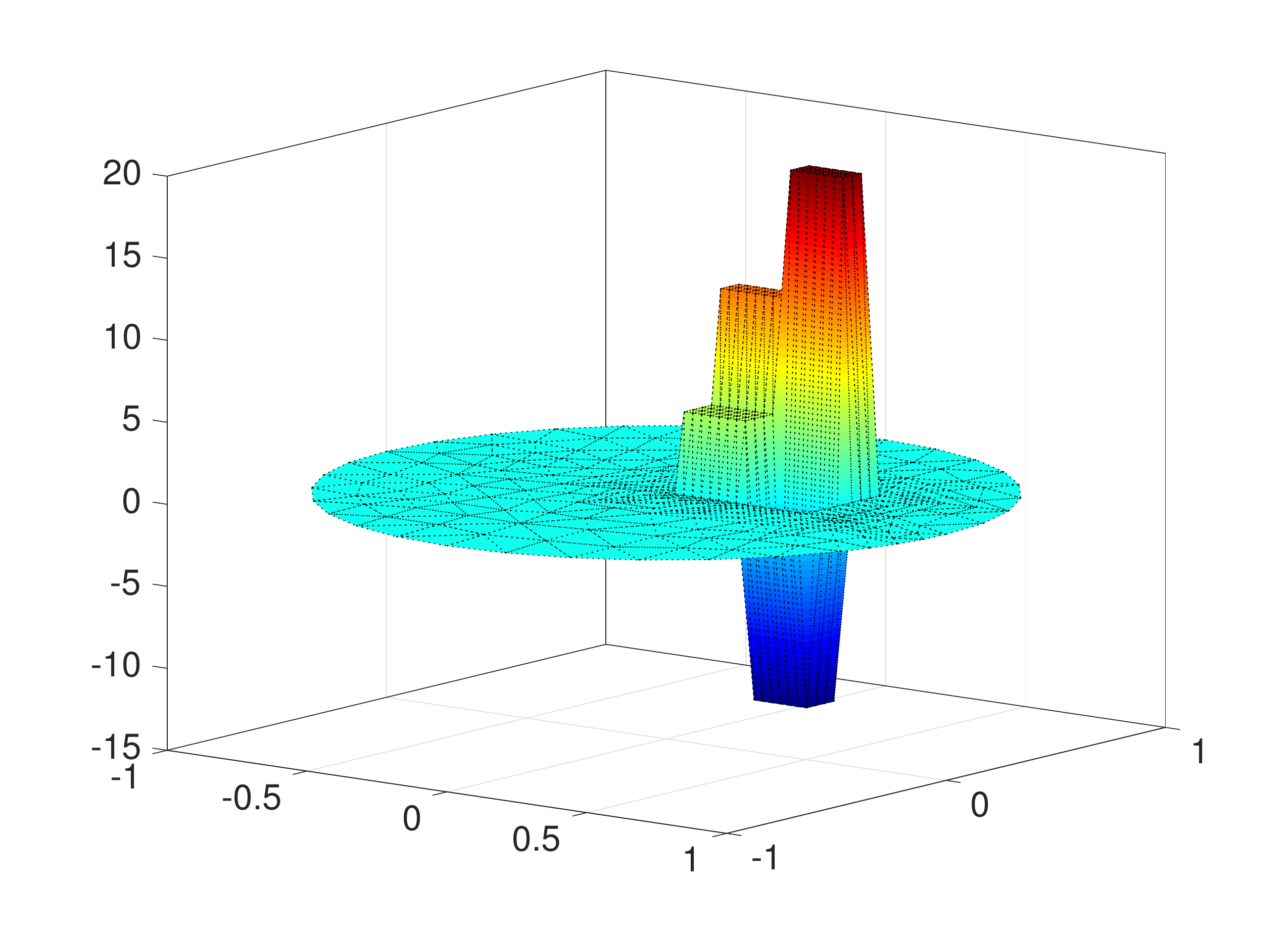}}
\subfigure[The boundary actuator $\Psi_2$.]
{\includegraphics[width=.495\linewidth]{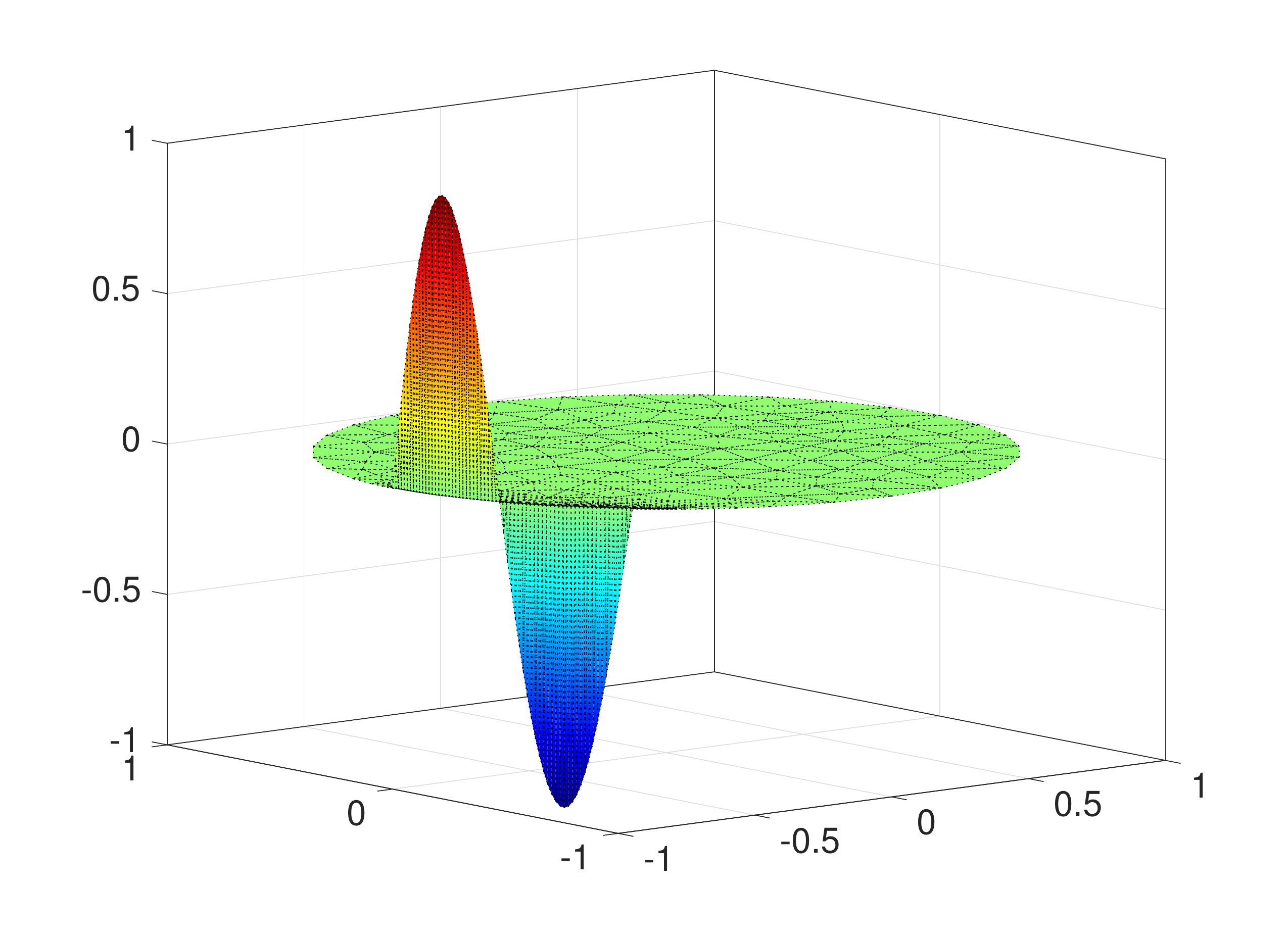}}
\caption{Internal and boundary actuators.}
\label{Controllers}
\end{figure}

\subsection{Testing with a family of functions~$(\hat a,\hat b)$.}
\label{FamilyTrajectories}
We set $\lambda = 2$, $\nu = \frac{1}{4}$,
, $\chi =  \mathbbm{1}_{\omega}$.
Next, we  choose a family of functions $\hat a,$ $\hat b_1$, and~$\hat b_2$ as follows
\begin{equation}\label{eq:FamilyTrajectories}
\begin{aligned}
\hat a (t,x)&= - \sin (t) \cos (\mathbf{i}x_1) + \sin(5t) \sin(\mathbf{j}x_2) -3, \\
\hat b_1(t,x) &=  \cos (t) \sin (-\mathbf{k}x_1) - \cos(3t) \cos (\mathbf{l} x_2),\\
\hat b_2(t,x) &= \sin(-t) \sin(\mathbf{m} x_1) - \cos(2t) \sin(\mathbf{n}x_2).
\end{aligned}
\end{equation}
We will test firstly with 6 piecewise-constant actuators, corresponding to~$(m,n)=(3,2)$, and then with 6 boundary
actuators~$\Psi_i$, $i\in\{1,2,\dots,6\}$. The initial condition is set $v_0 (x) \coloneqq \sin(2x_1) \cos(x_2)$ and the
time-interval is set $[0,8]$. In Figures~\ref{figure_familytrajectories}(b) and~\ref{figure_familytrajectories}(c),
we can observe that both internal and boundary feedback
control is able to stabilize the system with the desired rate $\frac{\lambda}{2} = 1$ for some 
parameters $(\mathbf{i},\mathbf{j},\mathbf{k},\mathbf{l},\mathbf{m},\mathbf{n}) \in \PP$, with
\begin{align*}
\PP \coloneqq \left\{(1,1,1,1,1,1), (1,2,2,1,1,1), (2,-1,1,-3,5,1),\right.\\
\qquad \qquad \left. (-1,5,3,1,1,5),(1,2,3,4,5,6),(6,-2,5,3,4,1) \right\}.
\end{align*}
In Figure~\ref{figure_familytrajectories}, as well in following ones, the squared norm $|z|_H^2$ is understood as the discrete
approximation $\overline{z}^{\top} \mathbf{M} \overline{z}$.
We would like to emphasize that without any control, the system is unstable for these parameters defined by $\PP$ as we can see from Figure~\ref{figure_familytrajectories}(a).
\begin{figure}[]  
  \centering
\subfigure[Without control.]
{\includegraphics[width=.325\linewidth]{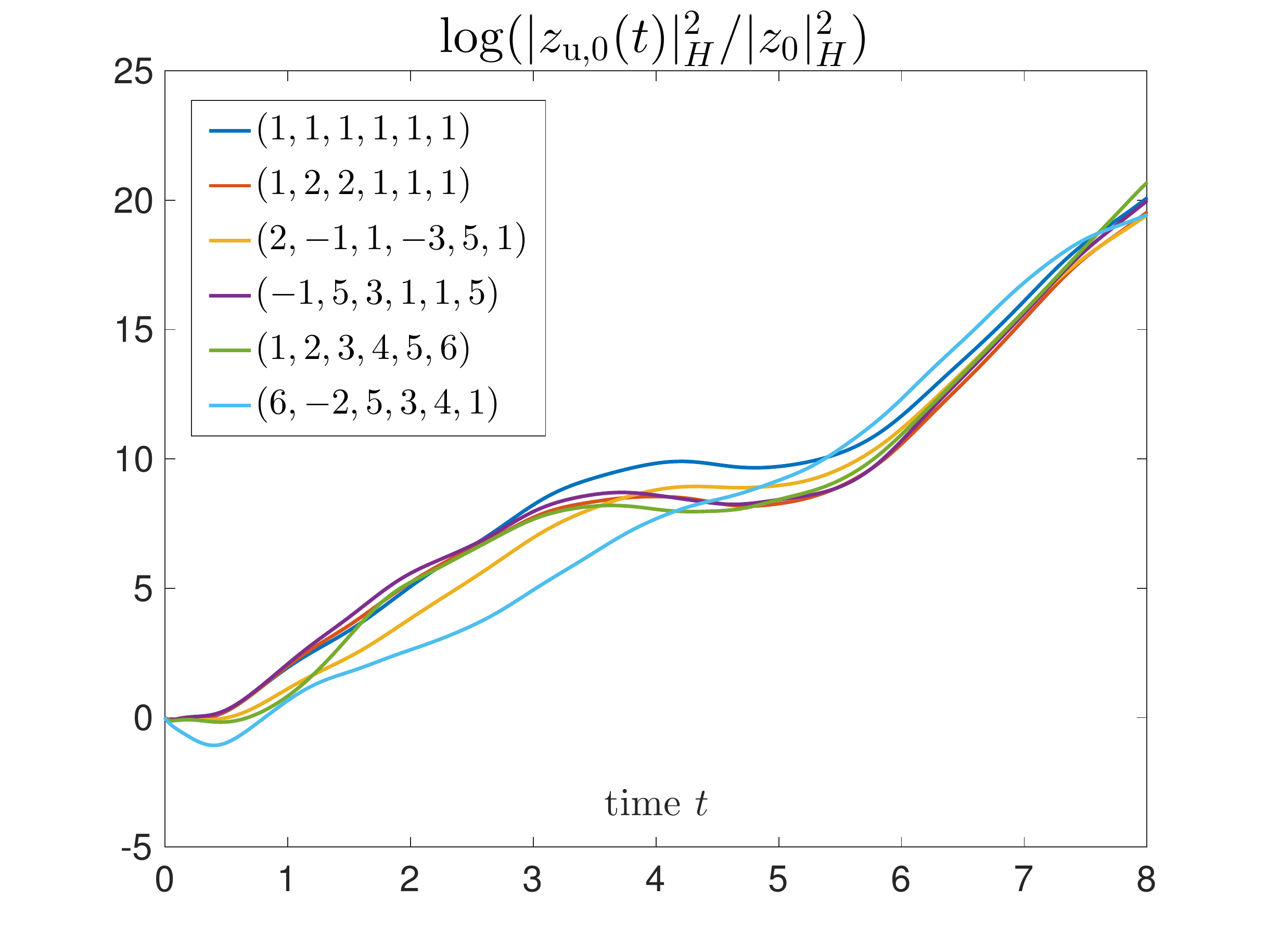}}
\subfigure[With internal feedback.]
{\includegraphics[width=.325\linewidth]{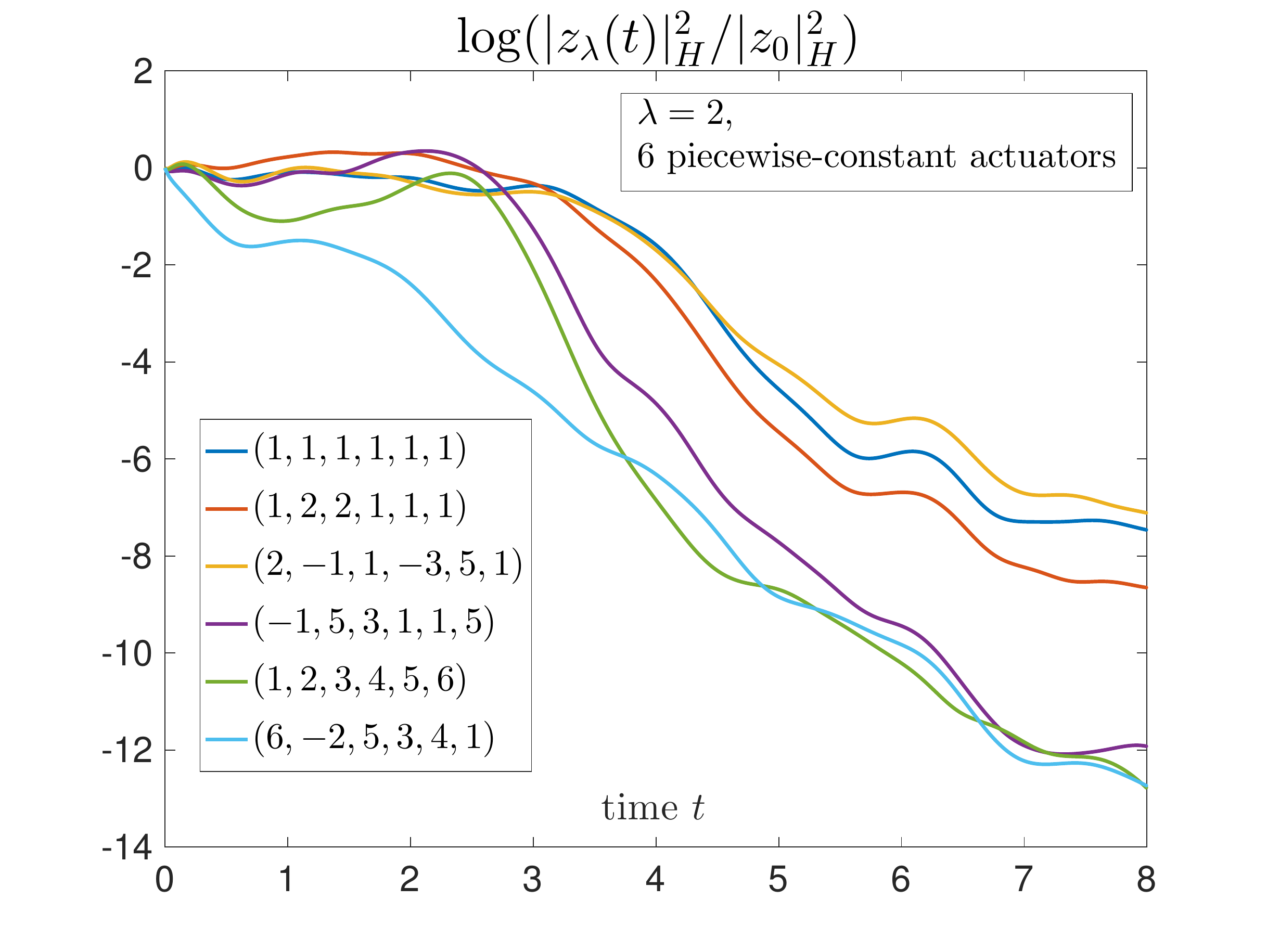}}
\subfigure[With boundary feedback.]
{\includegraphics[width=.325\linewidth]{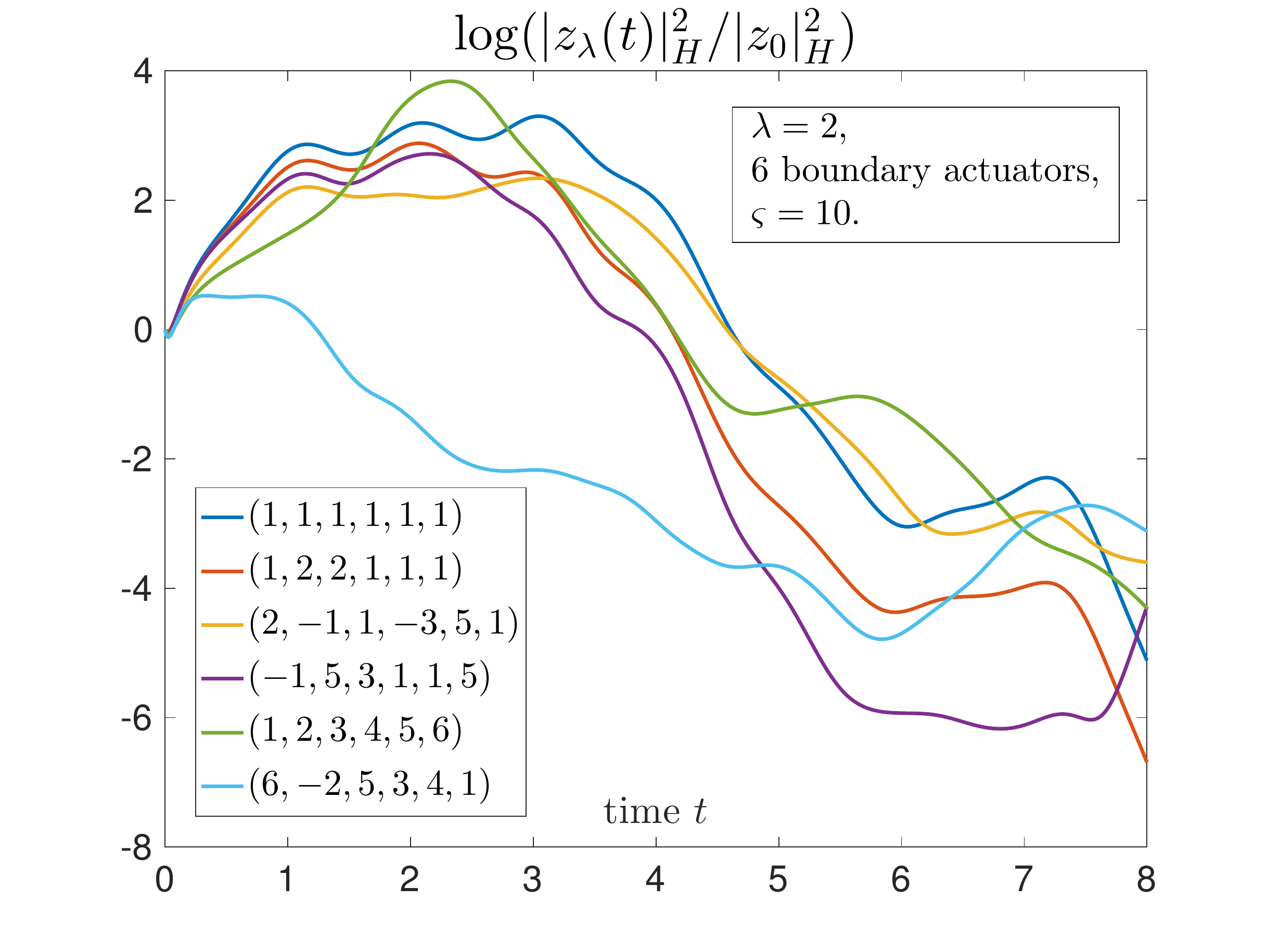}}
\caption{Stabilization rate is provided by the feedback control.}
\label{figure_familytrajectories}
\end{figure}
\subsection{Increasing the number of actuators}
We compare the results we obtain by changing the number of actuators. With no surprise, we see that with more actuators we obtain better results. We will take one element of the
family in~\eqref{eq:FamilyTrajectories}, namely the one corresponding to $(\mathbf{i},\mathbf{j},\mathbf{k},\mathbf{l},\mathbf{m},\mathbf{n}) = (2,-1,1,-3,5,1)$ and
also take $\lambda = 2$. 
\subsubsection{Internal feedback control}\label{ssS:intcontrols}
In
Figure~\ref{figure_ControlSol4DifferentNumControls_compare1}, we present the results for some rearrangements~$(m,n)$. We observe that
the cost~$(\Pi z(t),z(t))_H$, understood as~$\overline{z(t)}^\perp\Pi_D(t)\overline{z(t)}$,  decreases as the number of controls increase, as we can see in Figure~\ref{figure_ControlSol4DifferentNumControls_compare1}(b).

Notice that we cannot say a priori which among~$(m,n)=(2,2)$ and~$(m,n)=(4,1)$ is better because one set of actuators does not include the other. However,
for the considered example we observe that~$(2,2)$ is better than~$(4,1)$.

The case $(m,n)=(+\infty, +\infty)$ in Figure \ref{figure_ControlSol4DifferentNumControls_compare1} should be understood as the case we do not impose any restriction on
number of actuators. That is, to the case we take $P_M=1$, or in other words in the case our control operator in~\eqref{eq:kk2} is just~$1_\omega\chi 1_\omega$. In this case
$\RRR^{\rm in}=\DD_{\overline{\chi 1_\omega}}$ in~\eqref{RInternal}.

\begin{figure}[htb]  
  \centering
\subfigure[Controlled solution.]
{\includegraphics[width=.325\linewidth]{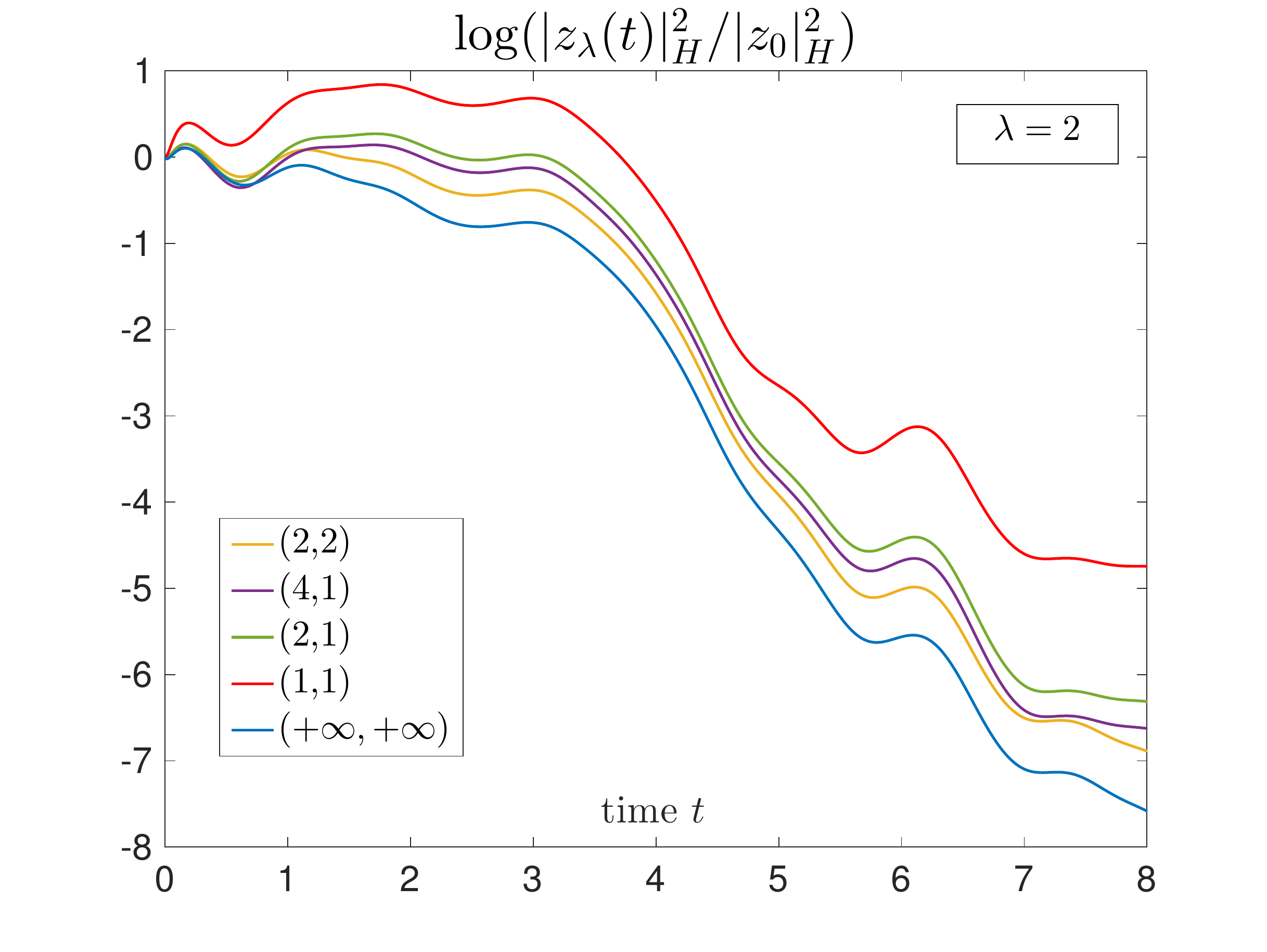}}
\subfigure[Cost function.]
{\includegraphics[width=.325\linewidth]{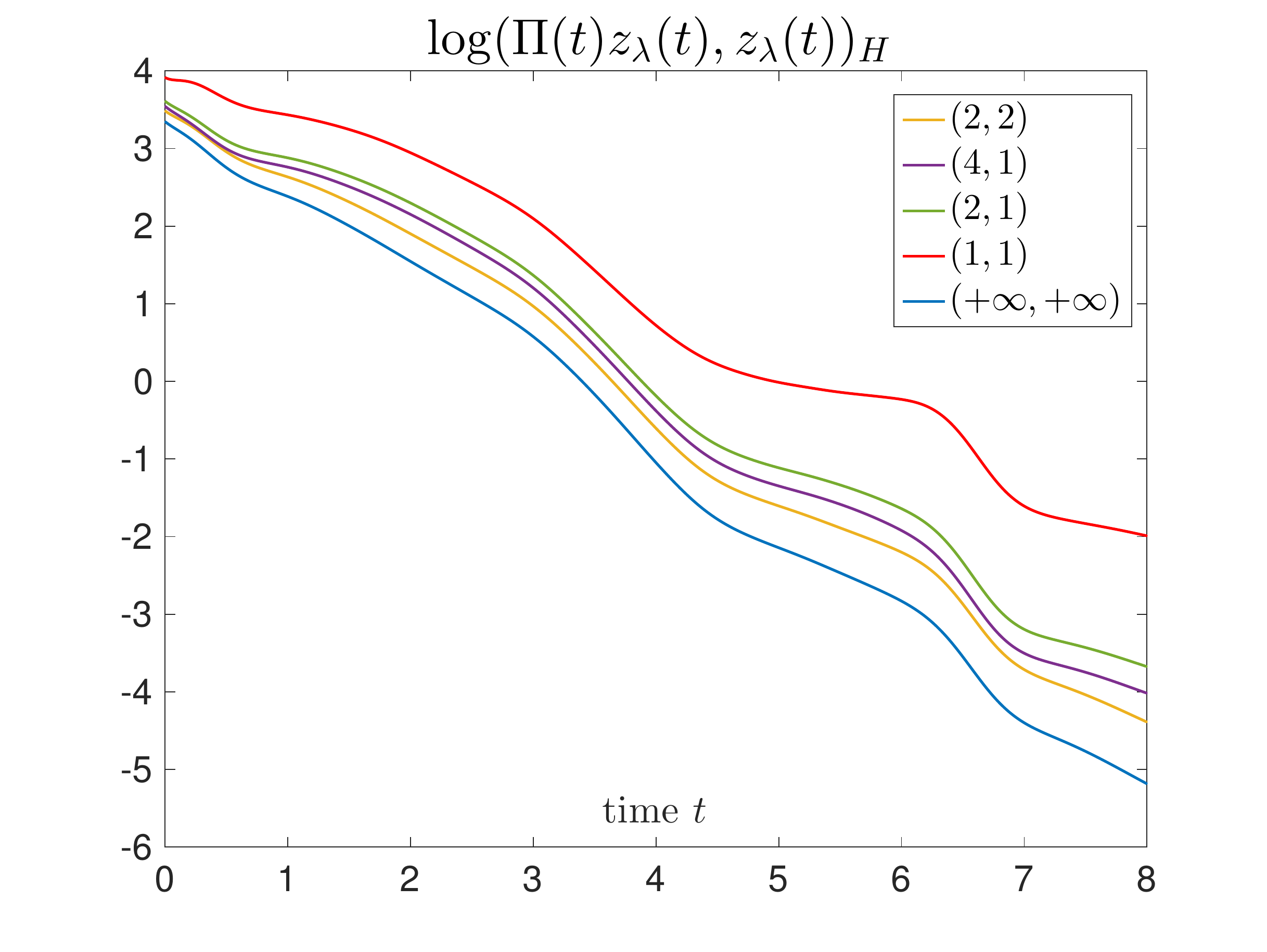}}
\subfigure[Control.]
{\includegraphics[width=.325\linewidth]{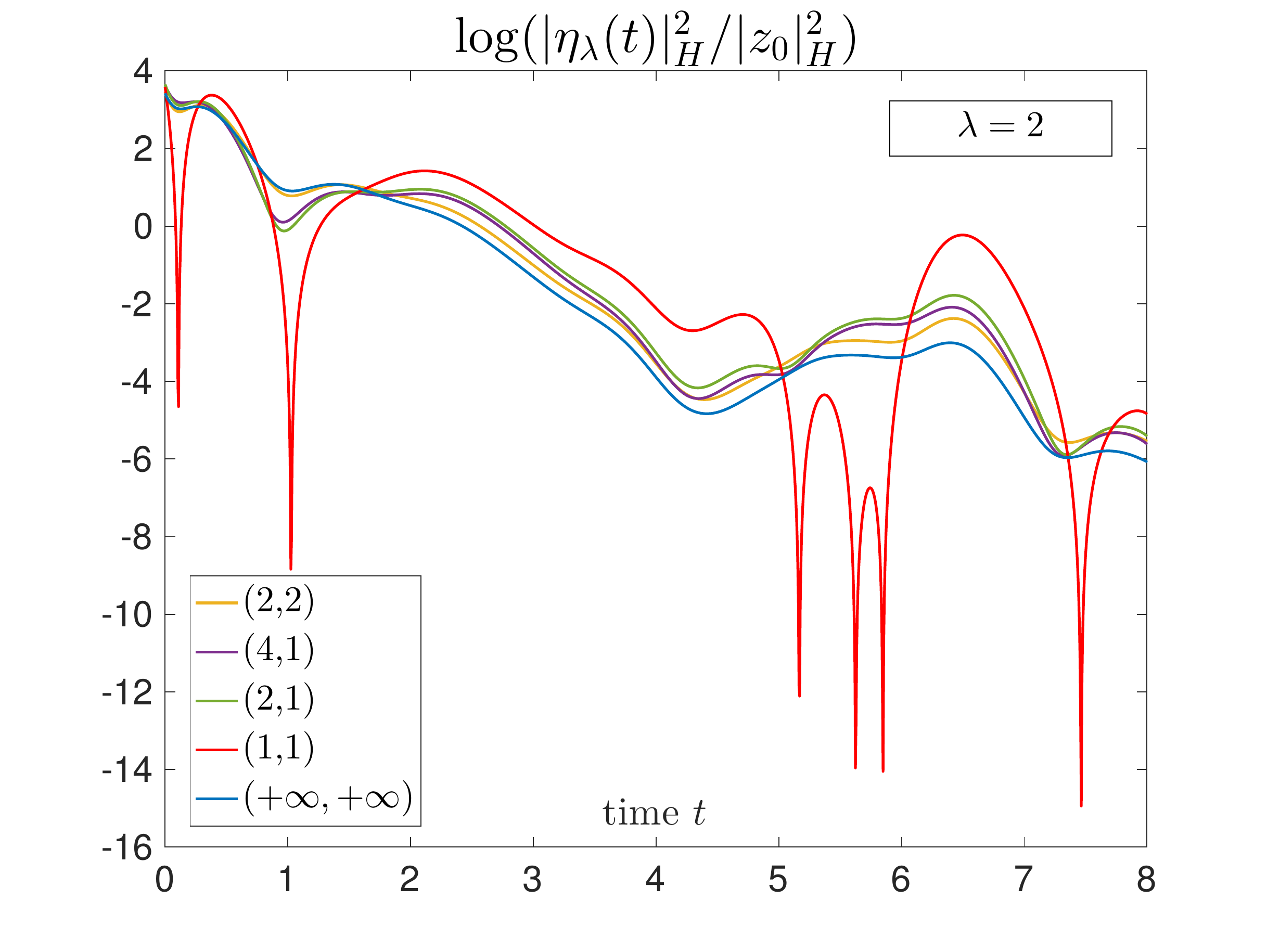}}
\caption{The convergence rate of the solution, the cost function and the control.}
\label{figure_ControlSol4DifferentNumControls_compare1}
\end{figure}

Recall that our control can be written as~$\eta(t,x)= \sum\limits_{l_1 =1}^{m}\sum\limits_{l_2 =1}^{n}{\eta_{l_1,l_2}(t) 1_{\omega_{l_1,l_2}} (x)}$. 
For example, in the left column of Figure \ref{figure_Controls_Sometimesteps}, we plot the value of the feedback control, in the case of a single actuator, at two time instants $t = 4.5$ and $t = 6.5$.
We refer also to Figure~\ref{figure_valuecontrol}(a) for this value at all time instants $t\in[0,8]$. 
Notice that, the ``cusps'' in Figure~\ref{figure_ControlSol4DifferentNumControls_compare1}(c), in the case~$(m,n)=(1,1)$, are due to the fact that the
control vanishes at the corresponding time instants.
In the case we take only one actuator a ``cusp'' will appear when the control changes sign, because the plotted function ``takes'' the value~$-\infty$ if the control
vanishes, see also Figure~\ref{figure_valuecontrol}(a).

In Figure~\ref{figure_Controls_Sometimesteps}, we plot the control at two time instants $t = 4.5$ and $t = 6.5$.
In Figure~\ref{figure_valuecontrol} we plot the value of each actuator for all time instants for some rearrangements~$(m,n)$.

\begin{figure}[htb]
{\includegraphics[width=1\linewidth]{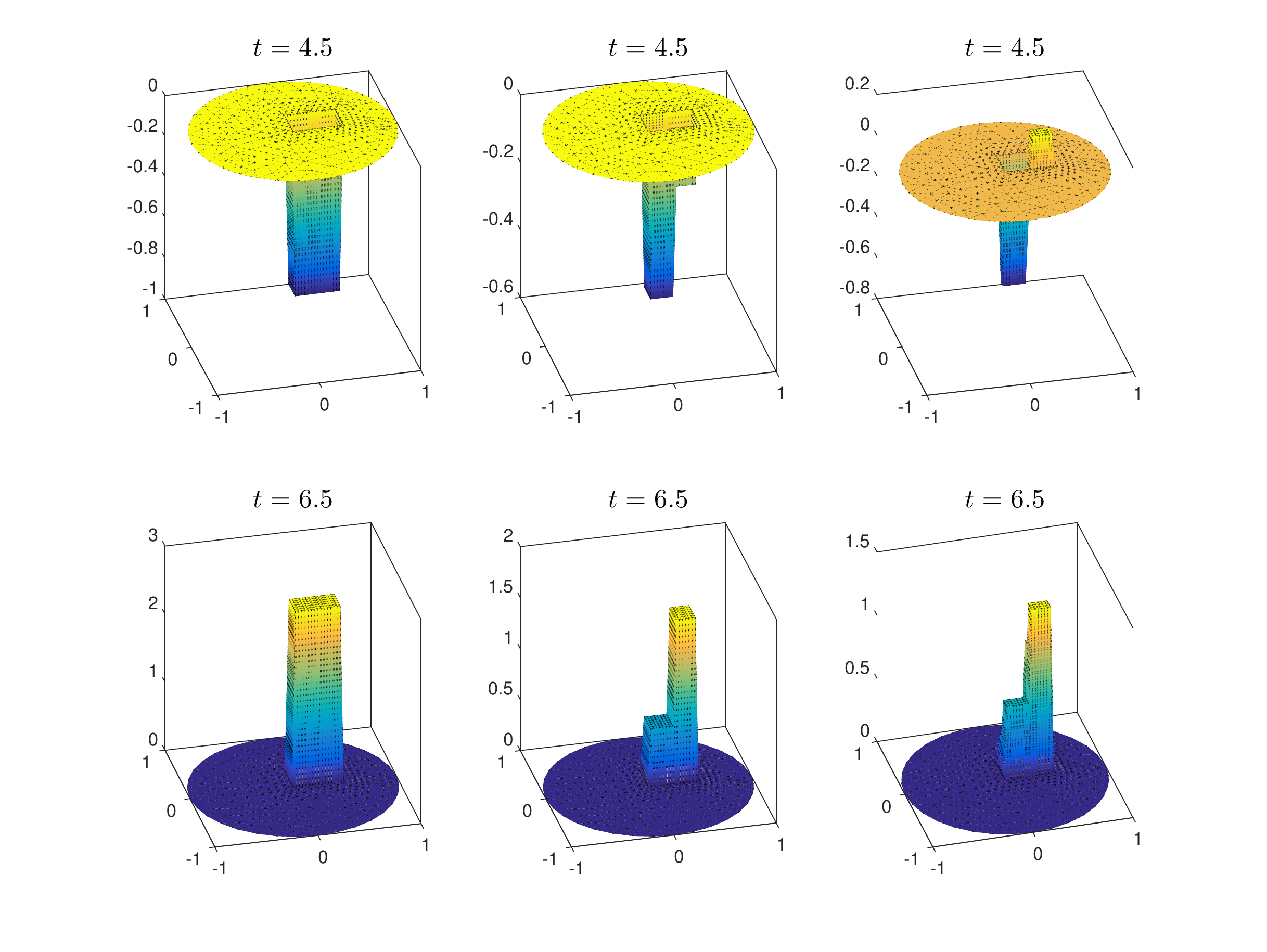}}
\caption{Feedback control at~$t\in\{4.5,6.5\}$ for $(m,n)\in\{(1,1),(2,1),(2,2)\}$.}
\label{figure_Controls_Sometimesteps}
\end{figure}

\begin{figure}[]  
  \centering
\subfigure[$(m,n)=(1,1)$, $1$~actuator.]
{\includegraphics[width=.325\linewidth]{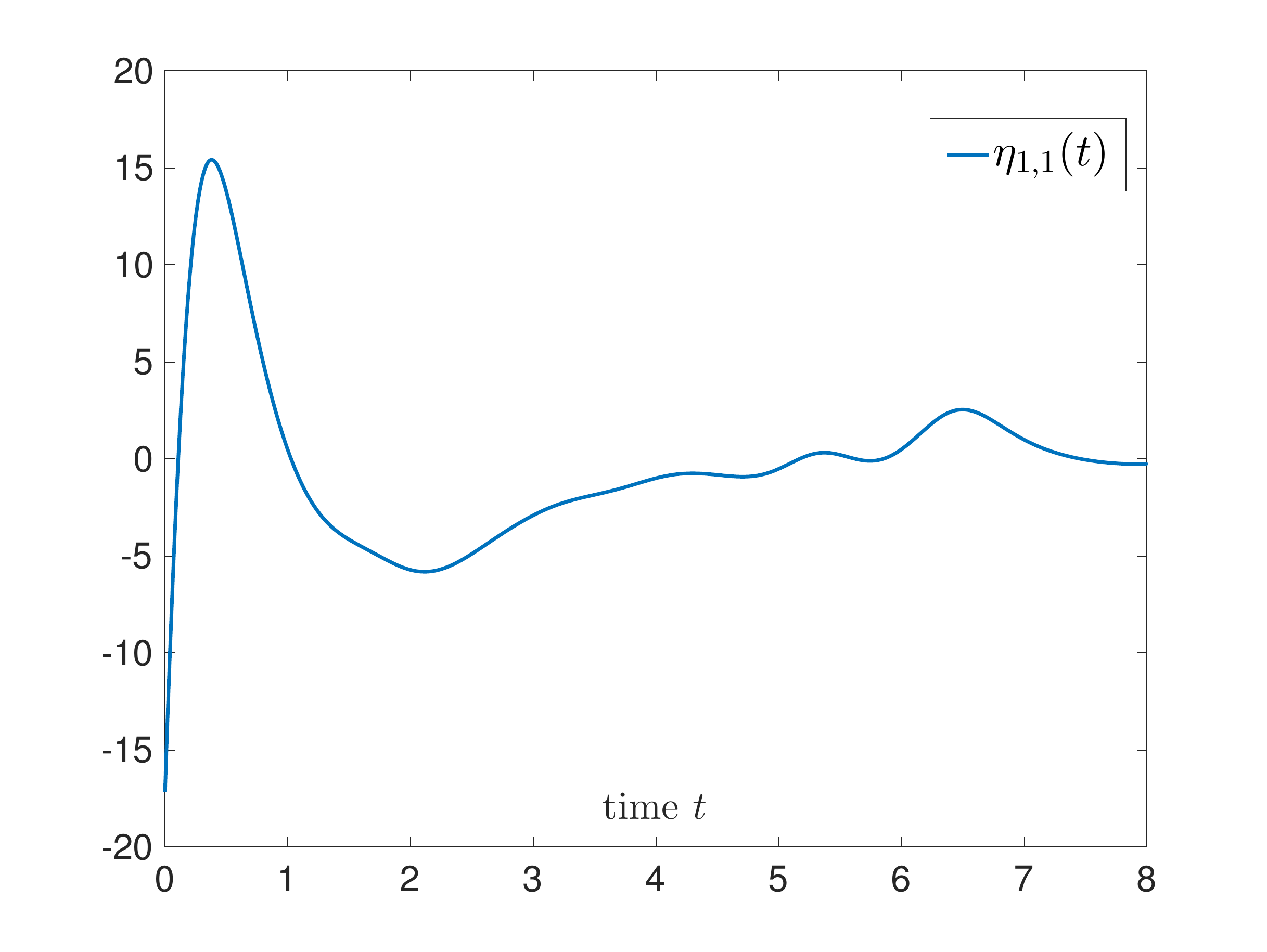}}
\subfigure[ $(m,n)=(2,1)$, $2$~actuators.]
{\includegraphics[width=.325\linewidth]{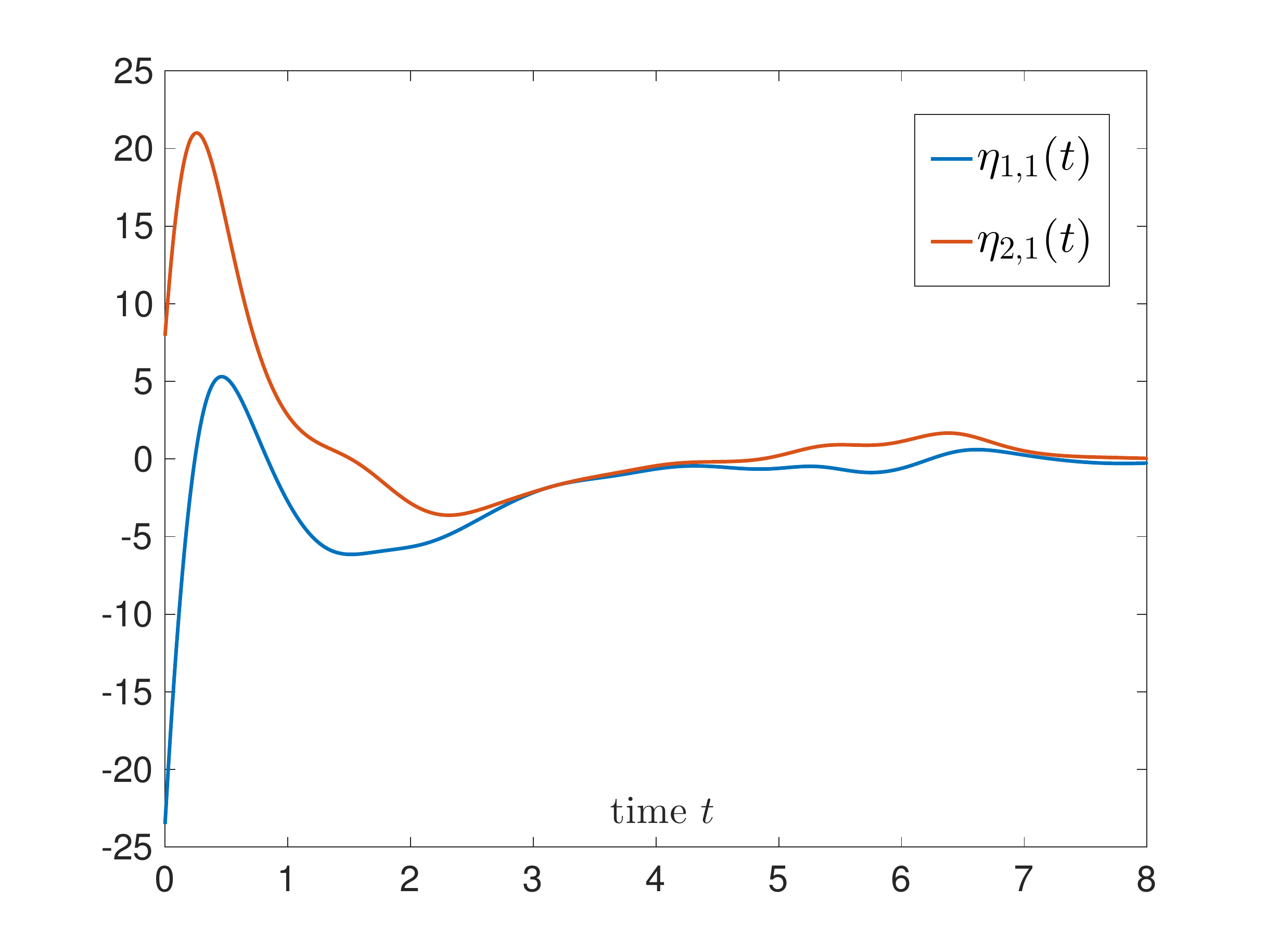}}
\subfigure[$(m,n)=(2,2)$, $4$~actuators.]
{\includegraphics[width=.325\linewidth]{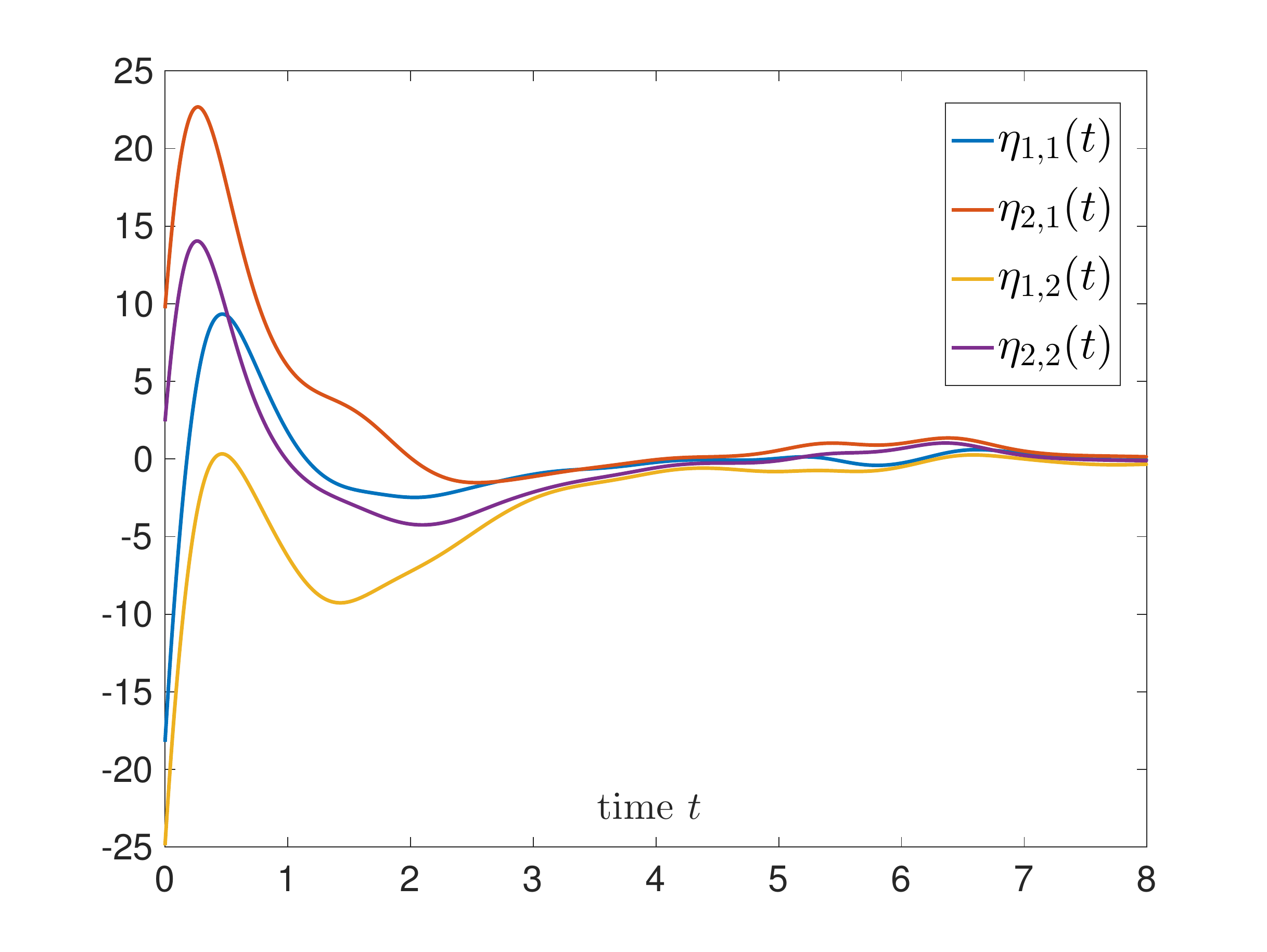}}
\caption{The magnitude(s) of the internal actuator(s).}
\label{figure_valuecontrol}
\end{figure}

\subsubsection{Boundary feedback control}\label{ssS:bdrycontrols}
Now we consider the reference function in~\ref{FamilyTrajectories} corresponding to $(\mathbf{i},\mathbf{j},\mathbf{k},\mathbf{l},\mathbf{m},\mathbf{n}) = (6,-2,5,3,4,1)$
and we take $\lambda = 2$ and~$\varsigma=10$. We take a family of actuators~$\{\Psi_i\mid i\in\{1,2,\dots,M\}\}$, as in~\eqref{sin_contbdry}.
We compare the result in the cases~$M\in\{1,2,4,6\}$. With no surprise, from Figure~\ref{figure_boundarycontrol}(a), taking more actuators leads to better results.
The boundary feedback control here can be written as~$\eta_M(t) = \sum_{i=1}^M { \kappa_i(t) \Psi_i}$.
In Figures~\ref{figure_boundarycontrol}(b,c) we plot the functions~$\kappa_i=\kappa_i(t,M)$ for the cases $M\in\{1,6\}$.
\begin{figure}[ht]  
  \centering
\subfigure[The cases~$M\in\{1, 2, 4, 6\}$.]
{\includegraphics[width=.325\linewidth]{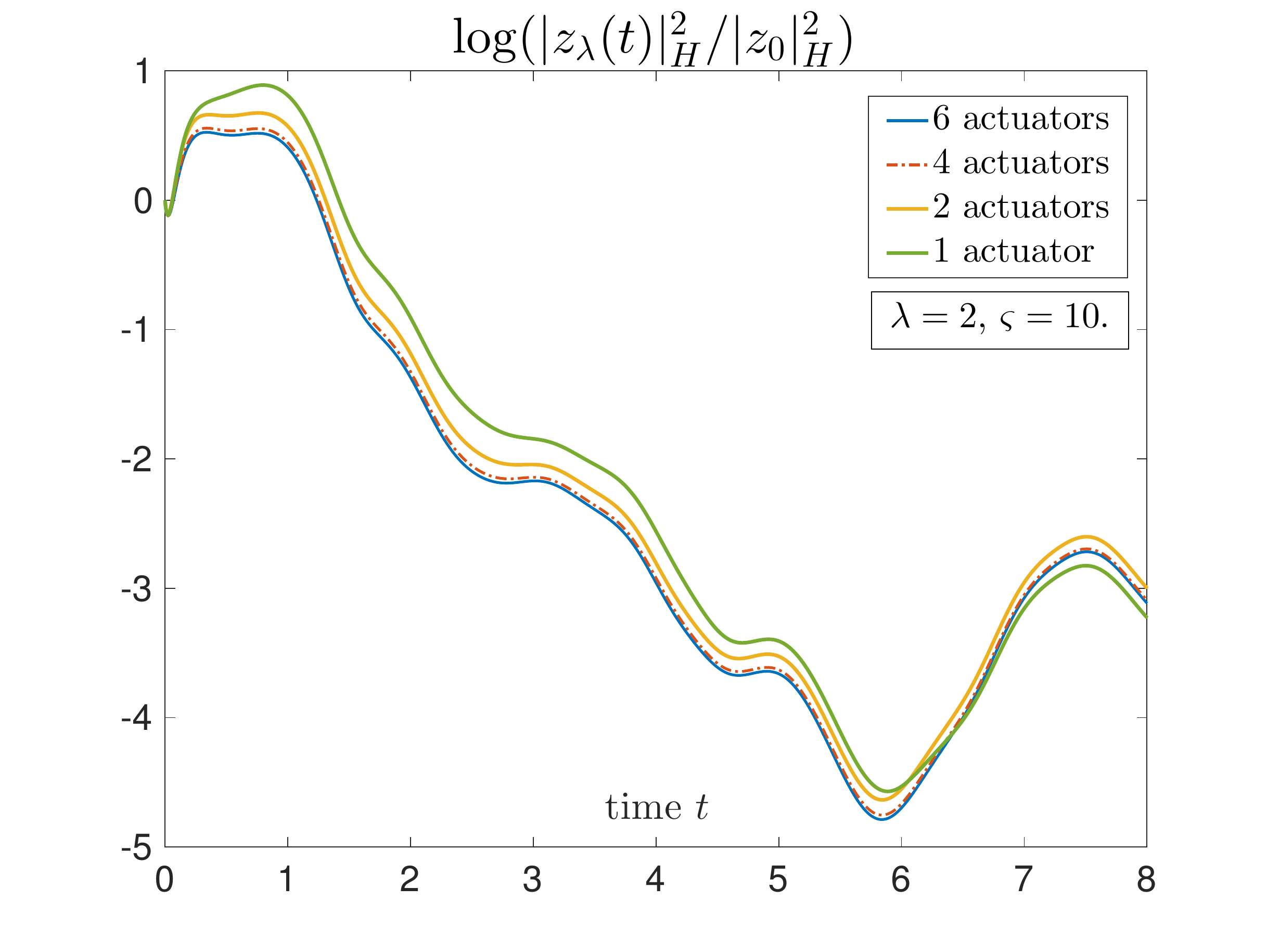}}
\subfigure[Magnitude for 1 actuator.]
{\includegraphics[width=.325\linewidth]{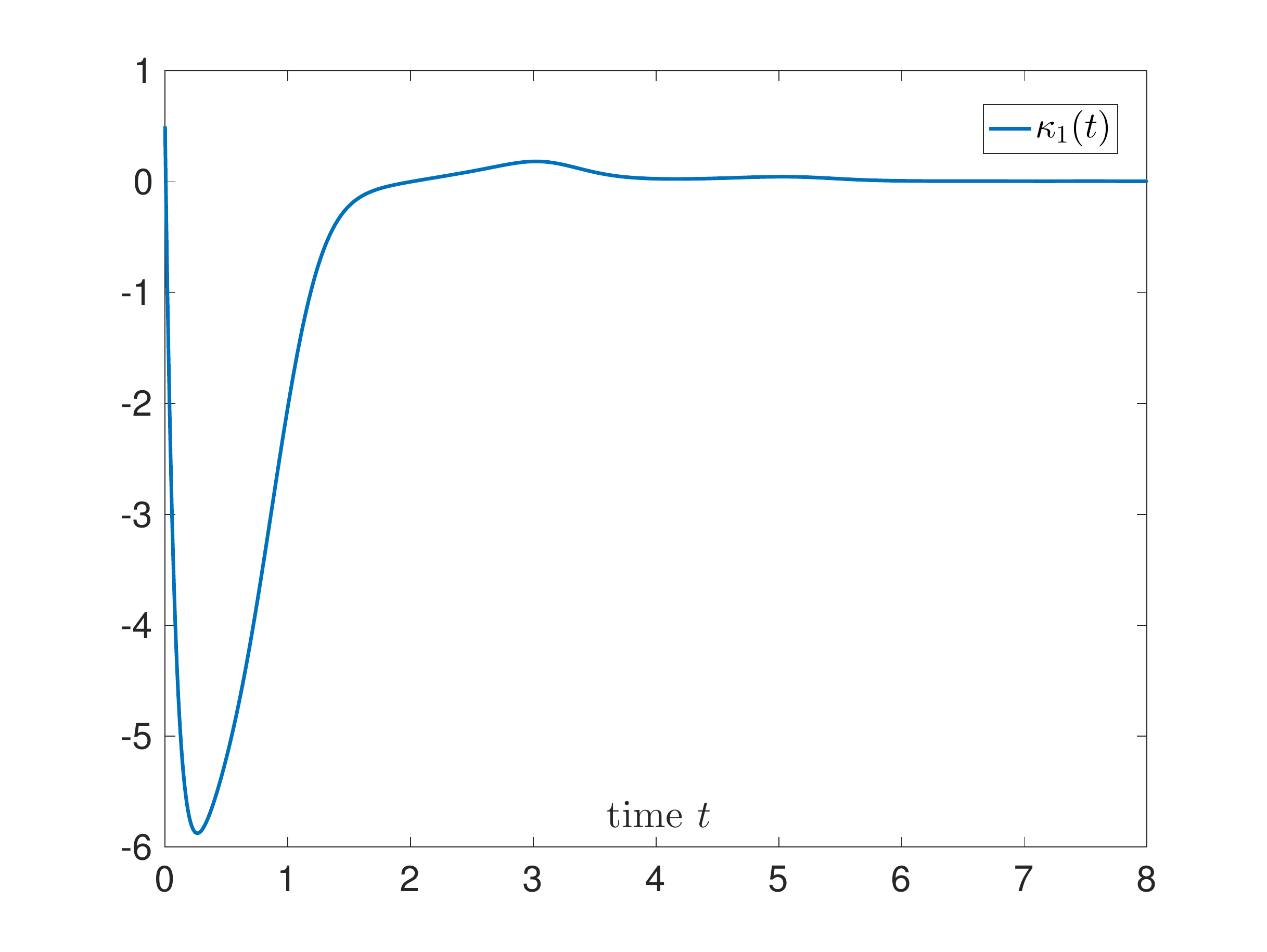}}
\subfigure[Magnitudes for 6 actuators.]
{\includegraphics[width=.325\linewidth]{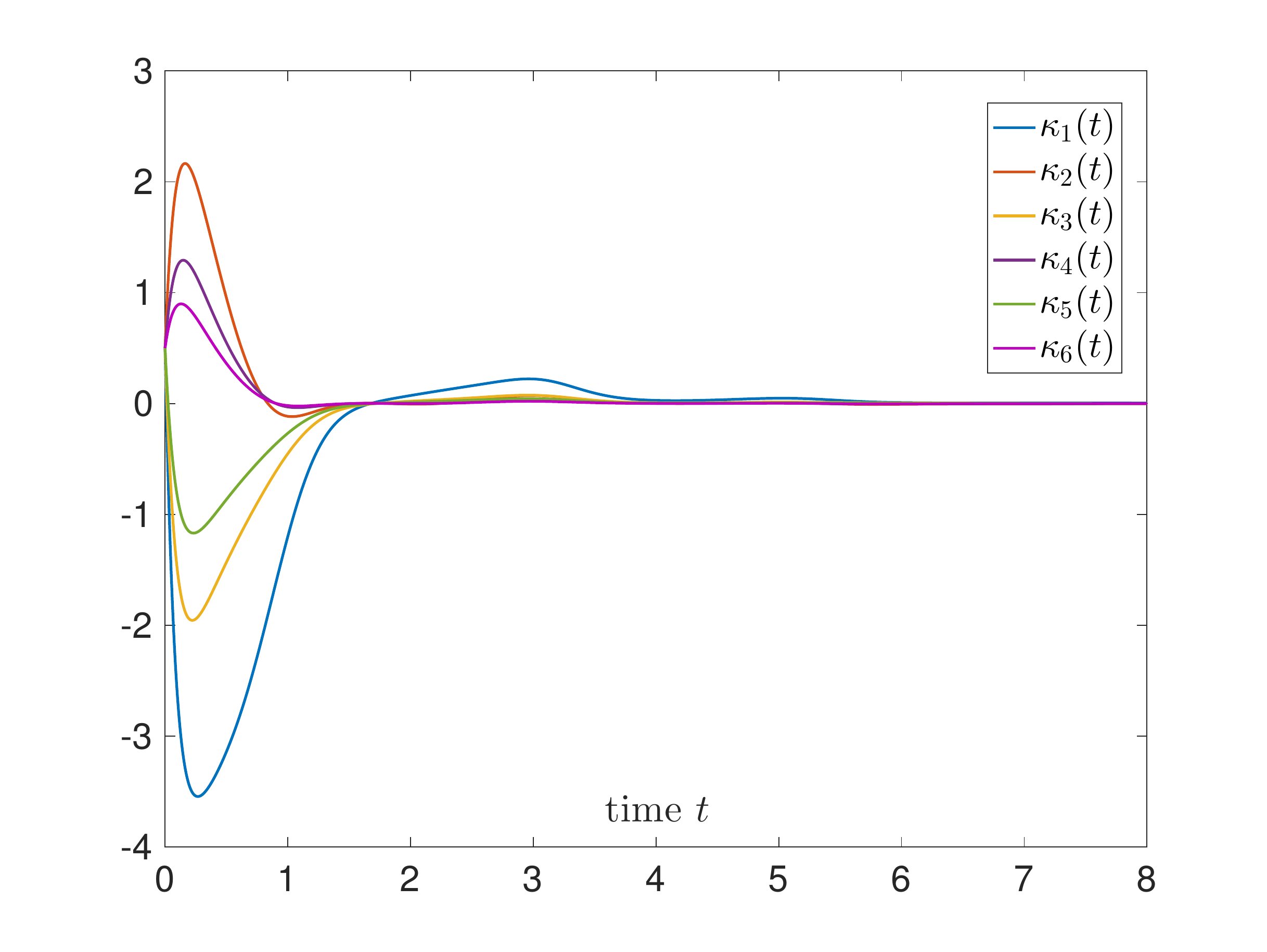}}
\caption{Convergence rate of solutions and the magnitude(s) of the boundary actuator(s).}
\label{figure_boundarycontrol}
\end{figure}
\subsection{Comparing with another placement of the actuators} We consider another placement of the actuators and
compare the controlled solutions. We recall one member of the family in \eqref{eq:FamilyTrajectories},
namely the one corresponding to $(\mathbf{i},\mathbf{j},\mathbf{k},\mathbf{l},\mathbf{m},\mathbf{n}) = (1,2,2,1,1,1)$ and
also set $\lambda = 2$.
\subsubsection{Internal feedback control}
Instead of defining all actuators in one rectangle $\omega = \left(0, \frac{1}{2} \right) \times \left(0, \frac{1}{3} \right)$
as in Figure \ref{Controllers}(a), we define 4 piecewise-constant actuators in 4 rectangles~$\omega_i$ away from each other as in
Figure~\ref{figure_TwoWaysDefineIC}(a). We want to say that the sum of the areas of the four separated rectangles~$\omega_i$ equals the area of~$\omega$, that is,
$\sum_{i=1}^4{|\omega_i|} = |\omega|$. Comparing Figures~\ref{figure_TwoWaysDefineIC}(b) and~\ref{figure_ControlSol4DifferentNumControls_compare1}(a),
we see that the latter placement is more effective than the former one (for this example). 
\begin{figure}[]  
  \centering
\subfigure[New placement of 4 internal actuators.]
{\includegraphics[width=.495\linewidth]{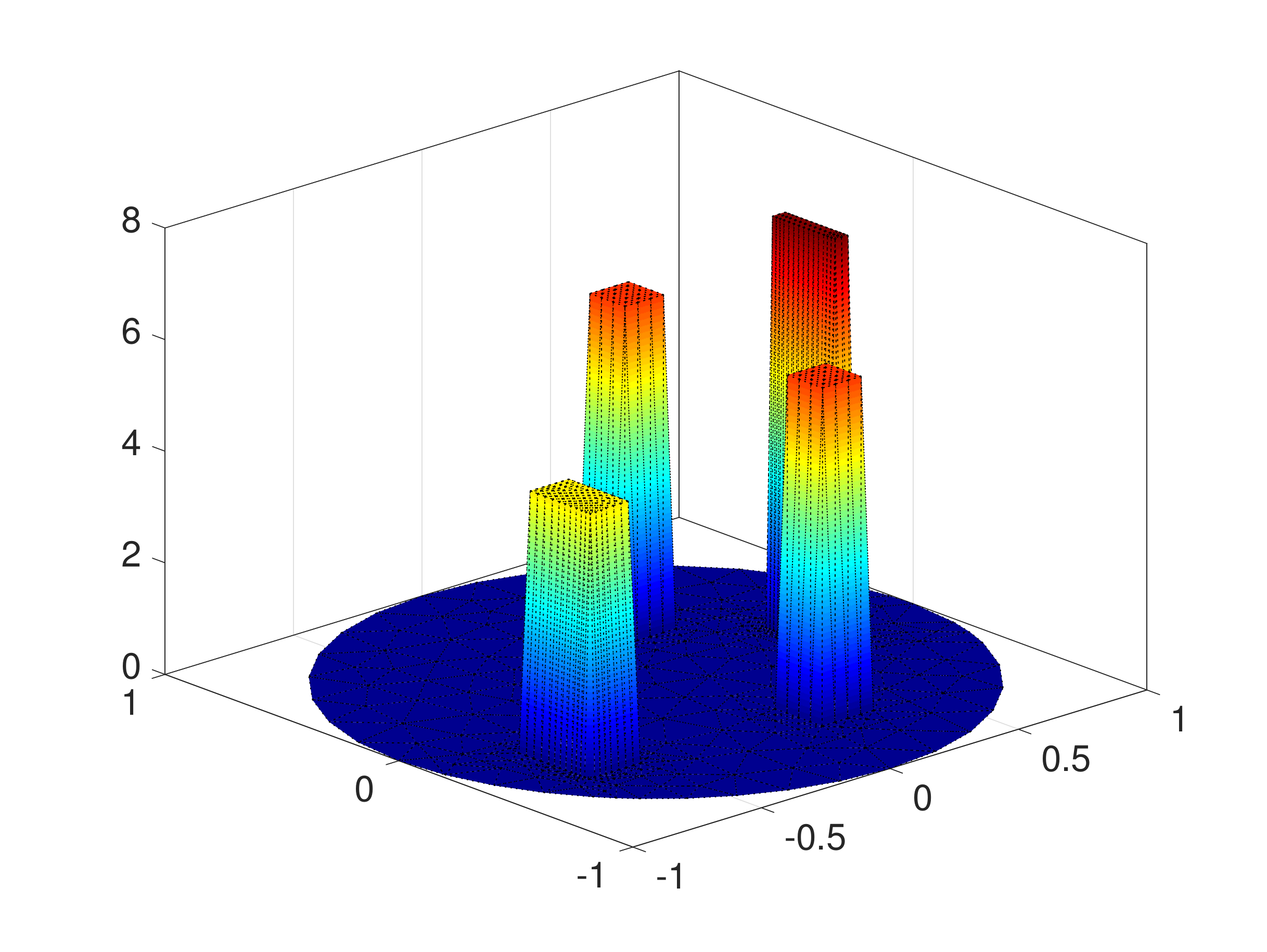}}
\subfigure[Controlled solution of the 2 placements.]
{\includegraphics[width=.495\linewidth]{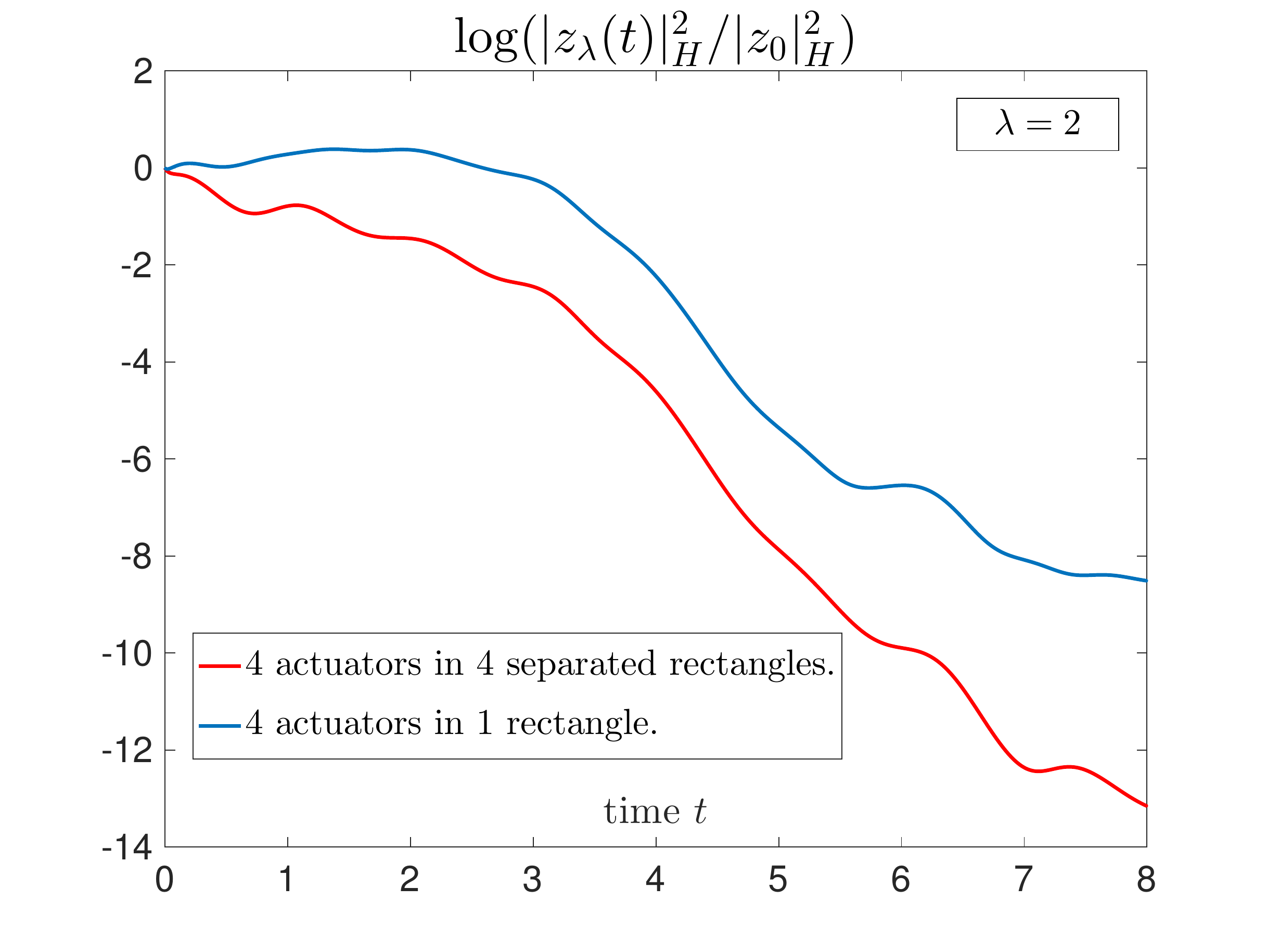}}
\caption{Compare two placements of internal actuators.}
\label{figure_TwoWaysDefineIC}
\end{figure}

\subsubsection{Boundary feedback control}
Here, we set $\zeta = 10$ and we use 4 boundary actuators. Instead of defining as in~\eqref{sin_contbdry}, see Figure~\ref{Controllers}(b), using 4 frequencies of the sinus
functions in an interval~$\left(\pi, \frac{5\pi}{4} \right)$, we now define the new actuators in domains away from each other as in Figure~\ref{figure_TwoWaysDefineBC}(a),
where each actuator is as in~\eqref{sin_contbdry} with $i=1$. Comparing Figures~\ref{figure_TwoWaysDefineBC}(b) and~\ref{figure_boundarycontrol}(a) we see
that the latter placement seems to be better than the former one (for this example).
\begin{figure}[]  
  \centering
\subfigure[New placement of 4 boundary actuators.]
{\includegraphics[width=.495\linewidth]{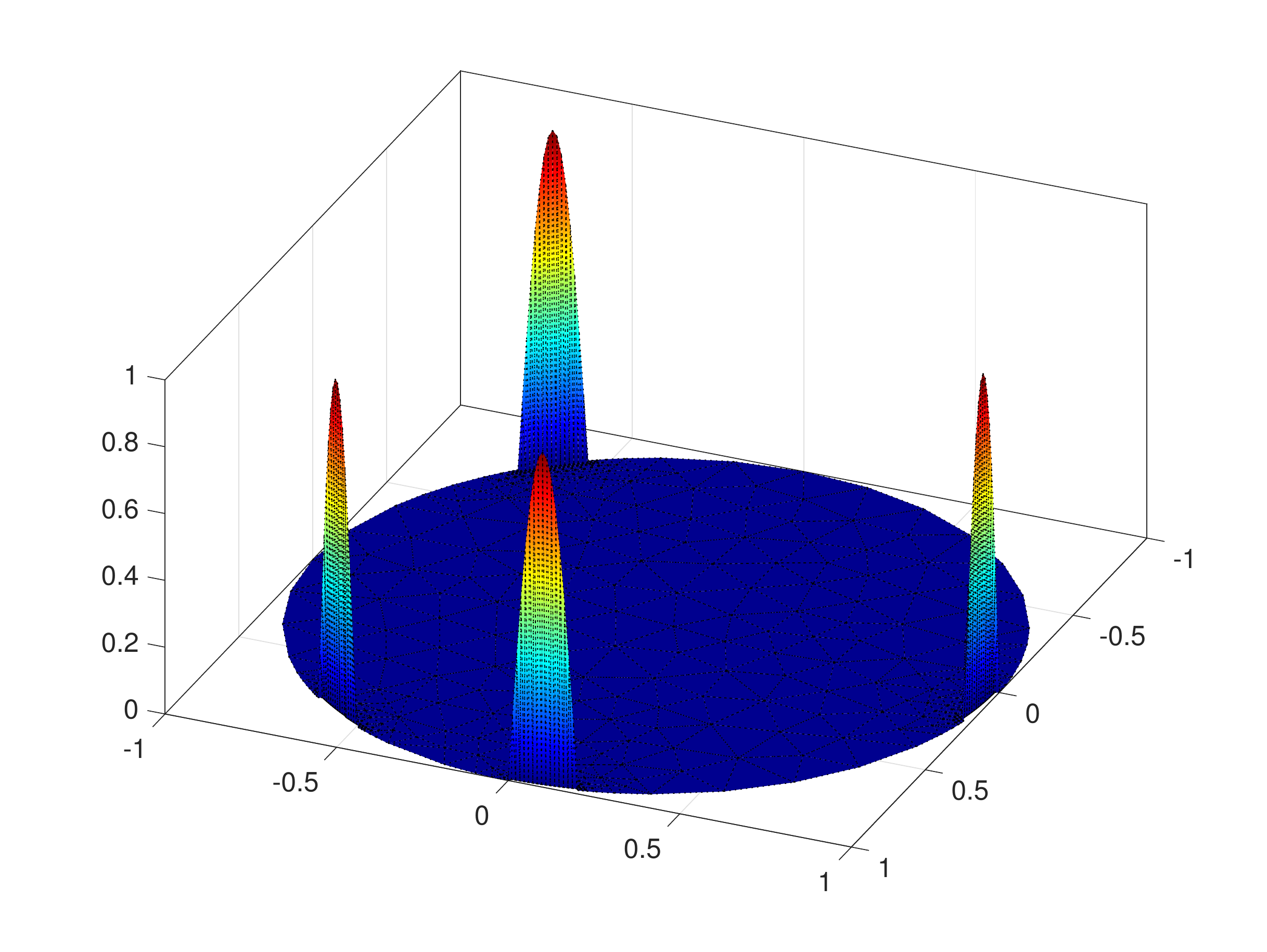}}
\subfigure[Controlled solution of the two placements.]
{\includegraphics[width=.495\linewidth]{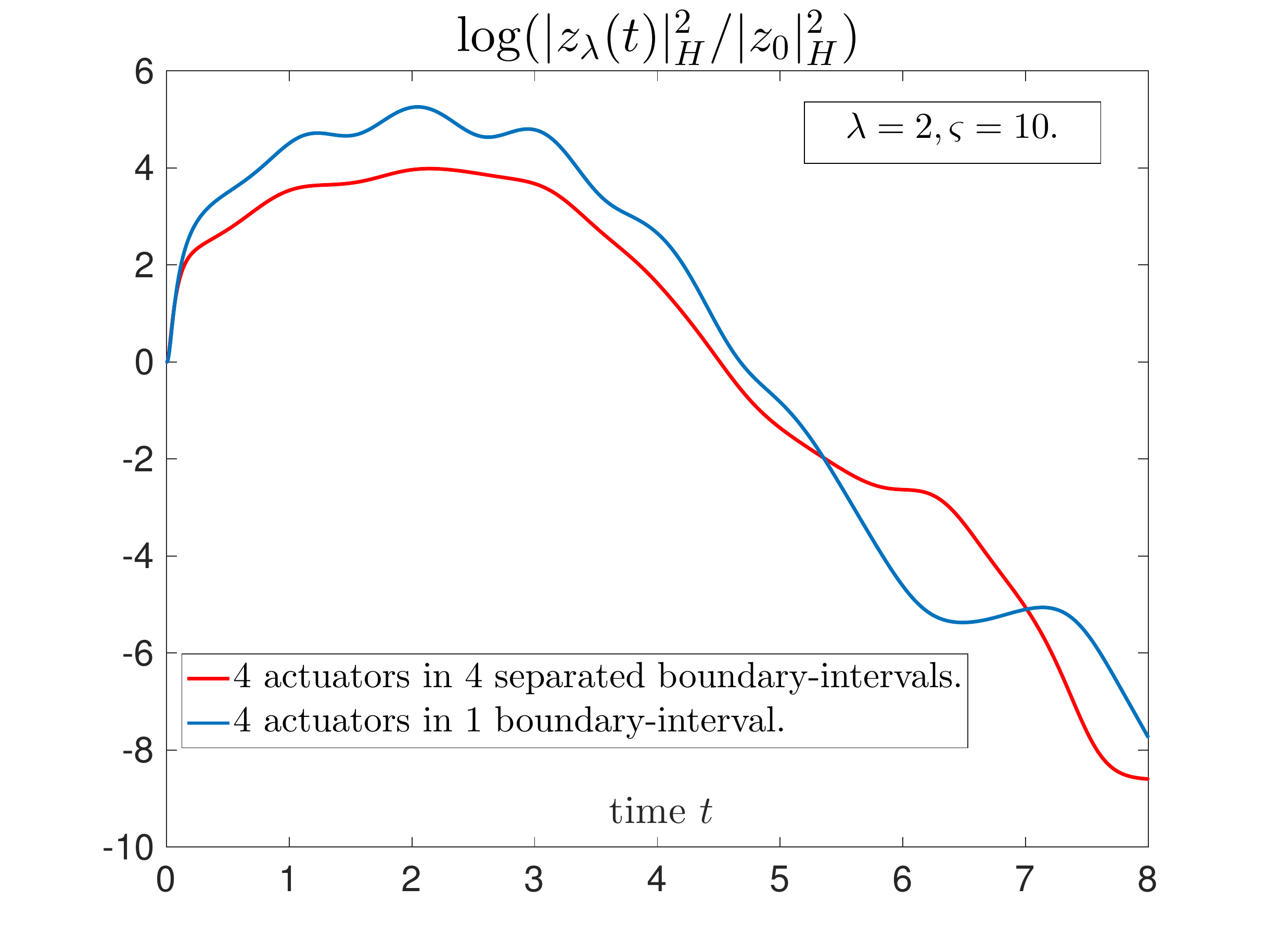}}
\caption{Compare two placements of the boundary actuators.}
\label{figure_TwoWaysDefineBC}
\end{figure}

\subsection{Dependence of the transient bound on the desired decreasing rate}
\label{sS:RateLambda}
Here we consider the case of internal controls with no restriction in the dimension of the control, we take~$P_M=1$, $\chi=1$
and~$\RRR^{\rm in}=\DD_{\overline{\chi1_\omega}}$ in~\eqref{RInternal}. We want to check whether the transient bound $\widehat C_{\lambda}$
associated to the Riccati based feedback control 
does depend on the desired exponential rate~$\lambda\ge0$ as
\begin{equation}\label{transalpha}
 C_{\lambda}\approx \ex^{\TT_{0,\alpha}+\TT_{1,\alpha}\lambda^\alpha}
\end{equation}
(cf. Section~\ref{sS:transient}), for a suitable~$\alpha\ge0$. As we will see, the answer is not clear.

We set $\nu = \frac{1}{2}$, $\hat a = -10 + 2x_1 + \cos x_2$, $\hat b = \left(-x_1^2 ,-\sin x_2 \right)$, and we impose no restriction on number of actuators.

Notice that~$(\hat a,\hat b)$ is independent of time. In the stationary case, we have ``just'' to solve an algebraic Riccati equation
for each~$\lambda$. Up to our best knowledge, even in the stationary case it is not known how precisely the constant~$C_{\lambda}$ depends on~$\lambda$,
for Riccati based stabilizing feedback.

Figure~\ref{figure_familyoflambda} shows that the system can be stabilized with rate~$\lambda\in[0,30]$, the figure also shows that
the constant~$C_{\lambda}$ increases with~$\lambda$.
\begin{figure}[htb]  
  \centering
\subfigure[With internal feedback control.]
{\includegraphics[width=.495\linewidth]{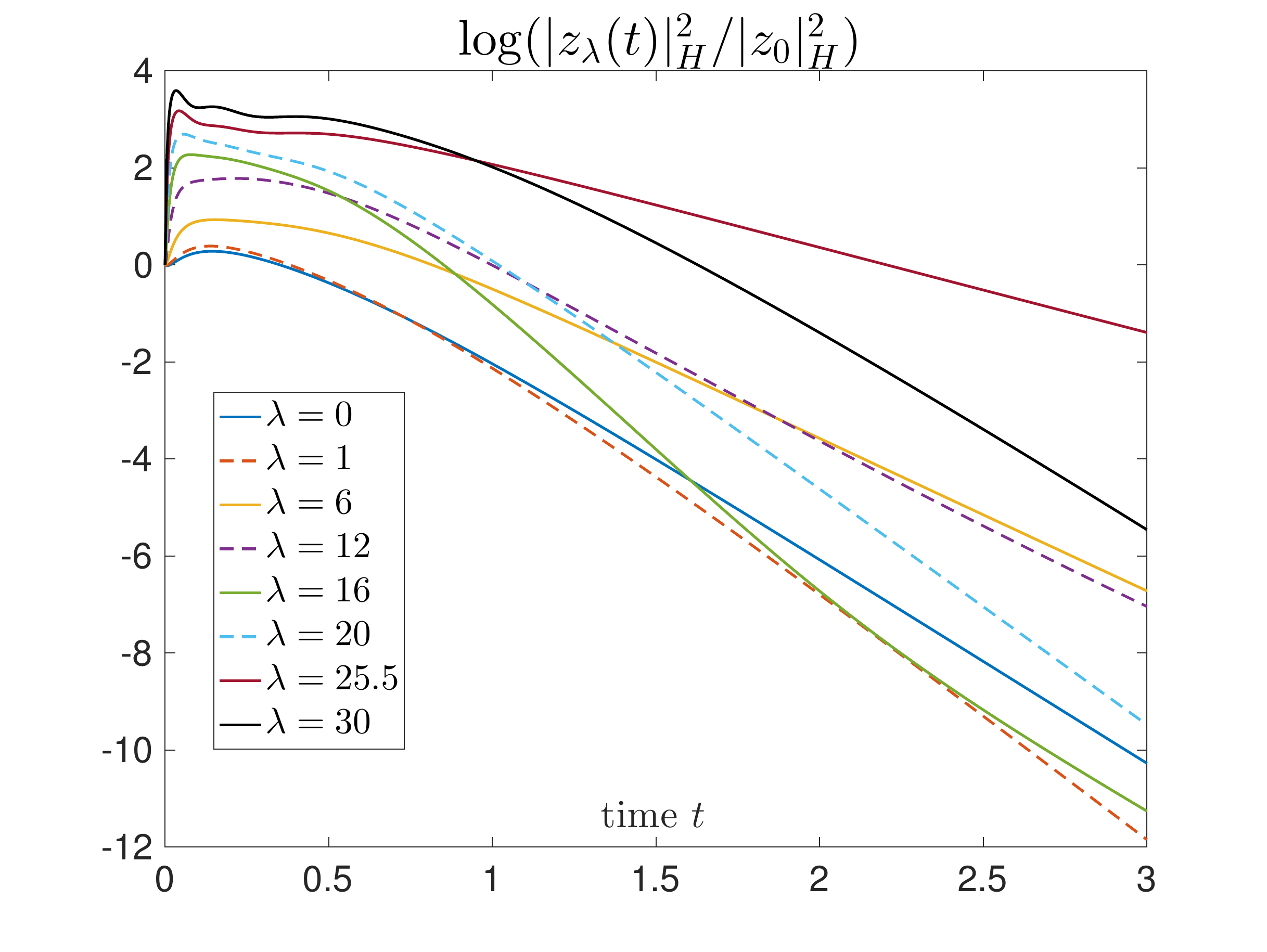}}
\subfigure[Without control.]
{\includegraphics[width=.495\linewidth]{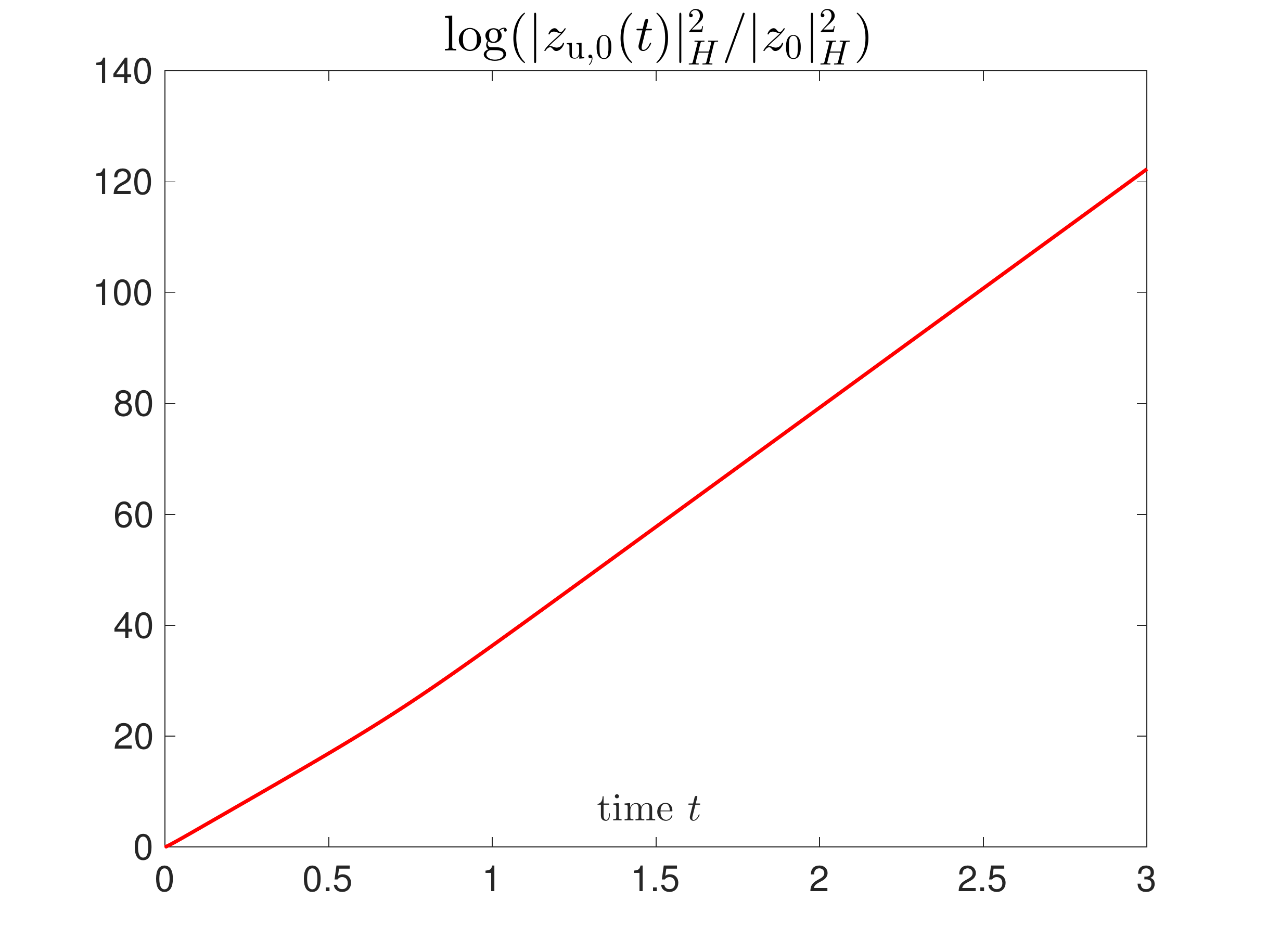}}
\caption{Stabilization rate of the solution is achieved under feedback control.}
\label{figure_familyoflambda}
\end{figure}

Let us observe that starting at time $t=s_0<T$, if for all~$t\in[s_0,T]$ we have~$|z(t)|^2_H\le C_{\lambda,s_0}\ex^{-\lambda(t-s_0)}|z(s_0)|^2_H$, then
we obtain the lower bound estimate
\[
C_{\lambda}\ge C_{\lambda,s_0}\ge\max_{t\in[s_0,T]}\ex^{\lambda(t-s_0)}\textstyle\frac{|z(t)|^2_H}{|z(s_0)|^2_H},
\]
and, considering all~$s_0\in[0,T)$ we obtain
\begin{align*}
\log C_{\lambda}\ge m_{\lambda}&\coloneqq \max_{s_0 \ge 0}\max_{t\in[s_0,T]} \left(\lambda(t-s_0) +\log\left(\textstyle\frac{|z(t)|^2_H}{|z(s_0)|^2_H}\right)\right)\\
&=\max_{s_0 \ge 0}\max_{t\in[s_0,T]} \left(r_{\lambda}(t) - r_{\lambda}(s_0)\right),
\end{align*}
with~$r_{\lambda}(t)
\coloneqq\lambda t + \log \left( \frac{|z(t)|^2_H}{|z_0|^2_H} \right)$.
In Figure \ref{figure_behaviorlambda} we plot the function~$m_{\lambda}$, which 
suggests that~\eqref{transalpha} might hold with~$\alpha\in(0,1]$, but we cannot confirm that~\eqref{trans_bd3} will hold for
``big~$\lambda$'', actually we cannot even say whether~$\lambda=30$ is big. Notice that for $\lambda\in[20,30]$ the figure suggests that~$\alpha=1$ is the best fitting parameter.
We must however say that we expect the magnitudes of control
and solution to increase as~$\lambda$ does, so we must be careful in reading the results for big~$\lambda$, because the mesh used was the same for all~$\lambda$, and for
bigger~$\lambda$ the numerical error will be more significant.
\begin{figure}[!]
{\includegraphics[width=.6\linewidth]{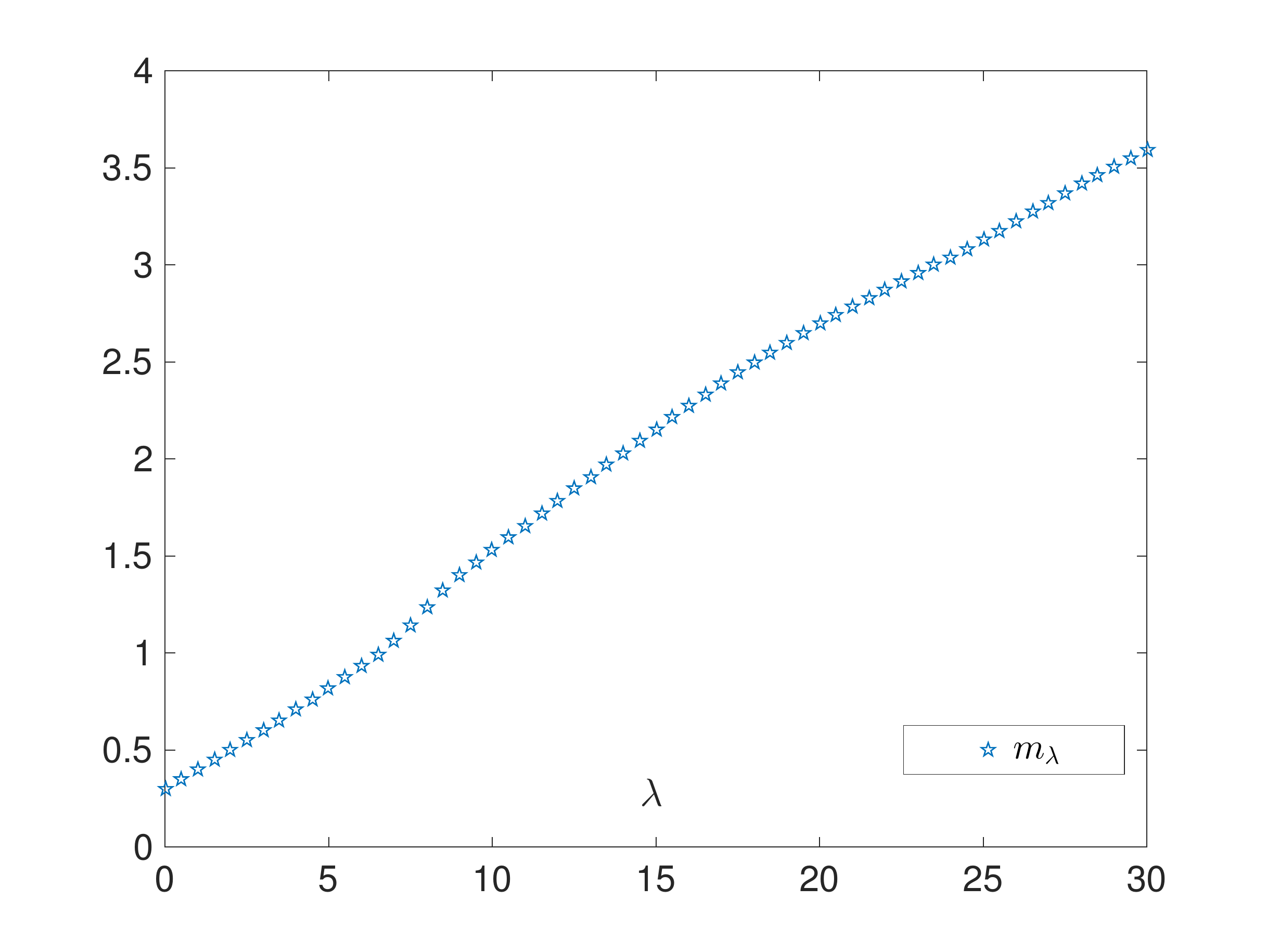}}
\caption{$m_{\lambda}$ with $\lambda \in \left\{0,\frac{1}{2},1,\frac{3}{2},\dots, 29,29 \frac{1}{2}, 30 \right\}$.} 
\label{figure_behaviorlambda}
\end{figure}

We present another example in the nonautonomous case. Notice that, this case is quite expensive, because we need to solve a differential Riccati equation for each~$\lambda$.
Here, we set $\nu = \frac{1}{4}$, $\hat a = -10 + 2x_1 + \sin 2t  \cos x_2$, $\hat b = \left(-x_1^3 \cos 3t ,-\sin t \sin x_2 \right)$,
and we impose no restriction on number of actuators. Figures~\ref{figure_familyoflambdaNS} and~\ref{figure_behaviorlambda2} show
that the system can be stabilized with rate $\lambda \in [0,40]$ and the constant $C_{\lambda}$ increases with $\lambda$. Again the Riccati feedback we use seems not to provide 
the behaviour~\eqref{trans_bd3} ``for big~$\lambda$''. So one question still remain: how small can we make the transient bound $C_{\lambda}$ and/or 
the ratio $\frac{C_{\lambda}}{\lambda}$ (cf.~Section~\ref{sS:transient}). Which feedback (Riccati based or not) makes the ratio smaller? 

\begin{figure}[]  
  \centering
\subfigure[With internal feedback control.]
{\includegraphics[width=.495\linewidth]{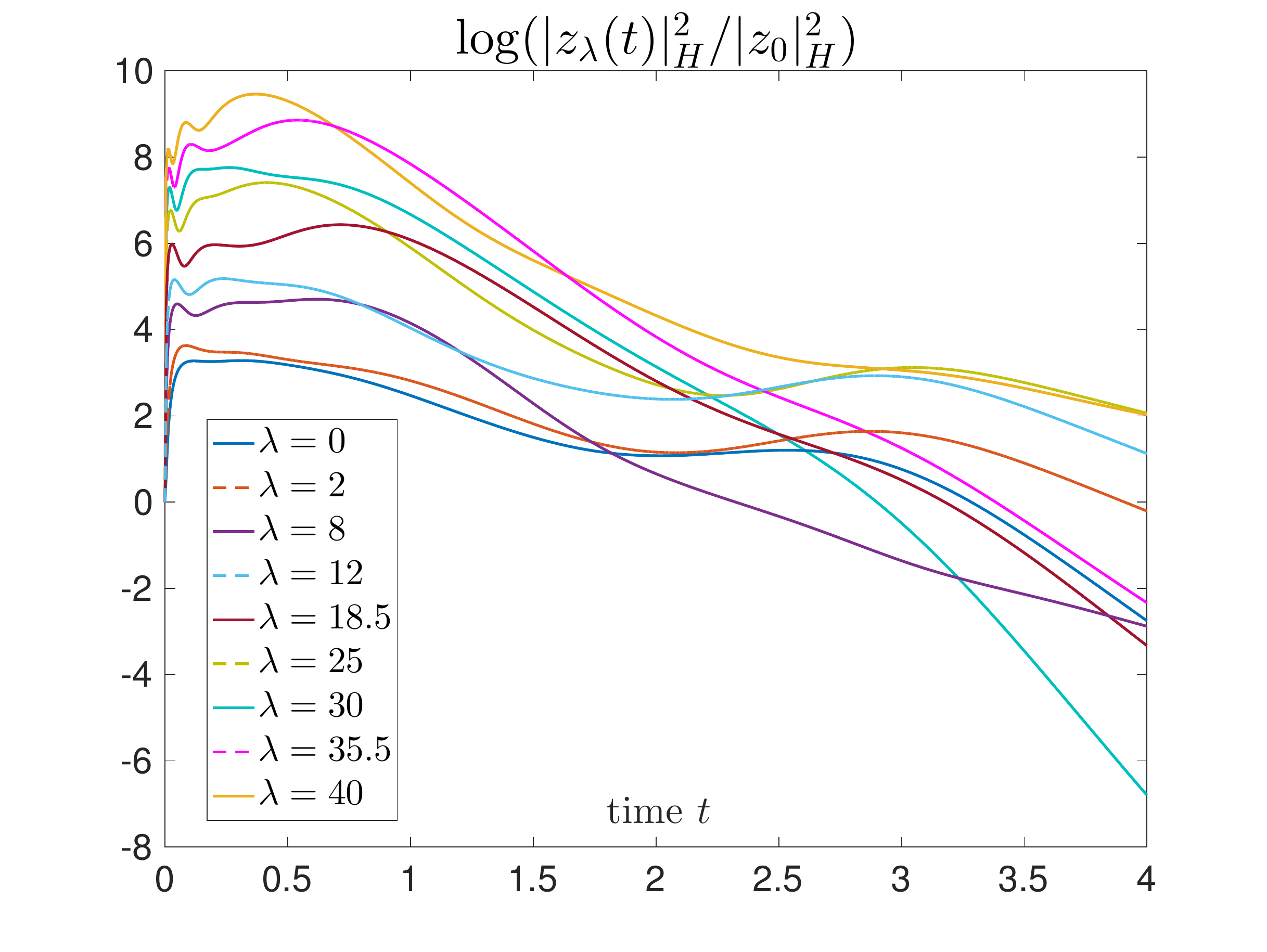}}
\subfigure[Without control.]
{\includegraphics[width=.495\linewidth]{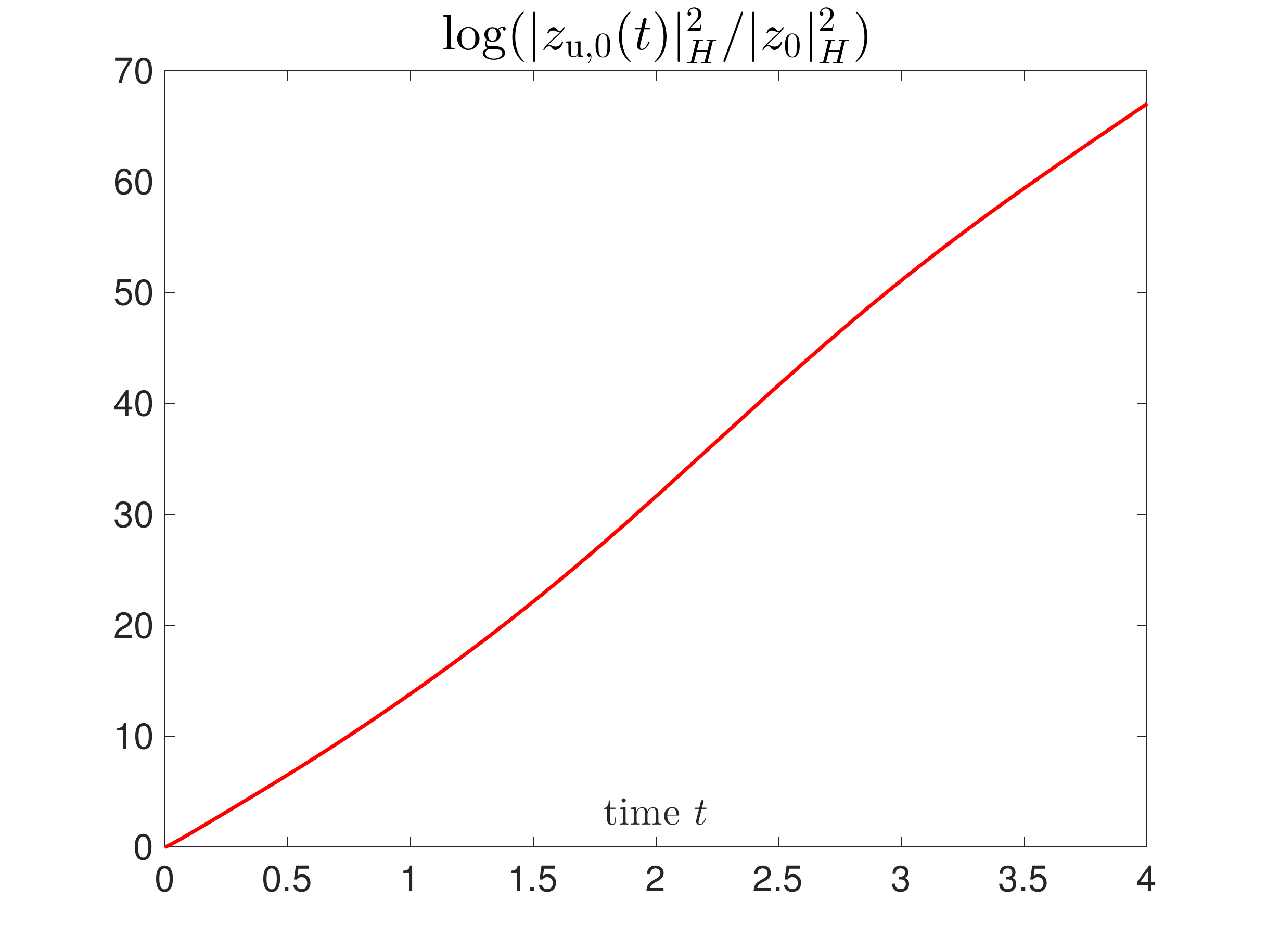}}
\caption{Stabilization rate of the solution is achieved under feedback control.}
\label{figure_familyoflambdaNS}
\end{figure}

\begin{figure}[!]
{\includegraphics[width=.6\linewidth]{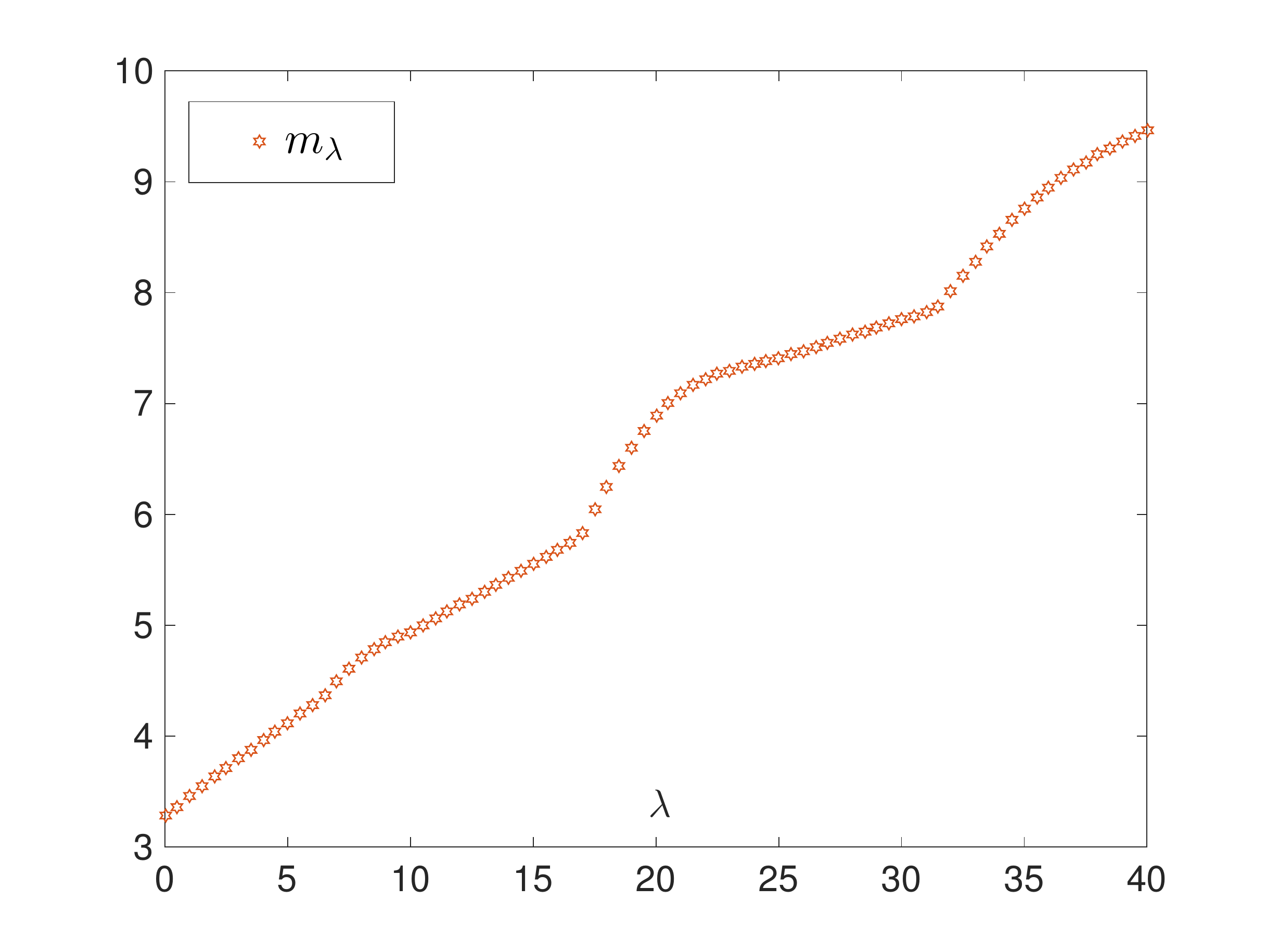}}
\caption{$m_{\lambda}$ with $\lambda \in \left\{0,\frac{1}{2},1,\frac{3}{2},\dots, 39,39 \frac{1}{2}, 40 \right\}$.} 
\label{figure_behaviorlambda2}
\end{figure}

In conclusion, the dependence of the transient bound $\widehat C_{\lambda}$,
associated to the Riccati based feedback control, 
on the desired exponential rate~$\lambda\ge0$ is not clear yet. This could be the subject of further researh and more simulations must be done, in particular for finer meshes.
This is however not trivial, because solving the Riccati equations is numerically quite demanding for fine meshes. However, the simulations strongly suggest that
(unfortunately) $\widehat C_{\lambda}$ depends exponentially-like on~$\lambda$, which suggests us that to treat the nonlinear system we should take $\lambda$ relatively small,
because~$\epsilon$ in~\eqref{goal} will decrease as~$\widehat C_{\lambda}$ increases (cf.Remark~\ref{R:trans_local}).

\subsection{On the parameter~$\varsigma$}
Here we consider the case of boundary controls. With $\nu$ and~$(\hat a,\hat b)$ as in Section~\ref{sS:RateLambda},
we try to understand the influence of the choice of the parameter~$\varsigma$ in the extended system~\eqref{PSBoundary}.
Notice that~$\varsigma$ does not appear in the ``original'' system~\eqref{PSBoundary-z}, it is a parameter which we can choose in the dynamics of the extension variable~$\kappa$
in~\eqref{PSBoundary-k}.

We will take 6 boundary actuators and fix $\lambda = 2$. 
Figure~\ref{figure_ParameterVarsigma} suggests that taking a bigger parameter $\varsigma$, we obtain a boundary feedback control providing a faster stabilization of the linear system
to zero. However, we must say that taking a bigger parameter~$\varsigma$, the free dynamics solution of~\eqref{PSBoundary-k} (without the feedback control)
will go faster to zero, which suggests that for bigger~$\varsigma$ we may need to take smaller time step~$k$ in the discretization~\eqref{SBoundary-sD} to get a good
approximation of the real dynamics, which will make the simulations more expensive.  

\begin{figure}[]
\centering
{\includegraphics[width=.6\linewidth]{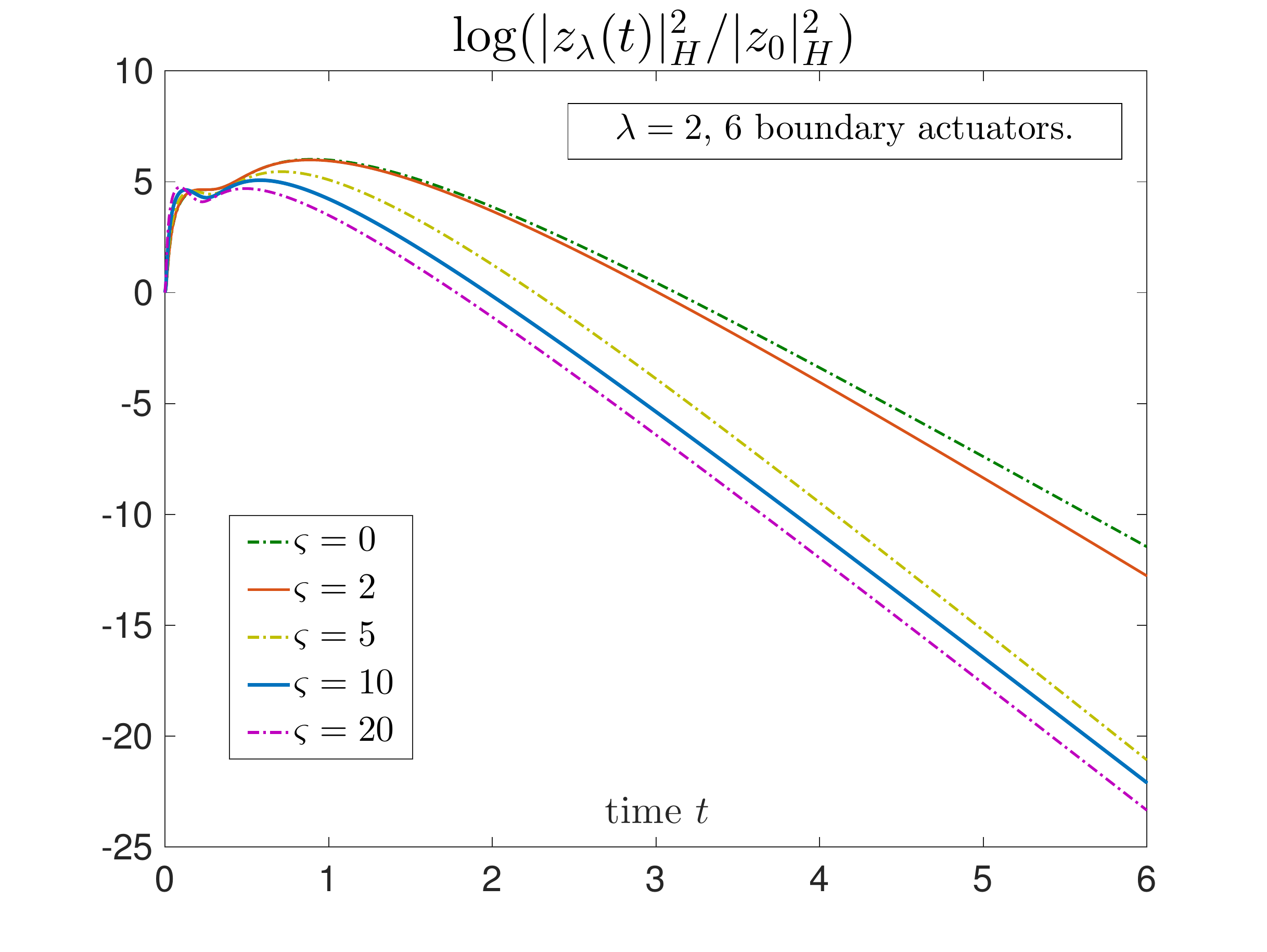}}
\caption{Increasing the parameter $\varsigma$.}
\label{figure_ParameterVarsigma}
\end{figure}

\subsection{On the feedback nature of the control} 
\label{RealFeedback}
Here, we remark the importance to have the control in a feedback form, because closed loop controls are known to be able to respond to small disturbances.
We consider the stationary
case with $\nu = \frac{1}{2}$, $\hat a = -10 + 2x_1 + \cos x_2$, $\hat b = \left(-x_1^2 ,-\sin x_2 \right)$ as in Section~\ref{sS:RateLambda} and with
$\lambda = 2$. 
However, here we consider 4 boundary actuators and set $\zeta = 10$. The initial condition here is $v_0(x)=3+\sum\limits_{i=1}^4\frac{1}{2}\widetilde\Psi_i(x)$

In Figure~\ref{figure_CopyandPasteFeedback}(a), we plot the results:
\begin{itemize}
\item  for the case of the
``original'' close-looped system~\eqref{PSBoundary},
\item for the case of the corresponding extended close-loop system~\eqref{sys-ext-Feed}
(with~$(\lambda,\varsigma)$ in the role of~$(\bar\lambda,\bar\varsigma)$ and with~$\mathbb A^{\hat a,\hat b,\nu}_{\lambda,\varsigma}$ as in~\eqref{A-Ric-orig}
in the role of~$\mathbb A^{a,b}_{\bar\lambda,\bar\varsigma}$),
and then recovering~$z=y_{\lambda}+B_\Psi \kappa_{\lambda}$, and
\item for the case where we save the~$\kappa_{\lambda}$ obtained for system~\eqref{sys-ext-Feed}
and plug it in the system~\eqref{PSBoundary-z} as an open-loop control.
\end{itemize}
Notice that, we do not expect that the numerical solutions obtained by solving the original closed-loop system
will coincide with the recovering from the solution of the extended closed-loop system (as Figure~\ref{figure_CopyandPasteFeedback}(a) could suggest),
but we do expect that they will be close to each other, this facts are confirmed in Figure~\ref{figure_CopyandPasteFeedback}(b). 

For a short period of time we do not see much difference among the three procedures. For a longer period of time, the closed-loop controls  
are able to stabilize the corresponding systems~\eqref{PSBoundary} and~\eqref{sys-ext-Feed} while, the open-loop control
is not able to stabilize the original one~\eqref{PSBoundary-z}, even though it stabilizes the ``equivalent'' extended system. This is due to the fact that the discretizations errors in the original and extended systems are different, the difference
between these two errors could be seen as a small disturbance to which the open-loop control cannot respond.

The conclusion is that it is important to have a control in feedback form in the original system~\eqref{PSBoundary}. It is not enough (at the discrete level) to find the control
for an ``equivalent'' auxiliary system, like~\eqref{sys-ext-Feed}, and plug it in~\eqref{PSBoundary-z}. 
\begin{figure}[]  
  \centering
\subfigure[Solutions from the three procedures.]
{\includegraphics[width=.495\linewidth]{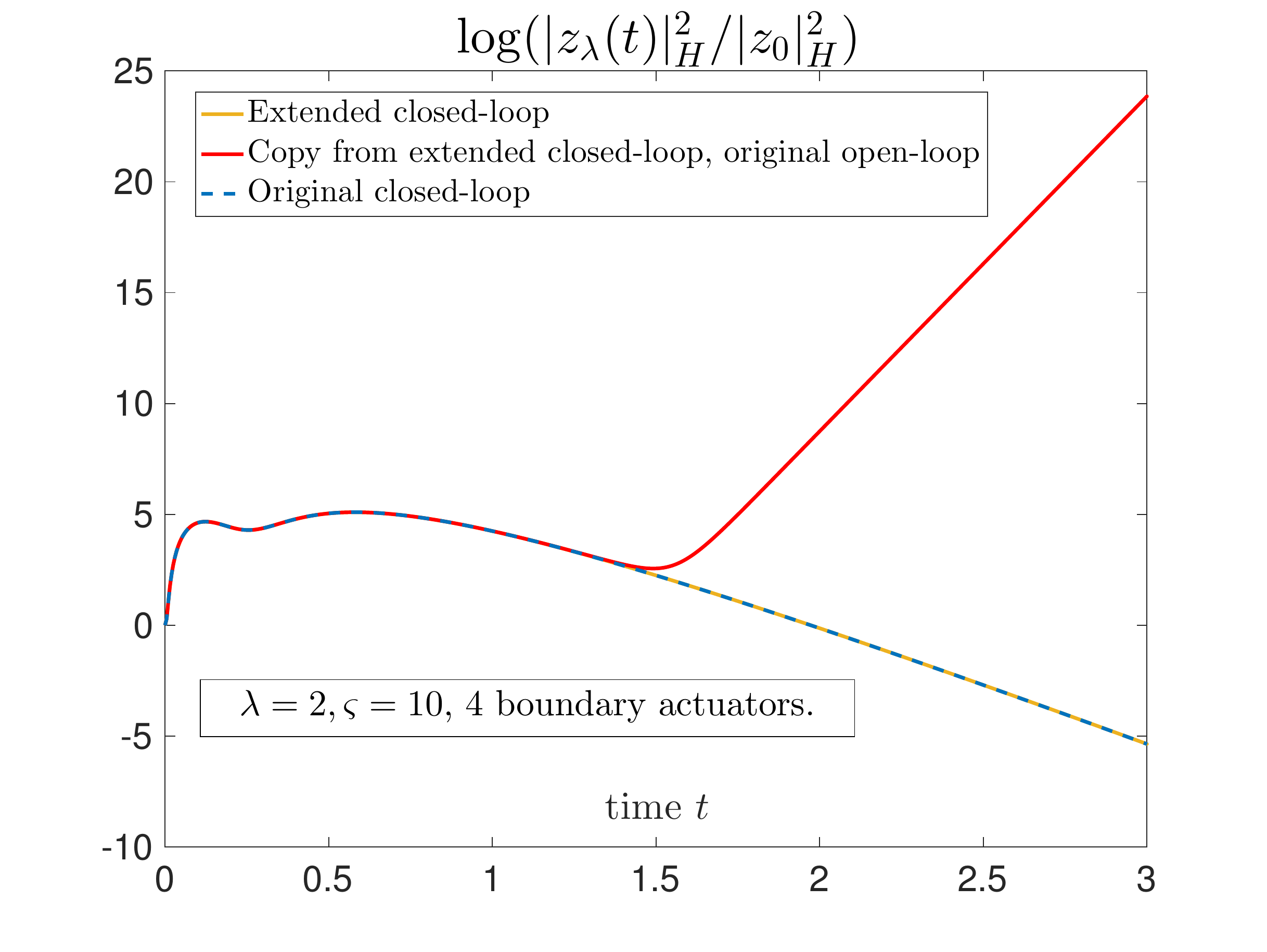}}
\subfigure[The difference between solutions from extended and original closed-loop.]
{\includegraphics[width=.495\linewidth]{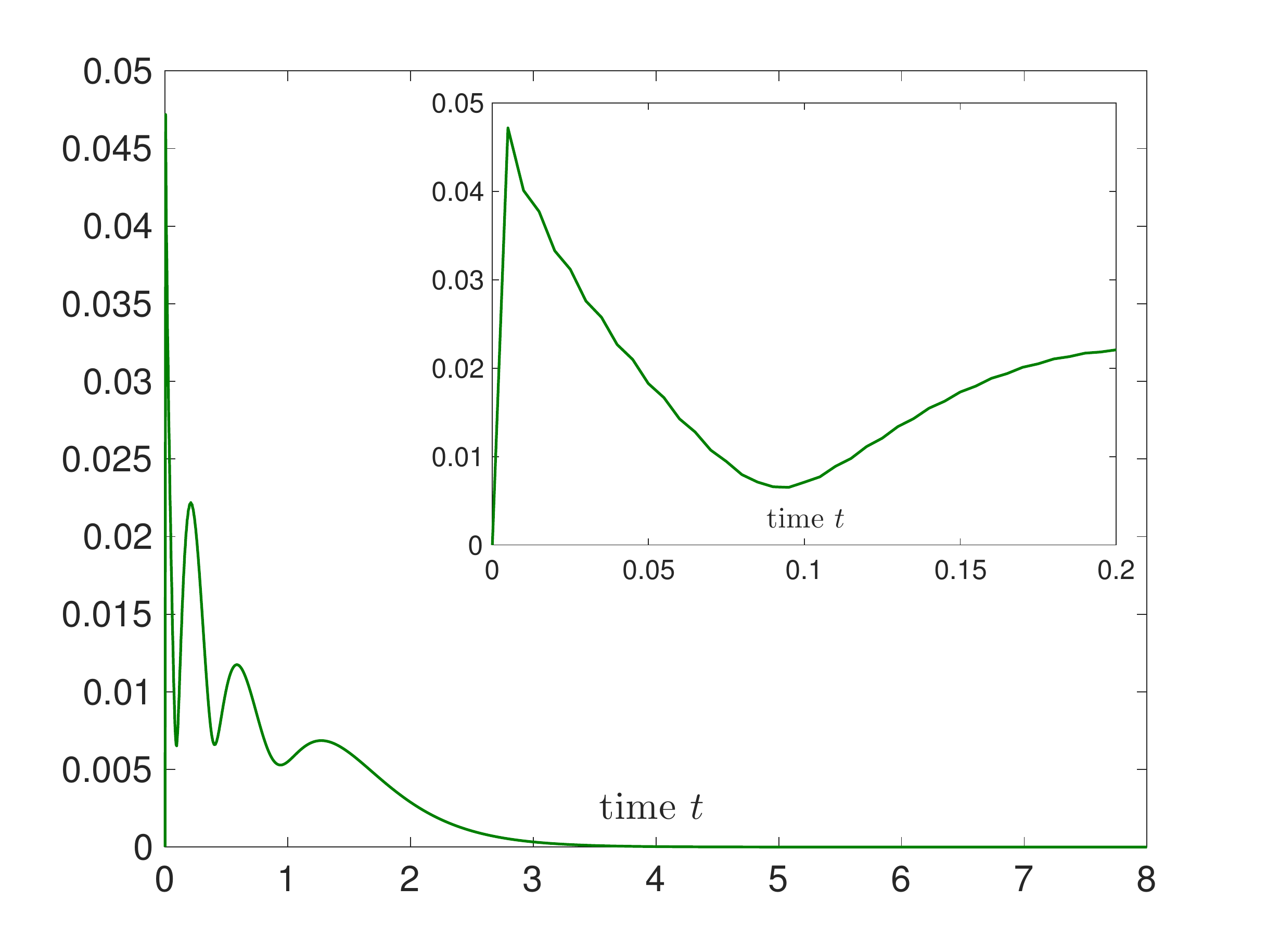}}
\caption{The ``real'' feedback to stabilize the system.}
\label{figure_CopyandPasteFeedback}
\end{figure}

\subsection{Switching the control off/on}
With the same setting as in Section \ref{RealFeedback}, here we perform an experiment to emphasize
the importance of feedback control to stabilize the system. We set $T = 6$.

In Figure~\ref{figure_SwitchOnSwitchOff} we plot the results for the 3 cases the control is switched on for time in~$[0,6]$, for time in~$[0,~3]\cup [4,~6]$, and for time in~$[0,3]$.

Switching the control off at time instant $t = 3$, we observe that the norm of the solution increases, and will (likely) not remain bounded.
Notice that at time~$t=3$ we have $|z(3)|^2_H \le \mathrm{e}^{-10-3\lambda}|z_0|^2_H < 1.2 \times 10^{-7} |z_0|^2_H$ is already quite small.
We also see that by switching on the control again, from time~$t=4$ on, then we recover the stability of the system.

Notice that at time~$t=3$ and~$t=4$, that is, we do not know the ``analytical'' expressions for $z(3)$ and~$z(4)$ which
come from the numerical solution. Therefore, in some sense, $z(3)$ and~$z(4)$ can be seen as ``random'' ``initial'' conditions showing the instability of the system and the
stabilizing property of the feedback control.
\begin{figure}[]
\centering
{\includegraphics[width=.6\linewidth]{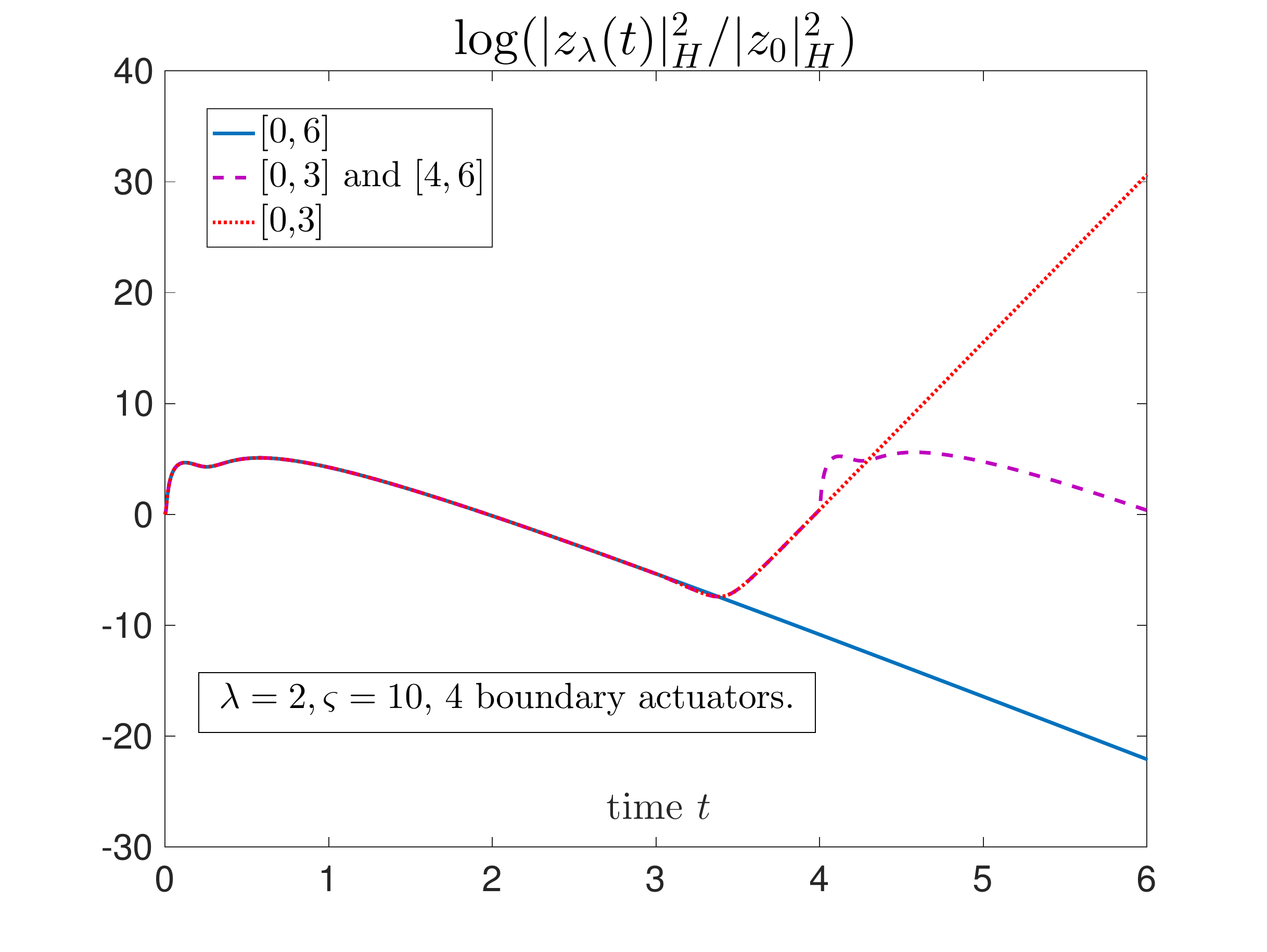}}
\caption{Switch on and switch off the feedback control.}
\label{figure_SwitchOnSwitchOff}
\end{figure}

\subsection{A nonlinear example}\label{sS:nonDexample}
We consider the following nonlinear parabolic equation, for time $t \in \left[0, T \right]$, again in the unit ball~$\Omega=\mathbb D=\{x=(x_1,x_2)\in\R^2\mid x_1^2+x_2^2<1\}$. 
\begin{equation}
 \begin{aligned}
\label{NonlinearSys}
\partial_t y - \nu \Delta y + c_3 y^3 + c_2 y^2 + c_1y + \frac{1}{2} \nabla \cdot \left(y^2, y^2 \right) + f_0 &= 0,\qquad y\rest\Gamma = g,\\
y(0,x) &= y_0(x), 
\end{aligned}
\end{equation}
where~$c_1$, $c_2$, and~$c_3$ are constants in~$\R$, and $f_0$ is a fixed appropriate function. 

Let us fix a smooth function~$\hat y$ which we will take as our reference trajectory. Then, as external forces, we (must) take the functions
\begin{equation}\label{exact_f0g}
 \begin{aligned}
f_0=f_0(\hat y) &= -\left(\partial_t \hat y - \nu \Delta \hat y + c_3 \hat y^3 + c_2 \hat y^2 + c_1\hat y + \frac{1}{2} \nabla \cdot \left(\hat y^2, \hat y^2 \right)\right),\\
g=g(\hat y)&=\hat y\rest\Gamma,
\end{aligned}
\end{equation}
We will also set the parameters
\[
\nu=0.2,\quad (c_1,\,c_2,c_3)=(-2,-1,-3),\quad\lambda=1.
\]

\begin{remark}
 In order to make a smooth function~$\hat{y}(t,x)$ a solution of~\eqref{NonlinearSys}, we have just to set the appropriate external forces~$f_0$ and
 $g$ as in~\eqref{exact_f0g}. Thus, to test with other reference trajectories, we would just have to take the corresponding~$f_0$ and~$g$.   
\end{remark}

\subsubsection{Discretization} 
Given a triangular mesh~$\DDD=(\ppp,\eee,\ttt)$ of~$\mathbb D$ and a time step~$k>0$, we discretize system~\eqref{NonlinearSys} as if it were the heat equation,
simply looking at the nonlinearity as an external forcing, that is, writing
\begin{align*}
f_1(y,\nabla y)&=c_3 y^3 + c_2 y^2 + c_1y + \frac{1}{2} \nabla \cdot \left(y^2, y^2 \right),\\
f=f(y,\nabla y)&=f_1(y,\nabla y)+ f_0,
\end{align*}
since~$f_0$ and~$g$ are smooth we can take at each time step $jk$, as approximations of~$f_0(jk)$ and~$g(jk)$ the corresponding evaluation
vectors~$\overline{f_0(jk)}$ and~$\overline{g(jk)}$ at the points~$\ppp$ of~$\DDD$.

For the nonlinear terms, at each time step, we take the aproximations
\begin{equation*}
(c_r y^r)^j\coloneqq c_r\Ma (\overline y^j)^r\approx c_r y^r(jk),
\qquad \textstyle\frac{1}{2}\left(\GGG_{x_1}(\overline y^j)^2+\GGG_{x_2}(\overline y^j)^2\right)\approx\frac{1}{2} \nabla \cdot \left(y^2(jk), y^2(jk) \right),
\end{equation*}
where $(\overline y^j)^r$ means that we take the $r$-th power of each coordinate of~$\overline y^j$.  
Hence, given vector solution~$\overline y^j$, at time~$jk$, we approximate the nonlinearity, as a function from~$V$ into~$V'$, as
\begin{align}\label{NN_D}
 \NN_{1,D}(\overline y^j)&\coloneqq \Ma\left(c_3 (\overline y^j)^3 +c_2 (\overline y^j)^2 +c_1\overline y^j\right)
 + \textstyle\frac{1}{2}\left(\GGG_{x_1}(\overline y^j)^2+\GGG_{x_2}(\overline y^j)^2\right).
 \end{align}

We arrive to a semi-discretization analogously to~\eqref{heat-sys-sD}:
\begin{equation*}
\partial_t\Ma_{\rm ii} \overline y_{\rm i}
= -\nu \St_{\rm ii} \overline y_{\rm i} 
-\nu\St_{\rm ib}\overline g
-\partial_t\Ma_{\rm ib} \overline g
-\begin{bmatrix}\Ma_{\rm ii} &\Ma_{\rm ib}\end{bmatrix}\overline f_0 -\left(\NN_{1,D}(\overline y) \right)_{\rm i}.
\end{equation*}
where~$(\NN_{1,D}(\overline y))_{\rm i}$ stands for the coordinates of the vector $\NN_{1,D}(\overline y)$ corresponding to the interior points of the mesh.
As before~$\overline y=\begin{bmatrix}\overline y_{\rm i}\\ \overline y_{\rm b}\end{bmatrix}=\begin{bmatrix}\overline y_{\rm i}\\ \overline g\end{bmatrix}$.

Then, with Crank-Nicolson scheme, and with the notations as in~\eqref{SBoundary-Dy}, we find
\begin{equation}\label{SNonlin-CN}
\begin{split}
\quad\mathbf A_{\rm ii}^\oplus\overline y_{\rm i}^{j+1}
&= \mathbf A_{\rm ii}^\ominus\overline y_{\rm i}^j -\mathbf A_{\rm ib}^\oplus \overline g^{j+1}
+\mathbf A_{\rm ib}^\ominus\overline g^{j}-k\begin{bmatrix}\Ma_{\rm ii} &\Ma_{\rm ib}\end{bmatrix}
\left(\overline f_0^{j+1}+\overline f_0^{j}\right)\\
&\quad-k\left( \left(\NN_{1,D}(\overline y^{j+1})\right)_{\rm i}+\left(\NN_{1,D}(\overline y^j)\right)_{\rm i}\right).
\end{split}
\end{equation}
Knowing~$y_{\rm i}^j$, the only unknown term in the right hand side is~$(\NN_{1,D}(\overline y^{j+1}))_{\rm i}$. To approximate this term, we take again the linear extrapolation
\begin{equation}\label{SNonlin-ourext}
 (\NN_{1,D}(\overline y^{j+1}))_{\rm i}=2(\NN_{1,D}(\overline y^{j}))_{\rm i}-(\NN_{1,D}(\overline y^{j-1}))_{\rm i},\quad j\ge0,
 \quad\mbox{with}\quad(\NN_{1,D}(\overline y^{-1}))_{\rm i}\coloneqq(\NN_{1,D}(\overline y^{0}))_{\rm i}.
\end{equation}
Therefore we arrive to the scheme
\begin{equation}\label{SNonlin-D}
\begin{split}
\quad\mathbf A_{\rm ii}^\oplus\overline y_{\rm i}^{j+1}
&= \mathbf A_{\rm ii}^\ominus\overline y_{\rm i}^j -\mathbf A_{\rm ib}^\oplus \overline g^{j+1}
+\mathbf A_{\rm ib}^\ominus\overline g^{j}-k\begin{bmatrix}\Ma_{\rm ii} &\Ma_{\rm ib}\end{bmatrix}
\left(\overline f_0^{j+1}+\overline f_0^{j}\right)\\
&\quad-k\left(3(\NN_{1,D}(\overline y^{j}))_{\rm i}-\NN_{1,D}(\overline y^{j-1}))_{\rm i}\right),
\end{split}
\end{equation}
which we can invert to obtain~$\overline y_{\rm i}^{j+1}$.

\subsubsection{Local feedback stabilization} With the setting as in~\eqref{sS:nonDexample}, we take~$T=8$ and the reference trajectory
\[
\hat{y}(t) =(2x_1^3 + x_2^2) \sin t,  
\]
which solves~\eqref{NonlinearSys}, provided we take the ``fixed'' external forces as in~\eqref{exact_f0g}.

We will confirm that the feedback control is able to stabilize locally system~\eqref{NonlinearSys} to the targetted trajectory~$\hat y$
(see Theorems~\ref{T:st-traj} and~\ref{T:st-traj-bdry}) with exponential rate~$\frac{\lambda}{2}$.
That is, the solutions of the systems
\begin{equation}
 \begin{aligned}
\label{NonlinearSys-feed_int}
\partial_t y - \nu \Delta y + f_1(y,\nabla y) + f_0 +B_MB_M^*\widehat\Pi_{\lambda}(y-\hat y)&= 0,\qquad y\rest\Gamma = g,\\
y(0) &= y_0, 
\end{aligned}
\end{equation}
and
\begin{subequations}\label{NonlinearSys-feed_bdry}
\begin{align}
\partial_t y-\nu\Delta y+ f_1(y,\nabla y) + f_0&=0,&z\rest\Gamma&=B_\Psi^\Gamma\kappa,\label{NonlinearSys-feed_bdry-y}\\
\partial_t\kappa+\varsigma\kappa+\FF^{\rm bo}(y-\hat y-B_\Psi\kappa,\kappa)&=0,&
(y(0),\kappa(0))&=(y_0,\kappa_0).\label{NonlinearSys-feed_bdry-k}
 \end{align}
\end{subequations} 
go exponential to~$\hat y$, with rate~$\frac{\lambda}{2}$,  provided the initial condition~$y_0$ is close enough to~$\hat y_0$.
On the other hand, $y_u$ is the solution of systems \eqref{NonlinearSys-feed_int} and~\eqref{NonlinearSys-feed_bdry} without any feedback control.

The feedback control is found to stabilize, respectively, the linearized systems~\eqref{PSInternal} and~\eqref{PSBoundary}. That is, respectively, by solving the 
differential Riccati equations~\eqref{eq:riccati_nu} and~\eqref{eq:riccati-bdry-traj}, with~$\hat a$ and~$\hat b$ as in~\eqref{abN_lineariz}.

Departing from~\eqref{SNonlin-D}, and proceeding as for the systems~\eqref{SInternal-D} and~\eqref{SBoundary-D}, by taking a suitable linear extrapolations
at time~$t=(j+1)k$,
we arrive to the following discretizations of the systems~\eqref{NonlinearSys-feed_int} and~\eqref{NonlinearSys-feed_bdry}.

In the internal case we arrive to
\begin{equation*}
\begin{split}
\quad\mathbf A_{\rm ii}^\oplus\overline y_{\rm i}^{j+1}
&= \mathbf A_{\rm ii}^\ominus\overline y_{\rm i}^j -\mathbf A_{\rm ib}^\oplus \overline g^{j+1}
+\mathbf A_{\rm ib}^\ominus\overline g^{j}-k\begin{bmatrix}\Ma_{\rm ii} &\Ma_{\rm ib}\end{bmatrix}
\left(\overline f_0^{j+1}+\overline f_0^{j}\right)\\
&\quad-k\left(3(\NN_{1,D}(\overline y^{j}))_{\rm i}-\NN_{1,D}(\overline y^{j-1}))_{\rm i}\right)
-k\Ma_{\rm ii}(3\FF^{j,\rm in}\overline y_{\rm i}^{j}-\FF^{j-1,\rm in}\overline y_{\rm i}^{j-1}).
\end{split}
\end{equation*}
In the boundary case we arrive to
 \begin{align*}
(2+k\varsigma)\kappa^{j+1}
&=(2-k\varsigma)\kappa^{j}-k\left(3\FF^{j,\rm bo}_{\rm b}
\begin{bmatrix}\overline y_{\rm i}^{j}\\ \kappa^{j}\end{bmatrix}
-\FF^{j-1,\rm bo}_{\rm b}
\begin{bmatrix}\overline y_{\rm i}^{j-1}\\ \kappa^{j-1}\end{bmatrix}\right),\\
\quad\mathbf A_{\rm ii}^\oplus\overline y_{\rm i}^{j+1}
&= \mathbf A_{\rm ii}^\ominus\overline y_{\rm i}^j -\mathbf A_{\rm ib}^\oplus \left(\overline g^{j+1}+B_{\overline\Psi}^\Gamma\kappa^{j+1}\right)
+\mathbf A_{\rm ib}^\ominus\left(\overline g^{j}+B_{\overline\Psi}^\Gamma \kappa^{j}\right)\\
&\quad-k\begin{bmatrix}\Ma_{\rm ii} &\Ma_{\rm ib}\end{bmatrix}
\left(\overline f_0^{j+1}+\overline f_0^{j}\right)-k\left(3(\NN_{1,D}(\overline y^{j}))_{\rm i}-\NN_{1,D}(\overline y^{j-1}))_{\rm i}\right),\notag
\end{align*}
where~$\NN_{1,D}(\overline y^{j})$ is as in~\eqref{NN_D} and the feedback rules $\FF^{\rm in}=\FF^{\rm in}_\lambda$ and $\FF^{\rm bo}=\FF^{\rm bo}_\lambda$
computed as
in Section~\ref{ssS:discRiccati}.

\subsubsection*{The case of internal controls} 
For system~\eqref{NonlinearSys-feed_int}, we take the initial condition in the form~$y_0=y_0(x)= \hat{y}_0(x) + \epsilon v_0(x)$, with~$v_0$ chosen as
\[
v_0(x) \coloneqq \begin{cases}  1, &\text{if} \quad x< \frac{-1}{3}, \\ 0,&\text{otherwise.} 
\end{cases} 
\]
We take~$6$ piecewise constant actuators, as in Section~\ref{ssS:intcontrols}, supported in the rectangular subset~$\omega$ corresponding to the arrangement~$(m,n)=(3,2)$.

In Figure \ref{figure_InternalControlNL-feed}(a), with some values of~$\epsilon$ in~$[-0.577,0.495]$, we see that the system is stable under the feedback control action.
However, in Figure \ref{figure_InternalControlNL-feed}(b,c),
feedback control can not stabilize anymore the system for~$\epsilon\in\{-0.578,0.496\}$.  

In Figure  \ref{figure_InternalControlNL-unc},
we see that uncontrolled solution is not stable. Actually, we observe that even for small~$\epsilon$ the solution explodes at some time in $(0,8)$. We would like to recall that
such blowing-up is somehow expected. We refer to~\cite[Section~3]{Ball77}.

\begin{figure}[ht]  
  \centering
\subfigure[With $\epsilon\in{[}-0.577, 0.495{]}$.]
{\includegraphics[width=.325\linewidth]{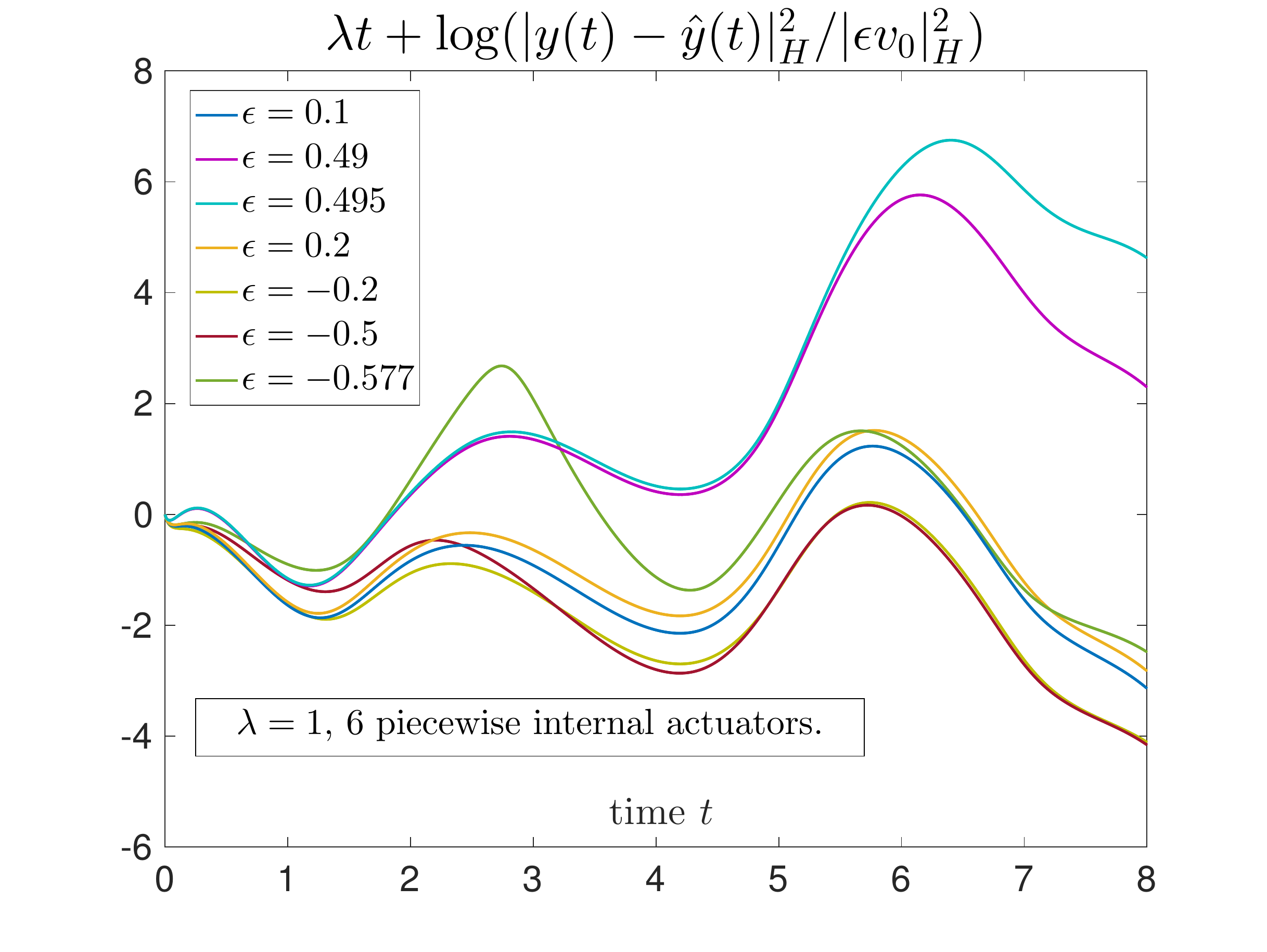}}
\subfigure[With $\epsilon = 0.496$.]
{\includegraphics[width=.325\linewidth]{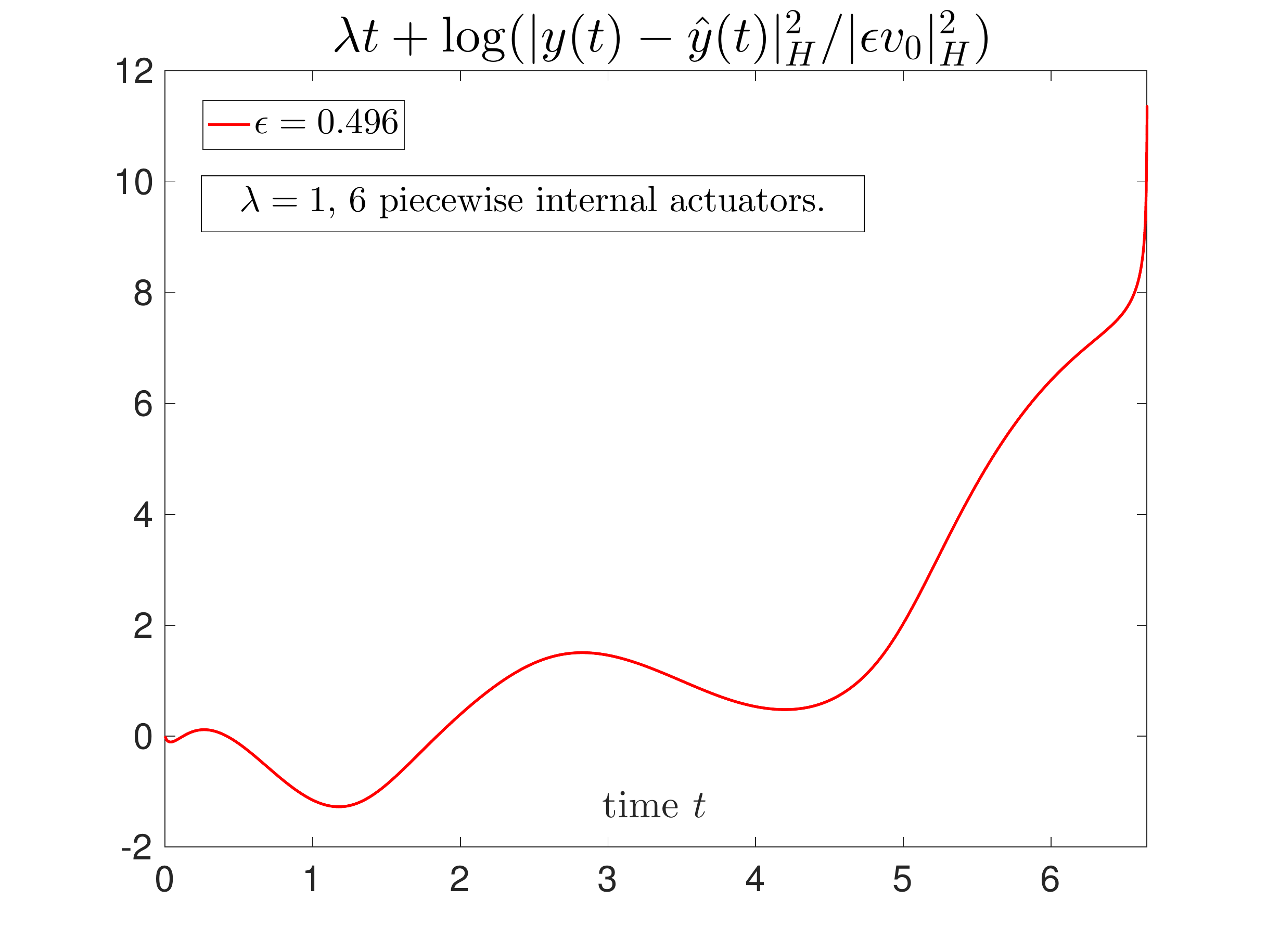}}
\subfigure[With $\epsilon = -0.578$.]
{\includegraphics[width=.325\linewidth]{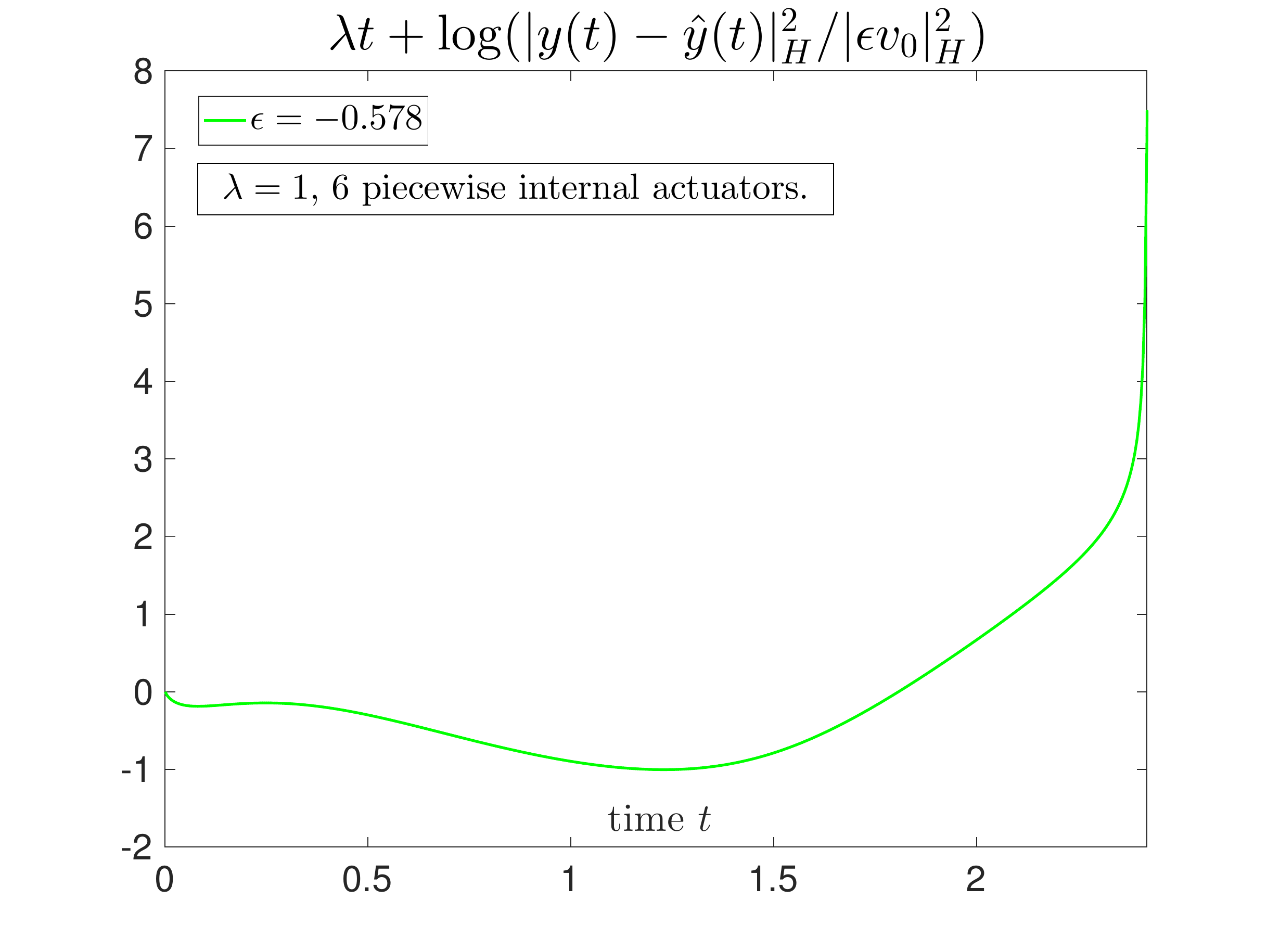}}
\caption{Internal feedback control. Convergence rate to $\hat{y}$ holds for small~$\epsilon$.}
\label{figure_InternalControlNL-feed}
\end{figure}

\begin{figure}[ht]  
  \centering
\subfigure[With $\epsilon = 0.1$.]
{\includegraphics[width=.325\linewidth]{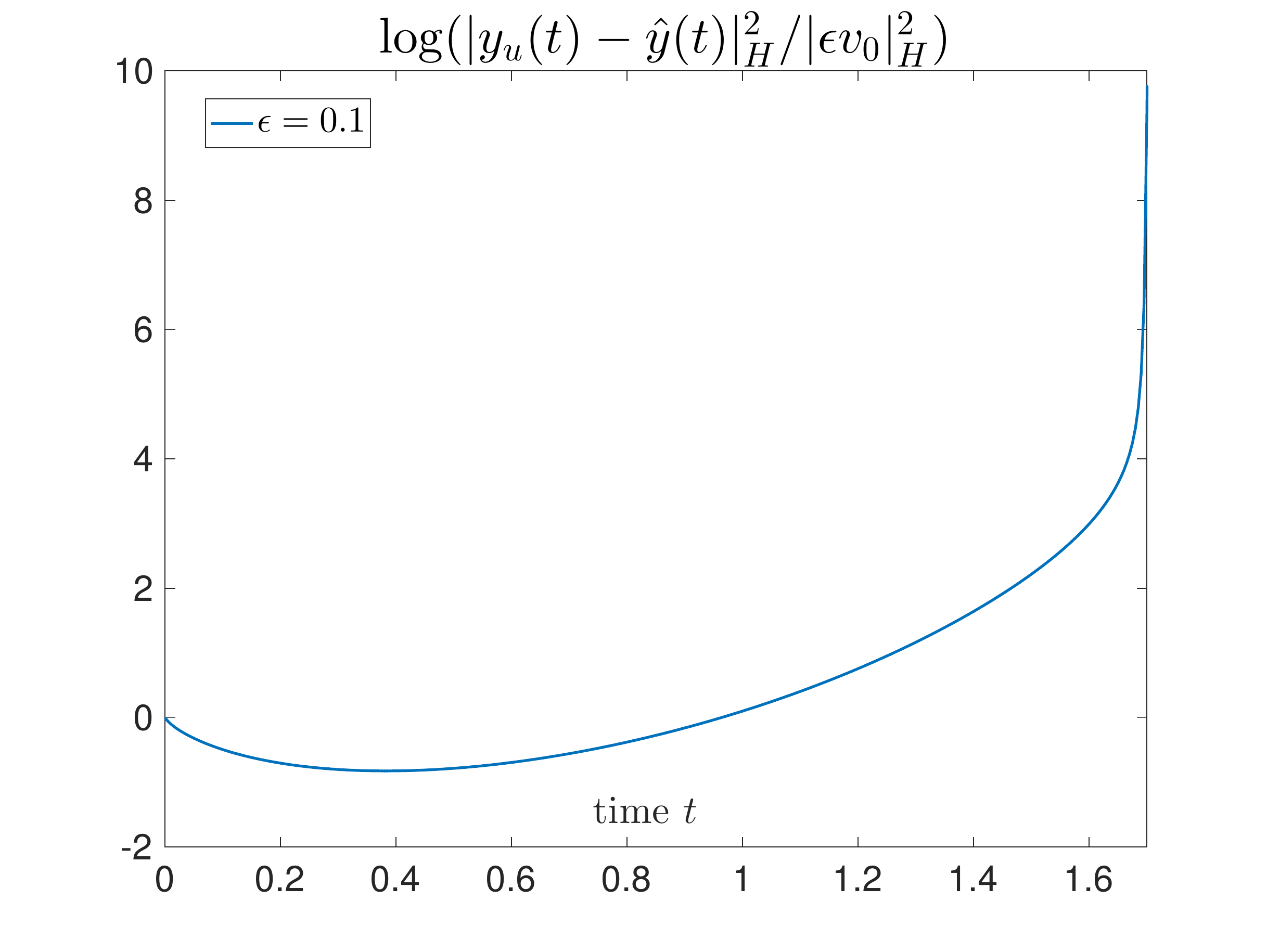}}
\subfigure[With $\epsilon = 0.496$.]
{\includegraphics[width=.325\linewidth]{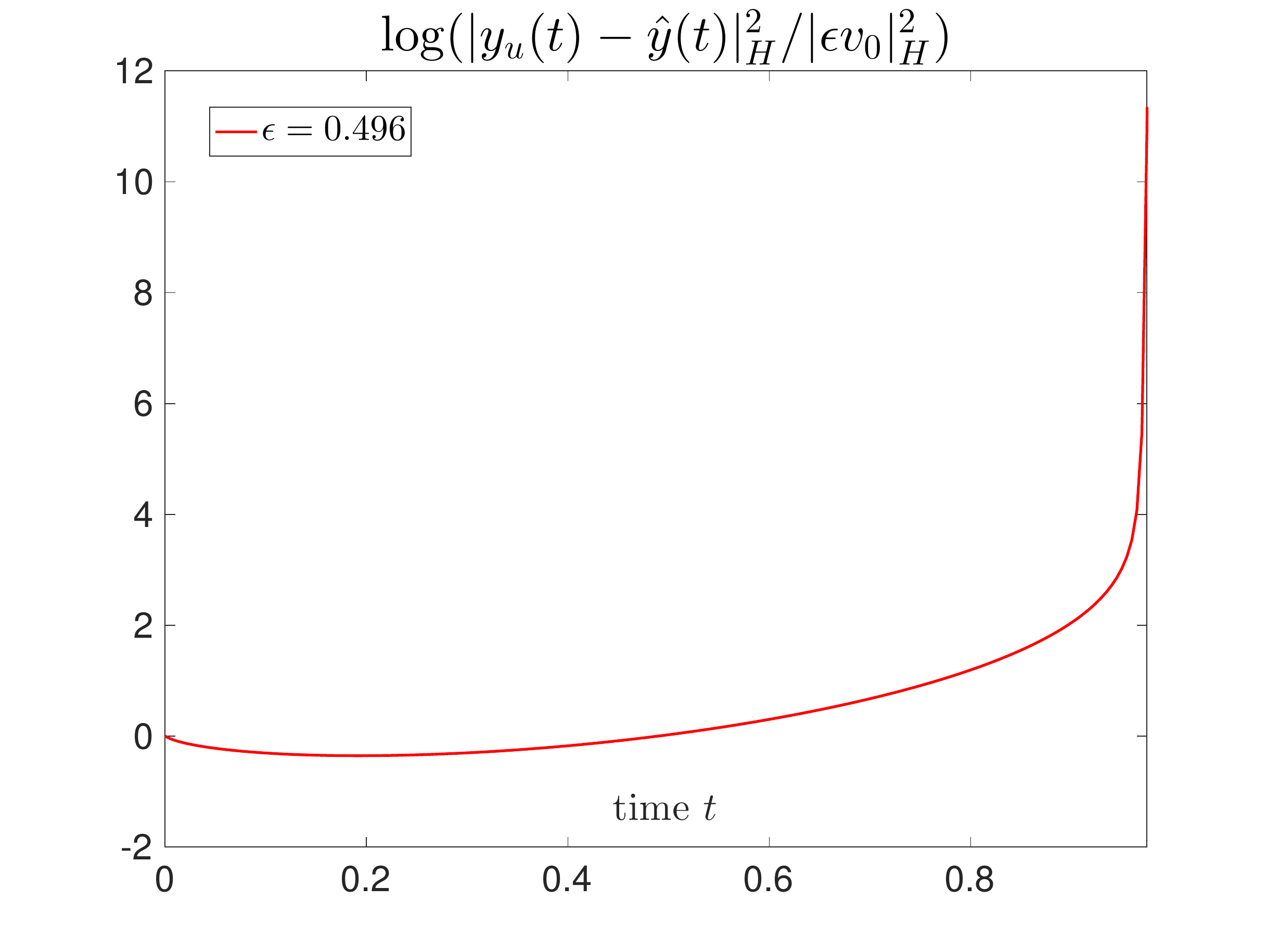}}
\subfigure[With $\epsilon = -0.578$.]
{\includegraphics[width=.325\linewidth]{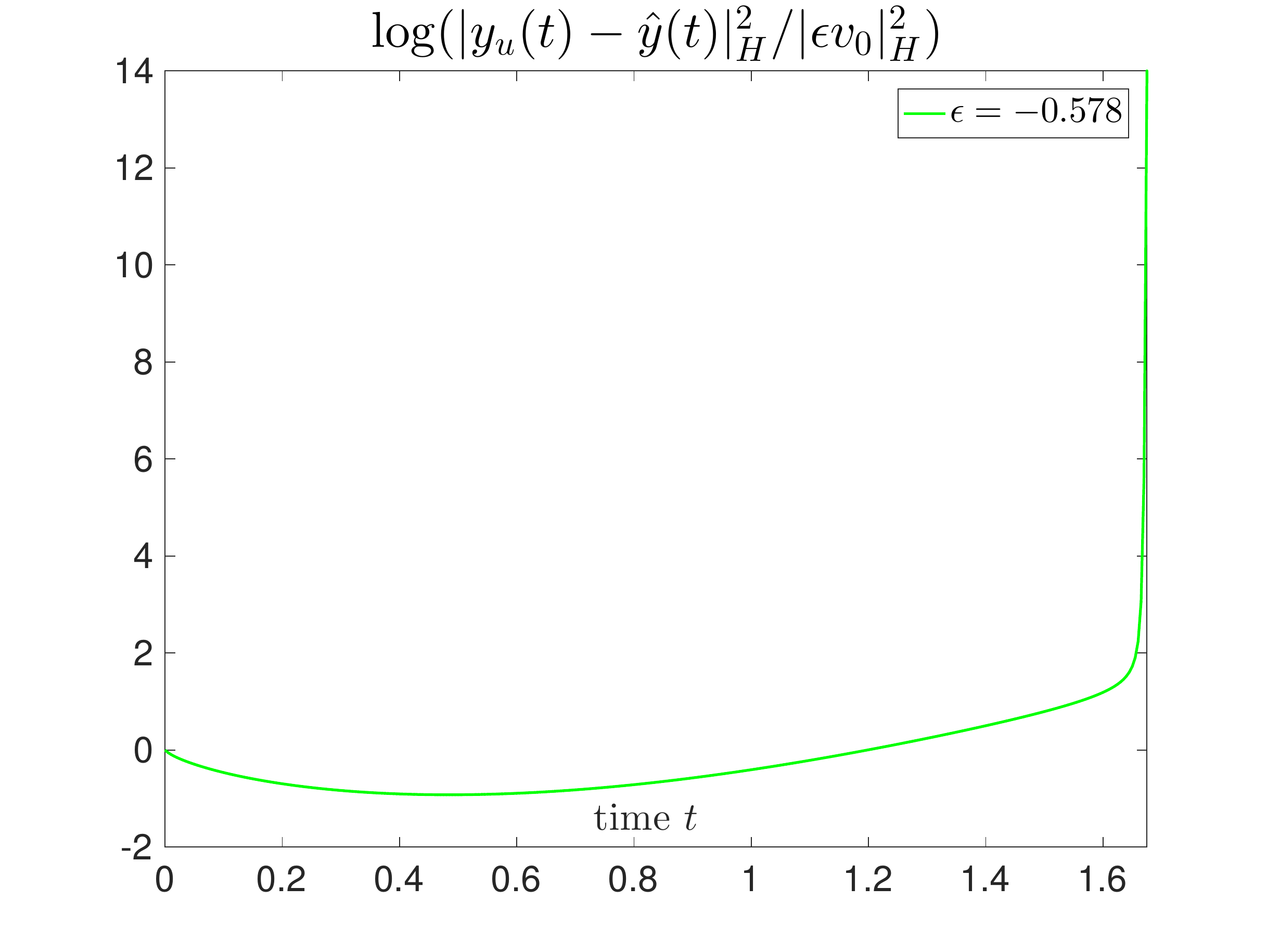}}
\caption{Uncontrolled solution.}
\label{figure_InternalControlNL-unc}
\end{figure}

Notice that, Theorem~\ref{T:st-traj} holds for~$v_0\in V$. The above choice does not satisfy this requirement, though the numerical results show that, in this example,
the feedback is still stabilizing the system locally.
Next, with the same actuators, we perform another simulation with $v_0\in V$ being the (numerical) solution of the elliptic system
\[
-\mu \Delta v_0 + \beta_{\rm r}v_0 + \nabla \cdot \left( \beta_{\rm c} v_0  \right) + h = 0,\quad v_0\rest\Gamma = 0,
\]
observing the analogous behaviour. Here, we choose
\begin{equation}\label{param_nonl}
\begin{split}
\mu = 0.5,\quad \beta_{\rm r}(x_1,x_2) &= \sin(x_1) + x_2,\quad \beta_{\rm c}(x_1,x_2) = \left( 2x_1x_2,~-2 \sin(x_2)  \right),\\
 \qquad \qquad\mbox{and}\quad h(x_1,x_2) &= \cos^2 (3y) + \sin (x) + 2.
\end{split} 
\end{equation}

In Figure \ref{figure_InternalControlNL-CompabilityCondition}(a), we plot the solution~$v_0$ of the equation above.
Again, in Figure \ref{figure_InternalControlNL-CompabilityCondition}, we observe that for small $\epsilon\in{[}-0.14, 0.1511{]}$ the
feedback control is able to stabilize the system.
In Figure \ref{figure_InternalControlNL-unc-CompabilityCondition}, the uncontrolled system is still unstable and exploding.

\begin{figure}[ht]  
  \centering
\subfigure[The function $v_0$.]
{\includegraphics[width=.325\linewidth]{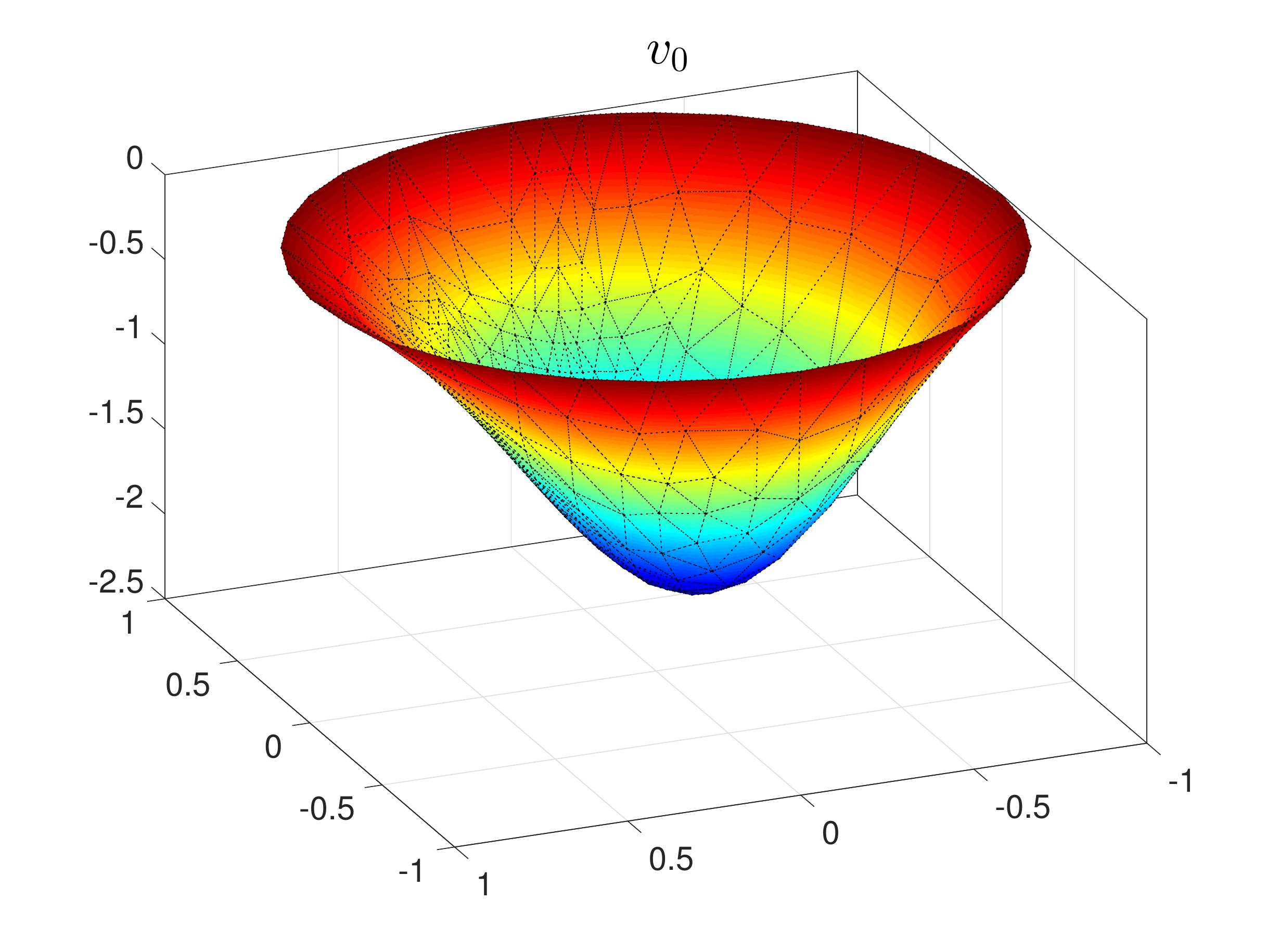}}
\subfigure[With $\epsilon\in{[}-0.14, 0.1511{]}$.]
{\includegraphics[width=.325\linewidth]{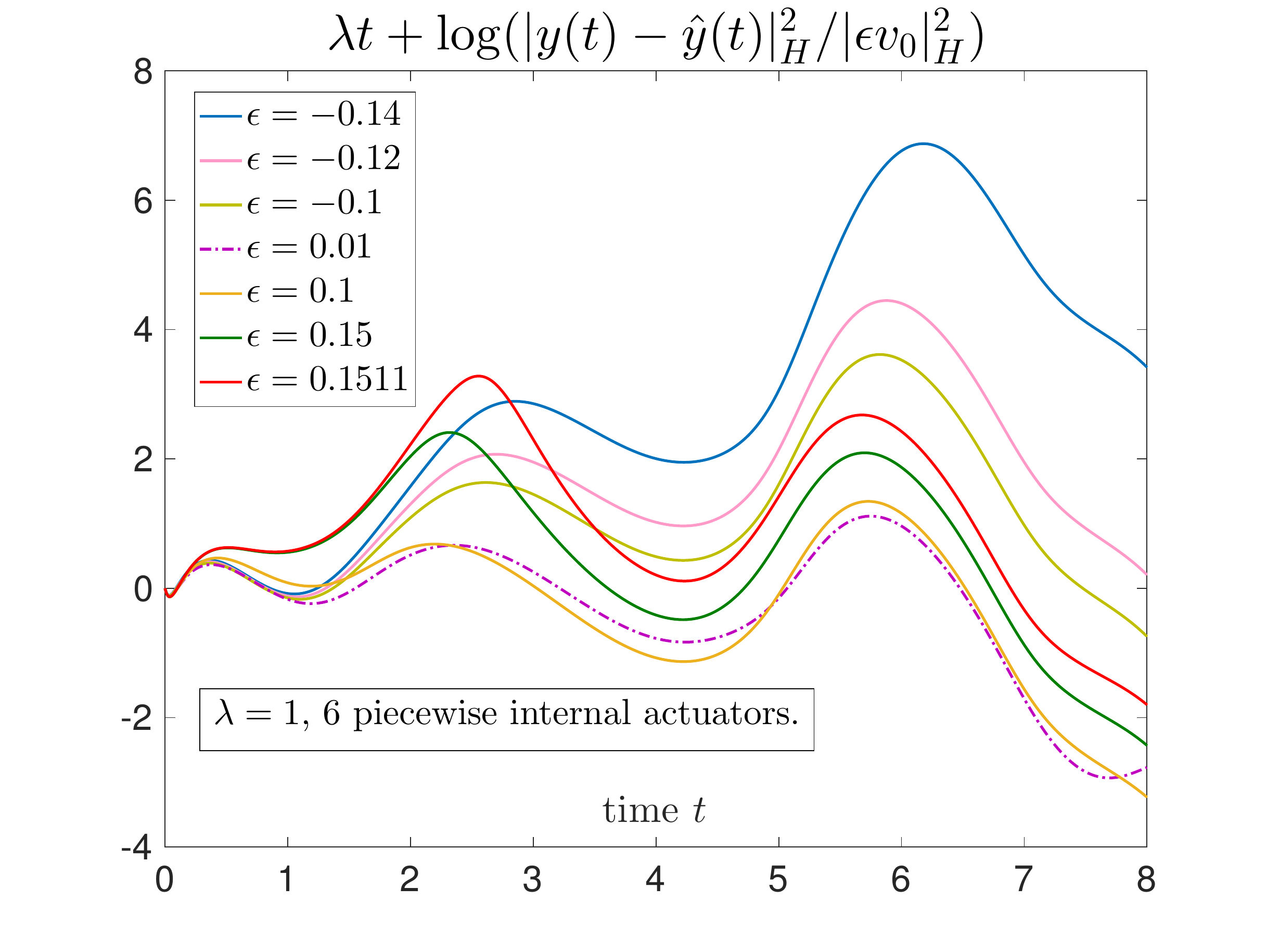}}
\subfigure[With $\epsilon \in \{ -0.142,~0.1512\}$.]
{\includegraphics[width=.325\linewidth]{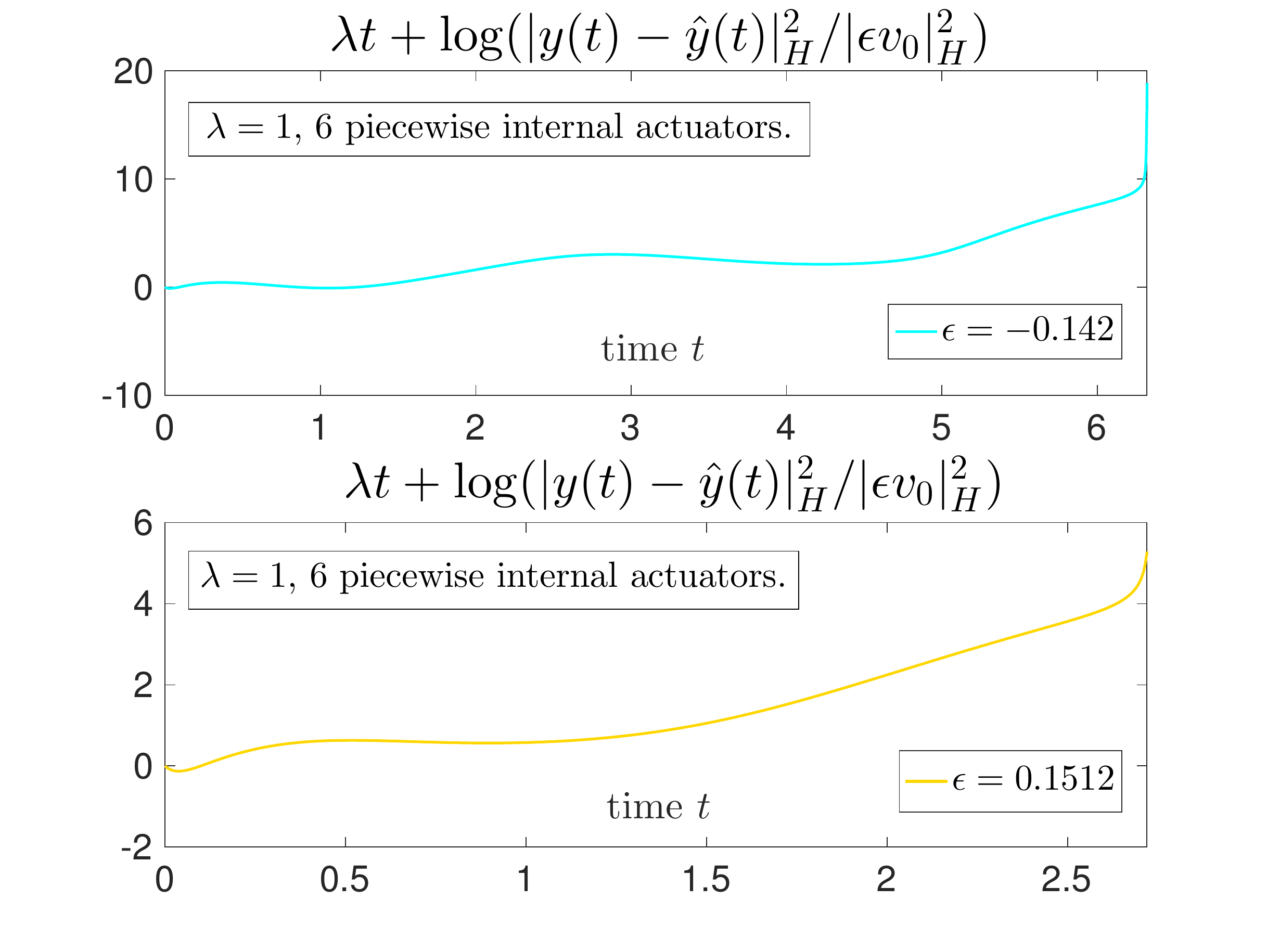}}
\caption{Internal feedback control. Convergence rate to $\hat{y}$ holds for small~$\epsilon$.}
\label{figure_InternalControlNL-CompabilityCondition}
\end{figure}

\begin{figure}[ht]  
  \centering
\subfigure[With $\epsilon = 0.01$.]
{\includegraphics[width=.325\linewidth]{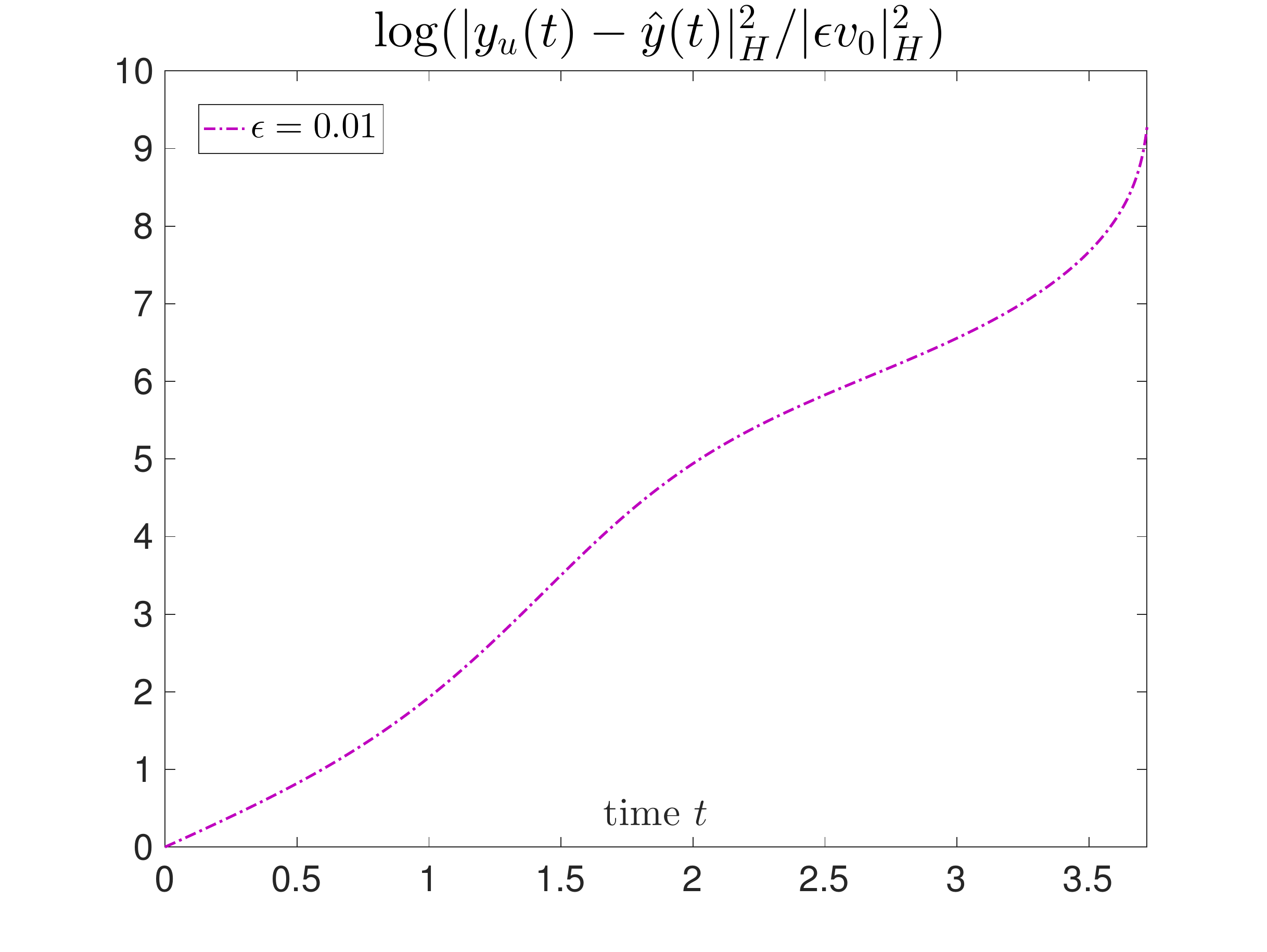}}
\subfigure[With $\epsilon = -0.14$.]
{\includegraphics[width=.325\linewidth]{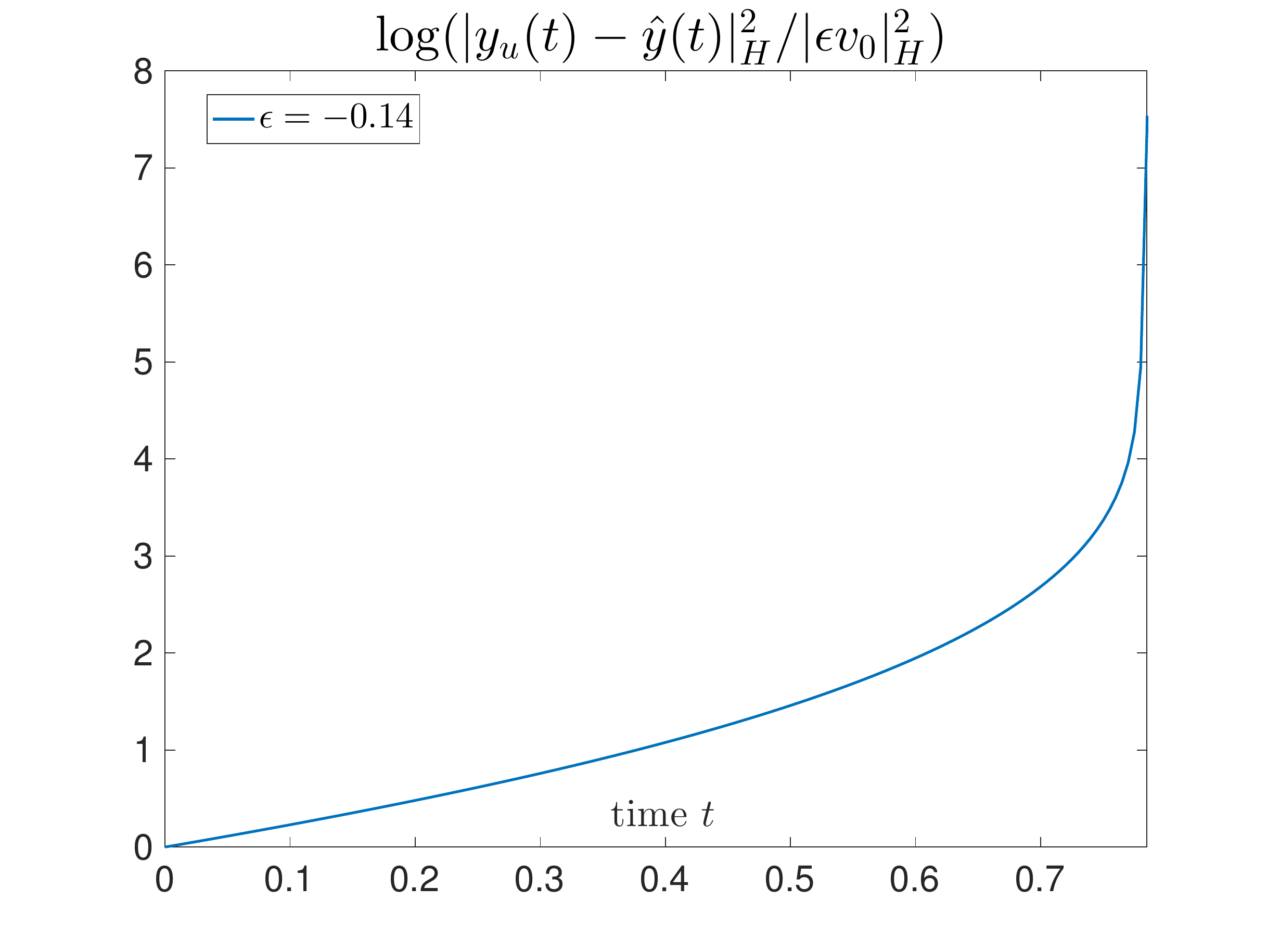}}
\subfigure[With $\epsilon = 0.1512$.]
{\includegraphics[width=.325\linewidth]{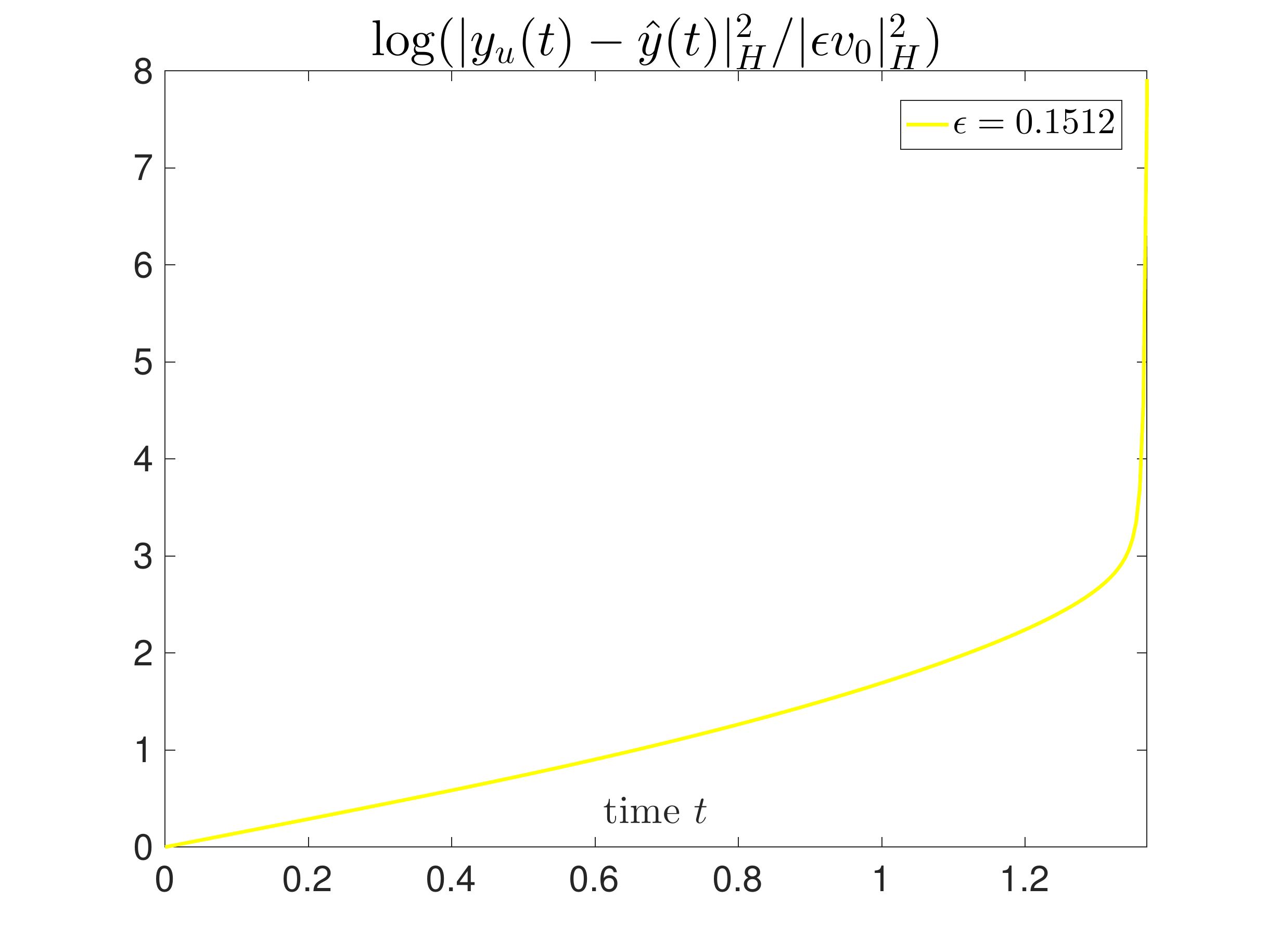}}
\caption{Uncontrolled solution.}
\label{figure_InternalControlNL-unc-CompabilityCondition}
\end{figure}

\subsubsection*{The case of boundary controls} 
Next, we use 6~boundary actuators as in~Section~\ref{ssS:bdrycontrols}, see~\eqref{sin_contbdry}.
Here in order to guarantee the compability condition~\eqref{CompatibilityCond}, we take $y_0 = \hat{y}_0+\epsilon v_0$ where $v_0$ is obtained by solving an elliptic equation 
\begin{align*}
-\mu \Delta v_0 + \beta_{\rm r} v_0 + \nabla \cdot \left( \beta_{\rm c} v_0  \right) + h = 0,\quad v_0\rest\Gamma = \sum_{i=1}^{M}{ \varrho_i \Psi_i},
\end{align*}
with~$\mu$, $\beta_{\rm r}$, $\beta_{\rm c}$, and $h$  as in~\eqref{param_nonl} and with $\varrho = \begin{bmatrix}
 1&1&0&0.5&0&0
\end{bmatrix}^{\top}.$
In Figure \ref{figure_BoundaryControlNL-feed}(a), we plot the function $v_0$ that we get by solving the system above.

Again, solving our system for different values of~$\epsilon$, in Figures~\ref{figure_BoundaryControlNL-feed}(b) we observe that under the boundary feedback control,
is stable for small~$\epsilon$. The feedback fails
to stabilize the system for bigger~$\epsilon$, as we see in Figure~\ref{figure_BoundaryControlNL-feed}(c). In this case we take~$\kappa^0=\kappa(0)=\epsilon\varrho$
in~\eqref{SBoundary-Dk}.

In Figure~\ref{figure_BoundaryControlNL-unc} we see that the uncontrolled solution is not stable, and it even explodes for small~$\epsilon$.
\begin{figure}[ht]
\centering
\subfigure[The function $v_0$.]
{\includegraphics[width=.325\linewidth]{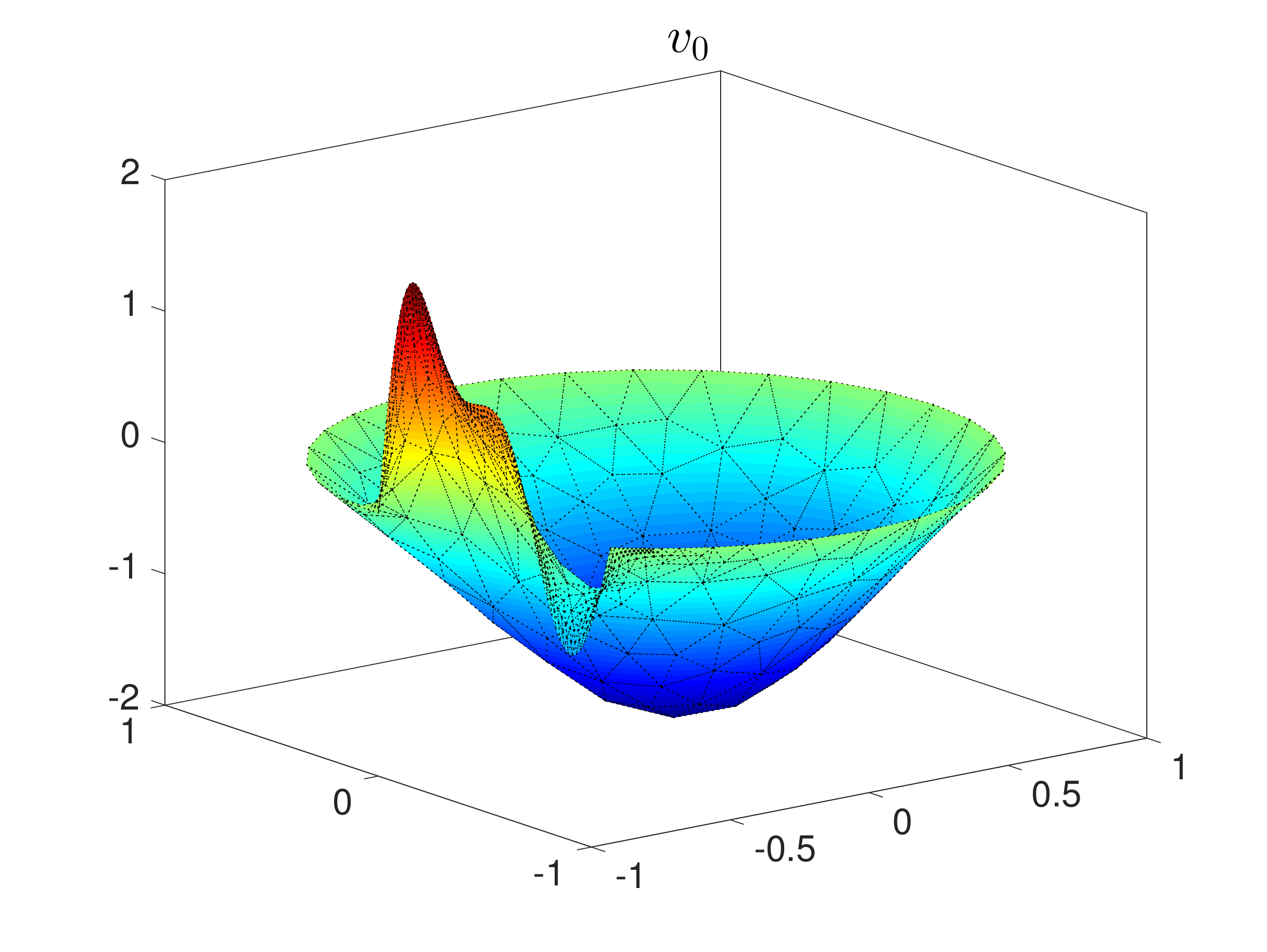}}
\subfigure[With $\epsilon\in{[}-0.008,0.041{]}$.]
{\includegraphics[width=.325\linewidth]{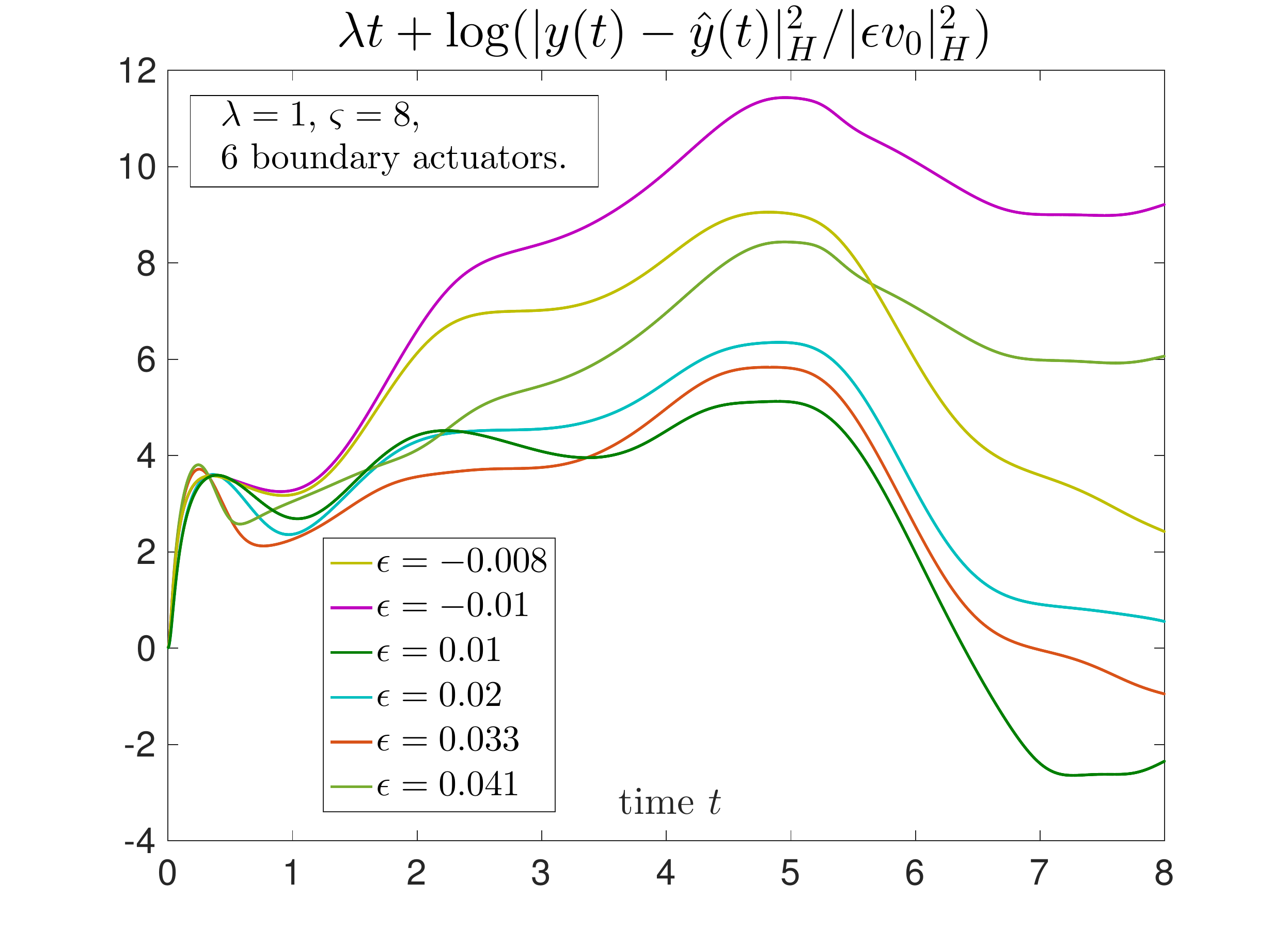}}
\subfigure[With $\epsilon \in \{ -0.011,~0.0429\}$.]
{\includegraphics[width=.325\linewidth]{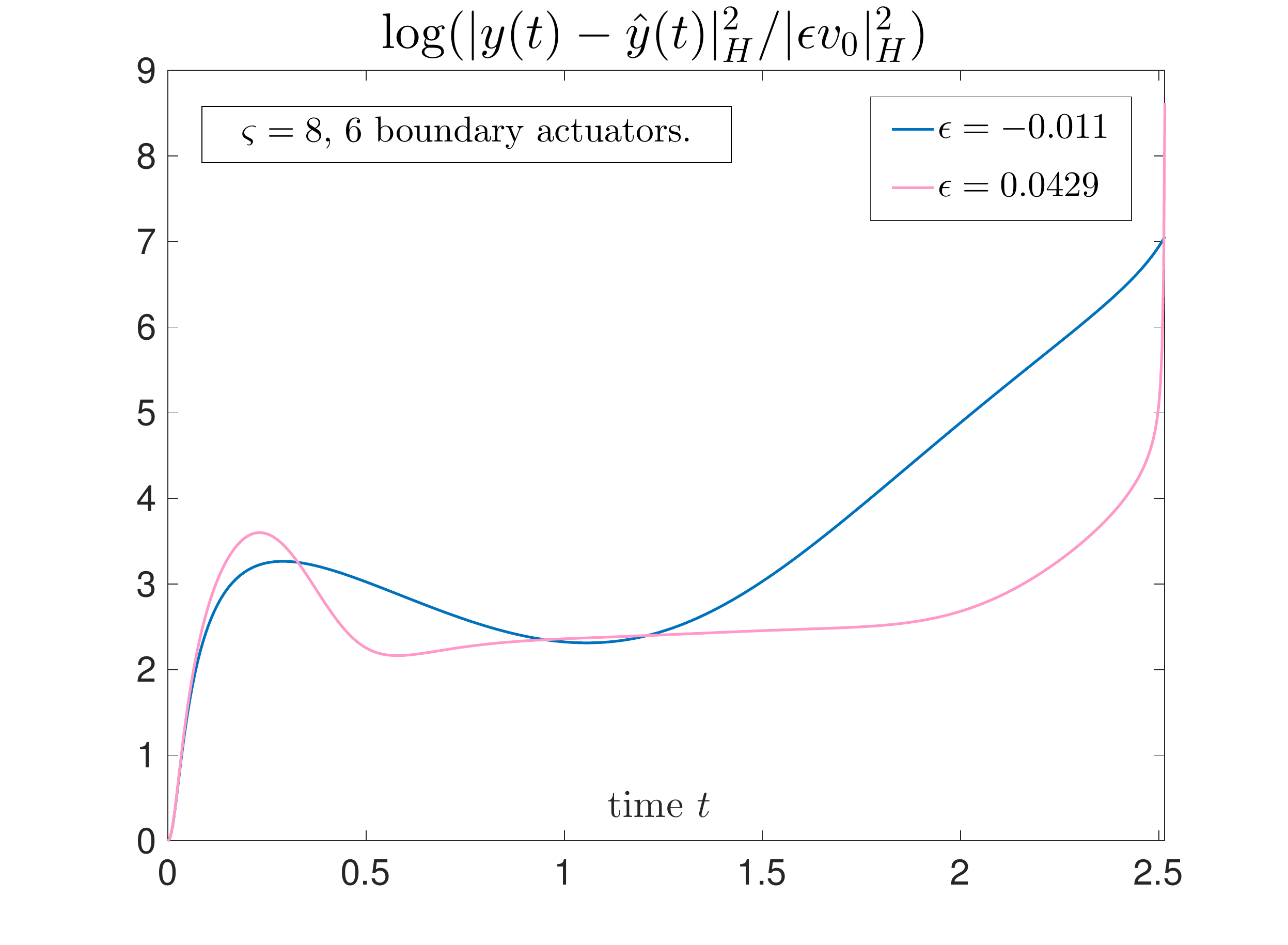}}
\caption{Boundary feedback control. Convergence rate to $\hat{y}$ holds for small~$\epsilon$.}
\label{figure_BoundaryControlNL-feed}
\end{figure}
\begin{figure}[ht]
\centering
\subfigure[With $\epsilon = -0.01$.]
{\includegraphics[width=.325\linewidth]{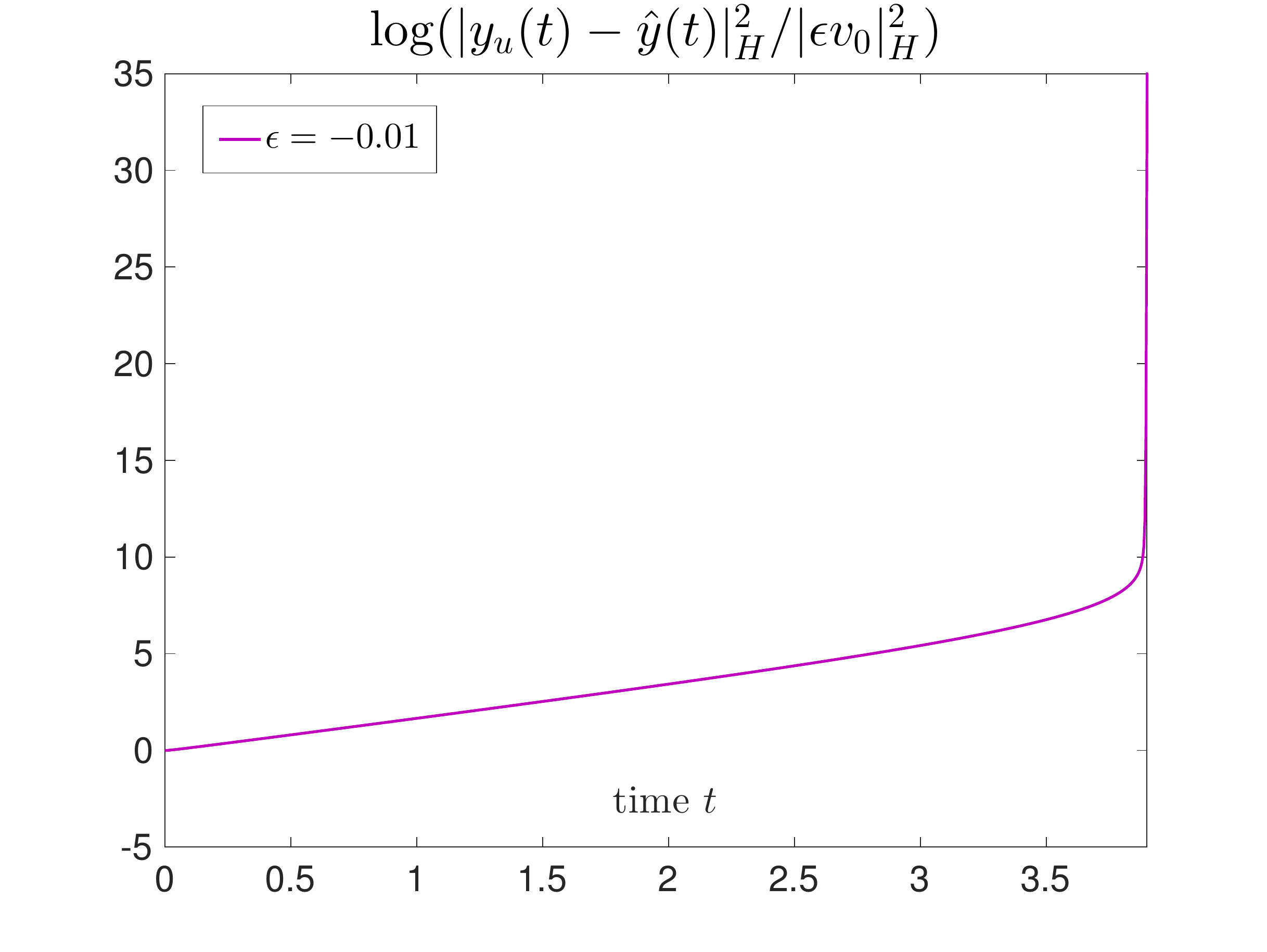}}
\subfigure[With $\epsilon = 0.01$.]
{\includegraphics[width=.325\linewidth]{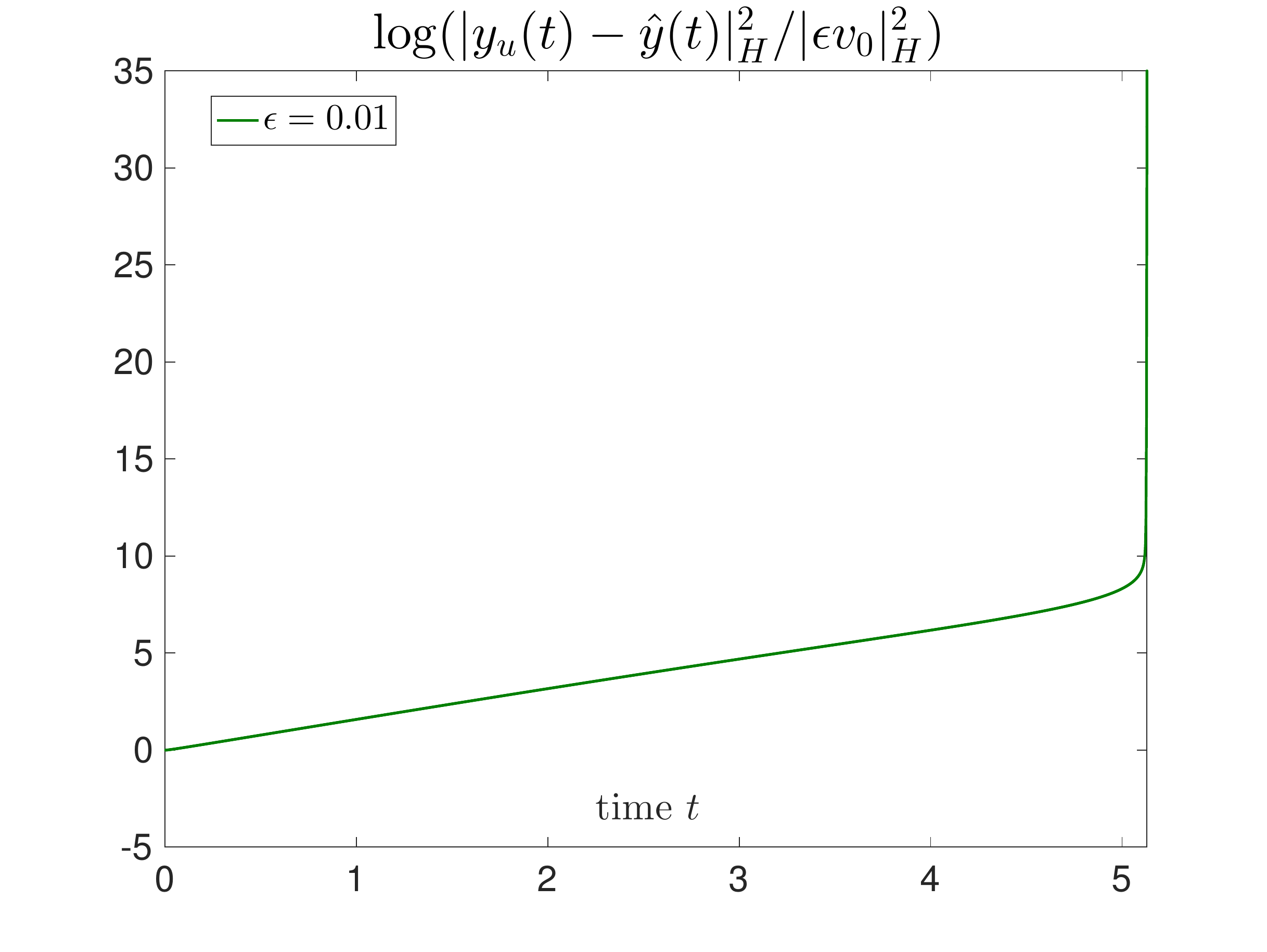}}
\subfigure[With $\epsilon = 0.0429$.]
{\includegraphics[width=.325\linewidth]{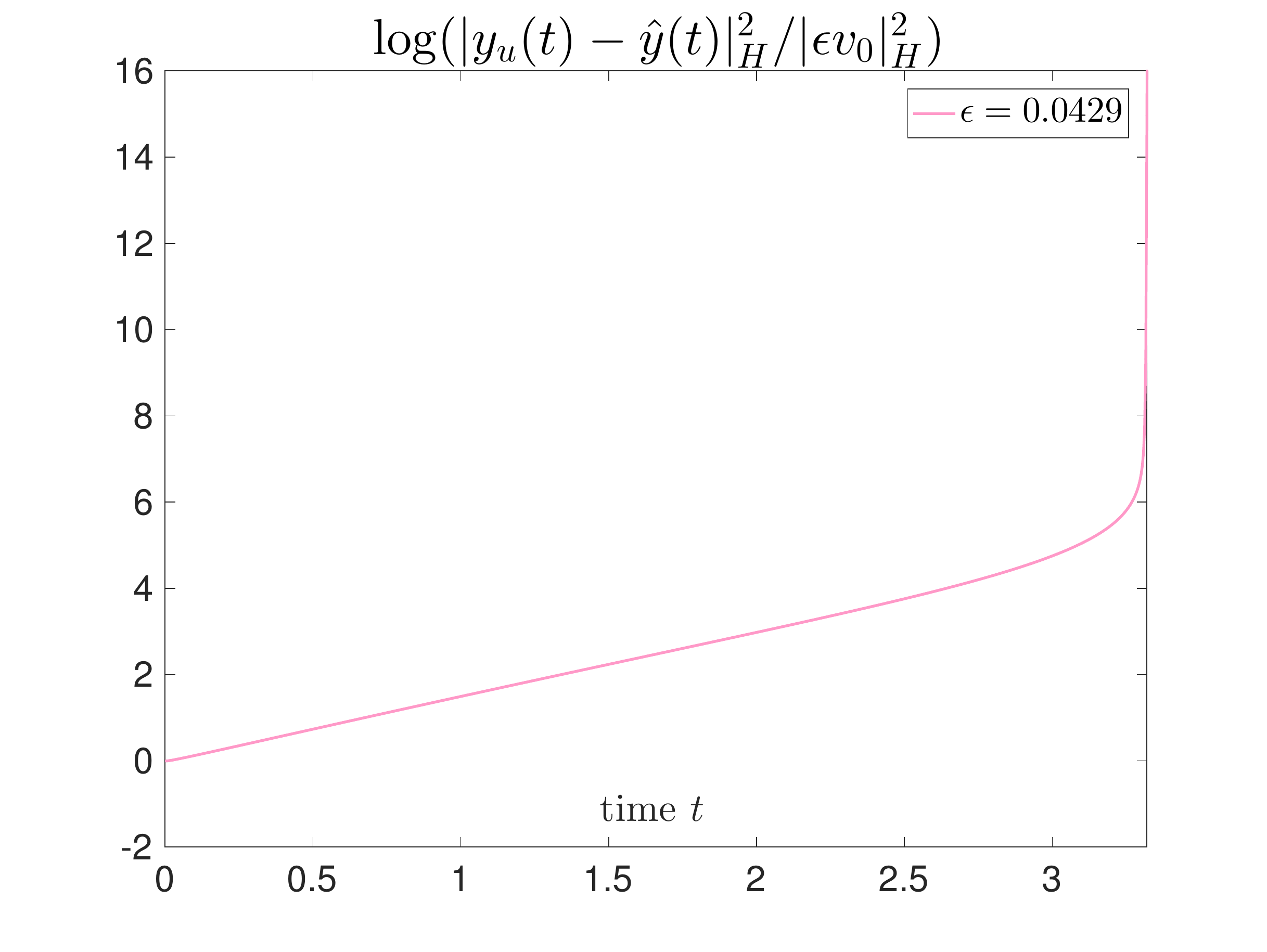}}
\caption{Uncontrolled solution.}
\label{figure_BoundaryControlNL-unc}
\end{figure}

\subsection{The discretization error} Here we take the time interval~$[0,5]$ and consider the function
\[
\hat{y}(x_1,x_2) = \left( t^2 - 2t \right) x_1^3 \sin^2 (x_2).
\]
Of course we expect the discrete solution vector~$\bar y$ to get closer to $\bar{\hat{y}}$ as the triangular mesh in~$\mathbb D$ gets finer and the timestep gets smaller.
Here we show, just through a simulation that this is the case for the discretization we propose.

We check the error that we obtain as the mesh pair~$(\DDD,k)$ is refined.
We start with a pair~$(\DDD_1,k_1)$ where~$\DDD_1$ is a triangular mesh of the cylinder~$\mathbb D$,
where each edge of each triangle in the mesh has a length bounded above by~$h_1 = 0.6$,
and~$k_1 = 0.01$ is the time step.  More precisely~$\DDD_1$ was generated
by the MATLAB function~{\tt initmesh} with the input~${\tt Hmax}= 0.6$.

Recursively, we construct the finer pair~$(\DDD_{r+1},k_{r+1})$ by refining
regularly the triangulation~$\DDD_{r}$ (by connecting the middle points of the edges of each triangle),
and by dividing the time-step by~$2$, $k_{r+1}=\frac{k_{r}}{2}$.
Hence, we want the solution~$y[r]$ obtained with~$(\DDD_r,k_r)$ to converge to~$\hat{y}$ as~$r$ increases.

We will also compare our approach with a Newton and a Heun based approach. The methods differ in the way we solve~\eqref{SNonlin-CN}.

\begin{remark}
We follow~\cite[Section~7.2.2]{Rosloniec08} and~\cite[Section~4.1.2]{McDonough07}
for the terminology ``Heun approach'' or ``Heun method''. However, in different references the terminology may vary, for example we find
``Modified Euler method'' in~\cite[Section~5.4]{BurdenFaires10},
or ``Explicit Trapezoidal method'' in~\cite[Section~4.1]{AscherPetzold98}. 
\end{remark}

\subsubsection*{The Heun based approach} Here instead of the extrapolation~\eqref{SNonlin-ourext}
it is used an explicit Euler step, see~\cite{McDonough07}, to find a preliminary guess $y^{*}_{\rm i}$ for~$y^{j+1}_{\rm i}$ as
\begin{align} \label{Heun-GuessbyEuler}
\begin{split}
\mathbf M_{\rm ii} \overline y^{*}_{\rm i} &=\left(\mathbf M_{\rm ii} - k \nu \mathbf{S}_{\rm ii} \right) \overline y^{j}_{\rm i}  +  \left( \mathbf M_{\rm ib} - k \nu \mathbf{S}_{\rm ib} \right)\overline g^{j} \\
												&\quad - \mathbf M_{\rm ib} \overline g^{j+1}  -k \begin{bmatrix}\Ma_{\rm ii} &\Ma_{\rm ib}\end{bmatrix} \overline f_0^{j} -k\left(\NN_{1,D}(\overline y^{j})\right)_{\rm}
\end{split}												
\end{align}
By defining $\overline y_{\mathbf G} = \begin{bmatrix}
\overline y^{*}_{\rm i} \\   \overline g^{j} 
\end{bmatrix}$,  we arrive to the scheme
\begin{equation*}
\begin{split}
\quad\mathbf A_{\rm ii}^\oplus\overline y_{\rm i}^{j+1}
&= \mathbf A_{\rm ii}^\ominus\overline y_{\rm i}^j -\mathbf A_{\rm ib}^\oplus \overline g^{j+1}
+\mathbf A_{\rm ib}^\ominus\overline g^{j}-k\begin{bmatrix}\Ma_{\rm ii} &\Ma_{\rm ib}\end{bmatrix}
\left(\overline f_0^{j+1}+\overline f_0^{j}\right)\\
&\quad-k\left( \left(\NN_{1,D}(\overline y_{\mathbf G})\right)_{\rm i}+\left(\NN_{1,D}(\overline y^j)\right)_{\rm i}\right).
\end{split}
\end{equation*}

\subsubsection*{The Newton based approach} 
We solve~\eqref{SNonlin-CN}, by a fixed point iterative procedure, see~\cite[Section 10.2]{BurdenFaires10}.
We write~\eqref{SNonlin-CN} in the form
\[
 F(\overline y^{j+1}_{\rm i})=0,
\]
with
\[
\begin{split}
F(w)&\coloneqq -\mathbf A_{\rm ii}^\oplus w -k\left(\NN_{1,D}\left(\begin{bmatrix}w\\ \overline g^{j+1}\end{bmatrix}\right)\right)_{\rm i}
+ H^j,\\
H^j &\coloneqq  \mathbf A_{\rm ii}^\ominus\overline y_{\rm i}^j -\mathbf A_{\rm ib}^\oplus \overline g^{j+1}
+\mathbf A_{\rm ib}^\ominus\overline g^{j}
-k\begin{bmatrix}\Ma_{\rm ii} &\Ma_{\rm ib}\end{bmatrix}
\left(\overline f_0^{j+1}+\overline f_0^{j}\right)-k\NN_{1,D}(\overline y^{j})_{\rm i}.
\end{split}
\]
Next, we take the derivative of $F$ at a given vector~$w^0$:
\[
\begin{split}
\ed F_{w^0}&\coloneqq-\mathbf A_{\rm ii}^\oplus
-k\left(\Ma\left(3c_3\begin{bmatrix}w^0\\ \overline g^{j+1}\end{bmatrix}^2+2c_2\begin{bmatrix}w^0\\ \overline g^{j+1}\end{bmatrix}+c_1\right)
+\GGG_{x_1}\begin{bmatrix}w^0\\ \overline g^{j+1}\end{bmatrix}+\GGG_{x_2}\begin{bmatrix}w^0\\ \overline g^{j+1}\end{bmatrix}\right)_{\rm i}\\
&=A^N_{w_0}+B(\overline g^{j+1}),
\end{split}
\]
with
\[
\begin{split} 
A^N_{w_0}&\coloneqq-\mathbf A_{\rm ii}^\oplus
-k\left(\Ma_{\rm ii}\left(3c_3(w^0)^2+2c_2w^0+c_1\right)
+\GGG_{x_1,\rm ii}w^0+\GGG_{x_2,\rm ii}w^0\right),\\
B(\overline g^{j+1})&\coloneqq-k\left(\Ma_{\rm ib}\left(3c_3(\overline g^{j+1})^2+2c_2\overline g^{j+1}+c_1\right)
+\GGG_{x_1,\rm ib}\overline g^{j+1}+\GGG_{x_2,\rm ib}\overline g^{j+1}\right),
\end{split}
\]
where the powers of vectors are understood to be taken coordinate-wise.

Now we find~$\overline y^{j+1}_{\rm i}$ iteratively, see~\cite{GraagTapia74} or \cite[Section 10.2]{BurdenFaires10}, as (or close to) the limit of
\[
w^{n+1}= w^n - \left [ F'(w_n)  \right]^{-1} F(w_n).
\]
By a continuity argument, if the time-step is small enough, we can expect the method to converge when we take
the starting vector $w^0=\overline y^{j}_{\rm i}$. Another option, see~\cite[page 88]{McDonough07}, is to
take $w^0 = \overline y^{*}_{\rm i}$ as in \eqref{Heun-GuessbyEuler}, which may allow us to reduce the number of iterations.
We used the latter in the simulations
presented here.
The stopping criteria which we used was
\begin{enumerate}
\item[i.] $\left| w^{n+1} - w^{n} \right|_{L^{\infty}} < \rm tol $.
\item[ii.] $\frac{\left| w^{n+1} - w^{n} \right|_{L^{\infty}}}{\left|  w^{n} \right|_{L^{\infty}}} < \rm tol $.
\item[iii.] $ F(w^{n+1}) < \rm tol$.
\end{enumerate}
with tolerance~${\rm tol}={\tt eps}$, where~${\tt eps}\approx 10^{-16}$ is the MATLAB's epsilon (i.e., ``zero'') for double precision.
Figure~\ref{figure_Refinement}(a), 
shows the $L_2$-norm, $\norm{y[r]-\hat{y}}{L^2((0, 5 ),L^2(\Omega,\R))}$ in the cylinder, of the discretization error with $r \in \{ 1,2,3,4,5\}$.
The Bochner norm $\norm{v}{L^2((0, 5 ),L^2(\Omega,\R))}$ is to be understood as the discrete approximation
\begin{align*}
\left(\underline{|v|_H} \right)^{\top} \mathbf{M}_{\rm t} \underline{|v|_H}\approx\textstyle\int_0^{5} |v(t)|_H|v(t)|_H\,\ed t =\norm{v}{L^2((0, 5 ),L^2(\Omega,\R))}^2, 
\end{align*}
where $\underline{|v|_H}$ is a column vector with $N_t +1$ components where each component is the discrete
approximation of the norm $|v(jk)|_H$, that is,  $\left(\underline{|v|_H} \right)_j \coloneqq \sqrt{\left(\overline{v}^j\right)^{\top} \mathbf{M} \overline{v}^j}$;
and $\mathbf{M}_{\rm t}$ is the mass matrix associated with the regular time mesh as in~\eqref{TimeMesh}, with~$T=5$. That is, (cf.~\cite[Section 5.1]{KroRod15}) 
\begin{align*}
\mathbf{M}_{\rm t} \coloneqq \frac{k}{6}
\begin{bmatrix}
    2 & 1 & 0 & 0 & \dots  & 0 \\
    1 & 4 & 1 & 0 & \dots  & 0  \\
     0 & 1 & 4 & 1 & \ddots  & \vdots  \\
\vdots & \ddots &  \ddots &  \ddots &  \ddots & 0 \\
0 & \dots & 0 &1 & 4 & 1 \\
0 & \dots & 0 & 0 & 1 & 2 
\end{bmatrix}.
\end{align*}

In the Figure~\ref{figure_Refinement}(b) we can see that the rate of convergence approaches~$4$ (i.e., second order convergence) for all the three approaches.

It is clear that the cheapest method is the one which we propose, and that the Newton method is the more expensive one.
Of course the Newton based approach is expected to be the most
accurate. However, we observe in the Figure~\ref{figure_Refinement} that, for this example,
the results of these two approaches do seem to match each other.

To find a difference between these two approaches, we have to take a bigger time step. We also take the bigger time interval $[0,6]$ and $h = 0.1$.
In Figure~\ref{figure_NewtonvsExtravsHeun}(a) and \ref{figure_NewtonvsExtravsHeun}(b), we can see that
with only~$60$ time nodes, the Newton approach gives already almost the same result as
with~$240$ time nodes. While, we can clearly see that there is an error associated with the time step for the approach that we propose using a linear extrapolation.
Notice, however that the approach we use can
be less expensive for~$240$ time nodes than the Newton approach with~$60$ time nodes,
depending on the number of Newton iterations at each time step, which is expected to increase with the time step (in our experiment with ~$60$ time nodes the average
number of iterations was $5.233$).

Another advantage of our method is
that we can invert the system~\eqref{SNonlin-D} iteratively, as we have done in our simulations by using the conjugate gradient method.
With the Newton approach the associated linearization matrices, at a given time $t=jk$,
may be not symmetric and so we may need to use a direct solver, for this in our simulations we have used the backslash ``$\backslash$'' solver from MATLAB.

The main disadvantage of the Heun method is the fact that it uses an extrapolation based on an explicit Euler guess for the solution~$y^{j+1}$ at time~$t=(j+1)k$,
which may lead to some oscillations/fluctuations in time if the time step is not small enough. With~$60$ time nodes, we observed that Heun approach fails, that is, the numerical
solution has exploded.
As we increase the number of time nodes to $240$, in Figure \ref{figure_NewtonvsExtravsHeun}(b) we can see that
the solution obtained by Heun approach is still the worst one. 
In Figure \ref{figure_NewtonvsExtravsHeun}(c), with a large enough number of time nodes, more precisely $600$, we cannot see a
remarkable difference among the three approaches.  
\begin{figure}[!]  
\centering
\subfigure[$L_2$-norm of the error.]
{\includegraphics[width=.495\linewidth]{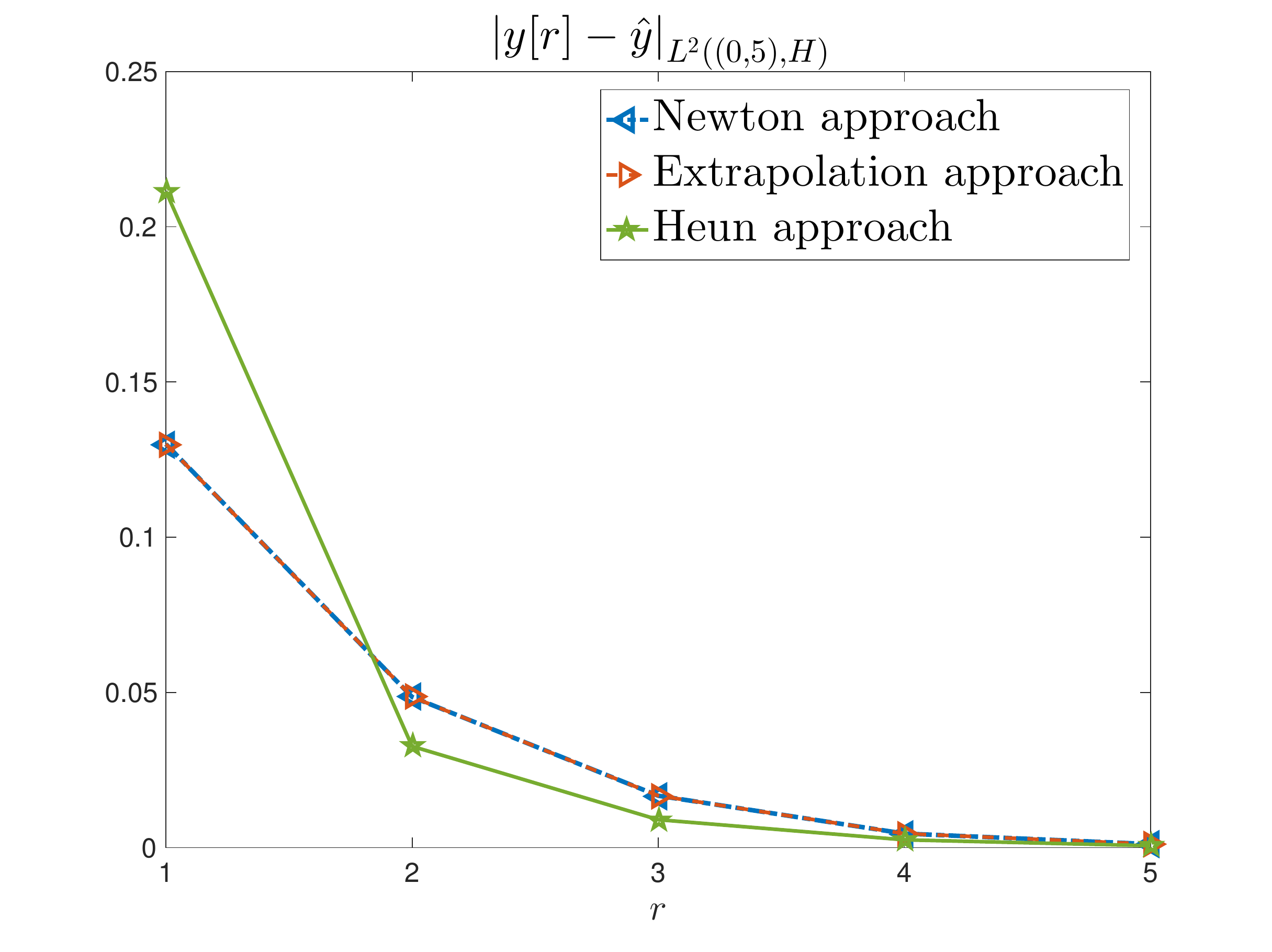}}
\subfigure[Ratio between two consecutive errors.]
{\includegraphics[width=.495\linewidth]{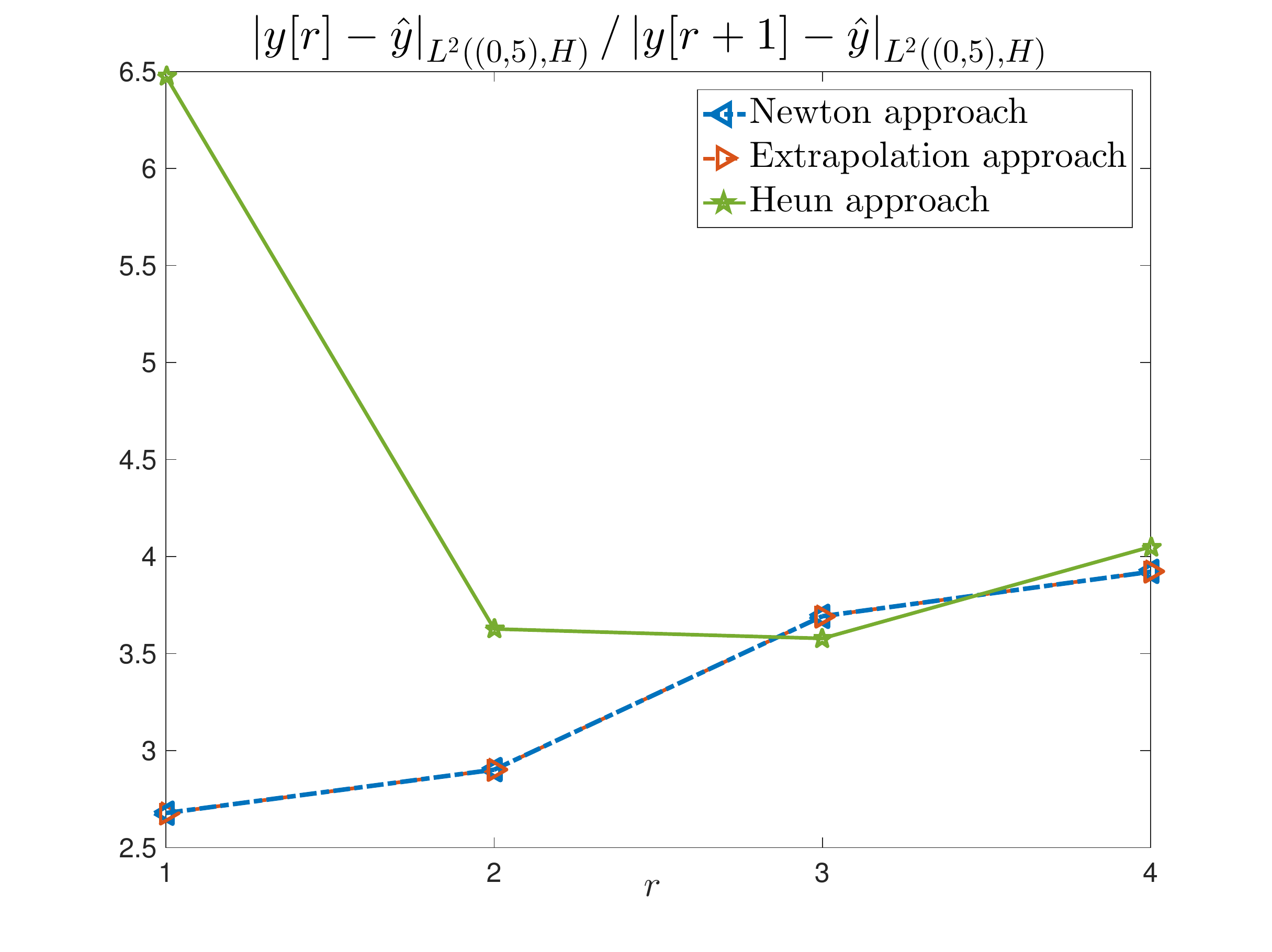}}
\caption{The discretization error for the mesh pairs~$(\DDD_r,k_r)$. }
\label{figure_Refinement}
\end{figure}
\begin{figure}[!]  
\centering
\subfigure[With 60 time nodes]
{\includegraphics[width=.325\linewidth]{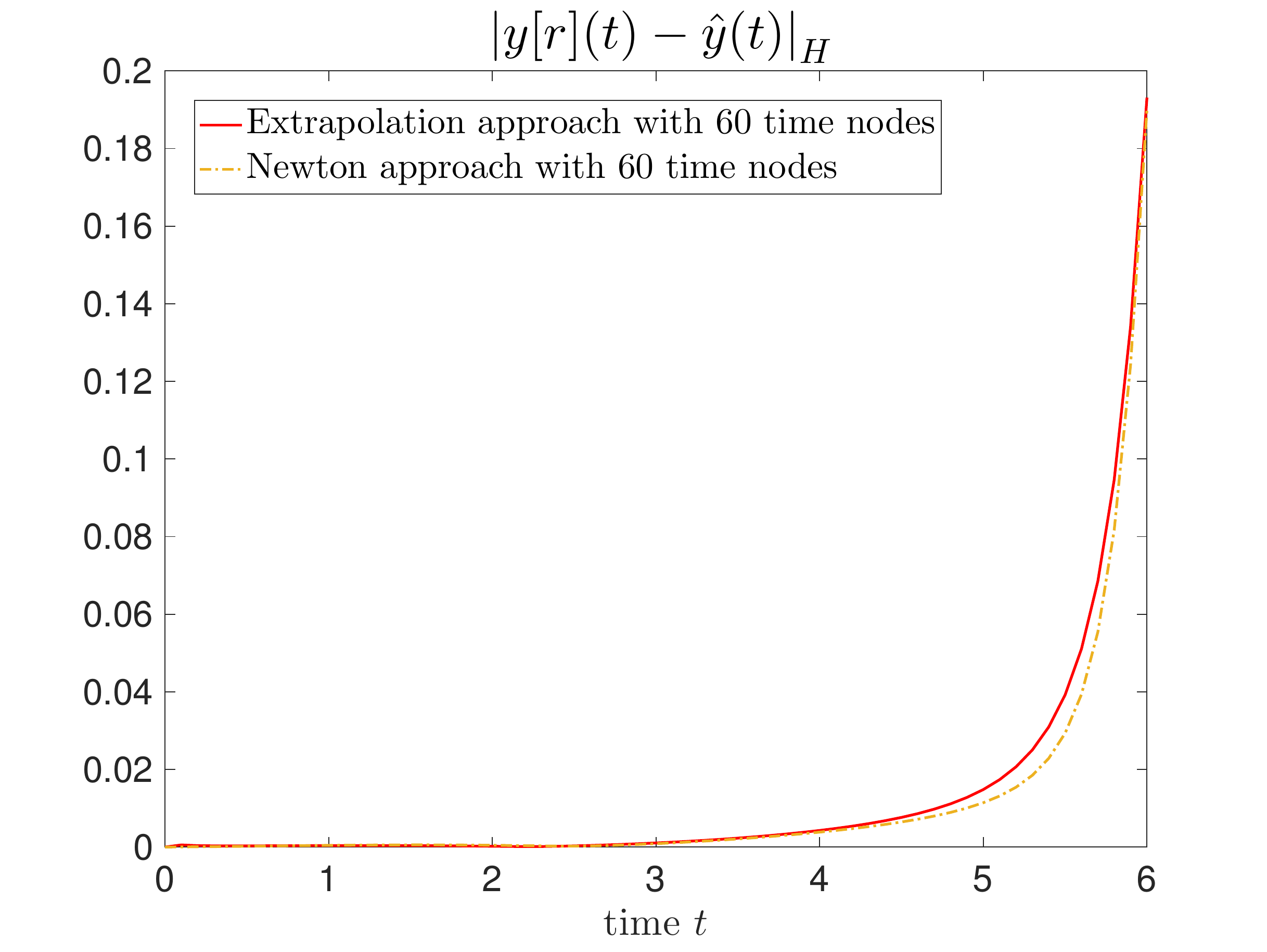}}
\subfigure[With 240 time nodes.]
{\includegraphics[width=.325\linewidth]{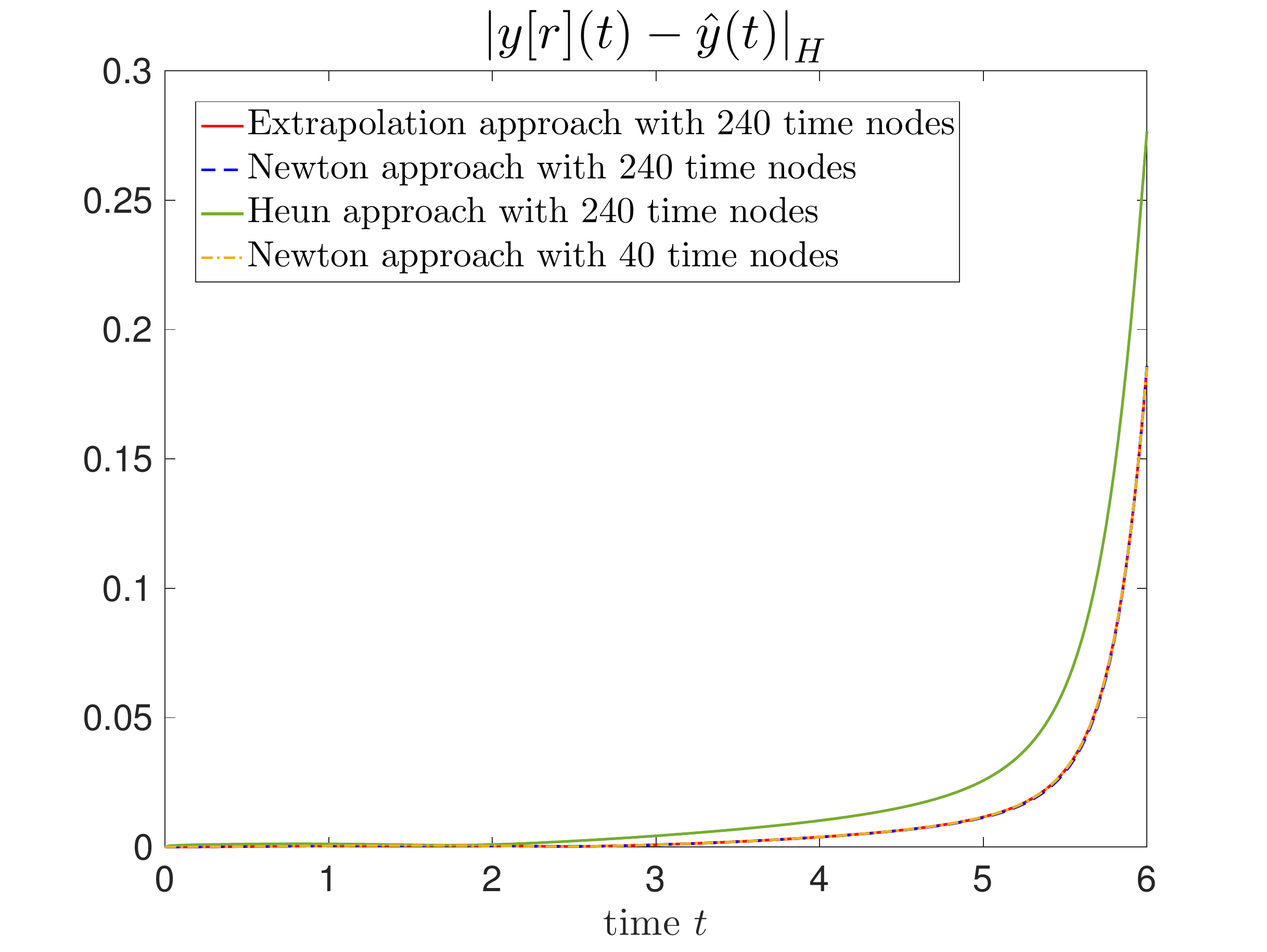}}
\subfigure[With 600 time nodes.]
{\includegraphics[width=.325\linewidth]{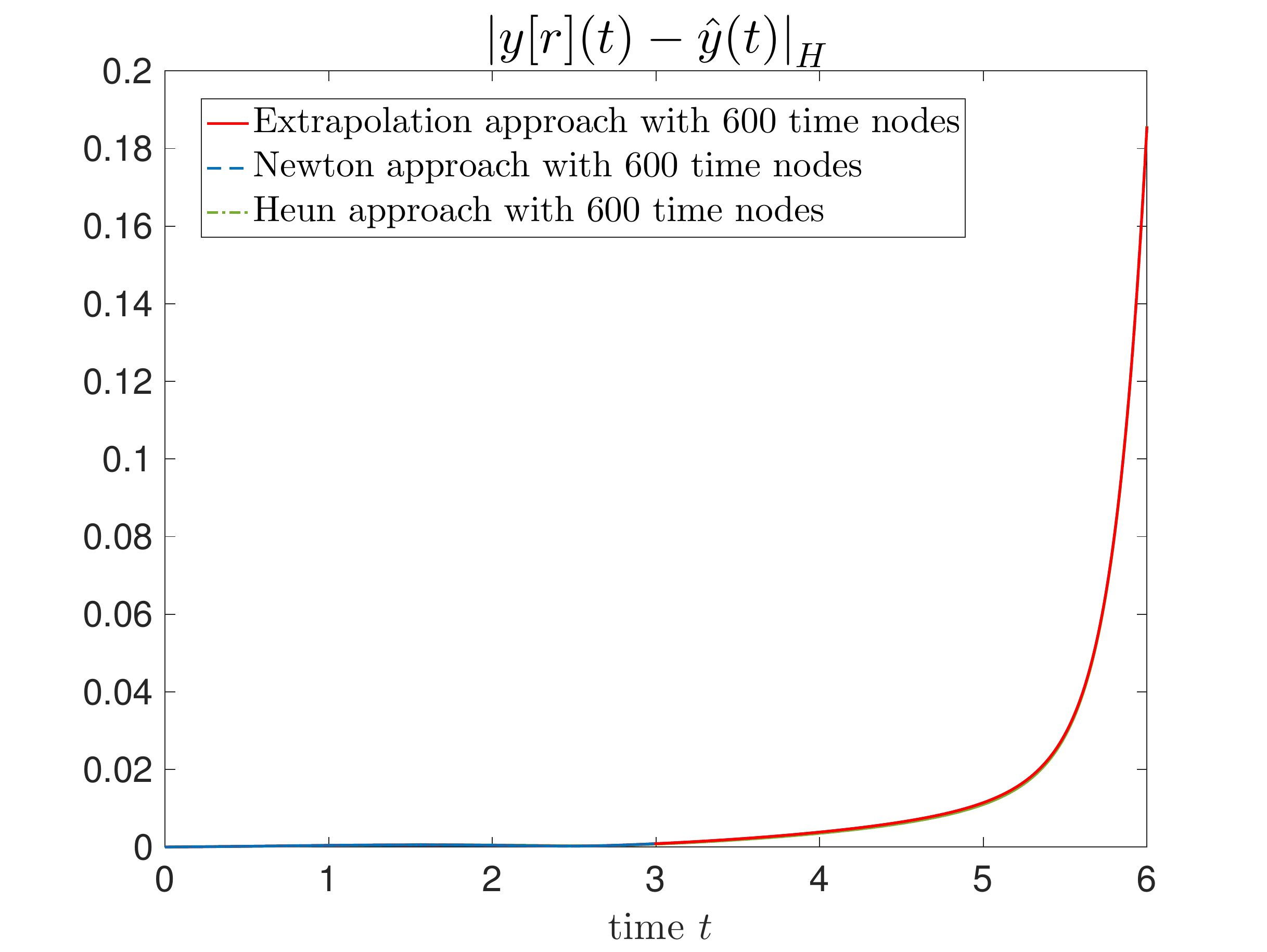}}
\caption{Comparing the three approaches for different time steps. }
\label{figure_NewtonvsExtravsHeun}
\end{figure}

\end{document}